\documentclass[aip,cha,amsmath,amssymb,reprint]{revtex4-1}

\usepackage{bm,calc,capt-of,color,dcolumn,docs,enumerate,euscript,graphicx}
\usepackage{ifthen,mathrsfs,multirow,theorem}
\DeclareGraphicsExtensions{.pdf}
\newcommand{\FixedSet}{\tmop{Fix}}
\newcommand{\Identity}{\mathrm{Id}}
\newcommand{\rmn}{\mathrm{n}}
\newcommand{\rmt}{\mathrm{t}}

\newcommand{\Qset}{\mathbb{Q}}
\newcommand{\Rset}{\mathbb{R}}
\newcommand{\Sset}{\mathbb{S}}
\newcommand{\Zset}{\mathbb{Z}}
\newcommand{\tmfloatcontents}{}
\newlength{\tmfloatwidth}
\newcommand{\tmfloat}[5]{
  \renewcommand{\tmfloatcontents}{#4}
  \setlength{\tmfloatwidth}{\widthof{\tmfloatcontents}+1in}
  \ifthenelse{\equal{#2}{small}}
    {\ifthenelse{\lengthtest{\tmfloatwidth > \linewidth}}
      {\setlength{\tmfloatwidth}{\linewidth}}{}}
    {\setlength{\tmfloatwidth}{\linewidth}}
  \begin{minipage}[#1]{\tmfloatwidth}
    \begin{center}
      \tmfloatcontents
      \captionof{#3}{#5}
    \end{center}
  \end{minipage}}
\newcommand{\tmop}[1]{\ensuremath{\operatorname{#1}}}
\newcommand{\tmtextit}[1]{{\itshape{#1}}}

\newcommand{\tmtextup}[1]{{\upshape{#1}}}
\newcommand{\DS}{\displaystyle}

\newtheorem{corollary}{Corollary}
\newtheorem{definition}{Definition}
\newtheorem{lemma}{Lemma}
\newenvironment{proof}{\noindent\textbf{Proof\ }}{\hspace*{\fill}$\Box$\medskip}
\newtheorem{proposition}{Proposition}
{\theorembodyfont{\rmfamily}\newtheorem{remark}{Remark}}
\newtheorem{theorem}{Theorem}
\newenvironment{descriptioncompact}{\begin{description} }{\end{description}}
\newenvironment{enumeratealpha}{\begin{enumerate}[a{\textup{)}}] }{\end{enumerate}}
\newenvironment{enumerateroman}{\begin{enumerate}[i.] }{\end{enumerate}}
\newenvironment{itemizedot}{\begin{itemize} }{\end{itemize}}

\definecolor{dark green}{RGB}{0,100,0}

\begin{document}

\title[Classification of SPTs in billiards]
      {Classification of symmetric periodic trajectories in ellipsoidal billiards}

\author{Pablo S. Casas}
\email{\texttt{pablo@casas.upc.es}.}
\author{Rafael Ram\'{\i}rez-Ros}
\email{\texttt{Rafael.Ramirez@upc.edu}.}
\affiliation{Departament de Matem\`atica Aplicada I,
             Universitat Polit\`ecnica de Catalunya,\\
             Diagonal 647, 08028 Barcelona, Spain}

\date{\today}

\begin{abstract}
We classify nonsingular symmetric periodic trajectories
(SPTs) of billiards inside ellipsoids of $\Rset^{n+1}$
without any symmetry of revolution.
SPTs are defined as periodic trajectories passing through some symmetry set.
We prove that there are exactly $2^{2n} (2^{n+1} - 1)$ classes of
such trajectories.
We have implemented an algorithm to find minimal SPTs of each of
the 12 classes in the 2D case ($\Rset^2$) and each of the 112 classes
in the 3D case ($\Rset^3$).
They have periods 3, 4 or 6 in the 2D case;
and 4, 5, 6, 8 or 10 in the 3D case.
We display a selection of 3D minimal SPTs.
Some of them have properties that cannot take place in the 2D case.
\end{abstract}

\maketitle

\begin{quotation}
Smooth convex billiards are a paradigm of conservative dynamics, in
which a particle collides with a fixed closed smooth convex hypersurface of
$\bm{\Rset^{n + 1}}$. They provide examples of different dynamics:
integrable, mostly regular, chaotic, etc. In this paper we tackle the
integrable situation. Concretely, we find and classify symmetric periodic
trajectories (SPTs) inside ellipsoids of $\bm{\Rset^{n + 1}}$. PTs show
different dynamics, which describe how they fold in $\bm{\Rset^{n + 1}}$. STs
present symmetry with respect to a coordinate subspace of $\bm{\Rset^{n + 1}}$.
Dynamics and symmetry are precisely the main aspects we consider in our
classification of SPTs. We establish 112 classes of SPTs in the 3D case, and
we find a representative of each class with the smallest possible period.
Those minimal SPTs have periods 4, 5, 6, 8, or 10. We depict a selection of
minimal 3D SPTs. Some of them have properties that cannot take place in the 2D
case. SPTs are preserved under symmetric deformations of the ellipsoid. In a
future paper we plan to study their bifurcations and the transition between
stability and instability under such deformations.
\end{quotation}

\section{Introduction}

Smooth convex billiards, in which a point particle moves uniformly until it
undergoes abrupt elastic collisions with a smooth strictly convex hypersurface $Q
\subset \Rset^{n + 1}$, are one of the typical examples of conservative
systems with both regular and chaotic dynamics. The 2D case ($n = 1$) was introduced
by Birkhoff\cite{Birkhoff1966} and there is an extensive literature about it. The
high-dimensional case remains much less studied, although we highlight two known
results. First,
there exists a Nekhoroshev-like theorem for billiards.~\cite{GramchevPopov1995}
Second, there are some general lower bounds on the number of periodic billiard
trajectories.~\cite{Babenko1992,Farber2002a,Farber2002b,FarberTabachnikov2002}
Of course, these lower bounds are not informative for integrable systems,
where periodic trajectories are organized in continuous families.

\begin{figure*}
  \resizebox{2.07\columnwidth}{!}{\includegraphics{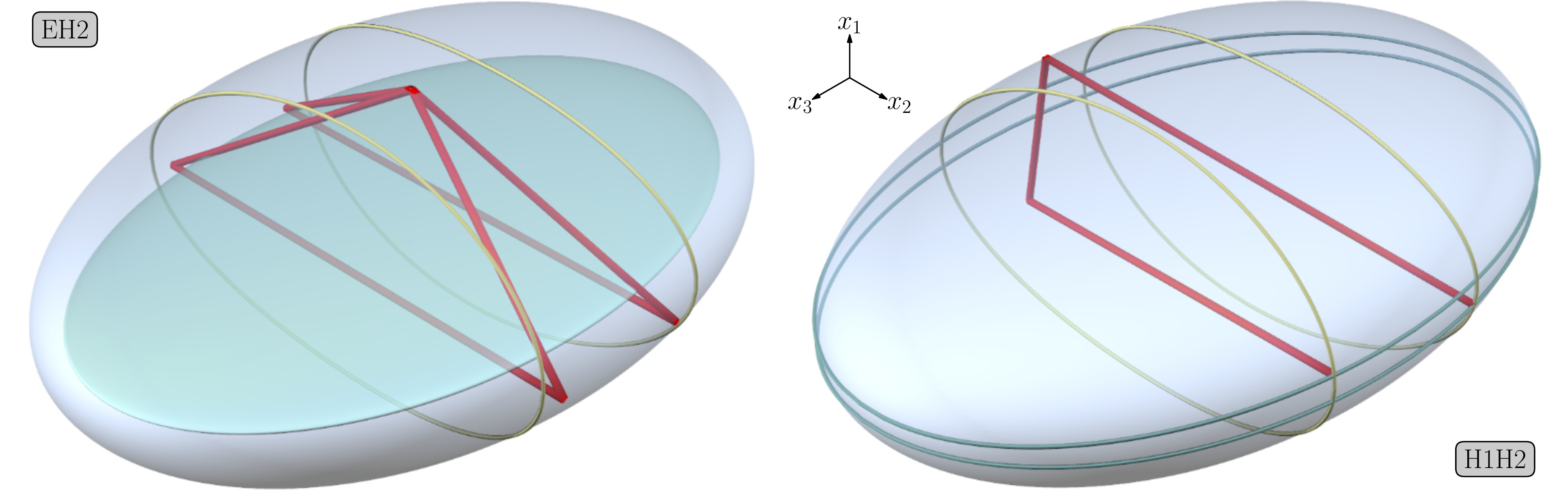}}
  \caption{\label{fig:TwoAmazingSPTs_3D}Two SPTs of period 6, but with only 5
  and 4 different impact points. Lines in green and yellow represent the
  intersections of the ellipsoid with H1-caustics and H2-caustics,
  respectively. The E-caustic is also shown in the EH2 figure.}
\end{figure*}

Ellipsoids are the only known examples of smooth strictly convex
closed hypersurfaces of the Euclidean space $\Rset^{n +1}$
that give rise to integrable billiards.
There is a larger class of integrable billiards,
called \emph{Liouville billiard tables},
but in none of the new integrable billiards the point particle follows
straight lines between two consecutive impacts;
instead it moves on the geodesics of some smooth compact
connected Riemannian manifold with
boundary.\cite{PopovTopalov2003,PopovTopalov2008,PopovTopalov2011}

In fact, an old conjecture atributed to Birkhoff states:
ellipsoids are the only smooth strictly convex closed hypersurfaces
of $\Rset^{n+1}$ that give rise to integrable billiards.
The conjecture remains still unproven,
but Berger\cite{Berger1995} and Gruber\cite{Gruber1995} proved that only
ellipsoids have smooth caustics when $n \geq 2$,
which is a powerful argument for the conjecture. A smooth \textit{caustic} is a
smooth hypersurface with the property that a billiard trajectory, once tangent
to it, stays tangent after every reflection.

Periodic trajectories (PTs) are the most distinctive trajectories, so their study is
the first task. There exist two remarkable results about PTs inside ellipsoids:
\textit{Poncelet theorem} and \textit{Cayley conditions}.
Poncelet\cite{Poncelet1822} showed that if a billiard trajectory inside an ellipse is
periodic, then all the trajectories sharing its conic caustic are also periodic.
Cayley\cite{Cayley1854} gave some algebraic conditions for determining conic
caustics whose trajectories are periodic. Later on, both results were generalized to
any dimension.~\cite{ChangFriedberg1988,DragovicRadnovic1998b}

Almost all papers on billiards inside ellipsoids follow an algebro-geometric
approach to the problem.\cite{Veselov1988,ChangFriedberg1988,MoserVeselov1991,
Chang_etal1993a,Chang_etal1993b,DragovicRadnovic1998a,DragovicRadnovic1998b,Fedorov1999, 
DragovicRadnovic2006,AbendaFedorov2006}
A more dynamical approach to billiards inside ellipsoids has recently been
considered in Refs.~\onlinecite{Waalkens_etal1999,WaalkensDullin2002,
DragovicRadnovic2009,CasasRamirez2011}.
We follow that approach, in which the description of the
dynamics is more important than the obtaining of explicit formulae. Indeed, this
paper is a natural continuation of the study on (minimal) periodic trajectories
carried out in Ref.~\onlinecite{CasasRamirez2011}. In the framework of smooth
convex billiards the minimal period is always two. Nevertheless, since all
two-periodic billiard trajectories inside ellipsoids are singular ---in the
sense that some of their caustics are singular elements of a family of confocal
quadrics---, two questions arise: which is the minimal period among
nonsingular billiard trajectories? which ellipsoids display such trajectories?
The authors settled both questions in Ref.~\onlinecite{CasasRamirez2011}.

We say that an ellipsoid of $\Rset^{n+1}$ whose $n+1$ axis have
different lengths (i.e., an ellipsoid without any symmetry of revolution)
is \textit{nondegenerate}.
Our current goal is to classify nonsingular symmetric
periodic trajectories (SPTs) inside nondegenerate
ellipsoids.
We also look for minimal SPTs
---representatives of each class with the smallest possible period.

The relevance of the SPTs relies in three facts.
First, they persist under small enough symmetric perturbations of
a nondegenerate ellipsoid.
Second, each perturbed SPT can be computed by solving
a nonlinear equation for which the corresponding unperturbed SPT
is a good initial approximation.
Third, this nonlinear equation is $n$-dimensional, not $2n$-dimensional.
That is, the dimension of the problem can be halved.
These properties greatly simplify the numerical computation of SPTs
inside symmetrically perturbed ellipsoids, which is our next aim.
To be more precise,
we plan to study the stability and bifurcations of perturbed SPTs,
as a way to gain some insight into the transition from integrability
to chaos in high-dimensional smooth convex billiard problems.

Roughly speaking, a symmetric trajectory (ST) inside
\[ Q = \left\{ \left( x_1, \ldots, x_{n + 1} \right) \in \Rset^{n + 1} :
   \frac{x_1^2}{a_1} + \cdots + \frac{x_{n + 1}^2}{a_{n + 1}} = 1 \right\}, \]
is a billiard trajectory invariant under a orthogonal reflection,
with respect to some of the $2^{n+1} - 1$ coordinate proper subspaces
of the Euclidean space $\Rset^{n+1}$.
For instance, the left trajectory in Fig.~\ref{fig:TwoAmazingSPTs_3D}
is invariant under the $x_1$-axial and
$x_1 x_3$-specular reflections, whereas the right one is invariant under the
$x_2$-axial reflection. There are twice as many ``billiard symmetries'' as
coordinate proper subspaces of $\Rset^{n + 1}$. This is easily
understood by means of the following example. We can associate two different
kinds of SPTs to any coordinate axis of $\Rset^{n + 1}$; namely, the
ones with an impact point on the axis and the ones with a segment parallel to
the axis. Note that the left SPT in Fig.~\ref{fig:TwoAmazingSPTs_3D} has an
impact point on the $x_1$-axis ---and so, it is invariant under the
$x_1$-axial reflection---, but also has a segment parallel to the
$x_2$-axis ---and so, it is invariant under the $x_1 x_3$-specular
reflection. We will check that this kind of double symmetry is the norm among
SPTs, not the exception.

In view of the above arguments it is easy to guess that there is a rich
casuistry of SPTs inside nondegenerate ellipsoids of $\Rset^{n + 1}$, so
much so that there is no simple way to describe them. For simplicity, let us
briefly consider the 3D case: $n = 2$. Then, any billiard trajectory inside $Q
\subset \Rset^3$ has two caustics of the form
\[
Q_\lambda =
\left\{
( x_1, x_2, x_3 ) \in \Rset^3 :
\sum_{j=1}^3 \frac{x_j^2}{a_j - \lambda}  = 1
\right\}.
\]
Let us assume that $0 < a_1 < a_2 < a_3$. We restrict our attention to
nonsingular trajectories. That is, trajectories whose caustics are ellipsoids:
$0 < \lambda < a_1$; 1-sheet hyperboloids: $a_1 < \lambda < a_2$; or 2-sheet
hyperboloids: $a_2 < \lambda < a_3$. The singular values $\lambda \in \left\{
a_1, a_2, a_3 \right\}$, are discarded. It is known that there are only four
types of couples of nonsingular caustics: EH1, H1H1, EH2, and H1H2. The
notation is self-explanatory. It is also known that each caustic type gives
rise to a different billiard dynamics. For instance, only trajectories with
EH2-caustics rotate around the $x_3$-axis.
The left trajectory in Fig.~\ref{fig:TwoAmazingSPTs_3D} is a sample.

Coming back to the general case, its billiard trajectories present $2^n$ different
caustic types, which describe how trajectories fold in $\Rset^{n + 1}$; see
Ref.~\onlinecite{CasasRamirez2011}. Then, we classify SPTs inside nondegenerate
ellipsoids of $\Rset^{n + 1}$ by their caustic types and their symmetries.
That is, two SPTs with the same symmetries, but different caustic type, are in
different classes. On the one hand, not every symmetry takes place among SPTs of a
given caustic type, but only $2^{n + 1}$. On the other hand, all SPTs are, in some
sense, doubly symmetric, which rise to $2^n \left( 2^{n + 1} - 1 \right)$
classes we may encounter with a given caustic type. Consequently, we know that there
are at most $2^{2 n} \left( 2^{n + 1} - 1 \right)$ classes of SPTs. Finally, by
means of an argument involving \textit{winding numbers}, we establish the existence
of SPTs for all those classes.

That concludes the analytical part of the work. We have also implemented a
numerical algorithm to find minimal SPTs of each of the 12 classes in the 2D
case, and each of the 112 classes in the 3D case. They have periods 3, 4 or 6
in the 2D case; and 4, 5, 6, 8 or 10 in the 3D case. We depict all 2D minimal
SPTs in Table~\ref{tab:SPTs_2D}, and just a gallery (for brevity of
exposition) of 3D minimal SPTs in Sec.~\ref{sec:3D}.

For instance, we show in Fig.~\ref{fig:TwoAmazingSPTs_3D} two minimal SPTs
of period 6, but with only 5 and 4 different impact points. Only the second
phenomena can take place in the 2D case. This has to do with the fact that,
from a generic point on the ellipsoid, we can trace four lines tangent to a
fixed couple of caustics. On the contrary, we can trace just two in the 2D
case.

Some of the ideas used in this paper were first introduced by Kook and
Meiss,\cite{KookMeiss1989} to classify symmetric periodic orbits of some standard-like
reversible $2 n$-dimensional symplectic maps. The technical details of our problem
are harder ---there are more symmetries, caustic types play a role, and the
classification does not depend only on the evenness or oddness of some
integers---, but the main arguments do not change.

Some of our billiard SPTs could be considered as discrete versions of the closed
geodesics on triaxial ellipsoids found in Ref.~\onlinecite{Fedorov2005}. For
instance, it is interesting to compare the SPT on the left side of
Fig.~\ref{fig:TwoAmazingSPTs_3D} with the symmetric closed geodesic shown in Fig.~4 of
that paper. Of course, the SPT on the right side of Fig.~\ref{fig:TwoAmazingSPTs_3D}
have no continuous version, because there are no geodesics with return points.

We complete this introduction with a note on the organization of the article. We
review the classical theory of reversible maps in Sec.~\ref{sec:ReversibleMaps}
following Ref.~\onlinecite{LambRoberts1998}. Next, we specialize that
theory to billiards inside symmetric hypersurfaces.
In Sec.~\ref{sec:BilliardsEllipsoids} we review briefly some well-known results about
billiards inside ellipsoids, in order to fix notations that will be used along the
rest of the paper. Billiards inside ellipses of $\Rset^2$ and inside triaxial
ellipsoids of $\Rset^3$, are thoroughly studied in Sec.~\ref{sec:2D} and
Sec.~\ref{sec:3D}, respectively. Billiards inside nondegenerate ellipsoids of
$\Rset^{n + 1}$ are revisited in Sec.~\ref{sec:nD}. Perspectives and
conclusions are drawn in Sec.~\ref{sec:Conclusions}.

\section{Reversible maps and symmetric orbits}\label{sec:ReversibleMaps}

In this section we state the concept of symmetric periodic orbit for a map. A
map is reversible if each orbit is related to its time reverse orbit by a
symmetry transformation. Reversible maps can be characterized as maps that
factorize as the composition of two involutions. Symmetric orbits must have
points on certain symmetry sets; namely, the fixed sets of the involutions
(reversors) which factorize the original map. A reversible map can have
different factorizations and, therefore, its symmetric orbits can be
classified according to the symmetry sets they intersect.

Let $f : M \to M$ be a diffeomorphism on a manifold $M$.

\begin{definition}
  The map $f$ is symmetric when there exists an involution $s : M \to M$
  (i.e., $s \circ s = \Identity$), such that $f \circ s = s \circ f$.
  Then, $s$ is called a symmetry of the map $f$.
\end{definition}

\begin{definition}
  \label{def:ReversibleMap}The map $f$ is reversible when there exists an
  involution $r : M \to M$ such that $f \circ r = r \circ f^{- 1}$. Then, $r$
  is called a reversor of the map $f$ (i.e., $f$ is $r$-reversible).
\end{definition}

\begin{remark}
  \label{rem:SymmetryReversor}The composition of a symmetry and a reversor
  (provided both commute) is another reversor.
\end{remark}

\begin{definition}
  If $r$ is a reversor, we denote by $\FixedSet (r) = \{m \in M :
  r (m) = m\}$ its set of fixed points, which is called the symmetry set of
  the reversor. 
\end{definition}

\begin{definition}
  An orbit of the map $f$ is a sequence $O = (m_j)_{j \in \Zset}$ such
  that $m_j = f (m_{j - 1}) = f^j (m_0)$.
\end{definition}

\begin{definition}
  An orbit $O$ of a $r$-reversible map $f$ is called $r$-symmetric when $r
  (O) = O$. 
\end{definition}

The following characterization of reversible maps goes back to G. D. Birkhoff.

\begin{lemma}
  \label{lem:ReversibleMaps}A map $f$ is reversible if and only if it can be
  factorized as the composition of two involutions, in which case both of them
  are reversors of $f$. 
\end{lemma}

\begin{proof}
  Let us assume that $f$ is $\tilde{r}$-reversible. Then $\hat{r} = f \circ
  \tilde{r} = \tilde{r} \circ f^{- 1}$ is another reversor, because:
  \begin{itemize}
    \item $f \circ \hat{r} = f \circ \tilde{r} \circ f^{- 1} = \hat{r} \circ
    f^{- 1}$, and
    
    \item $\hat{r}^2 = \hat{r} \circ \hat{r} = \tilde{r} \circ f^{- 1} \circ f
    \circ \tilde{r} = \tilde{r}^2 = \Identity$.
  \end{itemize}
  Therefore, the map $f = f \circ \tilde{r}^2 = \hat{r} \circ \tilde{r}$ is
  the composition of two involutions.
  
  On the other hand, if $f = \hat{r} \circ \tilde{r}$ and $\hat{r}^2 =
  \tilde{r}^2 = \Identity$, then:
  \begin{itemize}
    \item $f \circ \tilde{r} \circ f = \hat{r} \circ \tilde{r}^2 \circ \hat{r}
    \circ \tilde{r} = \hat{r}^2 \circ \tilde{r} = \tilde{r}$, and
    
    \item $f \circ \hat{r} \circ f = \hat{r} \circ \tilde{r} \circ \hat{r}^2
    \circ \tilde{r} = \hat{r} \circ \tilde{r}^2 = \hat{r}$,
  \end{itemize}
  so both involutions $\hat{r}$ and $\tilde{r}$ are reversors of $f$.
\end{proof}

The factorization of a reversible symmetric map as a composition of two
involutions is not unique. From any given factorization $f = \hat{r} \circ
\tilde{r}$, we can construct infinitely many more, namely, $f = \hat{r}_k
\circ \tilde{r}_k$, where $\hat{r}_k = f^{k + 1} \circ \tilde{r} = f^k \circ
\hat{r} = \hat{r} \circ f^{- k}$ and $\tilde{r}_k = f^k \circ \tilde{r}$ for
any $k \in \Zset$. Nevertheless, these new factorizations do not provide
new symmetric orbits, because $\hat{r} (O) = \tilde{r} (O) = \hat{r}_k (O) =
\tilde{r}_k (O)$, i.e., an orbit is invariant under $\tilde{r}$ (or $\hat{r}$)
if and only if it is invariant under all the reversors $\tilde{r}_k$ and
$\hat{r}_k$. On the contrary, the existence of a symmetry $s$ commuting with
the reversors $\tilde{r}$ and $\hat{r}$ is more promising, because then $f =
(s \circ \hat{r}) \circ (s \circ \tilde{r})$ is a new factorization, which
could give rise to new kinds of symmetric orbits. The more such symmetries
exist, the more kinds of symmetric orbits we can try to find.

Given a factorization $f = \hat{r} \circ \tilde{r}$, we say that $\tilde{r}$ and
$\hat{r}$ are \tmtextit{associated reversors}, and $\FixedSet ( \tilde{r})$ and
$\FixedSet ( \hat{r})$ are \tmtextit{associated symmetry sets}. The importance of
these concepts is clarified in the following characterization of symmetric orbits,
which can be found in Ref.~\onlinecite{LambRoberts1998}.

\begin{theorem}
  \label{thm:SymmetricOrbits}Let $f = \hat{r} \circ \tilde{r}$ be any
  factorization of a rever-sible map as the composing of two involutions. Then
  an orbit of this map is:
  \begin{enumeratealpha}
    \item $\hat{r}$-symmetric if and only if it is $\tilde{r}$-symmetric.
    
    \item \label{item:NoMoreThan2}$\tilde{r}$-symmetric if and only if it has
    at least one point on $\FixedSet ( \tilde{r}) \cup
    \FixedSet ( \hat{r})$, in which case it has no more than two
    points on $\FixedSet ( \tilde{r}) \cup \FixedSet (
    \hat{r})$.
    
    \item \label{item:SPOs}$\tilde{r}$-symmetric and periodic if and only if
    it has exactly two points on $\FixedSet ( \tilde{r}) \cup
    \FixedSet ( \hat{r})$, in which case it has a point on each
    symmetry set if and only if it has odd period. In particular,
    \begin{enumerateroman}
      \item An orbit is $\tilde{r}$-symmetric with period $2 k$ if and only if
      it has a point $m_{\ast}$ such that either $m_{\ast} \in
      \FixedSet ( \tilde{r}) \cap f^k \FixedSet (
      \tilde{r})$ or $m_{\ast} \in \FixedSet ( \hat{r}) \cap f^k
      \FixedSet ( \hat{r})$, in which case $f^k (m_{\ast})$ is on
      the same symmetry set as $m_{\ast}$.
      
      \item An orbit is $\tilde{r}$-symmetric orbit with period $2 k + 1$ if
      and only if it has a point $m_{\ast}$ such that $m_{\ast} \in
      \FixedSet ( \tilde{r}) \cap f^k \FixedSet (
      \hat{r})$, in which case $f^{k + 1} (m_{\ast}) \in \FixedSet
      ( \hat{r})$.
    \end{enumerateroman}
  \end{enumeratealpha}
\end{theorem}

\begin{proof}
  This is an old result, so we skip most of the details. We focus in those
  aspects that are more relevant for our goals.
  \begin{enumeratealpha}
    \item It is immediate as $\hat{r} \left( O \right) = \tilde{r} \left( O
    \right)$.
    
    \item To begin with, let us assume that an orbit $O = (m_j)_{j \in
    \Zset}$ is $\tilde{r}$-symmetric, so there exists some $l \in
    \Zset$ such that $m_l = \tilde{r} (m_0)$. Then $m_{l + 1} = f (m_l) =
    \hat{r} (m_0)$, because $\hat{r} = f \circ \tilde{r}$. We distinguish two
    cases:
    \begin{itemize}
      \item If $l = 2 k$ is even then $m_k \in \FixedSet (
      \tilde{r})$ as:
      \[ m_k = f^{- k} (m_{2 k}) = f^{- k} ( \tilde{r} (m_0)) = \tilde{r}
         (m_k) . \]
      \item If $l = 2 k + 1$ is odd then $m_{k + 1} \in \FixedSet
      ( \hat{r})$ as:
      \[
        m_{k + 1} = f^{- (k + 1)} (m_{2 k + 2})
        = f^{- (k + 1)} ( \hat{r} (m_0)) = \hat{r} (m_{k + 1}) .
      \]
    \end{itemize}
    On the other hand, let us assume that $O$ has a point, say $m_k$, in one
    of those fixed sets. Then there exist two possibilities. In the case $m_k
    \in \FixedSet ( \tilde{r})$, $\forall j \in \Zset$:
    \[ m_{k + j} = f^j (m_k) = f^j ( \tilde{r} (m_k)) = \tilde{r} (f^{- j}
       (m_k)) = \tilde{r} (m_{k - j}) . \]
    In the case $m_k \in \FixedSet ( \hat{r})$, from $\tilde{r} =
    f^{- 1} \circ \hat{r}$, $\forall j \in \Zset$:
    \begin{eqnarray*}
      m_{k + j - 1} & = & f^{j - 1} (m_k) \hspace{0.15cm} = \hspace{0.15cm}
      f^{j - 1} ( \hat{r} (m_k))\\
      & = & f^j ( \tilde{r} (m_k)) \hspace{0.15cm} = \hspace{0.15cm}
      \tilde{r} (f^{- j} (m_k)) \hspace{0.15cm} = \hspace{0.15cm} \tilde{r}
      (m_{k - j}) .
    \end{eqnarray*}
    Hence, $O$ turns out to be $\tilde{r}$-symmetric in both cases.
    
    \item Next, we study the symmetric periodic orbits. First, if $O$ has two
    points, say $m_0$ and $m_k$, on $\FixedSet ( \tilde{r})$, then
    $O$ is $2 k$-periodic, because:
    \begin{eqnarray*}
      m_{2 k} & = & f^k (m_k) \hspace{0.15cm} = \hspace{0.15cm} f^k (
      \tilde{r} (m_k))\\
      & = & \tilde{r} (f^{- k} (m_k)) \hspace{0.15cm} = \hspace{0.15cm}
      \tilde{r} (m_0) \hspace{0.15cm} = \hspace{0.15cm} m_0 .
    \end{eqnarray*}
    Second, if $O$ has one point on each symmetry set, say $m_0 \in
    \FixedSet ( \tilde{r})$ and $m_{k + 1} \in
    \FixedSet ( \hat{r})$, then $O$ is $(2 k + 1)$-periodic,
    because:
    \begin{eqnarray*}
      m_{2 k + 1} & = & f^k (m_{k + 1}) = f^k( \hat{r} (m_{k + 1}))\\
      & = & f^{k + 1} ( \tilde{r} (m_{k + 1})) = \tilde{r} (f^{- (k + 1)} (m_{k +
        1}))\\
      & = & \tilde{r} (m_0) = m_0.
    \end{eqnarray*}
  \end{enumeratealpha}
\vspace{-28pt}
\end{proof}

Therefore, the computation of symmetric orbits with period $2 k$ can be
reduced to the computation of points on symmetry sets, which are mapped onto
the same symmetry set after $k$ iterations of the map. On the other hand, the
computation of symmetric orbits with period $2 k + 1$ can be reduced to the
computation of points on symmetry sets, which are mapped onto their associated
symmetry sets after $k + 1$ iterations of the map.

\section{Billiards inside symmetric hypersurfaces of $\bm{\Rset^{n+1}}$}
\label{sec:SymmetricBilliards}

Let $Q$ be a (strictly) convex smooth hypersurface of $\Rset^{n + 1}$.
The billiard motion inside $Q$ can be modelled by means of a diffeomorphism
$f$ defined on the \tmtextit{phase space}
\[
M =
\left\{
(q,p) \in Q \times \Sset^n :
\mbox{$p$ is directed outward $Q$ at $q$}
\right\}.
\]
We define the \tmtextit{billiard map} $f : M \to M$, $f (q, p) = (q', p')$, as
follows. The new velocity $p'$ is the reflection of $p$, with
respect to the tangent hyperplane $T_q Q$. That is,
\[ p' = p_{\rmt} - p_{\rmn} = p - 2 p_{\rmn}, \]
where $p_\rmt$ and $p_\rmn$ are the tangent and normal components at $q$ of the old
velocity:
\begin{equation}
p = p_{\rmt} + p_{\rmn},\qquad
p_{\rmt} \in T_q Q, \qquad
p_{\rmn} \in N_q Q.
\label{eq:TangentNormal}
\end{equation}
Let $q + \langle p' \rangle$ be the line through $q$ with direction $p'$. Then
the new impact point $q'$ is determined by the condition
\[ Q \cap \left( q + \langle p' \rangle \right) = \{q, q' \} . \]
That is, $q'$ is the intersection of the ray $\{q + \mu p' : \mu > 0\}$ with
the hypersurface $Q  $. This intersection is unique and
transverse by convexity.

\begin{definition}
  A billiard orbit is a sequence of points $\left\{ m_j \right\}_{j \in
  \Zset} \subset M$ such that $m_j = f (m_{j - 1}) = f^j (m_0)$. If the points
  $m_j = (q_j, p_j)$ form a billiard orbit, then the sequence of impact points
  $\left\{ q_j \right\}_{j \in \Zset} \subset Q$ is a billiard (or impact)
  configuration, whose joining by polygonal lines form a billiard trajectory;
  the sequence of outward velocities $\left\{ p_j \right\}_{j \in \Zset}
  \subset \Sset^n$ is a velocity configuration. 
\end{definition}

The distinction between orbits and trajectories is clear. We refer to orbits
when we are working in the phase space $M$, whereas we refer to trajectories
when we are drawing in $\Rset^{n + 1}$. There exists an one-to-one
correspondence between billiard orbits and billiard configurations. The
velocities of the billiard orbit that correspond to the billiard configuration
$\left\{ q_j \right\}_{j \in \Zset}$ are:
\begin{equation}
  p_j = \frac{q_j - q_{j - 1}}{|q_j - q_{j - 1} |} . \label{eq:Relation-qp}
\end{equation}

From now on we shall assume that the hypersurface $Q$ is symmetric with respect
to \tmtextit{all} coordinate hyperplanes of $\Rset^{n + 1}$. Then the
billiard map admits $2^{n + 1}$ factorizations as a composition of two
involutions. We need the following notations in order to describe them.

Let $\Sigma$ be the set made up of the \tmtextit{reflections} ---that is,
involutive linear transformations--- with respect to the $2^{n + 1}$
coordinate subspaces of $\Rset^{n + 1}$. We represent its elements as
diagonal matrices whose diagonal entries are equal to $1$ or $- 1$, so
\[ \Sigma = \left\{ \sigma = \tmop{diag} \left( \sigma_1, \ldots, \sigma_{n +
   1} \right) : \sigma_j \in \left\{ - 1, 1 \right\} \tmop{for} \tmop{all} j
   \right\} . \]
Given $\left( q, p \right) \in M$, there exists a unique point $\hat{q}$
such that
\[ Q \cap \left( q + \langle p \rangle \right) = \{q, \hat{q} \} . \]
Thus, $\hat{q}$ denotes the previous impact point. Finally, let $\tilde{p}$ be
the reflection of the velocity $p$ with respect to the normal line $N_q Q$.
That is,
\[ \tilde{p} = - p' = p_{\rmn} - p_{\rmt} = p - 2 p_{\rmt}, \]
where $p_\rmt$ and $p_\rmn$ were defined in the
decomposition~(\ref{eq:TangentNormal}).

In what follows, symbols $q'$, $p'$, $\hat{q}$, and $\tilde{p}$ have the
meaning given in the previous paragraphs. We emphasize that they make sense
only after both impact point $q$, and unitary outer velocity $p$, are given.
Next we describe the $2^{n + 1}$ factorizations of the billiard map.

\begin{proposition}
  \label{prop:BilliardFactorizations}Let $f : M \to M$, $f \left( q, p \right)
  = \left( q', p' \right)$, be the billiard map inside a closed convex
  symmetric hypersurface $Q \subset \Rset^{n + 1}$. Let $\tilde{r},
  \hat{r} : M \rightarrow M$ be the maps
  \[ \tilde{r} (q, p) = (q, \tilde{p}), \qquad \hat{r} (q, p) = (
     \hat{q}, - p) . \]
  Let $s_{\sigma}, \tilde{r}_{\sigma}, \hat{r}_{\sigma} : M \to M$ be the maps
  \[ s_{\sigma} (q, p) = (\sigma q_{}, \sigma p_{}), \qquad
     \tilde{r}_{\sigma} = s_{\sigma} \circ \tilde{r}, \qquad
     \hat{r}_{\sigma} = s_{\sigma} \circ \hat{r} \]
  for any reflection $\sigma \in \Sigma$. Then:
  \begin{enumeratealpha}
    \item $f = \hat{r} \circ \tilde{r}$, and $\hat{r}^2 = \tilde{r}^2 =
    \Identity$.
    
    \item $f \circ s_{\sigma} = s_{\sigma} \circ f$, $s_{\sigma}^2 =
    \Identity$, $\hat{r} \circ s_{\sigma} = s_{\sigma} \circ \hat{r}$, and
    $\tilde{r} \circ s_{\sigma} = s_{\sigma} \circ \tilde{r}$.
    
    \item $f = \hat{r}_{\sigma} \circ \tilde{r}_{\sigma}$, and
    $\hat{r}_{\sigma}^2 = \tilde{r}_{\sigma}^2 = \Identity$.
  \end{enumeratealpha}
\end{proposition}

\noindent\textbf{Proof\ }
  \begin{enumeratealpha}
    \item $\hat{r} \left( \tilde{r} \left( q, p \right) \right) = \hat{r}
    \left( q, \tilde{p} \right) = \hat{r} \left( q, - p' \right) = \left( q',
    p' \right)$. Besides, $\hat{r}^2 \left( q, p \right) = \hat{r} \left(
    \hat{q}, - p \right) = \left( q, p \right)$, as $Q \cap \left( \hat{q} +
    \left\langle - p \right\rangle \right) = \left\{ q, \hat{q} \right\}$.
    Finally, $\tilde{r}^2 \left( q, p \right) = \tilde{r} \left( q, \tilde{p}
    \right) = \left( q, p \right)$.
    
    \item Clearly $s_{\sigma}$ is an involution. Since $Q$ is symmetric with
    respect to the coordinate hyperplanes, it turns out that if $p_\rmt$ and $p_\rmn$
    are the tangent and normal components of a velocity $p$ at an impact point
    $q$, then $\sigma p_\rmt$ and $\sigma p_\rmn$ are the tangent and normal
    components of $\sigma p$ at the impact point $\sigma q$. On the other
    hand, if we write $q' = q + \mu \left( q, p' \right) p'$, then $\mu \left(
    \sigma q, \sigma p' \right) = \mu \left( q, p' \right)$, again by the
    symmetry of $Q$. Hence, $f \circ s_{\sigma} = s_{\sigma} \circ f$. The
    proof of $\hat{r} \circ s_{\sigma} = s_{\sigma} \circ \hat{r}$ and
    $\tilde{r} \circ s_{\sigma} = s_{\sigma} \circ \tilde{r}$ follows the same
    lines.
    
    \item It is a direct consequence of Remark~\ref{rem:SymmetryReversor}.
    \hspace*{\fill}$\Box$\medskip
  \end{enumeratealpha}

No symmetry has been required to obtain the factorization $f = \hat{r} \circ
\tilde{r}$. Therefore, all convex hypersurfaces $Q \subset \Rset^{n +
1}$ give rise to reversible billiard maps, although symmetric hypersurfaces
have much more factorizations.

We introduce the acronyms SO, ST, SPO, and SPT for symmetric orbit, symmetric
trajectory, symmetric periodic orbit, and symmetric periodic trajectory,
respectively. If $r$ is any reversor of the billiard map, we can deal with
$r$-SOs, $r$-STs, $r$-SPOs, and $r$-SPTs.

Once these $2^{n + 1}$ factorizations $f = \hat{r}_{\sigma} \circ
\tilde{r}_{\sigma}$ have been found, we describe their symmetry sets. In the
next proposition we prove that only two symmetry sets are empty. We also
provide an explicit geometric description of the $2^{n + 2} - 2$ nonempty
symmetry sets where one can look for SOs. As a consequence, we shall see that
these symmetry sets are mutually exclusive ---that is, they do not
intersect--- but at some very specific points which are described in
detail. As before, some notations are required.

We recall that given a reflection defined on an Euclidean space, we can
decompose the space as the orthogonal sum of the eigenspaces of eigenvalues $-
1$ and $1$. This is called the \tmtextit{spectral decomposition} of the
reflection. Let $\Rset^{n + 1} = V^+_{\sigma} \perp V^-_{\sigma}$ be the
spectral decomposition associated to any reflection $\sigma \in \Sigma$.
That is,
\[ V^{\pm}_{\sigma} = \{p \in \Rset^{n + 1} : \sigma p_{} = \pm p\} . \]
For instance, if $\left( x_1, x_2, x_3 \right)$ are the Cartesian coordinates
in $\Rset^3$ and $\sigma = \tmop{diag} \left( - 1, 1, 1 \right)$, then
$V_{\sigma}^+$ is the $x_2 x_3$-plane, and $V_{\sigma}^-$ is the $x_1$-axis.
We also introduce the sections
\[ \begin{array}{lllll}
     Q_{\sigma} & = & Q \cap V^+_{\sigma} & = & \{q \in Q : \sigma q_{} =
     q\},\\
     P_{\sigma} & = & \Sset^n \cap V_{\sigma}^- & = & \{p \in \Sset^n : \sigma
     p_{} = - p\} .
   \end{array} \]
Given two linear varieties $V_1, V_2 \subset \Rset^{n + 1}$, the symbol
$V_1 \perp_{\ast} V_2$ means that they have a orthogonal intersection.
Finally, we recall that a line in $\Rset^{n + 1}$ is a \tmtextit{chord}
of $Q$ when it intersects orthogonally $Q$ at two different points.

\begin{proposition}
  \label{prop:SymmetrySets}The symmetry sets of the reversors
  $\tilde{r}_{\sigma}$ and $\hat{r}_{\sigma}$, $\sigma \in \Sigma$, are:
  \begin{eqnarray*}
    \FixedSet ( \tilde{r}_{\sigma}) & = & \left\{ \left( q, p \right) \in M :
    q \in Q_{\sigma} \tmop{and} p \in N_q Q_{\sigma} \right\},\\
    \FixedSet ( \hat{r}_{\sigma}) & = & \left\{ \left( q, p \right) \in M :
    \sigma q_{} \in q + \langle p \rangle \tmop{and} p \in P_{\sigma}
    \right\}\\
    & = & \left\{ \left( q, p \right) \in M : \left( q + \left\langle p
    \right\rangle \right) \perp_{\ast} V_{\sigma}^+ \right\} .
  \end{eqnarray*}
  Only the reversors $\tilde{r}_{-\Identity} $ and $\hat{r}_{\Identity}$ have
  empty symmetry sets. Moreover, if a point $\left( q, p \right) \in M$
  belongs simultaneously to two different symmetry sets, then the line $q +
  \left\langle p \right\rangle$:
  \begin{itemize}
    \item Is contained in some coordinate hyperplane of $\Rset^{n + 1}$,
    in which case so is its whole billiard trajectory; or
    
    \item Is a chord of the hypersurface $Q$ through the origin, in which case
    its billiard trajectory is 2-periodic.
  \end{itemize}
\end{proposition}

\begin{proof}
  To begin with, we deduce from the definitions of reversors
  $\tilde{r}_{\sigma}$ and $\hat{r}_{\sigma} $, and symmetries
  $s_{\sigma}$ that
  \begin{eqnarray*}
    (q, p) \in \FixedSet ( \tilde{r}_{\sigma}) & \Longleftrightarrow &
    s_{\sigma} (q, p) = \tilde{r} (q, p)\\
    & \Longleftrightarrow & \begin{Bmatrix}
      \sigma q = q\\
      \sigma p = \tilde{p}
    \end{Bmatrix} \Longleftrightarrow \begin{Bmatrix}
      q \in Q_{\sigma}\\
      \sigma p = \tilde{p}
    \end{Bmatrix},\\
    (q, p) \in \FixedSet ( \hat{r}_{\sigma}) & \Longleftrightarrow &
    s_{\sigma} (q, p) = \hat{r} (q, p)\\
    & \Longleftrightarrow & \begin{Bmatrix}
      \sigma q = \hat{q} \phantom{ql} \\
      \sigma p = - p
    \end{Bmatrix} \Longleftrightarrow \begin{Bmatrix}
      \sigma q \in q + \langle p \rangle\\
      p \in P_{\sigma}
    \end{Bmatrix}.
  \end{eqnarray*}
  To understand the condition $\sigma p_{} = \tilde{p}$, we compare the
  spectral decomposition $\Rset^{n + 1} = V^+_{\sigma} \perp
  V^-_{\sigma}$ with the spectral decomposition $\Rset^{n + 1} \simeq
  T_q \Rset^{n + 1} = \tilde{V}^+_q \perp \tilde{V}^-_q$ associated to
  the reflection $p \mapsto \tilde{p}$. We note that
  \[ \begin{array}{lllll}
       \tilde{V}^-_q & = & \{p \in T_q \Rset^{n + 1} : \tilde{p} = - p\}
       & = & T_q Q,\\
       \tilde{V}_q^+ & = & \{p \in T_q \Rset^{n + 1} : \tilde{p}
       = p\} & = & N_q Q.
     \end{array} \]
  If $q \in Q_{\sigma}$, then
  \[ \begin{array}{lllll}
       N_q Q_{\sigma} & = & N_q Q \perp V_{\sigma}^- & = & \tilde{V}^+_q \perp
       V^-_{\sigma},\\
       T_q Q_{\sigma} & = & T_q Q \cap V^+_{\sigma} & = & \tilde{V}^-_q \cap
       V_{\sigma}^+,
     \end{array} \qquad \begin{array}{l}
       \tilde{V}^+_q \subset V_{\sigma}^+,\\
       V^-_{\sigma} \subset \tilde{V}^-_q .
     \end{array} \]
  The equality on $N_q Q_{\sigma}$ is due to $\left( A \cap B \right)^{\bot} =
  A^{\bot} \bot B^{\bot}$. The inclusions follow from $V_{\sigma}^+$ being a
  symmetry subspace of $Q$. Those relations imply that
  \begin{eqnarray*}
    T_q  \Rset^{n + 1} & = & N_q Q_{\sigma} \perp T_q Q_{\sigma}
    \hspace{0.15cm} = \hspace{0.15cm} \tilde{V}^+_q \perp V^-_{\sigma} \perp (
    \tilde{V}^-_q \cap V^+_{\sigma})\\
    & = & (V^+_{\sigma} \cap \tilde{V}^+_q) \perp (V^-_{\sigma} \cap
    \tilde{V}^-_q) \perp ( \tilde{V}^-_q \cap V^+_{\sigma}) .
  \end{eqnarray*}
  Therefore, if we fix any $q \in Q_{\sigma}$, then $\sigma p = \tilde{p}
  \Leftrightarrow p \in (V^+_{\sigma} \cap \tilde{V}^+_q) \perp (V^-_{\sigma}
  \cap \tilde{V}^-_q) = \tilde{V}^+_q \perp V^-_{\sigma} = N_q Q_{\sigma}$, so
  \[ \FixedSet ( \tilde{r}_{\sigma}) = \left\{ (q, p) \in M : q \in
     Q_{\sigma}, p \in N_q Q_{\sigma} \right\} . \]
  Next, we study the other symmetry sets for $\sigma \neq \Identity$. If $p
  \in P_{\sigma}$, then $p \perp V^+_{\sigma}$ and $p \parallel V^-_{\sigma}$,
  which implies that $\sigma q_{} \in q + \langle p \rangle \Leftrightarrow q
  + \langle p \rangle \perp_{\ast} V^+_{\sigma}$. Therefore,
  \[ \FixedSet ( \hat{r}_{\sigma}) = \left\{ (q, p) \in M : q + \langle p
     \rangle \perp_{\ast} V^+_{\sigma} \right\} . \]
  Clearly, $Q_{\sigma}$ is empty if and only if $\sigma = - \Identity$, and
  $P_{\sigma}$ is empty if and only if $\sigma = \Identity$.
  
  Finally, let us check that if the point $\left( q, p \right) \in M$ belongs
  to two different symmetry sets, then $q + \left\langle p \right\rangle$ is a
  chord of $Q$ or is contained in some coordinate hyperplane like
  \begin{equation}
    H_j = \left\{ q = \left( x_1, \ldots, x_{n + 1} \right) \in \Rset^{n
    + 1} : x_j = 0 \right\} . \label{eq:Hyperplane}
  \end{equation}
  First, if $\left( q, p \right) \in \FixedSet \left( \tilde{r}_{\sigma}
  \right) \cap \FixedSet \left( \tilde{r}_{\tau} \right)$, then using the
  relations written at the beginning of this proof, we get
  \[ \sigma q = q = \tau q , \qquad \sigma p = \tilde{p} = \tau
     p. \]
  If $\sigma \neq \tau$, then $\sigma_j \neq \tau_j$ for some $j$, so that $q
  \in H_j$, $p \left| \right| H_j$, and $q + \left\langle p \right\rangle
  \subset H_j$.
  
  Second, if $\left( q, p \right) \in \FixedSet \left( \tilde{r}_{\sigma}
  \right) \cap \FixedSet \left( \hat{r}_{\tau} \right)$, then using again the
  same relations, we get that
  \[ \sigma q = q , \qquad \sigma p_{} = \tilde{p} = - p',
     \qquad \tau q = \hat{q}, \qquad \tau p_{} = - p. \]
  If $\sigma \neq \Identity$, then $\sigma_j = - 1$ for some $j$, so that $q,
  \hat{q} \in H_j$, $p \left| \right| H_j$, and $q + \left\langle p
  \right\rangle \subset H_j$. If $\sigma = \Identity$ but $\tau \neq -
  \Identity$, then $\tau_j = 1$ for some $j$, so that $p' = - p \in N_q Q$, $p
  \left| \right| H_j$, and $p' \left| \right| H_j$. Thus, $p_\rmn \left| \right|
  H_j$ and $N_q Q \subset H_j$, which implies that $q \in H_j$, because $Q$ is
  symmetric, smooth, and (strictly) convex. Therefore, $q + \left\langle p
  \right\rangle \subset H_j$. If $\sigma = \Identity$ and $\tau = -
  \Identity$, then $p' = - p \in N_q Q$ and $\hat{q} = - q$, so that the line
  $q + \left\langle p \right\rangle$ intersects orthogonally $Q$ at $q$ and $-
  q$.
  
  Third, if $\left( q, p \right) \in \FixedSet \left( \hat{r}_{\sigma}
  \right) \cap \FixedSet \left( \hat{r}_{\tau} \right)$, then
  \[ \sigma q = \hat{q} = \tau q , \qquad \sigma p = - p = \tau
     p, \]
  and we apply the same reasoning as in the first case.
\end{proof}

Hence, only very specific billiard orbits ---the ones contained in a
coordinate hyperplane or 2-periodic--- can have a point in the
intersection of two symmetry sets. For instance, it turns out that any
2-periodic billiard orbit inside a nondegenerate ellipsoid is contained in
more than half of the $2^{n + 2} - 2$ nonempty symmetry sets.

Nevertheless, there exists other billiard orbits that intersect different
symmetry sets at different points. Indeed, all SPOs with odd period have
points on two associated symmetry sets, whereas all SPOs with even period have
\tmtextit{two} points on some symmetry set; see item~\ref{item:SPOs}) of
Theorem~\ref{thm:SymmetricOrbits}. However, a especial situation arises when
$Q$ is an ellipsoid: most of its SPOs with even period have exactly
\tmtextit{four} ---instead of the expected two--- points in the
symmetry sets described in Proposition~\ref{prop:SymmetrySets}, in which case
they must intersect two different symmetry sets
(see item~\ref{item:NoMoreThan2}) of Theorem~\ref{thm:SymmetricOrbits}).

That motivates the following definition.

\begin{definition}
  An SPO inside a symmetric hypersurface $Q \subset \Rset^{n + 1}$, is a
  doubly SPO when it intersects two different symmetry sets.
\end{definition}

\begin{remark}\label{rem:CentralAxialSpecular}
Let $O$ be a $\tilde{r}_{\sigma}$-SO or, equivalently,
a $\hat{r}_{\sigma}$-SO (see Theorem~\ref{thm:SymmetricOrbits}).
By definition, $O$, as a subset of the phase space $M$,
is invariant under the reversors
  \[ \tilde{r}_{\sigma} (q, p) = (\sigma q_{}, \sigma \tilde{p}_{}),
     \qquad \hat{r}_{\sigma} (q, p) = (\sigma \hat{q}_{}, - \sigma p_{})
     . \]
\end{remark}

In particular, the billiard configuration associated to $O$, viewed as a
subset of the configuration space $Q$, is invariant under the map
$\sigma_{\left| Q \right.} : q \mapsto \sigma q$, whereas its velocity
configuration, viewed as a subset of the velocity space $\Sset^n$, is
invariant under the map $- \sigma_{\left| \Sset^n \right.} : p \mapsto -
\sigma p$. This motivates the following definitions for the 2D and 3D cases. A
billiard configuration inside a symmetric curve/surface is called
\tmtextit{central}, \tmtextit{axial} or \tmtextit{specular} when it is
symmetric with respect to the origin, some axis of coordinates or some plane of
coordinates, respectively. Similar definitions apply to velocity
configurations.

\section{Billiards inside ellipsoids of $\Rset^{n + 1}$}\label{sec:BilliardsEllipsoids}

The billiard dynamics inside ellipsoids has several important properties. For
instance, it is completely integrable in the sense of Liouville. We present some
results about such billiards. First, we list the classical ones, which can be found
in the monographs Refs.~\onlinecite{KozlovTreschev1991,Tabachnikov1995,Tabachnikov2005}.
Next, we describe a dual property discovered in
Refs.~\onlinecite{ChangFriedberg1988,Veselov1988}. Finally, we detail the behaviour
of elliptic coordinates of billiard trajectories inside ellipsoids given in
Ref.~\onlinecite{DragovicRadnovic2006}.

\subsection{Caustics and elliptic coordinates}\label{ssec:Caustics}

The following results go back to Jacobi, Chasles, Poncelet, and Darboux.
We consider a billiard inside the ellipsoid
\begin{equation}
  Q = \left\{ q \in \Rset^{n + 1} : \left\langle D q, q \right\rangle =
  1 \right\}, \label{eq:Ellipsoid}
\end{equation}
where $D^{- 1} = \tmop{diag} \left( a_1, \ldots, a_{n + 1} \right)$ is a
diagonal matrix such that $0 < a_1 < \cdots < a_{n + 1}$. The degenerate cases
in which the ellipsoid has some symmetry of revolution are not considered
here. This ellipsoid is an element of the family of confocal quadrics given by
\[ Q_{\mu} = \left\{ q \in \Rset^{n + 1} : \left\langle D_{\mu} q, q
   \right\rangle = 1 \right\}, \qquad \mu \in \Rset, \]
where $D_{\mu} = \left( D^{- 1} - \mu \Identity \right)^{- 1}$. The meaning of
$Q_{\mu}$ in the singular cases $\mu \in \left\{ a_1, \ldots, a_{n + 1}
\right\}$ must be clarified. If $\mu = a_j$, we define it as the
$n$-dimensional hyperplane~(\ref{eq:Hyperplane}).

\begin{theorem}\label{thm:Jacobi}
Let $Q$ be the nondegenerate ellipsoid~(\ref{eq:Ellipsoid}).
\begin{enumeratealpha}
\item
Any generic point $q \in \Rset^{n + 1}$ belongs to exactly $n+1$
distinct nonsingular quadrics $Q_{\mu_0},\ldots,Q_{\mu_n}$ such
that $\mu_0 < a_1 < \mu_1 < a_2 < \cdots < a_n < \mu_n < a_{n + 1}$.
Besides, those $n + 1$ quadrics are mutually orthogonal at $q$.
\item
Any generic line $\ell \subset \Rset^{n + 1}$ is tangent to
exactly $n$ distinct nonsingular confocal quadrics
$Q_{\lambda_1},\ldots,Q_{\lambda_n}$ such that
$\lambda_1 < \cdots < \lambda_n$, $\lambda_1 \in (-\infty,a_1) \cup(a_1, a_2)$,
and $\lambda_i \in (a_{i - 1}, a_i) \cup (a_i, a_{i + 1})$,
for $i = 2, \ldots, n$.
\end{enumeratealpha}
\end{theorem}

Set $a_0 = 0$. If a generic point $q$ is in the interior of the ellipsoid $Q$,
then $\mu_1 > 0$, so $a_0 < \mu_0 < a_1$. In the same way, if a generic line
$\ell$ has a transverse intersection with the ellipsoid $Q$, then $\lambda_1 >
0$, so $\lambda_1 \in \left( a_0, a_1 \right) \cup \left( a_1, a_2 \right)$.
The values $\mu_0 = 0$ and $\lambda_1 = 0$ are attained just when $q \in Q$
and $\ell$ is tangent to $Q$, respectively.

We denote by $\mathfrak{q}= \left( \mu_0, \ldots, \mu_n \right) \in
\Rset^{n + 1}$, the \tmtextit{Jacobi elliptic coordinates} of the point
$q = \left( x_1, \ldots, x_{n + 1} \right)$. Cartesian and elliptic
coordinates are linked by relations
\begin{equation}
  x^2_j = \frac{\prod_{i = 0}^n \left( a_j - \mu_i \right)}{\prod_{i \neq j}
  \left( a_j - a_i \right)}, \qquad j = 1, \ldots, n + 1.
  \label{eq:EllipticCoordinatesnD}
\end{equation}
Elliptic coordinates define a coordinate system on each of the $2^{n + 1}$
open orthants of the Euclidean space $\Rset^{n + 1}$, but they become
singular at the $n + 1$ coordinate hyperplanes, because the map $q \mapsto
\mathfrak{q}$ is not one-to-one in any neighborhood of theses hyperplanes.

A point is \tmtextit{generic}, in the sense of Theorem~\ref{thm:Jacobi},
if and only if it is outside all coordinate hyperplanes.
From Eq.~(\ref{eq:EllipticCoordinatesnD}),
we deduce that when the point $q$ tends to the hyperplane $H_j$,
some elliptic coordinate $\mu_i$ tends to $a_j$.
In fact, $i = j$ or $i = j - 1$,
because of the inequalities $a_i < \mu_i < a_{i+1}$.

A line is \tmtextit{generic}, in the sense of Theorem~\ref{thm:Jacobi}, if and
only if it is neither tangent to a singular confocal quadric nor contained in
a nonsingular confocal quadric.

If two lines obey the reflection law at a point $q \in Q$, then both lines are
tangent to the same confocal quadrics. This shows a tight relation between
elliptic billiards and confocal quadrics: all lines of a billiard trajectory
inside the ellipsoid $Q$ are tangent to exactly $n$ confocal quadrics
$Q_{\lambda_1}, \ldots, Q_{\lambda_n} $, which are called
\tmtextit{caustics} of the trajectory.
We will say that $\lambda = (\lambda_1, \ldots, \lambda_n) \in \Rset^n$
is the \tmtextit{caustic parameter} of the trajectory.

\begin{definition}
  A billiard trajectory inside a nondegenerate ellipsoid $Q$ is nonsingular
  when it has $n$ distinct nonsingular caustics; that is, when its caustic
  parameter belongs to the nonsingular caustic space
\[
\Lambda =
\left\{
(\lambda_1, \ldots, \lambda_n) \in \Rset^n :
\begin{array}{cc}
0  < \lambda_1 < \cdots < \lambda_n \\
\lambda_i \in (a_{i-1}, a_i) \cup (a_i, a_{i+1})
\end{array}
\right\}.
\]
\end{definition}

We will only deal with nonsingular billiard trajectories along this paper.

\begin{remark}
  \label{rem:NumberCausticTypes}The set $\Lambda$ has $2^n$ connected
  components, being each one associated to a different type of caustics. For
  instance, in the 2D case ($n = 1$) the two connected components correspond
  to ellipses and hyperbolas.
\end{remark}

\begin{theorem}
  If a nonsingular billiard trajectory closes after $l$ bounces, all
  trajectories sharing the same caustics also close after $l$ bounces.
\end{theorem}

Poncelet and Darboux proved this theorem for ellipses and triaxial ellipsoids of
$\Rset^3$, respectively. Later on, it was generalized to any dimension in
Ref.~\onlinecite{ChangFriedberg1988}.

\subsection{A dual transformation}

Let us present a map $g:M \to M$ that takes place only for
billiards inside ellipsoids.
It exchanges the role of positions and velocities
---so, it could be seen as a dual transformation---,
and was introduced in Refs.~\onlinecite{ChangFriedberg1988,Veselov1988}.
The following explicit formulae are required to define $g$.

\begin{proposition}
The billiard map $f : M \rightarrow M$ inside the ellipsoid~(\ref{eq:Ellipsoid})
is expressed by $f \left( q, p \right) = \left( q', p'
  \right)$, where
  \[ \left\{ \begin{array}{ll}
       q' = q + \mu \left( q, p' \right) p', & \mu \left( q, p' \right) = - 2
       \DS \frac{\langle D q, p' \rangle}{\left\langle D p', p'
       \right\rangle},\\
       p' = p + \nu \left( q, p \right) D q, & \nu \left( q, p \right) = - 2
       \DS \frac{\left\langle D q, p \right\rangle}{\left\langle D q, D q
       \right\rangle} .
     \end{array} \right. \]
  Besides, the reversors $\hat{r}, \tilde{r} : M \rightarrow M$ are given by:
  \[ \left\{ \begin{array}{ll}
       \hat{r} \left( q, p \right) = \left( \hat{q}, - p \right), & \hat{q} =
       q + \mu \left( q, p \right) p,\\
       \tilde{r} \left( q, p \right) = \left( q, \tilde{p} \right), &
       \tilde{p} = - p - \nu \left( q, p \right) D q.
     \end{array} \right. \]
\end{proposition}

\begin{proof}
  The formulae for the billiard map are well known. See, for instance,
  Ref.~\onlinecite{Veselov1988}. Next, to obtain the formula for $\hat{r}$ we recall
  that $\hat{q}$ is the previous impact point, so
  \[ \hat{q} = q + \mu \left( q, - p \right) \left( - p \right) = q + \mu
     \left( q, p \right) p. \]
  Finally, the formula for $\tilde{r}$ follows from $\tilde{p} = - p'$.
\end{proof}

We are ready to introduce the dual transformation $g$.
We check that, in certain sense, $g$ is the square root of $f$.
In particular, $g$ has the same symmetries and reversors as $f$.
From our point of view,
its most useful property is that it interchanges some symmetry sets.

\begin{proposition}
  \label{prop:DualSymmetry}Let $g : M \rightarrow M$, $g \left( q, p \right) =
  \left( \bar{q}, \bar{p} \right)$, where:
  \[ \left\{ \begin{array}{l}
       \bar{q} = Cp' = Cp + \nu \left( q, p \right) C^{- 1} q,\\
       \bar{p} = - C^{- 1} q,
     \end{array} \right. \]
  with $C = D^{- 1 / 2}$.
Then, the following relations hold:
  \begin{enumeratealpha}
    \item $g^2 = - f$, and $f \circ g = g \circ f$.
    
    \item $g \circ s_{\sigma} = s_{\sigma} \circ g$ for all $\sigma \in
    \Sigma$.
    
    \item $\hat{r} \circ g = - g \circ \tilde{r}$, $g \circ \tilde{r} \circ g
    = \tilde{r}$, and $g \circ \hat{r} \circ g = \hat{r}$.
    
    \item $g \left( \FixedSet \left( \tilde{r}_{\sigma} \right) \right) =
    \FixedSet \left( \hat{r}_{- \sigma} \right) = \FixedSet \left( f \circ
    \tilde{r}_{- \sigma} \right)$ for all $\sigma \in \Sigma$.
  \end{enumeratealpha}
\end{proposition}

\begin{proof}
  
  \begin{enumeratealpha}
    \item We observe two relations:
    \begin{eqnarray*}
      \nu \left( \bar{q}, \bar{p} \right) & = & - 2 \frac{\left\langle D
      \bar{q}, \bar{p} \right\rangle}{\left\langle D \bar{q}, D \bar{q}
      \right\rangle} = 2 \frac{\left\langle D q, p' \right\rangle}{\left\langle D
        p', p' \right\rangle}\\
      & = & - \mu \left( q, p' \right) = \mu \left( q, \tilde{p} \right),\\
      \mu \left( \bar{q}, \bar{p} \right) & = & - 2 \frac{\left\langle D
      \bar{q}, \bar{p} \right\rangle}{\left\langle D \bar{p}, \bar{p}
      \right\rangle} = 2 \frac{\left\langle D q, p' \right\rangle}{\left\langle D q,
        D q \right\rangle}\\
      & = & - \nu \left( q, p' \right) = \nu \left( q, \tilde{p} \right) .
    \end{eqnarray*}
    Thus $\nu \circ g = \mu \circ \tilde{r}$ and $\mu \circ g = \nu \circ
    \tilde{r}$, and then $g^2 = - f$:
    \begin{eqnarray*}
      g^2 \left( q, p \right) & = & g \left( \bar{q}, \bar{p} \right)
      = \left( C \bar{p} + \nu \left(
      \bar{q}, \bar{p} \right) C^{- 1}  \bar{q}, - C^{- 1}  \bar{q} \right)\\
      & = & \left( - q - \mu \left( q, p' \right) p', - p' \right)
      = \left( - q', - p' \right) .
    \end{eqnarray*}
    It is immediate that $g$ is odd. Therefore,
    \[ f \circ g = \left( - g^2 \right) \circ g = - g^3 = g \circ \left( - g^2
       \right) = g \circ f. \]
    \item It is obvious.
    
    \item By using a) and definition of $g$:
    \begin{eqnarray*}
      g \left( q, \tilde{p} \right) & = & \left( C \tilde{p} + \nu \left( q,
      \tilde{p} \right) C^{- 1} q, - C^{- 1} q \right)\\
      & = & \left( - \bar{q} - \mu \left( \bar{q}, \bar{p} \right)  \bar{p},
      \bar{p} \right) = - \hat{r} \left( \bar{q}, \bar{p} \right).
    \end{eqnarray*}
    As a consequence $\hat{r} \circ g = - g \circ \tilde{r}$ and besides:
    \begin{eqnarray*}
      \hat{r} & = & f \circ \tilde{r} = - g^2 \circ \tilde{r} = g \circ \hat{r}
      \circ g,\\
      \tilde{r} & = & f \circ \tilde{r} \circ f = \hat{r} \circ f = \hat{r} \circ
      ( -g^2 )\\
      & = & - \hat{r} \circ g^2 = g \circ \tilde{r} \circ g.
    \end{eqnarray*}
    \item Because $g$ is a diffeomorphism we have that:
    \begin{eqnarray*}
      m \in \FixedSet \left( \tilde{r}_{\sigma} \right) & \Longleftrightarrow
      & g \left( m \right) = g \left( \tilde{r}_{\sigma} \left( m \right)
      \right) = - \hat{r}_{\sigma} \left( g \left( m \right) \right)\\
      & \Longleftrightarrow & g \left( m \right) \in \FixedSet \left(
      \hat{r}_{- \sigma} \right) .
    \end{eqnarray*}
    We have used that $- \hat{r}_{\sigma} = \hat{r}_{- \sigma}$. Consequently,
    \[ g \left( \FixedSet \left( \tilde{r}_{\sigma} \right) \right) =
       \FixedSet \left( \hat{r}_{- \sigma} \right) = \FixedSet \left( f
       \circ \tilde{r}_{- \sigma} \right) . \]
  \end{enumeratealpha}
\vspace{-28pt}
\end{proof}

This proposition has a practical corollary.

\begin{corollary}\label{cor:DualCorrespondence}
The dual transformation $g : M \rightarrow M$ gives an explicit one-to-one
correspondence between $\tilde{r}_{\sigma}$-SPOs and
$\left( f \circ \tilde{r}_{-\sigma} \right)$-SPOs of the same period.
\end{corollary}

\begin{proof}
  Let $m \in \FixedSet \left( \tilde{r}_{\sigma} \right)$ be a point of the
  phase space such that its orbit $O$, is $l$-periodic. Thus, $O$ is a
  $\tilde{r}_{\sigma}$-SPO. Let us consider the orbit $\bar{O}$ by the point
  $\bar{m} = g \left( m \right)$. First,
  \[ f^l \left( \bar{m} \right) = f^l \left( g \left( m \right) \right) = g
     \left( f^l \left( m \right) \right) = g \left( m \right) = \bar{m} . \]
  Thus, $\bar{O}$ is $l$-periodic. Second, $\bar{m} = g \left( m \right) \in
  \FixedSet \left( \hat{r}_{- \sigma} \right)$, so $\bar{O}$ is a $\hat{r}_{-
  \sigma}$-SPO. Therefore, we have explicitly constructed the correspondence
  $g : O \mapsto \bar{O}$, which is one-to-one since $g^2 = - f$.
\end{proof}

In subsequent sections we carry out some computations on SPOs only for
reversors $\left\{ \tilde{r}_{\sigma} : \sigma \in \Sigma \right\}$, since
from this corollary, we deduce the analogous results for reversors $\left\{
\hat{r}_{\sigma} = f \circ \tilde{r}_{\sigma} : \sigma \in \Sigma \right\}$.

\subsection{Elliptic coordinates of SPTs}\label{ssec:EllipticCoordinates}

The behaviour of elliptic coordinates along nonsingular billiard trajectories inside
a nondegenerate ellipsoid, can be summarized as follows. If a particle obeys the
billiard dynamics, describing a polygonal curve of $\Rset^{n + 1}$ whose
vertexes are on the ellipsoid, then its $i$-th elliptic coordinate $\mu_i$
oscillates inside some interval $I_i = \left[ c_{2 i}, c_{2 i + 1} \right]$ in such
a way that its only critical points are attained when $\mu_i \in \partial I_i$. In
other words, the oscillation has amplitude $c_{2 i + 1} - c_{2 i}$. This is a
classical result that can be found, for instance, in
Ref.~\onlinecite{DragovicRadnovic2006}. Let us state it as a formal theorem.

In order to describe the intervals $I_i$, we set
\[ \left\{ c_1, \ldots, c_{2 n + 1} \right\} = \left\{ a_1, \ldots, a_{n + 1}
   \right\} \cup \left\{ \lambda_{1,} \ldots, \lambda_n \right\}, \]
once ellipsoid parameters $a_1, \ldots, a_{n + 1}$, and caustic parameters
$\lambda_{1,} \ldots, \lambda_n$, are fixed.

If $\lambda \in \Lambda$, then $c_1, \ldots, c_{2 n + 1}$ are pairwise
distinct and positive, so we can assume that
\[ c_0 := 0 < c_1 < \cdots < c_{2 n + 1} . \]
Then $I_i = \left[ c_{2 i}, c_{2 i + 1} \right]$, $0 \leq i \leq n$, are the
intervals that we were looking for.

\begin{theorem}\label{thm:Oscillations}
Let $q(t)$ be an arc-length parameterization of a nonsingular
billiard trajectory inside the ellipsoid~(\ref{eq:Ellipsoid}),
sharing caustics $Q_{\lambda_1}, \ldots, Q_{\lambda_n}$.
Let $\mathfrak{q}(t) = ( \mu_0(t), \ldots, \mu_n(t))$ be the corresponding
  parameterization in elliptic coordinates. The following properties hold:
  \begin{enumeratealpha}
    \item $c_{2 i} \leq \mu_i \left( t \right) \leq c_{2 i + 1}$ for all $t
    \in \Rset$.
    
    \item Functions $\mu_i \left( t \right)$ are smooth everywhere, except
    $\mu_0 \left( t \right)$ which is nonsmooth at impact points ---that
    is, when $q \left( t_{\star} \right) \in Q$---, in which case
    $\mu'_0 \left( t_{\star} + \right) = - \mu'_0 \left( t_{\star} - \right)
    \neq 0$.
    
    \item \label{item:CriticalPoints}If $\mu_i \left( t \right)$ is smooth at
    $t = t_{\star}$, then
    \[ \mu_i' \left( t_{\star} \right) = 0 \Longleftrightarrow \mu_i \left(
       t_{\star} \right) \in \left\{ c_{2 i}, c_{2 i + 1} \right\} . \]
    \item \label{item:WindingNumbers}If the trajectory is periodic with length
    $L_0$, then $\mathfrak{q} \left( t \right)$ is $L_0$-periodic and there
    exist some positive integers $m_0, \ldots, m_n$, called winding numbers,
    such that $\mu_i \left( t \right)$ makes exactly $m_i$ complete
    oscillations (round trips) inside the interval $I_i = \left[ c_{2 i}, c_{2
    i + 1} \right]$ along one period $0 \leq t \leq L_0$. Besides,
    \begin{enumerateroman}
      \item $m_0$ is the period of the billiard trajectory.
      
      \item $m_i$ is even when $\left\{ c_{2 i}, c_{2 i + 1} \right\} \cap
      \left\{ a_1, \ldots, a_{n + 1} \right\} \neq \emptyset$.
      
      \item \label{item:HalfPeriod}$\mathfrak{q} \left( t \right)$ has period
      $L_0 / 2$ if and only if all winding numbers are even, but it never has
      period $L_0 / 4$.
      
      \item \label{item:NotMultiplesOf4}Not all winding numbers can be
      multiples of four.
    \end{enumerateroman}
  \end{enumeratealpha}
\end{theorem}

Therefore,
all billiard trajectories sharing the caustics $Q_{\lambda_1},\ldots,Q_{\lambda_n}$
are contained in the $\left( n + 1 \right)$-dimensional \tmtextit{cuboid}
\[ \mathcal{} \mathcal{C}_{\lambda} : = I_0 \times \cdots \times I_n = \left[
   0, c_1 \right] \times \cdots \times \left[ c_{2 n}, c_{2 n + 1} \right]
   \subset \Rset^{n + 1}, \]
when they are expressed in elliptic coordinates. Besides, a billiard
trajectory drawn in elliptic coordinates, has elastic reflections with the
$n$-dimensional face $\left\{ \mu_0 = 0 \right\}$, of the cuboid, but inner
tangent contacts with the other $2 n + 1$ faces. This behaviour can be
observed in Table~\ref{tab:SPTs_2D}, where several SPTs are drawn in Cartesian
and elliptic coordinates. The cuboid is just a rectangle in those cases.

Based on extensive numerical experiments,
it has been conjectured\cite{CasasRamirez2011} that
\begin{equation}
  2 \leq m_n < \cdots < m_1 < m_0, \label{eq:Conjecture}
\end{equation}
but we are not aware of any proof.

\section{2D case}\label{sec:2D}

\subsection{Caustics and elliptic coordinates}

We adapt the previous setting of billiards inside ellipsoids of
$\Rset^{n + 1}$ to the 2D case; that is, $n = 1$. Then the configuration
space is an ellipse $Q$, which, in order to simplify the exposition, we write
as
\[
Q =
\left\{
q = (x,y) \in \Rset^2 : \frac{x^2}{a} + \frac{y^2}{b} = 1
\right\},\qquad
a > b > 0.
\]
As we said in Subsec.~\ref{ssec:Caustics}, any nonsingular billiard trajectory
inside $Q$ is tangent to one confocal caustic
\[ Q_{\lambda} = \left\{ q = \left( x, y \right) \in \Rset^2 :
   \frac{x^2}{a - \lambda} + \frac{y^2}{b - \lambda} = 1 \right\}, \]
where the caustic parameter $\lambda$ belongs to the nonsingular caustic space
\[
\Lambda = E \cup H, \qquad
E = \left( 0, b \right), \qquad
H = \left( b, a \right) .
\]
The names of the connected components of $\Lambda$ come from the fact that
$Q_{\lambda}$ is an confocal ellipse for $\lambda \in E$ and a confocal
hyperbola for $\lambda \in H$. The singular cases $\lambda = b$ and $\lambda =
a$ correspond to the $x$-axis and $y$-axis, respectively. We say that the
\tmtextit{caustic type} of a billiard trajectory is E or H, when its caustic
is an ellipse or a hyperbola. We also distinguish between E-caustics and
H-caustics.

\begin{figure}
\includegraphics{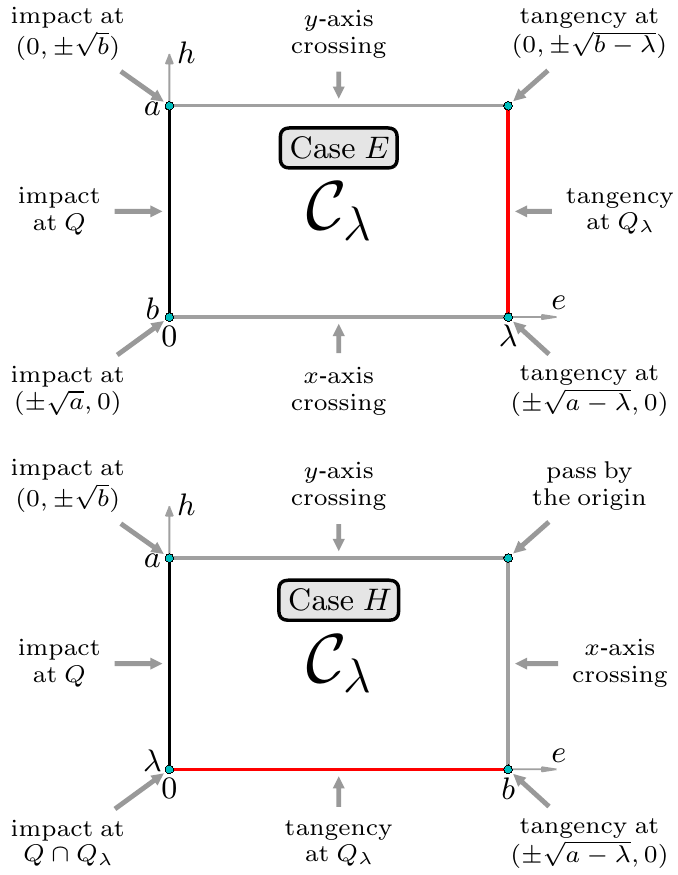}
\caption{\label{fig:SidesVertexes}Geometric meaning of sides and vertexes
       of $\EuScript{\mathcal{C}}_{\lambda}$.}
\end{figure}

We denote by $\mathfrak{q}= \left( e, h \right)$ the Jacobi elliptic
coordinates of the point $q = \left( x, y \right)$. That is, $q$ belongs to
the orthogonal intersection of the confocal ellipse $Q_e$ and the confocal
hyperbola $Q_h$. We recall from Eq.~(\ref{eq:EllipticCoordinatesnD}) that
\[
x^2 = \frac{\left( a - e \right)  \left( a - h \right)}{a - b},\qquad
y^2 = \frac{\left( b - e \right)  \left( h - b \right)}{a - b}.
\]
Besides, $0 < e < b < h < a$ if the point $q$ is contained in the interior of
$Q$, and $e = 0$ at impact points $q \in Q$. The $y$-axis is $\left\{ h = a
\right\}$, the $x$-axis is $\left\{ e = b \right\} \cup \left\{ h = b
\right\}$, and the foci of the ellipse verify $e = h = b$, in elliptic
coordinates. Likewise, the points $\left( \pm x, \pm y \right)$ have
associated the same elliptic coordinates $\left( e, h \right)$, so this system
of coordinates do not distinguish among the four quadrants in $\Rset^2$.

We also know that all billiard trajectories sharing the caustic
$Q_{\lambda_{}}$ are contained in the rectangle
\begin{equation}
  \EuScript{\mathcal{C}}_{\lambda} \mathscr{} = \left[ 0, \min \left(
  b, \lambda \right) \right] \times \left[ \max \left( b, \lambda
  \right), a \right] \subset \Rset^2 \label{eq:Rectangle}
\end{equation}
when they are expressed in terms of $\left( e, h \right)$. Moreover, the
coordinates $e$ and $h$ have a monotone behaviour except at the endpoints of
the intervals that enclose them. The geometric meaning of the sides and
vertexes of the rectangle $\EuScript{\mathcal{C}}_{\lambda}$ is described in
Fig.~\ref{fig:SidesVertexes}. In the figure, black, red and gray sides means
impact with the ellipse $Q$, tangency with the caustic $Q_{\lambda}$, and
crossing some coordinate axis, respectively. This color code is repeated in
Table~\ref{tab:SPTs_2D}.

\subsection{Rotation number and winding numbers}

We recall some concepts related to periodic trajectories of billiards inside
ellipses. These results can be found, for instance, in
Refs.~\onlinecite{CasasRamirez2011, ChangFriedberg1988}.

To begin with, we introduce the function $\rho : \Lambda \rightarrow
\Rset$ given by the quotient of elliptic integrals
\[ \rho \left( \lambda \right) = \frac{\int_0^{\min \left( b, \lambda
   \right)} \frac{\mathrm{d} s}{\sqrt{\left( \lambda - s \right) \left( b - s
   \right) \left( a - s \right)}}}{2 \int_{\max \left( b, \lambda
   \right)}^a \frac{\mathrm{d} s}{\sqrt{\left( \lambda - s \right) \left( b -
   s \right) \left( a - s \right)}}} . \]
It is called the \tmtextit{rotation number} and characterizes
the caustic parameters that give rise to periodic trajectories.
To be precise, the billiard trajectories sharing the caustic $Q_{\lambda}$ are
periodic if and only if
\[ \rho \left( \lambda \right) = m_1 / 2 m_0 \in \Qset \]
for some integers $2 \leq m_1 < m_0$, which are the \tmtextit{winding
numbers}. It turns out that $m_0$ is the period, $m_1 / 2$ is the number of
turns around the ellipse $Q_{\lambda}$ for E-caustics, and $m_1$ is the number
of crossings of the $y$-axis for H-caustics. Thus, $m_1$ is always even.
Besides, all periodic trajectories with H-caustics have even period.

\subsection{Reversors and their symmetry sets}\label{ssec:SymmetrySets_2D}

We also change the notations for reversors. We denote by $R$ the reversor
previously written as $\tilde{r}$. Besides, $R_x$, $R_y$, and $R_{x y}$ are
the reversors obtained by composing $R$ with the symmetries defined on the
phase space and associated to the reflections $\left( x, y \right) \mapsto
\left( - x, y \right)$, $\left( x, y \right) \mapsto \left( x, - y \right)$,
and $\left( x, y \right) \mapsto \left( - x, - y \right)$, respectively.
Finally, $f \circ R$, $f \circ R_x$, $f \circ R_y$, and $f \circ R_{x y}$
denote the four reversors of the form $\left\{ \hat{r}_{\sigma} = f \circ
\tilde{r}_{\sigma} : \sigma \in \Sigma \right\}$.
We proved in Proposition~\ref{prop:SymmetrySets} that the symmetry sets
of $R_{x y}$ and $f \circ R$ are empty sets,
and following the same proposition we get:
\begin{itemizedot}
  \item $\FixedSet ( R ) = \{ ( q, p ) \in M : p
  \in N_q Q \}$,
  
  \item $\FixedSet ( R_x ) = \{ (q, p) \in M : q
  = ( 0, \pm \sqrt{b} ) \}$,
  
  \item $\FixedSet ( R_y ) = \{ ( q, p ) \in M : q
  = ( \pm \sqrt{a}, 0 ) \}$,
  
  \item $\FixedSet ( f \circ R_x ) = \{ ( q, p )
  \in M : p = ( \pm 1, 0 ) \}$,
  
  \item $\FixedSet ( f \circ R_y ) = \{ ( q, p )
  \in M : p = ( 0, \pm 1 ) \}$, and
  
  \item $\FixedSet ( f \circ R_{xy} ) = \{ ( q, p
  ) \in M : 0 \in q + \langle p \rangle \}$.
\end{itemizedot}
Here,
$M =
 \{ (q, p) \in Q \times \Sset:
 \mbox{$p$ directed outward $Q$ at $q$}
 \}$ is the phase space.
We write $q = (x,y)$ and $p = (u,v)$.

\begin{table}\caption{\label{tab:Pointsqp_ST_2D}
Points $(q, p) \in \FixedSet(r)$ such that the line $q + \langle p \rangle$
is tangent to the caustic $Q_{\lambda}$.}
\begin{ruledtabular}
\begin{tabular}{ll}
     $r$ & $q = (x,y)$, $p = (u,v)$ \\
     [0.05cm]
     \hline \\
     [-0.3cm]
     $R$ &
     $x^2 = \frac{a ( a - \lambda)}{a - b}$,
     $y^2 = \frac{b ( \lambda - b)}{a - b}$,
     $p = \mu (\frac{x}{a}, \frac{y}{b})$,
     $\mu = \sqrt{a b/\lambda}$ \\
     $R_x$ &
     $q = (0, \pm \sqrt{b})$, $u^2 = (a - \lambda)/a$, $v^2 = \lambda/a$ \\
     $R_y$ &
     $q = (\pm\sqrt{a}, 0)$, $u^2 = \lambda/b$, $v^2 = (b - \lambda)/b$ \\
     $f \circ R_x$ &
     $x^2 = a \lambda/b$, $y^2 = b - \lambda$, $p = (\pm 1, 0)$ \\
     $f \circ R_y$ &
     $x^2 = a - \lambda$, $y^2 = b \lambda/a$, $p = ( 0, \pm 1)$ \\
     $f \circ R_{xy}$ &
     $u^2 = \frac{a - \lambda}{a - b}$,
     $v^2 = \frac{\lambda - b}{a - b}$,
     $q = \sqrt{a b / \lambda} p$\\
\end{tabular}
\end{ruledtabular}
\end{table}

\subsection{Characterization of STs}

\begin{table}
\caption{\label{tab:Vertexes_2D}Points in STs that correspond to
   vertexes of $\mathcal{C}_{\lambda}$.}
\begin{ruledtabular}
\begin{tabular}{cccc}
     Vertex $\left( e, h \right)$ & $q_{\star}$ belongs to & Reversor & Type \\
     \hline
     $\left( 0, \lambda \right)$ & $Q \cap Q_{\lambda}$ & $R$ & H\\
     $\left( 0, a \right)$ & $Q \cap \left\{ y\textrm{-axis} \right\}$
     & $R_x$ & any\\
     $\left( 0, b \right)$ & $Q \cap \left\{ x\textrm{-axis} \right\}$
     & $R_y$ & E\\
     $\left( \lambda, a \right)$ & $Q_{\lambda} \cap \left\{ y\textrm{-axis} \right\}$ & $f \circ R_x$ & E\\
     $\left( \lambda, b \right)$ or $\left( b, \lambda \right)$ & $Q_{\lambda}
     \cap \left\{ x\textrm{-axis} \right\}$ & $f \circ R_y$ & any\\
     $\left( b, a \right)$ & $\left\{ \left( 0, 0 \right) \right\}$ & $f \circ
     R_{x y}$ & H\\
   \end{tabular}
\end{ruledtabular}
\end{table}

In order to find caustics $Q_{\lambda}$, whose tangent trajectories are
periodic, we solve equation $\rho \left( \lambda \right) = m_1 / 2 m_0$, for
some winding numbers $2 \leq m_1 < m_0$. Hence, once we fix a caustic
parameter $\lambda$ such that $\rho \left( \lambda \right) = m_1 / 2 m_0 \in
\Qset$, and a reversor $r : M \rightarrow M$, it is natural to look for
points $\left( q, p \right) \in \FixedSet \left( r \right)$ such that $q +
\left\langle p \right\rangle$ is tangent to $Q_{\lambda}$. If we can find
those points, billiard trajectories associated to them are $r$-SPTs of period
$m_0$.

In the 2D case there are only $6 \left( = 2^3 - 2 \right)$ nonempty symmetry sets
where one can look for STs; see previous subsection. We write the six formulae for
$\left( q, p \right)$ in Table~\ref{tab:Pointsqp_ST_2D}. The cases where $p$ is
directed inward $Q$ at $q$ are excluded from this table, because we require $\left(
q, p \right) \in M$. Let us check the first $3$ cases in
Table~\ref{tab:Pointsqp_ST_2D}.
The formulae for $f \circ R_x$, $f \circ R_y$ and $f \circ
R_{xy}$ follow from their dual reversors $R_y$, $R_x$ and $R_{}$, respectively. See
Corollary~\ref{cor:DualCorrespondence}. Suffice it to perform the changes $u^2
\leftrightarrow x^2 / a$ and $v^2 \leftrightarrow y^2 / b$ in the first formulae.
This has to do with the fact that the dual transformation $g : M \rightarrow M $,
$g \left( q, p \right) = \left( \bar{q}, \bar{p} \right)$, $\bar{p} = \left(
\bar{u}, \bar{v} \right) $ verifies that $\bar{u} = - x / \sqrt{a}$ and
$\bar{v} = - y / \sqrt{b}$.

\begin{descriptioncompact}
  \item[$\boldsymbol{R}.$] We look for $\left( q, p \right) \in M$ such that $p \in N_q Q$
  and $q + \left\langle p \right\rangle$ is tangent to $Q_{\lambda}$. If $p
  \in N_q Q$, then there exists $\, \mu > 0$ such that $p = \mu \left( x / a,
  y / b \right)$. If, in addition, $q + \left\langle p \right\rangle$ is
  tangent to $Q_{\lambda} $, then $Q_{\lambda}$ is a hyperbola and $q
  \in Q \cap Q_{\lambda}$. Here, we have used that confocal ellipses and
  confocal hyperbolas are mutually orthogonal. Thus, the elliptic coordinates
  of $q = \left( x, y \right)$ are $\left( e, h \right) = \left( 0, \lambda
  \right)$, and so $x^2 = a \left( a - \lambda \right) / \left( a - b \right)$
  and $y^2 = b \left( \lambda - b \right) / \left( a - b \right)$. Finally, we
  recall that $p$ is a unit velocity:
  \[ 1 = u^2 + v^2 = \mu^2  \left( \frac{x^2}{a^2} + \frac{y^2}{b^2} \right) =
     \frac{\lambda \mu^2}{ab} \Longrightarrow \mu^2 = \frac{ab}{\lambda} . \]
  \item[$\boldsymbol{R_x}.$] We look for a velocity $p = \left( u, v \right) \in \Sset$ such
    that $q + \left\langle p \right\rangle$ has exactly one point on $Q_{\lambda}$
    for $q = ( 0, \pm \sqrt{b} )$. This condition turns into the uniqueness of
    solutions in $t$ of
  \[ \alpha t^2 + 2 \beta t + \gamma = 0, \]
  where
  \[ \alpha = \frac{u^2}{a - \lambda} + \frac{v^2}{b - \lambda},\qquad
     \beta = \frac{\pm \sqrt{b} v}{b - \lambda},\qquad
     \gamma = \frac{\lambda}{b - \lambda}, \]
  or equivalently $\beta^2 = \alpha \gamma$. This last equality together with
  $u^2 + v^2 = 1$ gives $v^2 = \lambda / a$ and $u^2 = \left( a - \lambda
  \right) / a$.
  
  \item[$\boldsymbol{R_y}.$] It is obtained directly from the previous case by exchanging
  the role of coordinates $x$ and $y$.
\end{descriptioncompact}

Next, we characterize nonsingular STs inside noncircular ellipses, as the
trajectories passing, in elliptic coordinates, through some vertex of the
rectangle $\mathcal{C}_{\lambda}$.

We only have to check that any vertex $\mathfrak{q}_{\star} \in
\mathcal{C}_{\lambda}$, corresponds to some point $q_{\star} \in q +
\left\langle p \right\rangle$, such that $q + \left\langle p \right\rangle$ is
tangent to $Q_{\lambda}$ and $\left( q, p \right) \in \FixedSet \left( r
\right)$ for some reversor $r$. We list such correspondences in
Table~\ref{tab:Vertexes_2D}. The point $q_{\star}$ is not unique, because elliptic
coordinates do not distinguish among the four quadrants in $\Rset^2$. In
its last column we describe the feasible reversors for each caustic type; see
Subsec.~\ref{ssec:ForbiddenReversors_2D} for more detailed information.
For instance, let us focus on the vertex $\mathfrak{q}_{\star} = \left( 0,
\lambda \right)$. It is a vertex of $\mathcal{C}_{\lambda}$ if and only if $b
< \lambda < a$, so the type of caustic is H. We note that $q_{\star} \in Q$
since $e = 0$, and $q_{\star} \in Q_{\lambda}$ since $h = \lambda$. Let $p \in
\Sset$ be the unique outward velocity such that $q_{\star} +
\left\langle p \right\rangle$ is tangent to $Q_{\lambda}$. Then $p \in
N_{q_{\star}} Q$, since the confocal hyperbola $Q_{\lambda}$ has a orthogonal
intersection at $q_{\star}$ with the ellipse $Q$. Hence, $\left( q_{\star}, p
\right) \in \FixedSet \left( R \right)$. The other cases are similar. We only
observe that $e \neq 0$ in the last three cases, so $q_{\star} \not\in Q$; but
$q_{\star}$ is the middle point of two consecutive impact points.

\begin{table}
\caption{\label{tab:FeasibleReversors_2D}
Feasible and forbidden reversors for each caustic type.}
\begin{ruledtabular}
\begin{tabular}{ccc}
Type &  Feasible &  Forbidden \\
[0.03cm] \hline \\ [-0.3cm]
E & $R_x$, $R_y$, $f \circ R_x$, $f \circ R_y$ & $R$, $f \circ R_{xy}$ \\
H & $R$, $R_x$, $f \circ R_y$, $f \circ R_{xy}$ & $R_y$, $f \circ R_x$ \\
\end{tabular}
\end{ruledtabular}
\end{table}

\subsection{Forbidden reversors for each type of caustic}
\label{ssec:ForbiddenReversors_2D}

Although there are 6 nonempty symmetry sets, once the caustic type (E or H) is
fixed, only 4 of them give rise to SPTs of that type. The correspondences between
caustic types and reversors whose symmetry sets are feasible or forbidden are listed
in Table~\ref{tab:FeasibleReversors_2D}. To check that all couples caustic
type/feasible reversor take really place,
we show some examples in Table~\ref{tab:SPTs_2D}.

We can prove Table~\ref{tab:FeasibleReversors_2D} in three different ways.
Firstly, by means of analytical arguments based on the formulae listed in
Table~\ref{tab:Pointsqp_ST_2D}. Secondly, we could deduce it by using elliptic
coordinates, as shown in Table~\ref{tab:Vertexes_2D}. Finally, we could write
a geometric proof. The first way is the simplest one.
If $(q,p) \in \FixedSet(R)$, then $y^2 = b (\lambda - b)/(a - b)$,
so $\lambda > b$ and $Q_{\lambda}$ is a hyperbola.
If $(q,p) \in \FixedSet (f \circ R_{x y})$,
then $v^2 = (\lambda - b)/(a - b)$,
so $\lambda > b$ and $Q_{\lambda}$ is a hyperbola.
If $(q,p) \in \FixedSet(R_y)$,
then $v^2 = (b - \lambda)/b$,
so $\lambda < b$ and $Q_{\lambda}$ is an ellipse.
If $(q,p) \in \FixedSet(f \circ R_x)$,
then $y^2 = b - \lambda$, so $\lambda < b$ and
$Q_{\lambda}$ is an ellipse.
This ends the proof.

\subsection{Characterization and classification of SPTs}
\label{ssec:12TypesSPTs}

\begin{table}
\caption{\label{tab:Propertiesqp_2D}
Symmetry properties of the billiard orbits $(q_j, p_j) = f^j(q,p)$
such that $(q,p) \in \FixedSet(r)$.}
\begin{ruledtabular}
\begin{tabular}{@{\hspace{0.6cm}}l@{\hspace{0.6cm}}l@{\hspace{0.6cm}}}
$r$ & Symmetry properties \\ \hline \\ [-0.3cm]
$R$ & $q_{- j} = q_j$, $p_{- j + 1} = - p_j$\\
$R_x$ & $q_{- j} = ( - x_j, y_j )$, $p_{- j} = ( u_{j+1}, - v_{j + 1} )$\\
$R_y$ & $q_{- j} = ( x_j, - y_j )$, $p_{- j} = ( - u_{j+1}, v_{j + 1} )$\\
$f \circ R_x$ & $q_{- ( j + 1 )} = ( - x_j, y_j )$, $p_{- j} = ( u_j, - v_j )$\\
$f \circ R_y$ & $q_{- ( j + 1 )} = ( x_j, - y_j )$, $p_{- j} = ( - u_j, v_j )$\\
$f \circ R_{xy}$ & $q_{- ( j + 1 )} = - q_j$, $p_{- j} = p_j$
\end{tabular}
\end{ruledtabular}
\end{table}

Next, we give a complete classification of nonsingular SPTs inside noncircular
ellipses. We classify them by caustic type (E or H) and symmetry sets they
intersect. We also characterize them as trajectories connecting, in elliptic
coordinates, two vertexes of the rectangle $\mathcal{C}_{\lambda}$.

Along this subsection, $O$ denotes a $r$-SO through a point $( q, p ) \in
\FixedSet (r)$, for some reversor $r : M \to M$. We set
$( q_j, p_j ) = f^j ( q, p )$, $q_j = ( x_j, y_j)$, and $p_j = ( u_j, v_j )$.
The properties listed in Table~\ref{tab:Propertiesqp_2D} are easily deduced using
geometric arguments from the symmetry of the ellipse. For instance, let us explain
the case of reversor $R_y$. If $( q, p ) \in \FixedSet ( R_y
)$, then $x_0 = \pm \sqrt{a}$ and $y_0 = 0$. Let us assume that $x_0 =
\sqrt{a}$. Then it is geometrically evident that $q_{- j} = ( x_j, - y_j
)$ and the billiard configuration is symmetric with respect to the $x$-axis;
see Fig.~\ref{fig:Ry-ST}.

We note that all properties concerning velocities $p_j$ can be directly deduced from
the properties of impact points $q_j$, by using identity~(\ref{eq:Relation-qp}).

Below, we suppose $O$ is periodic as well, with winding numbers $2 
\leq m_1 < m_0$, $m_0$ being its period. We need a technical lemma on SPOs of
even period.

\begin{lemma}
  \label{lem:Technical_EH}Let us assume that the period is even: $m_0 = 2 l$,
  and the reversor is $r = R_x$, so that $x_0 = 0$.
  \begin{enumeratealpha}
    \item If the caustic is an ellipse, $q_l = - q_0$ and $p_l = - p_0$.
    Thus, $q_{l + j} = - q_j$ and $p_{l + j} = - p_j$, for all $j
    \in \Zset$.
    
    \item If the caustic is a hyperbola, then $q_l = \left( - 1 \right)^l q_0$
    and
    \[ p_l = \left\{ \begin{array}{ll}
         \left( u_0, - v_0 \right), & \text{if } l \text{ odd and } m_1 / 2
         \text{ even,}\\
         \left( - u_0, v_0 \right), & \text{if } l \text{ even,}\\
         \left( - u_0, - v_0 \right), & \text{if } l \text{ odd and } m_1 / 2
         \text{ odd.}
       \end{array} \right. \]
  \end{enumeratealpha}
\end{lemma}

\begin{proof}
  The hypothesis $r = R_x$ in a), implies $x_{- j} = - x_j$;
  see Table~\ref{tab:Propertiesqp_2D}.
  Then, as $m_0 = 2 l$, we have $x_l = x_{- l} = - x_l$,
  so $x_l = 0$ and $q_l \in \left\{ q_0, - q_0 \right\}$.
  If $q_l = q_0$, then $p_l = p_0$, since the line $q_l + \left\langle p_l
  \right\rangle$ must be tangent to the ellipse $Q_{\lambda}$ and the billiard
  trajectory turns around $Q_{\lambda}$ in a constant ---clockwise or
  anticlockwise--- direction. But the period cannot be smaller than
  $m_0$ by hypothesis. Hence, $q_l = - q_0$. Finally, the identity $p_l = -
  p_0$ is necessary to keep the rotation direction around $Q_{\lambda}$. The
  formulae for a general integer $j$ are obvious.
  
  We skip the proof for H-caustics in b): it is similar.
\end{proof}

Following we prove that all SPOs inside a noncircular ellipse are doubly SPOs
and that there are 12 classes of SPOs, as claimed in the introduction. We
recall that SPOs are classified by caustic type and intersected symmetry sets.

\begin{theorem}
  \label{thm:SPTs_2D}All SPOs inside a noncircular ellipse are doubly SPOs.
  SPTs inside noncircular ellipses are characterized as trajectories
  connecting, in elliptic coordinates, two different vertexes of their
  rectangle $\mathcal{C}_{\lambda}$. There are 12 classes of SPTs, listed in
  Table~\ref{tab:SPTs_2D}.
\end{theorem}

\begin{proof}
  Let $O$ be a $r$-SPO through a point $\left( q_{}, p_{} \right) \in
  \FixedSet \left( r \right)$ with winding numbers $2 \leq m_1 < m_0$, being
  $m_0$ its period. We must prove that there exists another reversor
  $\breve{r} \neq r$ and some index $j \in \Zset$ such that $\left( q_j,
  p_j \right) \in \FixedSet \left( \breve{r} \right)$.
  
  First, let us consider the caustic type E.
  Then, according to Table~\ref{tab:FeasibleReversors_2D},
  we have only four possibilities:
  \begin{equation}
    r \in \left\{ R_x, f \circ R_x, R_y, f \circ R_y \right\} .
    \label{eq:4FeasibleReversors}
  \end{equation}
  The easiest situation is when $m_0 = 2 k + 1$, because we know from the last
  item in Theorem~\ref{thm:SymmetricOrbits} that
\[
(q,p) \in \FixedSet(R_x)
\Longleftrightarrow
(q_{k + 1},p_{k + 1}) \in \FixedSet(f \circ R_x),
\]
  and the same occurs with the couple of associated reversors $R_y$ and
  $f \circ R_y$. This proves the first row of Table~\ref{tab:SPTs_2D}.
  The formulae in the last column are deduced from the symmetry sets in
  Subsec.~\ref{ssec:SymmetrySets_2D}.

\begin{figure}
\resizebox{0.8\columnwidth}{!}{\includegraphics{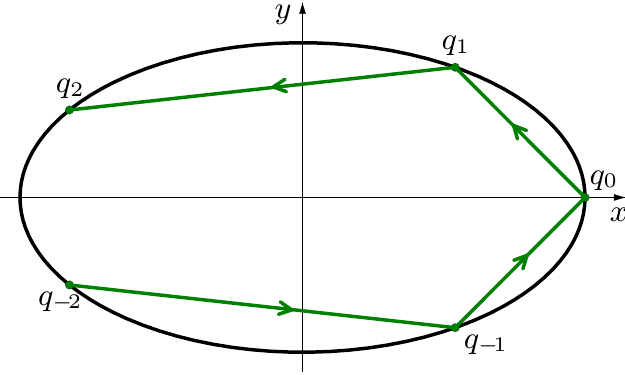}}
\caption{\label{fig:Ry-ST}A $R_y$-ST.}
\end{figure}
  
  Next, we assume that $m_0 = 4 k$. That is, we deal with the second row of
  Table~\ref{tab:SPTs_2D}. We study each one of the four feasible reversors
  listed in Eq.~(\ref{eq:4FeasibleReversors}).
  \begin{itemizedot}
    \item If $r = R_x$, then $q_0 \in \left\{ y \text{-axis} \right\}$, so
    $x_0 = 0$. Thus, since $m_0 = 4 k$, we have $y_{3 k} = y_{- k} = y_k = -
    y_{3 k}$. The last equality follows from Lemma~\ref{lem:Technical_EH}a)
    for $l = 2 k$ and $j = k$. The above equalities imply $y_{3 k} = y_k = 0$
    and thus we can take $\breve{r} = R_y$ and $j \in \left\{ k, 3 k
    \right\}$.
    
    \item The case $r = R_y$ is similar. Suffice it to exchange variables $x$
    and $y$.
    
    \item The cases $r \in \left\{ f \circ R_x, f \circ R_y \right\}$ follow
    directly from the two previous ones by using
    Corollary~\ref{cor:DualCorrespondence}.
  \end{itemizedot}
  Finally, we assume that $m_0 = 4 k + 2$, which corresponds to the third row
  of Table~\ref{tab:SPTs_2D}. As before, the four feasible reversors listed in
  Eq.~(\ref{eq:4FeasibleReversors}) are studied separately.
  \begin{itemize}
    \item If $r = R_x$, then $q_0 \in \left\{ y \text{-axis} \right\}$, so
    $x_0 = 0$. Thus, since $m_0 = 4 k + 2$, we have $u_{3 k + 2} = u_{- k} =
    u_{k + 1} = - u_{3 k + 2}$. The last equality follows from
    Lemma~\ref{lem:Technical_EH}a) for $l = 2 k + 1$ and $j = k + 1$. These
    equalities imply $u_{3 k + 2} = u_{k + 1} = 0$ and thus we can take
    $\breve{r} = f \circ R_y$ and $j \in \left\{ k + 1, 3 k + 2 \right\}$.
    
    \item The case $r = f \circ R_y$ follows directly from the case $r = R_x$
    by using Corollary~\ref{cor:DualCorrespondence}.
    
    \item The cases $r \in \left\{ R_y, f \circ R_x \right\}$ are similar to
    the previous ones. Suffice it to exchange variables $x$ and $y$.
  \end{itemize}
  The proofs for the three rows about H-caustics follow the same lines, but
  using Lemma~\ref{lem:Technical_EH}b). We skip the details.
  
  Once we know that all SPOs are doubly SPOs, we deduce that, any doubly SPT
  connects two different vertexes, namely, the ones corresponding to the
  couple of reversors of the doubly SPT. Indeed, we simply take into account
  the 1-to-1 correspondence between reversors and vertexes of rectangles
  $\mathcal{C}_{\lambda}$, as listed in Table~\ref{tab:Vertexes_2D}. Besides,
  there are $12 = 2 \times 6$ classes of SPTs, since there are two types of
  caustic (E and H) and any rectangle has four vertexes, and so, six couples
  of vertexes.
  
  Finally, we check that the $12$ classes are \tmtextit{realizable}. That is,
  we find one SPT of each class. This can be achieved by properly choosing the
  winding numbers. See Table~\ref{tab:SPTs_2D}.
\end{proof}

\subsection{Minimal SPTs}

We end the study of SPTs inside ellipses by showing an SPT of each one of the
12 classes with the smallest period $m_0$. We call \tmtextit{minimal} such
SPTs. Therefore, $m_0 \in \left\{ 3, 4, 6 \right\}$ for E-caustics, and $m_0
\in \left\{ 4, 6 \right\}$ for H-caustics. Period $m_0 = 2$ is discarded,
since two-periodic billiard trajectories are singular. The 12 minimal SPTs are
drawn in Table~\ref{tab:SPTs_2D}, both in Cartesian and elliptic coordinates.
All of them connect \tmtextit{two} different vertexes of the rectangle
$\EuScript{\mathcal{C}}_{\lambda}$, as claimed in Theorem~\ref{thm:SPTs_2D}.

We realize that any minimal SPT of type E has a ``twin'' of type H when both
are depicted in elliptic coordinates.
This has to do with the fact that both trajectories have
the same rotation number and connect the same vertexes.

\begin{table*}
  \caption{\label{tab:SPTs_2D}Classification and symmetry properties of SPTs
    together to the minimal SPTs in Cartesian and elliptic coordinates. Black, red,
    grey, blue and green colors agree in both coordinates. The first three colors
    also agree with the ones displayed in Fig.~\ref{fig:SidesVertexes}.}
\begin{tabular}{*{4}{c|}c}
  Type & Period & \begin{tabular}{c}Minimal SPTs in\\Cartesian coordinates
  \end{tabular} & \begin{tabular}{c}Minimal SPTs in\\elliptic coordinates
  \end{tabular} & \begin{tabular}{c}Couples of reversors\\and symmetry properties
  \end{tabular} \\\hline
  \multirow{15}*{E} & $2 k + 1$ &
  \begin{tabular}{c}\includegraphics{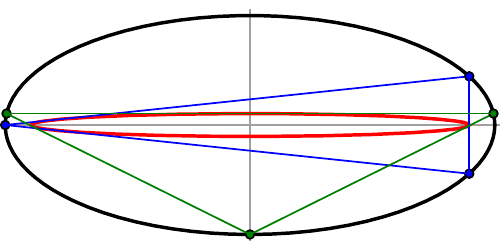}\end{tabular} &
  \begin{tabular}{c}\includegraphics{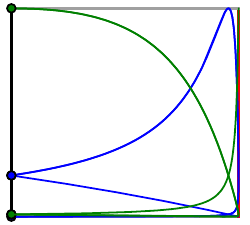}\end{tabular} &
  \begin{tabular}{c}
    {\color{dark green} $R_x$, $f \circ R_x$}\\
    {\color{dark green} }{\color{dark green} $q_0 = q_{2 k + 1} \in \left\{
      y \text{-axis} \right\}$}\\
    {\color{dark green} $p_{k + 1} \left| \right| \left\{ x \text{-axis}
      \right\}$}\\\\
    {\color{blue} $R_y$, $f \circ R_y$}\\
    {\color{blue} $q_0 = q_{2 k + 1} \in \left\{ x \text{-axis} \right\}$}\\
    {\color{blue} $p_{k + 1} \left| \right| \left\{ y \text{-axis}
      \right\}$}
  \end{tabular}\\\cline{2-5}\noalign{\vspace{0.5pt}}
  & $4 k$ &
  \begin{tabular}{c}\includegraphics{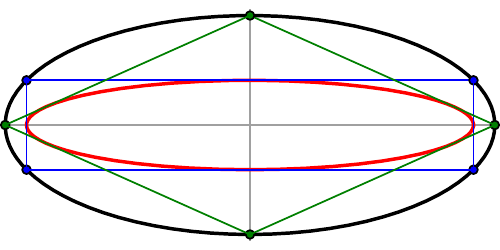}\end{tabular} &
  \begin{tabular}{c}\includegraphics{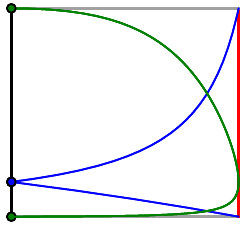}\end{tabular} &
  \begin{tabular}{c}
    {\color{dark green} $R_x$, $R_y$}\\
    {\color{dark green} }{\color{dark green} $q_0 = q_{4 k} = - q_{2 k} \in
      \left\{ y \text{-axis} \right\}$}\\
    {\color{dark green} $q_{3 k} = - q_k \in \left\{ x \text{-axis}
      \right\}$}\\
    \\
      {\color{blue} $f \circ R_x$, $f \circ R_y$}\\
      {\color{blue} $p_0 = p_{4 k} = - p_{2 k} \left| \right| \left\{ x
        \text{-axis} \right\}$}\\
      {\color{blue} $p_{3 k} = - p_k \left| \right| \left\{ y \text{-axis}
        \right\}$}
  \end{tabular}\\\cline{2-5}\noalign{\vspace{0.5pt}}
    & $4 k + 2$ &
  \begin{tabular}{c}\includegraphics{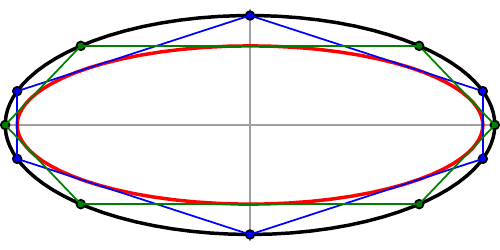}\end{tabular} &
  \begin{tabular}{c}\includegraphics{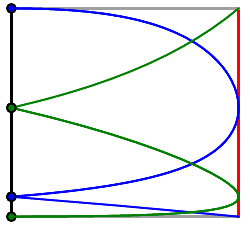}\end{tabular} &
  \begin{tabular}{c}
    {\color{dark green} $R_y$, $f \circ R_x$}\\
    {\color{dark green} $q_0 = q_{4 k + 2} = - q_{2 k + 1} \in \left\{ x
      \text{-axis} \right\}$}\\
    {\color{dark green} $p_{3 k + 2} = - p_{k + 1} \left| \right| \left\{ x
      \text{-axis} \right\}$}\\\\
    {\color{blue} $R_x$, $f \circ R_y$}\\
    {\color{blue} $q_0 = q_{4 k + 2} = - q_{2 k + 1} \in \left\{ y
      \text{-axis} \right\}$}\\
    {\color{blue} $p_{3 k + 2} = - p_{k + 1} \left| \right| \left\{ y
      \text{-axis} \right\}$}
  \end{tabular}\\\hline
  \multirow{17}{*}{H} &
  \begin{tabular}{l}$4 k + 2$\\\\$\DS\frac{m_1}{2}$ even\end{tabular} &
  \begin{tabular}{c}\includegraphics{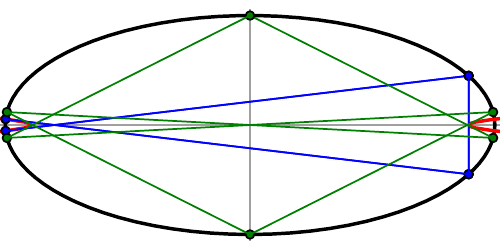}\end{tabular} &
  \begin{tabular}{c}\includegraphics{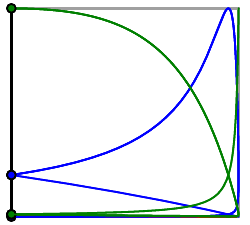}\end{tabular} &
  \begin{tabular}{c}
    {\color{dark green} $R_x$, $f \circ R_{x y}$}\\
    {\color{dark green} $q_0 = q_{4 k + 2} = - q_{2 k + 1} \in \left\{ y
      \text{-axis} \right\}$}\\
    {\color{dark green} $q_{k + 1} = - q_k$, $q_{3 k + 2} = - q_{3 k +
        1}$}\\
    {\color{dark green} $x_{3 k + 1} = x_k$, $y_{3 k + 1} = - y_k$}\\
    {\color{blue} $R$, $f \circ R_y$}\\
    {\color{blue} $q_0 = \left( x_0, y_0 \right) = q_{4 k + 2} \in Q \cap
      Q_{\lambda}$}{\color{blue} }\\
    {\color{blue} $q_{2 k + 1} = \left( x_0, - y_0 \right) \in Q \cap
      Q_{\lambda}$}\\
    {\color{blue} $p_{3 k + 2} = - p_{k + 1} \left| \right| \left\{ y
      \text{-axis} \right\}$}
  \end{tabular}\\\cline{2-5}
  & $4 k$ &
  \begin{tabular}{c}\includegraphics{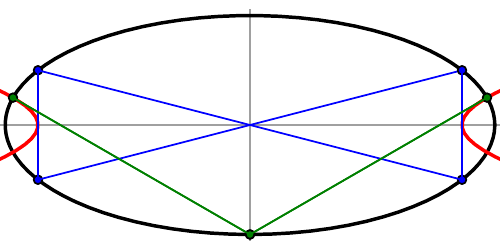}\end{tabular} &
  \begin{tabular}{c}\includegraphics{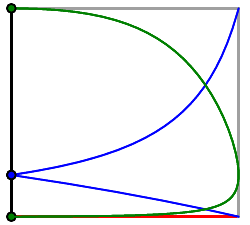}\end{tabular} &
  \begin{tabular}{c}
    {\color{dark green} $R$, $R_x$}\\
    {\color{dark green} $q_0 = q_{4 k} = q_{2 k} \in \left\{ y \text{-axis}
      \right\}$}\\
    {\color{dark green} $q_k = \left( x_k, y_k \right) \in Q \cap
      Q_{\lambda}$}\\
    {\color{dark green} $q_{3 k} = \left( - x_k, y_k \right) \in Q \cap
      Q_{\lambda}$}\\
    {\color{blue} $f \circ R_y$, $f \circ R_{x y}$}\\
    {\color{blue} $q_0 = q_{4 k} = - q_1, q_{2 k + 1} = - q_{2 k}$}\\
    {\color{blue} $x_{2 k} = - x_0, y_{2 k} = y_0$}\\
    {\color{blue} $p_{3 k} = - p_k \left| \right| \left\{ y \text{-axis}
      \right\}$}
  \end{tabular}\\\cline{2-5}\noalign{\vspace{0.5pt}}
  & \begin{tabular}{l}$4 k + 2$\\\\$\DS\frac{m_1}{2}$ odd \end{tabular} &
  \begin{tabular}{c}\includegraphics{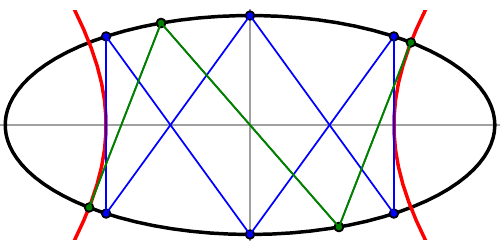}\end{tabular} &
  \begin{tabular}{c}\includegraphics{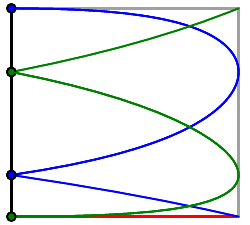}\end{tabular} &
  \begin{tabular}{c}
    {\color{dark green} $R$, $f \circ R_{x y}$}\\
    {\color{dark green} $q_0 = q_{4 k + 2} = - q_{2 k + 1} \in Q \cap
      Q_{\lambda}$}\\
    {\color{dark green} $q_{3 k + 2} = q_k = - q_{k + 1} = - q_{3 k + 1}$}\\\\
    {\color{blue} $R_x$, $f \circ R_y$}\\
    {\color{blue} $q_0 = q_{4 k + 2} = - q_{2 k + 1} \in \left\{ y
      \text{-axis} \right\}$}\\
    {\color{blue} $p_{3 k + 2} = - p_{k + 1} \left| \right| \left\{ y
      \text{-axis} \right\}$}
  \end{tabular}\\\hline
\end{tabular}
\end{table*}

Any SPT, minimal or not, can be computed as follows. First, we fix a caustic
type (E or H), a reversor $r$, a period $m_0$ (which must be even for
H-caustics), and an even winding number $m_1$, such that $2  \leq m_1
< m_0$. We have listed in Table~\ref{tab:SPTs_2D} all possible combinations.
Then we find the caustic parameter through numerical inversion of the relation
$\rho \left( \lambda \right) = m_1 / 2 m_0$. Of course, we solve that equation
for $\lambda \in E$ or $\lambda \in H$, depending on the caustic type we are
looking for. Finally, we take from Table~\ref{tab:Pointsqp_ST_2D} the starting
point $\left( q, p \right) \in \FixedSet \left( r \right)$ such that $q +
\left\langle p \right\rangle$ is tangent to $Q_{\lambda}$, and iterating the
billiard map $f : M \rightarrow M$, we get the whole SPO in the phase space
$M$ and the whole SPT in $\Rset^2$.

Any $R$-SPT is travelled twice in opposite directions, since it hits
orthogonally the ellipse at points $q \in Q \cap Q_{\lambda}$. It can be
observed in the last three rows of Table~\ref{tab:SPTs_2D}. Therefore, there
exist SPTs of even period $m_0 = 2 l \geq 4$, with only $l + 1$ distinct
impact points on the ellipse. We have found similar SPTs inside ellipsoids of
$\Rset^3$. See Fig.~\ref{fig:TwoAmazingSPTs_3D}.

Finally, we stress that any SPT has a dual version of the same period and
caustic type, given by the dual transformation $g$;
see Corollary~\ref{cor:DualCorrespondence}.
For instance, this duality is evident for the
two SPTs with caustic type E that have common period $m_0 = 4 k$. Suffice it
to compare the two groups of associated formulae in the second row of
Table~\ref{tab:SPTs_2D},
and to recall that $g$ exchanges the role of positions and velocities.
In particular, the rectangular trajectory of that row is the dual
version of the diamond-shaped one. The same applies to rows 1, 4 and 5.
Instead, trajectories in rows 3 and 6 are dual of themselves.

\section{3D case}\label{sec:3D}

The previous section sets the basis of this one. Roughly speaking, we apply
similar ideas to find analogous results.

\subsection{Caustics and elliptic coordinates}

Let us recall some classical facts about caustics and elliptic coordinates.
All the notations introduced in this subsection will be used later on in this
section.

We write the ellipsoid as
\begin{equation}
  Q = \left\{ \left( x_1, x_2, x_3 \right) \in \Rset^3 :
  \frac{x_1^2}{a_1} + \frac{x_2^2}{a_2} + \frac{x_3^2}{a_3} = 1 \right\},
  \label{eq:Ellipsoid3D}
\end{equation}
where $0 < a_1 < a_2 < a_3$. Any nonsingular billiard trajectory inside $Q$ is
tangent to two distinct confocal caustics $Q_{\lambda_1}$ and $Q_{\lambda_2}$
(with $\lambda_1 < \lambda_2$) of the family
\[
Q_{\mu} =
\left\{
( x_1, x_2, x_3 ) \in \Rset^3 :
\sum_{j = 1}^3 \frac{x_j^2}{a_j - \mu} = 1
\right\},
\]
where $\mu \not\in \left\{ a_1, a_2, a_3 \right\}$. If we set
\[
E =   ( 0,   a_1), \quad
H_1 = ( a_1, a_2),\quad
H_2 = ( a_2, a_3),
\]
then $Q_{\mu}$ is an ellipsoid for $\mu \in E$, an one-sheet hyperboloid if
$\mu \in H_1 $, and a two-sheet hyperboloid when $\mu \in H_2$.

Not all combinations of caustics exist. For instance, both caustics cannot be
ellipsoids. The four possible combinations are denoted by EH1, H1H1, EH2, and
H1H2. Hence, the caustic parameter $\lambda = \left( \lambda_1, \lambda_2
\right)$ belongs to the nonsingular caustic space
\[ \Lambda = \left( E \times H_1 \right) \cup \left( H_1 \otimes H_1 \right)
   \cup \left( E \times H_2 \right) \cup \left( H_1 \times H_2 \right), \]
for $H_1 \otimes H_1 = \left\{ \left( \lambda_1, \lambda_2 \right) \in H_1
\times H_1 : \lambda_1 < \lambda_2 \right\}$, in order to avoid the singular
case $\lambda_1 = \lambda_2$.

Let us write the three coordinates planes of $\Rset^3$ as
\[ \Pi_l = \{ \left( x_1, x_2, x_3 \right) \in \Rset^3 : x_l = 0\},
   \qquad l \in \left\{ 1, 2, 3 \right\} . \]
We know (see Theorem~\ref{thm:Jacobi}) that given any $q \not\in \Pi_1 \cup \Pi_2
\cup \Pi_3$, there exist an ellipsoid $Q_e$, a one-sheet hyperboloid $Q_{h_1}$
and a two-sheet hyperboloid $Q_{h_2}$ such that
\begin{equation}
  q \in Q_e \cap Q_{h_1} \cap Q_{h_2}, \label{eq:TripleIntersection}
\end{equation}
those being quadrics mutually orthogonal at $q$. Therefore, $e < a_1 < h_1 <
a_2 < h_2 < a_3$. Let $\mathfrak{q}= \left( e, h_1, h_2 \right)$ be the
\tmtextit{elliptic coordinates} of $q = \left( x_1, x_2, x_3 \right)$.
The relation between Cartesian and elliptic
coordinates~(\ref{eq:EllipticCoordinatesnD}) becomes
\begin{equation}
  x_i^2 =  \frac{\left( a_i - e \right)  \left( a_i - h_1 \right)  \left( a_i
  - h_2 \right)}{\left( a_i - a_j \right)  \left( a_i - a_k \right)},
  \label{eq:EllipticCoordinates3D}
\end{equation}
for any $\left\{ i, j, k \right\} = \left\{ 1, 2, 3 \right\}$. To understand
what happens at nongeneric points, we define the three singular caustics as
\[ Q_{a_l} = \Pi_l, \qquad l \in \left\{ 1, 2, 3 \right\} . \]
We also recall that the \tmtextit{focal ellipse} and \tmtextit{focal
hyperbola} shared by the family of confocal quadrics
$\{ Q_{\mu} : \mu \in \Rset\}$ are
\begin{eqnarray*}
  \mathcal{E} & = & \left\{ \left( 0, x_2, x_3 \right) \in \Rset^3 :
  \frac{x_2^2}{a_2 - a_1} + \frac{x_3^2}{a_3 - a_1} = 1 \right\} \subset
  \Pi_1,\\
  \mathcal{H} & = & \left\{ \left( x_1, 0, x_3 \right) \in \Rset^3 :
  \frac{x_3^2}{a_3 - a_2} - \frac{x_1^2}{a_2 - a_1} = 1 \right\} \subset \Pi_2
  .
\end{eqnarray*}
These notations are useful because of the following lemma, which is a known
extension of the first item in Theorem~\ref{thm:Jacobi}.

\begin{lemma}\label{lem:TripleIntersection}
  If $q \not\in \mathcal{E} \cup \mathcal{H}$,
  there exists $\mathfrak{q}= \left( e, h_1, h_2 \right)$ such that $e \leq
  a_1 \leq h_1 \leq a_2 \leq h_2 \leq a_3$, $e \neq h_1 $, $h_1 \neq
  h_2$, and property~(\ref{eq:TripleIntersection}) holds, including its
  orthogonal character. 
\end{lemma}

Thus, we can express in elliptic coordinates all points along nonsingular
billiard trajectories. The trajectory is contained inside $Q$, so $e \geq 0$.
Impact points $q \in Q$ correspond to $e = 0$. Tangency points $q \in
Q_{\lambda_j}$ correspond to $e = \lambda_j \in E$, $h_1 = \lambda_j \in H_1$
or $h_2 = \lambda_j \in H_2$ depending on the type of the caustic
$Q_{\lambda_j}$. Crossing points $q \in \Pi_l$ have the following elliptic
coordinates: $\Pi_1$ is $\left\{ e = a_1 \right\} \cup \left\{ h_1 = a_1
\right\}$, $\Pi_2$ is $\left\{ h_1 = a_2 \right\} \cup \left\{ h_2 = a_2
\right\}$, and $\Pi_3$ is $\left\{ h_2 = a_3 \right\}$.

Indeed, one could agree that $\mathcal{E}= \left\{ e = a_1 = h_1 \right\}$ and
$\mathcal{H}= \left\{ h_1 = a_2 = h_2 \right\}$, although we do not need it,
because all nonsingular billiard trajectories sharing the caustics
$Q_{\lambda_1}$ and $Q_{\lambda_2}$ are contained, in elliptic coordinates,
inside the cuboid
\begin{equation}
  \EuScript{\mathcal{C}}_{\lambda} = \left\{ \begin{array}{ll}
    \left[ 0, \lambda_1 \right] \times \left[ a_1, \lambda_2 \right] \times
    \left[ a_2, a_3 \right], & \text{for EH1},\\
    \left[ 0, a_1 \right] \times \left[ \lambda_1, \lambda_2 \right] \times
    \left[ a_2, a_3 \right], & \text{for H1H1},\\
    \left[ 0, \lambda_1 \right] \times \left[ a_1, a_2 \right] \times \left[
    \lambda_2, a_3 \right], & \text{for EH2},\\
    \left[ 0, a_1 \right] \times \left[ \lambda_1, a_2 \right] \times \left[
    \lambda_2, a_3 \right], & \text{for H1H2}.
  \end{array} \right. \label{eq:Cuboid3D}
\end{equation}
We note that $e \neq h_1 \neq h_2 $ for all $\mathfrak{q}= \left( e,
h_1, h_2 \right) \in \mathcal{C}_{\lambda}$, since $\lambda = \left(
\lambda_1, \lambda_2 \right) \in \Lambda$. Hence, all points belonging to
$\EuScript{\mathcal{C}}_{\lambda}$ verify Eqs.~(\ref{eq:TripleIntersection})
and~(\ref{eq:EllipticCoordinates3D}).

\subsection{Frequency map and winding numbers}\label{ssec:FrequencyMap}

The following results related to periodic trajectories of billiards inside
ellipsoids can be found in Ref.~\onlinecite{CasasRamirez2011}.

There exists a map $\omega : \Lambda \rightarrow \Rset^2$,
called \tmtextit{frequency map} and defined by means of six
hyperelliptic integrals,
that characterizes all caustic parameters $\lambda = (\lambda_1, \lambda_2)$
that give rise to periodic trajectories. (Its explicit formula is given in
Ref.~\onlinecite{CasasRamirez2011}.) To be precise, the billiard trajectories
sharing the caustics $Q_{\lambda_1}$ and $Q_{\lambda_2}$ are periodic if and only if
\begin{equation}
  \omega \left( \lambda_1, \lambda_2 \right) = \left( m_1, m_2 \right) / 2 m_0
  \in \Qset^2 \label{eq:FrequencyWindingNumbers}
\end{equation}
for some positive integers $m_0$, $m_1$, and $m_2$. These are the
\tmtextit{winding numbers} introduced in item~\ref{item:WindingNumbers}) of
Theorem~\ref{thm:Oscillations}. They describe how trajectories fold in
$\Rset^3$. First, $m_0$ is the period. Second, $m_1$ is the number of
$\Pi_1$-crossings and $m_2$ is twice the number of turns around the $x_1$-axis
for EH1-caustics; $m_1$ is twice the number of turns around the $x_3$-axis and
$m_2$ is the number of $\Pi_3$-crossings for EH2-caustics; $m_1$ is the number
of tangential touches with each one-sheet hyperboloid caustic and $m_2$ is
twice the number of turns around the $x_1$-axis for H1H1-caustics; whereas
$m_1$ is the number of $\Pi_2$-crossings and $m_2$ is the number of
$\Pi_3$-crossings for H1H2-caustics. Only EH1-caustics and EH2-caustics have
trajectories of odd period.

\subsection{Reversors and their symmetry sets}

Let $\sigma_l, \sigma_{m n}, \sigma_{1 2 3} : \Rset^3 \rightarrow
\Rset^3$ be the $x_m x_n$-specular reflection, the $x_l$-axial
reflection, and the central reflection, respectively. Here, $\left\{ l, m, n
\right\} = \left\{ 1, 2, 3 \right\}$. Two samples are $\sigma_3 \left( x_1,
x_2, x_3 \right) = \left( x_1, x_2, - x_3 \right)$ and $\sigma_{1 3} \left(
x_1, x_2, x_3 \right) = \left( - x_1, x_2, - x_3 \right)$. Let $s_l, s_{m n},
s_{1 2 3} : M \rightarrow M$ be the symmetries given by $s_l \left( q, p
\right) = \left( \sigma_l q, \sigma_l p \right)$, $s_{m n} \left( q, p \right)
= \left( \sigma_{m n} q, \sigma_{m n} p \right)$, and $s_{1 2 3} \left( q, p
\right) = \left( \sigma_{1 2 3} q, \sigma_{1 2 3} p \right)$. We denote by $R$
the reversor written as $\tilde{r}$ in Sec.~\ref{sec:SymmetricBilliards}.
Likewise, $R_l = s_l \circ R$, $R_{m n} = \sigma_{m n} \circ R$, and $R_{1 2
3} = s_{1 2 3} \circ R$.

\begin{remark}
  \label{rem:lmn}Henceforth, we shall adopt the following index notation for
  the sake of brevity. If $l $, $m$, and $n$ appear without an
  explicit explanation of their ranges, it must be understood that they
  represent \tmtextit{any} choice such that $\left\{ l, m, n \right\} =
  \left\{ 1, 2, 3 \right\} .$ This convention is used, for instance, in
  Tables~\ref{tab:SymmetrySets_3D} and~\ref{tab:VertexesPointsReversors_3D}.
  Moreover, in the second to fifth rows of
  Table~\ref{tab:VertexesPointsReversors_3D}, index $l$ has a distinguished meaning
  from indexes $m$ and $n$; and the formulae contained in those rows are
  invariant under permutation of $m$ and $n$. This has to do with the fact
  that $\sigma_{m n} = \sigma_{n m}$.
\end{remark}

The three planar sections of the ellipsoid are
\[ S_l = Q \cap \Pi_l, \qquad l \in \left\{ 1, 2, 3 \right\} . \]
The symmetry sets of the $16 = 2^4$ reversors are listed in
Table~\ref{tab:SymmetrySets_3D}, which follows directly from
Proposition~\ref{prop:SymmetrySets}. We note that only the symmetry sets of $R_{1 2 3}$
and $f \circ R$ are empty. We usually write coordinates $q = \left( x_1, x_2,
x_3 \right)$ in the configuration space $Q$, and $p = \left( u_1, u_2, u_3
\right)$ in the velocity space $\Sset^2$.

\begin{table}
\caption{\label{tab:SymmetrySets_3D}The symmetry sets of the 16
  reversors of the billiard inside a triaxial ellipsoid
  $Q \subset \Rset^3$.}
\begin{ruledtabular}
  \begin{tabular}{lcc}
    $r$ & $\FixedSet \left( r \right)$ & $\FixedSet (f \circ r)$\\
    \hline \\ [-0.3cm]
    $R$ & $p \in N_q Q$ & $\emptyset$\\
    $R_l$ & $q \in S_l$ and $p \in N_q S_l$ & $u_m = u_n = 0, u^2_l = 1$\\
    $R_{m n}$ & $x_m = x_n = 0, x^2_l = a_l$ &
    $( q + \langle p \rangle) \bot_{\ast} \{ \mbox{$x_l$-axis} \}$\\
    $R_{1 2 3}$ & $\emptyset$ & $0 \in q + \left\langle p \right\rangle$\\
  \end{tabular}
\end{ruledtabular}
\end{table}

\subsection{Characterization of STs}\label{ssec:STsVertexes_3D}

We want to characterize nonsingular STs inside nondegenerate ellipsoids, as
trajectories passing, in elliptic coordinates, through some vertex of the
cuboid $\mathcal{C}_{\lambda}$. In order to accomplish this goal, we must
prove both implications. In Proposition~\ref{prop:Vertex2Reversor} we check
that any vertex of $\mathcal{C}_{\lambda}$ is visited by some (indeed, eight)
STs. In Proposition~\ref{prop:Reversor2Vertex} we see that any ST passes
through some vertex of $\mathcal{C}_{\lambda}$.

\begin{proposition}
  \label{prop:Vertex2Reversor}Let $\lambda = \left( \lambda_1, \lambda_2
  \right) \in \Lambda$ be fixed. Given any vertex $\mathfrak{q}_{\star} =
  \left( e, h_1, h_2 \right)$ of $\mathcal{C}_{\lambda}$, let $q_{\star} \in
  Q_e \cap Q_{h_1} \cap Q_{h_2}$. Then there exists some reversor $r$, and
  some point $\left( q, p \right) \in \FixedSet \left( r \right)$ such that
  $q_{\star} \in q + \left\langle p \right\rangle$ and the line $q +
  \left\langle p \right\rangle$ is tangent to $Q_{\lambda_1}$ and
  $Q_{\lambda_2}$. If $e = 0$, then $q = q_{\star}$;
  otherwise $q_{\star}$ is the middle point of the consecutive
  impact points $\hat{q}$ and $q$.
\end{proposition}

We clarify some aspects of this result before its proof:
\begin{itemizedot}
  \item We follow the index notation explained in Remark~\ref{rem:lmn}.
  
  \item We know that two symmetry sets only intersect at points $(q,p) \in M$
  such that the line $q + \left\langle p \right\rangle$ is
  contained in some coordinate plane; see Proposition~\ref{prop:SymmetrySets}.
  But we are dealing with nonsingular trajectories. Therefore, we can
  associate just one reversor to each vertex.
  
  \item Correspondence $\mathfrak{q}_{\star} \mapsto (q,p)$
  is 1-to-8; see Table~\ref{tab:VertexesPointsReversors_3D}.
  This has to do with two facts.
  First, elliptic coordinates do not distinguish among the
  eight octants in $\Rset^3$. Moreover, unit velocities $p$ directed
  inward $Q$ at impact points $q$ are excluded from
  Table~\ref{tab:VertexesPointsReversors_3D}; otherwise,
  $(q,p) \not\in M$. The billiard trajectory associated to any of
  those eight points $(q, p) \in M$ is $r$-symmetric and passes,
  in elliptic coordinates, through vertex $\mathfrak{q}_{\star}$.
\end{itemizedot}

\noindent\textbf{Proof\ }
  To begin with, we know from Lemma~\ref{lem:TripleIntersection} that $Q_e$,
  $Q_{h_1}$ and $Q_{h_2}$ are mutually orthogonal at $q_{\star}$. There are
  several kinds of vertexes. The first distinction is $e = 0$ or $e \neq 0$;
  that is, $q_{\star} \in Q$ or $q_{\star} \not\in Q$. The second distinction is
  the number of coordinate planes through $q_{\star}$. It turns out that there
  are six kinds of vertexes. We study each kind separately.
  \begin{enumeratealpha}
    \item If the point $q_{\star}$ belongs to $Q$ but is outside all three
    coordinate planes, then $\mathfrak{q}_{\star} = \left( 0, \lambda_1,
    \lambda_2 \right)$. From Eq.~(\ref{eq:Cuboid3D}), the type is H1H2. Thus,
    $Q_{\lambda_1}$ is a one-sheet hyperboloid, $Q_{\lambda_2}$ is a two-sheet
    hyperboloid, and $Q$, $Q_{\lambda_1}$ and $Q_{\lambda_2}$ are pairwise
    orthogonal at $q = q_{\star}$. Thus, if $p$ is the outward unit normal
    velocity to $Q$ at $q$, then $\left( q, p \right) \in \FixedSet \left( R
    \right)$ and the line $q + \left\langle p \right\rangle$ is tangent to
    $Q_{\lambda_1}$ and $Q_{\lambda_2}$. Finally, let us check the formulae
    for $q$ and $p$ given in Table~\ref{tab:VertexesPointsReversors_3D}.

\begin{table*}
      \caption{\label{tab:VertexesPointsReversors_3D}Relations among vertexes
      $\mathfrak{q}_{\star}$ of the cuboid $\mathcal{C}_{\lambda}$, types of
      caustics, reversors $r$, and points $q_{\star} \in q + \left\langle p
      \right\rangle$ such that $q + \left\langle p \right\rangle$ is tangent
      to $Q_{\lambda_1}$ and $Q_{\lambda_2}$, with $\left( q, p \right) \in
      \FixedSet \left( r \right)$. Notation for indexes $\left\{ l, m, n
      \right\} = \left\{ 1, 2, 3 \right\}$ and $\left\{ j, k \right\} =
      \left\{ 1, 2 \right\}$ is described in the text.}
\begin{ruledtabular}
      \begin{tabular}{lcccl}
        Vertex $\mathfrak{q}_{\star} = \left( e, h_1, h_2 \right)$ & Type &
        $q_{\star}$ belongs to & $r$ & $q = \left( x_1, x_2, x_3 \right)$, $p
        = \left( u_1, u_2, u_3\right)$\\ [0.05cm]
        \hline\\ [-0.36cm]
        $\left( 0, \lambda_1, \lambda_2 \right)$ & H1H2 & $Q \cap
        Q_{\lambda_1} \cap Q_{\lambda_2}$ & $R$ & $x_l^2 =  \DS\frac{a_l  \left(
        a_l - \lambda_1 \right)  \left( a_l - \lambda_2 \right)}{\left( a_l -
        a_m \right)  \left( a_l - a_n \right)}$, $u_l = \DS\frac{x_l}{a_l}
        \sqrt{\frac{a_1 a_2 a_3}{\lambda_1 \lambda_2}}$\\ [0.25cm]
        \hline\\ [-0.36cm]
        $e = 0$, $\left\{ h_1, h_2 \right\} = \left\{ a_l, \lambda_j \right\}$
        & all & $Q \cap \Pi_l \cap Q_{\lambda_j}$ & $R_l$ &
        \begin{tabular}{l}$x_l = 0$, $x_m^2
         = \DS\frac{a_m \left( a_m - \lambda_j \right)}{a_m - a_n}$,
        $x_n^2  = \DS\frac{a_n \left( a_n - \lambda_j \right)}{a_n -
        a_m}$\\
        $u_l^2 = \DS\frac{a_l - \lambda_k}{a_l}$, $\left( u_m, u_n
        \right) = \sqrt{\DS\frac{\lambda_k}{a_1 a_2 a_3 \lambda_j}} \left( a_n
        x_m, a_n x_m \right)$\end{tabular}\\ [0.7cm]
        \hline\\ [-0.36cm]
        $e = 0$, $\left\{ h_1, h_2 \right\} = \left\{ a_m, a_n \right\}$ &
        \begin{tabular}{c}not\\H1H1\end{tabular} &
        $Q \cap \Pi_m \cap \Pi_n$ & $R_{m n}$ &
        \begin{tabular}{l}$x^2_l = a_l$,
        $x_m = x_n = 0$, $u^2_l = \DS\frac{\lambda_1 \lambda_2}{a_m a_n}$,
        $\tmop{sign} \left( u_l \right) = \tmop{sign} \left( x_l \right)$\\
        $u_m^2 =  \DS\frac{\left( a_m - \lambda_1 \right)  \left( a_m
        - \lambda_2 \right)}{a_m  \left( a_m - a_n \right)}$, $u_n^2 = 
        \DS\frac{\left( a_n - \lambda_1 \right)  \left( a_n - \lambda_2
        \right)}{a_n  \left( a_n - a_m \right)}$\end{tabular}\\ [0.6cm]
        \hline\\ [-0.36cm]
        $\left\{ e, h_1, h_2 \right\} = \left\{ a_l, \lambda_1, \lambda_2
        \right\}$ & \begin{tabular}{c}not\\H1H1\end{tabular} &
        $\Pi_l \cap Q_{\lambda_1} \cap Q_{\lambda_2}$ & $f \circ R_l$ &
        \begin{tabular}{l}$u^2_l = 1$, $u_m = u_n = 0$, $x_l^2 =  \DS\frac{a_l
        \lambda_1 \lambda_2}{a_m a_n}$, $\tmop{sign} \left( u_l \right) =
        \tmop{sign} \left( x_l \right)$\\
        $x_m^2 =  \DS\frac{\left( a_m - \lambda_1 \right)  \left( a_m
        - \lambda_2 \right)}{a_m - a_n}$, $x_n^2 = \DS\frac{\left( a_n -
        \lambda_1 \right)  \left(a_n - \lambda_2 \right)}{a_n - a_m}$
        \end{tabular}\\ [0.6cm]
        \hline\\ [-0.36cm]
        $\left\{ e, h_1, h_2 \right\} = \left\{ a_m, a_n, \lambda_j \right\}$
        & all & $\Pi_m \cap \Pi_n \cap Q_{\lambda_j}$ & $f \circ R_{m n}$ &
        \begin{tabular}{l}$u_l = 0$, $u_m^2 = \DS\frac{a_m - \lambda_j}{a_m - a_n}$,
         $u_n^2 = \DS\frac{a_n - \lambda_j}{a_n - a_m}$\\
        $x_l^2 = a_l - \lambda_k$, $\left( x_m, x_n \right) =
        \sqrt{\DS\frac{a_m a_n \lambda_k}{a_l \lambda_j}} \left( u_m, u_n
        \right)$\end{tabular}\\ [0.7cm]
        \hline\\ [-0.36cm]
        $\left( a_1, a_2, a_3 \right)$ & H1H2 & $\Pi_1 \cap \Pi_2 \cap \Pi_3$
        & $f \circ R_{123}$ & $u_l^2 =  \DS\frac{\left( a_l - \lambda_1 \right) 
        \left( a_l - \lambda_2 \right)}{\left( a_l - a_m \right)  \left( a_l -
        a_n \right)}$, $x_l = \sqrt{\DS\frac{a_1 a_2 a_3}{\lambda_1 \lambda_2}}
        u_l$\\
      \end{tabular}
\end{ruledtabular}
\end{table*}
    
    The formula for $q$ follows directly from
    Eq.~(\ref{eq:EllipticCoordinates3D}). Thus, there are exactly eight choices
    for $q$. If $p \in N_q Q $ and $p$ points outward $Q$ at $q$, then
    there exists $\mu > 0$ such that $u_l = \mu x_l / a_l$. And using that $p
    \in \Sset^2$ we get that
    \[ 1 = \mu^2  \left( \frac{x_1^2}{a_1^2} + \frac{x_2^2}{a_2^2} +
       \frac{x_3^2}{a_3^2} \right) \Longrightarrow \mu^{} = \sqrt{\frac{a_1
       a_2 a_3}{\lambda_1 \lambda_2}} . \]
    In summary, only $R$-STs with H1H2-caustics can take place for this
    vertex.
    
    \item \label{item:QPil}
    If the point $q_{\star}$ belongs to $Q$ and is contained in only one
    coordinate plane, say $\Pi_l$, then $\mathfrak{q}_{\star} = \left( 0, h_1,
    h_2 \right)$ with $\left\{ h_1, h_2 \right\} = \left\{ a_l, \lambda_j
    \right\}$ for some $j \in \left\{ 1, 2 \right\}$. In that case, $q_{\star}
    \in Q \cap \Pi_l \cap Q_{\lambda_j} = S_l \cap Q_{\lambda_j}$. We do not
    get any restriction on the type of the caustics, but only some
    restrictions on the indexes; see Eq.~(\ref{eq:Cuboid3D}). Namely,
    \begin{itemize}
      \item $l \in \left\{ 2, 3 \right\}$ and $j = 2$ for type EH1;
      
      \item $l \in \left\{ 2, 3 \right\}$ for type H1H1;
      
      \item $l \in \left\{ 1, 2 \right\}$ and $j = 2$ for type EH2; and
      
      \item $l = j = 2$ or $\left( l, j \right) = \left( 3, 1 \right)$ for
      type H1H2.
    \end{itemize}
    We observe that any $\left( q, p \right) \in M$ such that $q = q_{\star}$
    and $p \in N_q S_l$ is contained in $\FixedSet \left( R_l \right)$.
    Thus, the formula for $q$ follows again from
    relations~(\ref{eq:EllipticCoordinates3D}),
    and we find four choices for $q$. Next,
    we look for an outward velocity $p = \left( u_1, u_2, u_3 \right) \in
    \Sset^2$ such that $p \bot T_q S_l$ and $q + \left\langle p
    \right\rangle$ is tangent to both $Q_{\lambda_j}$ and $Q_{\lambda_k}$ with
    $\left\{ j, k \right\} = \left\{ 1, 2 \right\}$.
    
    Condition $p \bot T_q S_l$ reads as $x_n u_m / a_n = x_m u_n / a_m$. In
    other words, we look for some $\nu \in \Rset$ such that
    \[ \left( u_m, u_n \right) = \nu \left( a_n x_m, a_m x_n \right) . \]
    Then, tangency with $Q_{\lambda_j}$ is immediate, since $Q_{\lambda_j}$ is
    orthogonal to the planar section $S_l = Q \cap \Pi_l$ at $q$. It only
    remains to impose the tangency with $Q_{\lambda_k}$. Thus, the quadratic
    equation
    \[ \frac{\left( x_l + t u_l \right)^2}{a_l - \lambda_k} + \frac{\left( x_m
       + t u_m \right)^2}{a_m - \lambda_k} + \frac{\left( x_n + t u_n
       \right)^2}{a_n - \lambda_k} = 1 \]
    must have a unique solution in $t$, which is equivalent to the vanishing
    of its discriminant: $\beta^2 = \alpha \gamma$, where
\begin{eqnarray*}
\alpha & = &
\frac{u_l^2}{a_l - \lambda_k} + \frac{u_m^2}{a_m - \lambda_k} + \frac{u_n^2}{a_n - \lambda_k} =
\frac{1}{a_l - \lambda_k} \times \\
& &
\left( 1 + \left( \frac{a_l - a_m}{a_m - \lambda_k} a_n^2 x_m^2 +
                    \frac{a_l - a_n}{a_n - \lambda_k} a_m^2 x_n^2 \right) \nu^2
                    \right),\\ 
\beta & = &
\DS\frac{x_m u_m}{a_m - \lambda_k} + \frac{x_n u_n}{a_n - \lambda_k} =
\frac{a_m a_n ( \lambda - \beta) \nu}{(a_m - \lambda_k)(a_n - \lambda_k)},\\
\gamma & = &
\frac{x_m^2}{a_m - \lambda_k} + \frac{x_n^2}{a_n - \lambda_k} - 1 \\
& = &
\frac{\lambda_k(\lambda_j - \lambda_k)}{(a_m - \lambda_k)( a_n - \lambda_k)}.
\end{eqnarray*}
    Here, we have used that $u_l^2 = 1 - u_m^2 - u_n^2$. After some calculations we
    find $\nu^2 = \lambda_k / \lambda_j a b c$, and from here we get the desired
    formulae in Table~\ref{tab:VertexesPointsReversors_3D}. We must impose in turn,
    that $p$ is an outward velocity, so
    \[ 0 < \left\langle Dq, p \right\rangle = \nu \left( a_n x_m^2 / a_m + a_m
       x^2_n / a_n \right), \]
    and $\nu > 0$. Thus, once we fix $q$, there are two choices for $u_l$, and
    exactly eight choices for $\left( q, p \right)$.
    
    We finish this item by stressing the consequences of the restrictions on
    the index $l$. They mean that the reversor $R_1$ cannot take place for
    types EH1, H1H1, and H1H2, whereas $R_3$ is forbidden for type EH2. This
    information is displayed in Table~\ref{tab:ForbiddenReversors3D}.
    
    \item \label{item:QPimPin}
    If $q_{\star}$ belongs to $Q$ and is contained in two coordinate
    planes, say $\Pi_m$ and $\Pi_n$, then $\mathfrak{q}_{\star} = \left( 0,
    h_1, h_2 \right)$ with $\left\{ h_1, h_2 \right\} = \left\{ a_m, a_n
    \right\}$. This prevents H1H1-caustics; see Eq.~(\ref{eq:Cuboid3D}).
    Furthermore, we get again some restrictions on the indexes. Namely,
    \begin{itemize}
      \item $1 \in \left\{ m, n \right\}$ for type EH1;
      
      \item $3 \in \left\{ m, n \right\}$ for type EH2; and
      
      \item $1 \not\in \left\{ m, n \right\}$ for type H1H2.
    \end{itemize}
    Besides, $q_{\star}$ is one of the two vertexes of the ellipsoid on the
    $x_l$-axis. Hence, $(q_\star,p) \in \FixedSet(R_{m n})$ for any
    $(q_\star,p) \in M$. Finally,
    let us check that the formulae given for $p = (u_1, u_2, u_3)$
    in Table~\ref{tab:VertexesPointsReversors_3D} are correct.
    
    We look for an outward velocity $p \in \Sset^2$ such that the line
    $q + \left\langle p \right\rangle$ is tangent to $Q_{\lambda_k}$, for $k =
    1, 2$. Thus, the quadratic equation
    \[ \frac{\left( x_l + t u_l \right)^2}{a_l - \lambda_k} + \frac{\left( x_m
       + t u_m \right)^2}{a_m - \lambda_k} + \frac{\left( x_n + t u_n
       \right)^2}{a_n - \lambda_k} = 1 \]
    must have a unique solution in $t$, which is equivalent to the vanishing
    of its discriminant: $\beta^2 = \alpha \gamma$, where
    \begin{eqnarray*}
      \alpha & = & \frac{u_l^2}{a_l - \lambda_k} + \frac{u_m^2}{a_m -
      \lambda_k} + \frac{u_n^2}{a_n - \lambda_k}\\
      & = & \frac{1}{a_n - \lambda_k} \left( 1 + \frac{a_n - a_l}{a_l -
      \lambda_k} u_l^2 + \frac{a_n - a_m}{a_m - \lambda_k} u_m^2 \right),\\
      \beta & = & \frac{x_l u_l}{a_l - \lambda_k},\\
      \gamma & = & \frac{x_l^2}{a_l - \lambda_k} - 1 = \frac{\lambda_k}{a_l -
      \lambda_k} .
    \end{eqnarray*}
    After some calculations we obtain that
    \[ a_n  \left( a_m - \lambda_k \right) u_l^2 + \left( a_m - a_n \right)
       \lambda_k u_m^2 = \lambda_k  \left( a_m - \lambda_k \right) . \]
    for $k = 1, 2$. These two equations along with condition $u_l^2 + u^2_m +
    u_n^2 = 1 $ form a system of three linear equations in the
    unknowns $u_l^2$, $u_m^2$, and $u_n^2$. A simple computation shows that
    the formulae given in Table~\ref{tab:VertexesPointsReversors_3D} are the
    unique solution of this system.
    
    Finally, we impose that $p$ is an outward velocity, so
    \[ 0 < \left\langle Dq, p \right\rangle = x_l u_l / a_l, \]
    and $x_l u_l > 0$. Thus, once fixed the vertex $q$, there are two choices
    for $u_m$ and two others for $u_n$. Therefore, there are exactly eight
    choices for $\left( q, p \right)$.
    
    We finish again by stressing the consequences of the restrictions on the
    indexes $m$ and $n$. They mean that the reversor $R_{2 3}$ cannot take
    place for type EH1, $R_{1 2}$ is forbidden for type EH2, both $R_{1 3}$
    and $R_{1 2}$ are not allowed for type H1H2, and all reversors of the form
    $R_{m n}$ cannot be found for type H1H1. This information is also
    displayed in Table~\ref{tab:ForbiddenReversors3D}.
    
    \item
    If $q_{\star}$ does not belong to $Q$ but it is in all coordinate
    planes, then $q_{\star} = \left( 0, 0, 0 \right) = 0$ and
    $\mathfrak{q}_{\star} = \left( a_1, a_2, a_3 \right)$. In that case the
    type is H1H2; see Eq.~(\ref{eq:Cuboid3D}). We look for lines $q_{\star} +
    \left\langle p \right\rangle = \left\langle p \right\rangle$ that are
    tangent to the one-sheet hyperboloid $Q_{\lambda_1}$ and to the two-sheet
    hyperboloid $Q_{\lambda_2}$. Since $q_{\star}$ is the center of both
    quadrics, $p = \left( u_1, u_2, u_3 \right)$ must be an asymptotic
    direction of them, so
    \[ \frac{u^2_1}{a_1 - \lambda_k} + \frac{u^2_2}{a_2 - \lambda_k} +
       \frac{u^2_3}{a_3 - \lambda_k} = 0  , \qquad
       k \in \left\{ 1, 2 \right\} . \]
    These equations along with $u_l^2 + u^2_m + u_n^2 = 1 $ form a
    system of three linear equations in the unknowns $u_l^2$, $u_m^2$, and
    $u_n^2$. A simple computation shows that the formulae given in
    Table~\ref{tab:VertexesPointsReversors_3D} are the unique solution of this
    system. Thus, there are eight choices for $p$.
    
    The line $q_{\star} + \left\langle p \right\rangle = 0 + \left\langle p
    \right\rangle$ intersects the ellipsoid $Q$ at some points $q = \mu p$ and
    $\hat{q} = - \mu p$ where $\mu > 0$. So, $x_l = \mu u_l$. And using that
    $q \in Q$ we get that
    \[ 1 = \mu^2  \left( \frac{u_1^2}{a_1^2} + \frac{u_2^2}{a_2^2} +
       \frac{u_3^2}{a_3^2} \right) \Longrightarrow \mu^{} = \sqrt{\frac{a_1
       a_2 a_3}{\lambda_1 \lambda_2}} . \]
    Finally, we realize that $\left( q, p \right) \in \FixedSet \left( R_{1 2
    3} \right)$ and $q_{\star}$ is the middle point of $\hat{q}$ and $q$.
    
    In summary, only $(f \circ R_{123})$-STs with H1H2-caustics
    can take place for this vertex.
    
    \item If $q_{\star}$ does not belong to $Q$ but it is in two coordinate
    planes, say $\Pi_m$ and $\Pi_n$, then $\mathfrak{q}_{\star} = \left( e,
    h_1, h_2 \right)$ with $\left\{ e, h_1, h_2 \right\} = \left\{ a_m, a_n,
    \lambda_j \right\}$. This case is the ``dual'' of case~\ref{item:QPil}),
    so it is associated to the reversors $f \circ R_{mn}$.
    The computations do not involve new ideas. We skip them.
    We just stress that there are no restriction on the type of the caustics,
    but only some restrictions on the indexes; see Eq.~(\ref{eq:Cuboid3D}).
    Namely,
    \begin{itemize}
      \item $1 \in \left\{ m, n \right\}$ for types EH1, H1H1, and H1H2; and
      
      \item $3 \in \left\{ m, n \right\}$ for types EH2.
    \end{itemize}
    They mean that the reversor $f \circ R_{2 3}$ cannot take place for types
    EH1, H1H1, and H1H2, whereas $f \circ R_{1 2}$ is forbidden for type EH2.
    
    \item If $q_{\star}$ does not belong to $Q$ and is contained in only one
    coordinate plane, say $\Pi_l$, then $\mathfrak{q}_{\star} = \left( e, h_1,
    h_2 \right)$ with $\left\{ h_1, h_2 \right\} = \left\{ a_l, \lambda_1,
    \lambda_2 \right\}$. In that case, $q_{\star} \in \Pi_l \cap Q_{\lambda_1}
    \cap Q_{\lambda_2}$. This case is the ``dual'' of case~\ref{item:QPimPin}),
    so it is associated to the reversors $f \circ R_l$.
    We also skip the computations.
    
    We only remark that no H1H1-caustics can be associated to this case.
    Besides, we get again some restrictions on the indexes;
    see Eq.~(\ref{eq:Cuboid3D}).
    Namely, $l \neq 1$ for type EH1; $l \neq 3$ for type
    EH2; and $l = 1$ for type H1H2. Consequently, $f \circ R_1$ cannot take
    place for type EH1, $f \circ R_3$ is forbidden for type EH2, both $f \circ
    R_2$ and $f \circ R_3$ are not allowed for type H1H2, and all reversors of
    the form $f \circ R_l$ cannot be found for type H1H1.
    \hspace*{\fill}$\Box$\medskip
  \end{enumeratealpha}

\begin{remark}
  Opposed vertexes of a given cuboid $\mathcal{C}_{\lambda}$ provide points on dual
  symmetry sets. At the same time, the formulae for $\left( q, p \right)$ of
  reversors $f \circ R_l$, $f \circ R_{m n}$ and $f \circ R_{1 2 3}$ follow from the
  formulae of their dual reversors $R_{m n}$, $R_l$ and $R_{}$, respectively, by the
  change $u_i^2 \leftrightarrow x_i^2 / a_i$. This does not come as a surprise; see
  Corollary~\ref{cor:DualCorrespondence}. Suffice it to realize that the dual
  transformation $g : M \rightarrow M $, $g \left( q, p \right) = \left( \bar{q},
  \bar{p} \right)$, verifies that $\bar{p} = - C^{- 1} q$, $C^2 = \tmop{diag}
  \left( a_1, a_2, a_3 \right)$.
\end{remark}

All implications in the proof of Proposition~\ref{prop:Vertex2Reversor} can be
reversed. Hence, we can move along Table~\ref{tab:VertexesPointsReversors_3D}
in both directions: from left to right, and from right to left.

Nevertheless, we prefer to prove Proposition~\ref{prop:Reversor2Vertex}
following a reasoning whose generalization to an arbitrary dimension is
straightforward.

\begin{proposition}
  \label{prop:Reversor2Vertex}Let $O$ be a nonsingular $r$-SO through a point
  $\left( q_{}, p \right) \in \FixedSet \left( r \right)$ for some reversor
  $r : M \rightarrow M$. Let $\lambda = \left( \lambda_1, \lambda_2 \right)
  \in \Lambda$ be its caustic parameter. Then there exists a vertex
  $\mathfrak{q}_{\star} = \left( e, h_1, h_2 \right)$ of the cuboid
  $\mathcal{C}_{\lambda}$ and a point $q_{\star} \in Q_e \cap Q_{h_1} \cap
  Q_{h_2}$ such that $q_{\star} \in q + \left\langle p \right\rangle$.
\end{proposition}

\begin{proof}
  First, we consider the reversors of the form $R$, $R_l$ or $R_{m n}$. Let
  $\tilde{q} : \Rset \rightarrow \Rset^3$, $\tilde{q} \left( t
  \right) = \left( \tilde{x}_1 \left( t \right), \tilde{x}_2 \left( t
  \right), \tilde{x}_3 \left( t \right) \right)$, be any of the
  arc-length parameterizations of the billiard trajectory through
  the impact point $q_{\star} : = q$ with unit velocity $p$,
  determined by the conditions $\tilde{q}(0) = q_{\star}$
  and $\tilde{q}'(0) = p$.
  Clearly, it is smooth except at
  impact points. Besides, it has the symmetry properties listed in
  Table~\ref{tab:Propertiesq_3D}, which can be deduced from comments after
  Remark~\ref{rem:CentralAxialSpecular} and elementary geometric arguments.
  
\begin{table}
\caption{\label{tab:Propertiesq_3D}Symmetry properties of the
arc-length parameterization $\tilde{q}(t)$ of $r$-STs
such that $\tilde{q}(0) = q_{\star}$ and $\tilde{q}'(0) = p$.}
\begin{ruledtabular}
    \begin{tabular}{@{\hspace{1.9cm}}l@{\hspace{0.3cm}}l@{\hspace{2.1cm}}}
      $r$ & Symmetry properties\\
      \hline \\ [-0.3cm]
      $R$ & $\tilde{q} \left( - t \right) = \tilde{q} \left( t \right)$\\
      $R_l$ & $\tilde{q} \left( - t \right) = \sigma_l  \tilde{q} \left( t
      \right)$\\
      $R_{m n}$ & $\tilde{q} \left( - t \right) = \sigma_{m n}  \tilde{q}
      \left( t \right)$\\
      $f \circ R_l$ & $\tilde{q} \left( - t \right) = \sigma_l  \tilde{q}
      \left( t \right)$\\
      $f \circ R_{m n}$ & $\tilde{q} \left( - t \right) = \sigma_{m n} 
      \tilde{q} \left( t \right)$\\
      $f \circ R_{1 2 3}$ & $\tilde{q} \left( - t \right) = - \tilde{q} \left(
      t \right)$\\
    \end{tabular}
\end{ruledtabular}
  \end{table}
  
  Let $\tilde{\mathfrak{q}} : \Rset \rightarrow \mathcal{C}_{\lambda}$,
  $\tilde{\mathfrak{q}} \left( t \right) = ( \tilde{e} \left( t \right),
  \tilde{h}_1 \left( t \right), \tilde{h}_2 \left( t \right) )$, be the
  corresponding parameterization in elliptic coordinates. We know that
  components $\tilde{e} \left( t \right)$, $\tilde{h}_1 \left( t \right)$, and
  $\tilde{h}_2 \left( t \right)$, oscillate in some intervals in such a way
  that their only critical points are attained at the extremes of these
  intervals. The cuboid $\mathcal{C}_{\lambda}$ is the product of these three
  intervals. Besides, the functions $\tilde{h}_1 \left( t \right)$ and
  $\tilde{h}_2 \left( t \right)$ are smooth everywhere.
  
  The above symmetries of $\tilde{q} \left( t \right)$ imply that
  $\tilde{\mathfrak{q}} \left( - t \right) = \tilde{\mathfrak{q}} \left( t
  \right)$, since elliptic coordinates do not distinguish among the eight
  octants in $\Rset^3$. Therefore,
  \[ \tilde{h}'_1 \left( 0 \right) = 0, \qquad \tilde{h}'_2 \left( 0
     \right) = 0. \]
  In addition, $\tilde{e} \left( 0 \right) = 0$, since we have taken
  $q_{\star} = q \in Q.$ Hence, $\mathfrak{q}_{\star} = \left( e, h_1, h_2
  \right) := \tilde{\mathfrak{q}} \left( 0 \right) = ( 0,
  \tilde{h}_1 \left( 0 \right), \tilde{h}_1 \left( 0 \right) )$ is a
  vertex of $\mathcal{C}_{\lambda}$; see item~\ref{item:CriticalPoints}) of
  Theorem~\ref{thm:Oscillations}.
  
  It remains to consider reversors of the form $f \circ R_l$ $f \circ R_{m n}$
  and $f \circ R_{1 2 3}$. Let $\tilde{q} : \Rset \rightarrow
  \Rset^3$ be the arc-length parameterization of the billiard trajectory
  that begins at the middle point $q_{\star} := \left( q + \hat{q}
  \right) / 2$ with unit velocity $p$. We recall that $\hat{q}$ is the
  previous impact point, that is, $Q \cap \left( q + \langle p \rangle \right)
  = \{q, \hat{q} \}$.
  
  The symmetry properties of these parameterizations are also listed in
  Table~\ref{tab:Propertiesq_3D}. For instance, if $r = f \circ R_{1 2 3}$, then
  $q_{\star} = \left( 0, 0, 0 \right)$, which implies $\tilde{q} \left( - t
  \right) = - \tilde{q} \left( t \right)$. Then, we apply exactly the same
  argument as before, with just one difference. The function $\tilde{e} \left(
  t \right)$ is smooth at $t = 0$, because $\tilde{q} \left( 0 \right) =
  q_{\star} \not\in Q$ and $\tilde{e} \left( 0 \right) \neq 0$. Hence, we also
  get that $\tilde{e}' \left( 0 \right) = 0$, so $\mathfrak{q}_{\star} =
  \left( e, h_1, h_2 \right) := \tilde{\mathfrak{q}} \left( 0 \right) =
  ( \tilde{e} \left( 0 \right), \tilde{h}_1 \left( 0 \right), \tilde{h}_1
  \left( 0 \right) )$ is a vertex of $\mathcal{C}_{\lambda}$.
  
  Finally, the intersection property
  $q_{\star} \in Q_e \cap Q_{h_1} \cap Q_{h_2}$ follows
  from Lemma~\ref{lem:TripleIntersection}.
\end{proof}

\begin{corollary}
  Nonsingular STs inside triaxial ellipsoids of $\Rset^3$ are
  characterized as trajectories passing, in elliptic coordinates, through some
  vertex of their cuboid.
\end{corollary}

\subsection{Forbidden reversors for each type of caustics}%
\label{ssec:ForbiddenReversors_3D}

Next, we emphasize that, although there are 14 non\-empty symmetry sets, once fixed
the caustic type, some of them cannot take place. To be more precise, there are ten
forbidden symmetry sets for H1H1-caustics, but only six otherwise. These results
have been obtained along the proof of Proposition~\ref{prop:Vertex2Reversor}, and we
organize them in Table~\ref{tab:ForbiddenReversors3D}. Forbidden reversors appear in
\tmtextit{dual couples} ---couples whose symmetry sets are interchanged by the
dual transformation $g$.

\begin{table}
\caption{\label{tab:ForbiddenReversors3D}Forbidden reversors for each
type of caustics.}
\begin{ruledtabular}
\begin{tabular}{@{\hspace{1cm}}l@{\hspace{0.3cm}}@{\hspace{0.3cm}}l@{\hspace{1cm}}}
Type & Forbidden reversors \\
[0.1cm] \hline
EH1 & $R$, $f \circ R_{123}$, $R_1$,
      $f \circ R_{23}$, $R_{23}$, $f \circ R_1$ \\
EH2 & $R$, $f \circ R_{123}$, $R_3$, $f \circ R_{12}$,
      $R_{1 2}$, $f \circ R_3$\\
H1H1 & \begin{tabular}{l} $R$, $R_1$, $R_{12}$, $R_{13}$, $R_{23}$, $f \circ R_{123}$ \\
                          $f \circ R_1$, $f \circ R_2$, $f\circ R_3$, $f \circ R_{23}$
       \end{tabular}                       \\
H1H2 & $R_1$, $f \circ R_{23}$, $R_{13}$, $f \circ R_2$,
       $R_{12}$, $f \circ R_3$\\
\end{tabular}
\end{ruledtabular}
\end{table}

A glance at cuboid~(\ref{eq:Cuboid3D}) and
Table~\ref{tab:VertexesPointsReversors_3D} shows that,
two different vertexes of the same cuboid cannot be associated to
the same reversor, but for H1H1-caustics.
This agrees with the previous discussion, where we found $8 = 14 - 6$ feasible
symmetry sets for caustics of type EH1, EH2, and H1H2, but only $4 = 14 - 10$ for
H1H1-caustics.

\subsection{Characterization and classification of SPTs}

Let us characterize and classify SPTs inside triaxial ellipsoids of
$\Rset^3$. We classify them by the caustic type and the couple of
vertexes of the cuboid they connect.

\begin{theorem}
  \label{thm:SPTs_3D}Nonsingular SPTs inside triaxial ellipsoids of
  $\Rset^3$ are characterized as trajectories connecting, in elliptic
  coordinates, two different vertexes of their cuboid. All nonsingular SPOs
  are doubly SPOs, but a few with H1H1-caustics. Besides, there are exactly
  112 classes of nonsingular SPTs,
  listed in Tables~\ref{tab:ClassificationSPTsEH1}--\ref{tab:ClassificationSPTsH1H1}.
\end{theorem}

\begin{proof}
  Let $O$ be a nonsingular $r$-SPO through a point $\left( q_{}, p \right) \in
  \FixedSet \left( r \right)$, for some reversor $r : M \rightarrow M$, with
  winding numbers $m_0, m_1, m_2$. Let $\lambda = \left( \lambda_1, \lambda_2
  \right) \in \Lambda$ be its caustic parameter. Let $L_0$ be its length.
  
  Let $\tilde{q} : \Rset \rightarrow \Rset^3$ be the arc-length
  parameterization of the billiard trajectory that begins at the distinguished
  point $q_{\star}$ introduced in Proposition~\ref{prop:Reversor2Vertex}.
  
  Let $\tilde{\mathfrak{q}} : \Rset \rightarrow \mathcal{C}_{\lambda}$
  be the corresponding parameterization in elliptic coordinates. Then
  $\tilde{\mathfrak{q}} \left( t \right)$ is even ---see the proof of
  Proposition~\ref{prop:Reversor2Vertex}--- and $L_0$-periodic
  ---see Theorem~\ref{thm:Oscillations}.
  
  The key trick is as simple as to realize that
  \[ \tilde{\mathfrak{q}} \left( L_0 / 2 - t \right) = \tilde{\mathfrak{q}}
     \left( t - L_0 / 2 \right) = \tilde{\mathfrak{q}} \left( t + L_0 / 2
     \right), \qquad \forall t \in \Rset. \]
  Hence, $\tilde{\mathfrak{q}} \left( t \right)$ is even with respect to $L_0 /
  2$, and so
  \[ \mathfrak{q}_{\star} = \left( e^{\star}, h_1^{\star}, h^{\star}_2 \right)
     := \tilde{\mathfrak{q}} \left( 0 \right), \quad
     \mathfrak{q}_{\bullet} = \left( e^{\bullet}, h_1^{\bullet}, h^{\bullet}_2
     \right) := \tilde{\mathfrak{q}} \left( L_0 / 2 \right) \]
  are vertexes of the cuboid $\mathcal{C}_{\lambda}$. The proof for
  $\mathfrak{q}_{\star}$ was already explained in
  Proposition~\ref{prop:Reversor2Vertex};
  the proof for $\mathfrak{q}_{\bullet}$ is equal.
  
  Now, taking into account the interpretation of the winding numbers as, the
  number of complete oscillations of each elliptic coordinate when the
  arc-length parameter $t$ moves from $0$ to $L_0$
  ---see item~\ref{item:WindingNumbers}) of Theorem~\ref{thm:Oscillations}---,
  we consider two cases:
  \begin{itemize}
    \item If some winding number is odd, then $\mathfrak{q}_{\bullet} \neq
    \mathfrak{q}_{\star}$, because
    \[ e^{\star} = e^{\bullet} \Leftrightarrow m_0 \in 2\Zset,
       \qquad h^{\star}_i = h^{\bullet}_i \Leftrightarrow m_i \in
       2\Zset. \]
    \item If all winding numbers are even, then $\tilde{\mathfrak{q}} \left( t
    \right)$ has period $L_0 / 2$ and $\mathfrak{q}_{\bullet}
    =\mathfrak{q}_{\star}$
    ---see item~\ref{item:WindingNumbers}) of Theorem~\ref{thm:Oscillations}---,
    in which case we repeat the same strategy
    to find a second vertex $\mathfrak{q}_{\circ} = \left( e^{\circ},
    h_1^{\circ}, h^{\circ}_2 \right) := \tilde{\mathfrak{q}} \left( L_0 /
    4 \right)$ such that
    \[ e^{\star} = e^{\circ} \Leftrightarrow m_0 \in 4\Zset,
       \qquad h^{\star}_i = h^{\circ}_i \Leftrightarrow m_i \in
       4\Zset. \]
    Then, $\mathfrak{q}_{\circ} \neq \mathfrak{q}_{\star}$
    ---see item~\ref{item:WindingNumbers}) of Theorem~\ref{thm:Oscillations}.
  \end{itemize}
  Therefore, the trajectory connects two different vertexes in both cases.
  Next, we prove that any trajectory $\tilde{q} \left( t \right)$ connecting,
  in elliptic coordinates, two vertexes is an SPT. Thus, we assume that
  $\mathfrak{q}_{\star} = \tilde{\mathfrak{q}} \left( t_{\star} \right)$ and
  $\mathfrak{q}_{\bullet} = \tilde{\mathfrak{q}} \left( t_{\bullet} \right)$
  are two different vertexes with $t_{\star} < t_{\bullet}$. From
  Proposition~\ref{prop:Vertex2Reversor} we know that this trajectory is
  $r_{\star}$-symmetric and $r_{\bullet}$-symmetric for some reversors
  $r_{\star}$ and $r_{\bullet}$. The case $r_{\star} = r_{\bullet}$ is not
  excluded. Repeating the reasoning in Proposition~\ref{prop:Reversor2Vertex},
  we deduce that $\tilde{\mathfrak{q}} \left( t \right)$ is symmetric with
  respect to $t = t_{\star}$ and $t = t_{\bullet}$. In particular,
  $\tilde{\mathfrak{q}} \left( 2 t_{\star} + t \right) = \tilde{\mathfrak{q}}
  \left( - t \right)$ and $\tilde{\mathfrak{q}} \left( 2 t_{\bullet} + t
  \right) = \tilde{\mathfrak{q}} \left( - t \right)$ for all $t \in \Rset$.
  Set $T = 2 \left(t_{\bullet} - t_{\star} \right)$. Then
  \[ \tilde{\mathfrak{q}} \left( t + T \right) = \tilde{\mathfrak{q}} \left( 2
     t_{\bullet} + t - 2 t_{\star} \right) = \tilde{\mathfrak{q}} \left( 2
     t_{\star} - t \right) = \tilde{\mathfrak{q}} \left( t \right) \]
  for all $t \in \Rset$.
  Hence, $\tilde{\mathfrak{q}} \left( t \right)$ is $T$-periodic, and so,
  $\tilde{q} \left( t \right)$ is periodic with period $T$ or $2 T$; see
  item~\ref{item:WindingNumbers}) of Theorem~\ref{thm:Oscillations}.
  This proves the characterization of SPTs.
  
  Once we have established that any SPT connects two different vertexes of its
  cuboid, it is easy to deduce that, all SPOs inside triaxial ellipsoids of
  $\Rset^3$ whose type of caustics are EH1, EH2 or H1H2, must be doubly
  SPOs. Suffice it to recall that each vertex of the cuboid is associated to a
  different reversor for those three types.
  
  The number 112 comes from $4 \times 28$, since there are 4 types of caustics
  and any 3-dimensional cuboid has eight vertexes, and so $28 = \left( 8
  \times 7 \right) / 2$ couples of vertexes.
  
  We must check that none of the 112 classes of SPTs is fictitious. To do it,
  we present an algorithm that provides a minimal SPT of each class.
  \tmtextit{Minimal} means that it has the smallest possible period. The key
  step of the algorithm is to properly choose the winding numbers $\left( m_0,
  m_1, m_2 \right)$. In accordance with the previous discussion, we
  distinguish four kinds of winding number. Namely, even: ``e'', odd: ``o'',
  multiple of four: ``f'', and even but not multiple of four: ``t''. For
  instance, all winding numbers of kind $\left( \text{t}, \text{t}, \text{t}
  \right)$ connect opposite vertexes of their cuboids. The kind $\left(
  \text{f}, \text{f}, \text{f} \right)$ never takes place
  ---see item~\ref{item:WindingNumbers}) of Theorem~\ref{thm:Oscillations}.
  
  Let us explain the algorithm by using an example. We want to connect
  opposite vertexes $\mathfrak{q}_{\star} = \left( e^{\star}, h_1^{\star},
  h^{\star}_2 \right) = \left( 0, \lambda_2, a_3 \right)$ and
  $\mathfrak{q}_{\bullet} = \left( e^{\bullet}, h_1^{\bullet}, h^{\bullet}_2
  \right) = \left( \lambda_1, a_1, a_2 \right)$ of the cuboid
  \[ \EuScript{\mathcal{C}}_{\lambda} = \left[ 0, \lambda_1 \right] \times
     \left[ a_1, \lambda_2 \right] \times \left[ a_2, a_3 \right], \]
  which corresponds to caustic type EH1.

\begin{table}
\caption{\label{tab:ClassificationSPTsEH1}Classification of SPTs for
  caustic type EH1. Notation for kinds ``e'', ``o'', ``f'', and ``t'' is
  described in the text. For each kind of winding numbers we list its minimal
  representative, and its four couples of reversors, whose order inside the
  couple is irrelevant.}
\begin{ruledtabular}
  \begin{tabular}{@{\hspace{0.5cm}}c@{\hspace{0.3cm}}@{\hspace{0.3cm}}l@{\hspace{0.5cm}}}
    \begin{tabular}{c}$( m_0, m_1, m_2 )$\\ kind/minimal\end{tabular} &
      Couples of reversors\\\hline
    $\left( \text{t}, \text{t}, \text{t} \right)$ & $\left\{ R_3, f \circ R_{1
    2} \right\}$, $\left\{ R_2, f \circ R_{1 3} \right\}$\\
    $\left( 10, 6, 2 \right)$ & $\left\{ R_{1 3}, f \circ R_2 \right\}$,
    $\left\{ R_{1 2}, f \circ R_3 \right\}$\\
    \hline
    $\left( \text{t}, \text{t}, \text{f} \right)$ & $\left\{ R_3, f \circ R_{1
    3} \right\}$, $\left\{ R_2, f \circ R_{1 2} \right\}$\\
    $\left( 10, 6, 4 \right)$ & $\left\{ R_{1 3}, f \circ R_3 \right\}$,
    $\left\{ R_{1 2}, f \circ R_2 \right\}$\\
    \hline
    $\left( \text{t}, \text{f}, \text{t} \right)$ & $\left\{ R_3, f \circ R_2
    \right\}$, $\left\{ R_2, f \circ R_3 \right\}$\\
    $\left( 6, 4, 2 \right)$ & $\left\{ R_{1 3}, f \circ R_{1 2} \right\}$,
    $\left\{ R_{1 2}, f \circ R_{1 3} \right\}$\\
    \hline
    $\left( \text{o}, \text{e}, \text{e} \right)$ & $\left\{ R_3, f \circ R_3
    \right\}$, $\left\{ R_2, f \circ R_2 \right\}$\\
    $\left( 5, 4, 2 \right)$ & $\left\{ R_{1 3}, f \circ R_{1 3} \right\}$,
    $\left\{ R_{1 2}, f \circ R_{1 2} \right\}$\\
    \hline
    $\left( \text{f}, \text{t}, \text{t} \right)$ & $\left\{ R_3, R_{1 2}
    \right\}$, $\left\{ f \circ R_3, f \circ R_{1 3} \right\}$\\
    $\left( 8, 6, 2 \right)$ & $\left\{ R_2, R_{1 3} \right\}$, $\left\{ f
    \circ R_2, f \circ R_{1 3} \right\}$\\
    \hline
    $\left( \text{f}, \text{t}, \text{f} \right)$ & $\left\{ R_3, R_{1 3}
    \right\}$, $\left\{ f \circ R_3, f \circ R_{1 3} \right\}$\\
    $\left( 8, 6, 4 \right)$ & $\left\{ R_2, R_{1 2} \right\}$, $\left\{ f
    \circ R_2, f \circ R_{1 2} \right\}$\\
    \hline
    $\left( \text{f}, \text{f}, \text{t} \right)$ & $\left\{ R_3, R_2
    \right\}$, $\left\{ f \circ R_3, f \circ R_2 \right\}$\\
    $\left( 8, 4, 2 \right)$ & $\left\{ R_{1 3}, R_{1 2} \right\}$, $\left\{ f
    \circ R_{1 3}, f \circ R_{1 2} \right\}$\\
  \end{tabular}
\end{ruledtabular}
\end{table}

\begin{table}
\caption{\label{tab:ClassificationSPTsEH2}Analogous of
Table~\ref{tab:ClassificationSPTsEH1} for caustic type EH2.}
\begin{ruledtabular}
  \begin{tabular}{@{\hspace{0.5cm}}c@{\hspace{0.3cm}}@{\hspace{0.3cm}}l@{\hspace{0.5cm}}}
    \begin{tabular}{c}$( m_0, m_1, m_2 )$\\ kind/minimal\end{tabular} &
      Couples of reversors\\\hline
    $\left( \text{t}, \text{t}, \text{t} \right)$ & $\left\{ R_1, f \circ R_{2
    3} \right\}$, $\left\{ R_2, f \circ R_{1 3} \right\}$\\
    $\left( 10, 6, 2 \right)$ & $\left\{ R_{1 3}, f \circ R_2 \right\}$,
    $\left\{ R_{2 3}, f \circ R_1 \right\}$\\
    \hline
    $\left( \text{t}, \text{t}, \text{f} \right)$ & $\left\{ R_1, f \circ R_2
    \right\}$, $\left\{ R_2, f \circ R_1 \right\}$\\
    $\left( 10, 6, 4 \right)$ & $\left\{ R_{13}, f \circ R_{2 3} \right\}$,
    $\left\{ R_{2 3}, f \circ R_{1 3} \right\}$\\
    \hline
    $\left( \text{t}, \text{f}, \text{t} \right)$ & $\left\{ R_1, f \circ R_{1
    3} \right\}$, $\left\{ R_2, f \circ R_{2 3} \right\}$\\
    $\left( 6, 4, 2 \right)$ & $\left\{ R_{1 3}, f \circ R_1 \right\}$,
    $\left\{ R_{2 3}, f \circ R_2 \right\}$\\
    \hline
    $\left( \text{o}, \text{e}, \text{e} \right)$ & $\left\{ R_1, f \circ R_1
    \right\}$, $\left\{ R_2, f \circ R_2 \right\}$\\
    $\left( 5, 4, 2 \right)$ & $\left\{ R_{1 3}, f \circ R_{1 3} \right\}$,
    $\left\{ R_{2 3}, f \circ R_{2 3} \right\}$\\
    \hline
    $\left( \text{f}, \text{t}, \text{t} \right)$ & $\left\{ R_1, R_{2 3}
    \right\}$, $\left\{ f \circ R_1, f \circ R_{2 3} \right\}$\\
    $\left( 8, 6, 2 \right)$ & $\left\{ R_2, R_{1 3} \right\}$, $\left\{ f
    \circ R_2, f \circ R_{1 3} \right\}$\\
    \hline
    $\left( \text{f}, \text{t}, \text{f} \right)$ & $\left\{ R_1, R_2
    \right\}$, $\left\{ f \circ R_1, f \circ R_2 \right\}$\\
    $\left( 8, 6, 4 \right)$ & $\left\{ R_{1 3}, R_{2 3} \right\}$, $\left\{
    f \circ R_{1 3}, f \circ R_{2 3} \right\}$\\
    \hline
    $\left( \text{f}, \text{f}, \text{t} \right)$ & $\left\{ R_1, R_{1 3}
    \right\}$, $\left\{ f \circ R_1, f \circ R_{1 3} \right\}$\\
    $\left( 8, 4, 2 \right)$ & $\left\{ R_2, R_{2 3} \right\}$, $\left\{ f
    \circ R_2, f \circ R_{2 3} \right\}$\\
  \end{tabular}
\end{ruledtabular}
\end{table}

\begin{table}
\caption{\label{tab:ClassificationSPTsH1H2}Analogous of
Table~\ref{tab:ClassificationSPTsEH1} for caustic type H1H2.}
\begin{ruledtabular}
  \begin{tabular}{@{\hspace{0.5cm}}c@{\hspace{0.3cm}}@{\hspace{0.3cm}}l@{\hspace{0.5cm}}}
    \begin{tabular}{c}$( m_0, m_1, m_2 )$\\ kind/minimal\end{tabular} &
      Couples of reversors\\\hline
    $\left( \text{t}, \text{t}, \text{t} \right)$ & $\left\{ R, f \circ R_{1 2
    3} \right\}$, $\left\{ R_2, f \circ R_{1 3} \right\}$\\
    $\left( 10, 6, 2 \right)$ & $\left\{ R_3, f \circ R_{1 2} \right\}$,
    $\left\{ R_{2 3}, f \circ R_1 \right\}$\\
    \hline
    $\left( \text{t}, \text{t}, \text{f} \right)$ & $\left\{ R, f \circ R_{1
    2} \right\}$, $\left\{ R_2, f \circ R_3 \right\}$\\
    $\left( 10, 6, 4 \right)$ & $\left\{ R_3, f \circ R_{1 2 3} \right\}$,
    $\left\{ R_{2 3}, f \circ R_{1 3} \right\}$\\
    \hline
    $\left( \text{t}, \text{f}, \text{t} \right)$ & $\left\{ R, f \circ R_{1
    3} \right\}$, $\left\{ R_3, f \circ R_1 \right\}$\\
    $\left( 6, 4, 2 \right)$ & $\left\{ R_2, f \circ R_{1 2 3} \right\}$,
    $\left\{ R_{2 3}, f \circ R_{1 2} \right\}$\\
    \hline
    $\left( \text{t}, \text{f}, \text{f} \right)$ & $\left\{ R, f \circ R_1
    \right\}$, $\left\{ R_2, f \circ R_{1 2} \right\}$\\
    $\left( 10, 8, 4 \right)$ & $\left\{ R_3, f \circ R_{1 3} \right\}$,
    $\left\{ R_{2 3}, f \circ R_{1 2 3} \right\}$\\
    \hline
    $\left( \text{f}, \text{t}, \text{t} \right)$ & $\left\{ R, R_{2 3}
    \right\}$, $\left\{ f \circ R_1, f \circ R_{1 2 3} \right\}$\\
    $\left( 8, 6, 2 \right)$ & $\left\{ R_2, R_3 \right\}$, $\left\{ f \circ
    R_{1 2}, f \circ R_{13} \right\}$\\
    \hline
    $\left( \text{f}, \text{t}, \text{f} \right)$ & $\left\{ R, R_2 \right\}$,
    $\left\{ f \circ R_{1 3}, f \circ R_{1 2 3} \right\}$\\
    $\left( 8, 6, 4 \right)$ & $\left\{ R_3, R_{2 3} \right\}$, $\left\{ f
    \circ R_1, f \circ R_{1 2} \right\}$\\
    \hline
    $\left( \text{f}, \text{f}, \text{t} \right)$ & $\left\{ R, R_3 \right\}$,
    $\left\{ f \circ R_{1 2}, f \circ R_{1 2 3} \right\}$\\
    $\left( 8, 4, 2 \right)$ & $\left\{ R_2, R_{2 3} \right\}$, $\left\{ f
    \circ R_1, f \circ R_{1 3} \right\}$\\
  \end{tabular}
\end{ruledtabular}
\end{table}
  
  This problem about vertexes is equivalent to find an SPT of type EH1
  that is simultaneously $R_3$-symmetric and $(f \circ R_{1 2})$-symmetric,
  because the reversors associated to $\mathfrak{q}_{\star}$ and
  $\mathfrak{q}_{\bullet}$ are $R_3$ and $f \circ R_{1 2}$,
  respectively (cf. Table~\ref{tab:VertexesPointsReversors_3D}).
  As we said above, winding numbers
  must be of kind $\left( \text{t}, \text{t}, \text{t} \right)$. Then, taking
  for granted conjecture~(\ref{eq:Conjecture}), the \tmtextup{minimal} choice
  is $\left( m_0, m_1, m_2 \right) = \left( 10, 6, 2 \right) $.
  Finally, if we solve Eq.~(\ref{eq:FrequencyWindingNumbers}) for
  $\lambda = \left( \lambda_1, \lambda_2 \right) \in E \times H_1$, with
  $\left( m_0, m_1, m_2 \right) = \left( 10, 6, 2 \right)$, and draw the
  billiard trajectory through the point $\left( q, p \right)$, given by the
  formulae of the second row of Table~\ref{tab:VertexesPointsReversors_3D} for
  $l = 3$, then we obtain a trajectory of period 10, type EH1,
  $R_3$-symmetric, and $(f \circ R_{1 2})$-symmetric.
  
  All winding numbers of kind $\left( \text{t}, \text{t}, \text{t} \right)$
  connect opposite vertexes of their cuboids. Hence, their SPTs of type EH1
  can display the following double symmetries: $\left\{ R_3, f \circ R_{12}
  \right\}$, $\left\{ R_2, f \circ R_{13} \right\}$, $\left\{ R_{13}, f \circ
  R_2 \right\}$, or $\left\{ R_{12}, f \circ R_3 \right\}$. They are obtained
  by looking at Table~\ref{tab:VertexesPointsReversors_3D} the relations
  between vertexes and reversors.
  
  Other kinds of winding numbers and other caustic types can be studied in a
  completely analogous way. And so, we complete
  Tables~\ref{tab:ClassificationSPTsEH1}--\ref{tab:ClassificationSPTsH1H1}.
\end{proof}

The caustic type H1H1 has the following peculiarity: both caustics
$Q_{\lambda_1}$ and $Q_{\lambda_2}$ are 1-sheet hyperboloids. We will denote
$Q_{\lambda_1}$ as the outer hyperboloid and $Q_{\lambda_2}$ as the inner one,
since $\lambda_1 < \lambda_2$. Note that two vertexes of the form $\left( e,
\lambda_1, h_2 \right)$ and $\left( e, \lambda_2, h_2 \right)$ are associated
to the same reversor; see Table~\ref{tab:VertexesPointsReversors_3D}. In
particular, SPTs connecting vertexes of this form are not doubly symmetric.
Likewise, different connections can give rise to the same couple of reversors.
We still classify those connections as different classes, because they have
different geometries. Symbols $\left( R_2, R_3 | \right)$,
$\left( R_3 | R_2 \right)$, $\left( R_2 | R_3
\right)$, and $\left( | R_2, R_3 \right)$ denote four classes
with caustic type H1H1 and the same couple of reversors: $\left\{ R_2, R_3
\right\}$. They are depicted in Table~\ref{tab:SPTs_3D_H1H1_period6}.
But the points $q_{\star}$, associated to each reversor
(Table~\ref{tab:VertexesPointsReversors_3D}) are on the outer
(respectively, inner)
hyperboloid when the reversor is written to the left (respectively, right) of
symbol $|$.
This notation is used in Table~\ref{tab:ClassificationSPTsH1H1}.{\medskip}

We have seen that, once fixed the caustic type, there exists a tight relation
between the symmetry sets associated to an SPT and the kind of winding
numbers. A similar result for standard-like maps was obtained by Kook and
Meiss.\cite{KookMeiss1989}
Compare Tables~\ref{tab:ClassificationSPTsEH1}--\ref{tab:ClassificationSPTsH1H1}
with their Table~I, where they list the 36 classes of SPOs for the 4D symplectic
Froeschl\'e map. Our 4D symplectic billiard map has more classes because a new
factor enters in the classification: the caustic type.

\begin{table}
\caption{\label{tab:ClassificationSPTsH1H1}Analogous of
Table~\ref{tab:ClassificationSPTsEH1} for caustic type H1H1. See comments after
  the proof of Theorem~\ref{thm:SPTs_3D} for the meaning of $(|)$.}
\begin{ruledtabular}
  \begin{tabular}{@{\hspace{0.5cm}}c@{\hspace{0.3cm}}@{\hspace{0.3cm}}l@{\hspace{0.5cm}}}
    \begin{tabular}{c}$( m_0, m_1, m_2 )$\\ kind/minimal\end{tabular} &
      Couples of reversors\\\hline
    $\left( \text{t}, \text{t}, \text{t} \right)$ & $\left( R_3  | f \circ R_{1 2} \right)$,
    $\left( R_2 | f \circ R_{13} \right)$\\
    $\left( 10, 6, 2 \right)$ & $\left( f \circ R_{1 2} | R_3
    \right)$, $\left( f \circ R_{13} | R_2 \right)$\\
    \hline
    $\left( \text{t}, \text{t}, \text{f} \right)$ & $\left( R_2  | f \circ R_{1 2} \right)$, $\left( R_3  | f \circ R_{1
    3} \right)$\\
    $\left( 10, 6, 4 \right)$ & $\left( f \circ R_{1 2} |  R_2
    \right)$, $\left( f \circ R_{1 3} | R_3 \right)$\\
    \hline
    $\left( \text{t}, \text{f}, \text{t} \right)$ & $\left( R_2, f \circ R_{1
    3} | \right)$, $\left( | R_2, f \circ R_{1 3}
    \right)$\\
    $\left( 6, 4, 2 \right)$ & $\left( R_3, f \circ R_{1 2} |
    \right)$, $\left( | R_3, f \circ R_{1 2} \right)$\\
    \hline
    $\left( \text{t}, \text{f}, \text{f} \right)$ & $\left( R_2, f \circ R_{1
    2} | \right)$, $\left( | R_2, f \circ R_{1 2}
    \right)$\\
    $\left( 10, 8, 4 \right)$ & $\left( R_3, f \circ R_{1 3} |
    \right)$, $\left( | R_3, f \circ R_{1 3} \right)$\\
    \hline
    $\left( \text{f}, \text{t}, \text{t} \right)$ & $( R_2 | R_3)$,
    $( f \circ R_{1 3} | f \circ R_{12})$\\
    $\left( 8, 6, 2 \right)$ & $\left( R_3   R_2 \right)$,
    $\left( f \circ R_{1 2}   f \circ R_{13} \right)$\\
    \hline
    $\left( \text{e}, \text{o}, \text{e} \right)$ & $(R_2 | R_2)$,
    $(f \circ R_{1 2} | f \circ R_{12})$\\
    $\left( 4, 3, 2 \right)$ & $(R_3 | R_3)$,
    $(f \circ R_{1 3} | f \circ R_{1 3})$\\
    \hline
    $\left( \text{f}, \text{f}, \text{t} \right)$ & $( R_2, R_3 | )$, $( f \circ R_{1 2}, f \circ R_{1 3}  | )$\\
    $\left( 8, 4, 2 \right)$ & $( | R_2, R_3 )$,
    $( | f \circ R_{12}, f \circ R_{13})$\\
  \end{tabular}
\end{ruledtabular}
\end{table}

\subsection{Gallery of minimal SPTs}

We have 112 classes of SPTs listed in
Tables~\ref{tab:ClassificationSPTsEH1}--\ref{tab:ClassificationSPTsH1H1},
so we tackle the task of finding a minimal representative in each class.
They have periods $m_0 \in \{4, 5, 6, 8, 10\}$.

The algorithm for the caustic type EH1 is:
1) Choose one of the seven minimal winding numbers $(m_0, m_1, m_2)$
   in Table~\ref{tab:ClassificationSPTsEH1};
2) Find caustic parameters $\lambda_1 \in E$ and $\lambda_2 \in H_1$
   such that Eq.~(\ref{eq:FrequencyWindingNumbers}) holds;
3) Choose one of the four couples of reversors $\{r, \breve{r}\}$
   in the corresponding row of Table~\ref{tab:ClassificationSPTsEH1};
4) Get a point $(q, p)$ from Table~\ref{tab:VertexesPointsReversors_3D} using
   the reversor $r$ (or $\breve{r}$); and
5) Draw the doubly SPT through $q$ with velocity $p$.

Only step~2) is problematic, because it requires the inversion of the frequency map.
The main obstacle is that Eq.~(\ref{eq:FrequencyWindingNumbers}) might not have
solution in $E \times H_1$ for some of the minimal winding numbers at hand.
Nevertheless, it is known~\cite{CasasRamirez2011} that
$\omega|_{E \times H_1}$ is a diffeomorphism and
\[ \lim_{a_1 \rightarrow 0^+} \omega \left( E \times H_1 \right) = \left\{
   \left( \omega_1, \omega_2 \right) \in \Rset^2 : 0 < \omega_2 <
   \omega_1 < \textstyle{\frac{1}{2}} \right\} . \]
Consequently, Eq.~(\ref{eq:FrequencyWindingNumbers}) has a unique
solution in $E \times H_1$ provided that ellipsoid~(\ref{eq:Ellipsoid3D}) is
flat enough.

\begin{table*}
  \caption{\label{tab:SPTs_3D_periods45}Two minimal SPTs of period 4 and four
  minimal SPTs of period 5. The ``Data'' column includes consecutively:
  caustic type, winding numbers $\left( m_0, m_1, m_2 \right)$, ellipsoid
  parameters $a_2$ and $a_1$ ---we have set $a_3 = 1$---, caustic
  parameters $\lambda_2$ and $\lambda_1$, and all reversors whose symmetry
  sets intersect the corresponding SPOs.}
  \begin{tabular}{c|l|c|c|c}
    3D $\left( x_1 : \uparrow, x_2 : \searrow, x_3 : \swarrow \right)$ & Plane
    $\Pi_1$ $\left( x_2 : \uparrow, x_3 : \rightarrow \right)$ & Plane $\Pi_2$
    $\left( x_1 : \uparrow  , x_3 : \leftarrow \right)$ &
    Plane $\Pi_3$ $\left( x_1 : \uparrow, x_2 : \rightarrow \right)$ & Data\\
    \hline
    \scalebox{0.2}{\begin{tabular}{c}\includegraphics{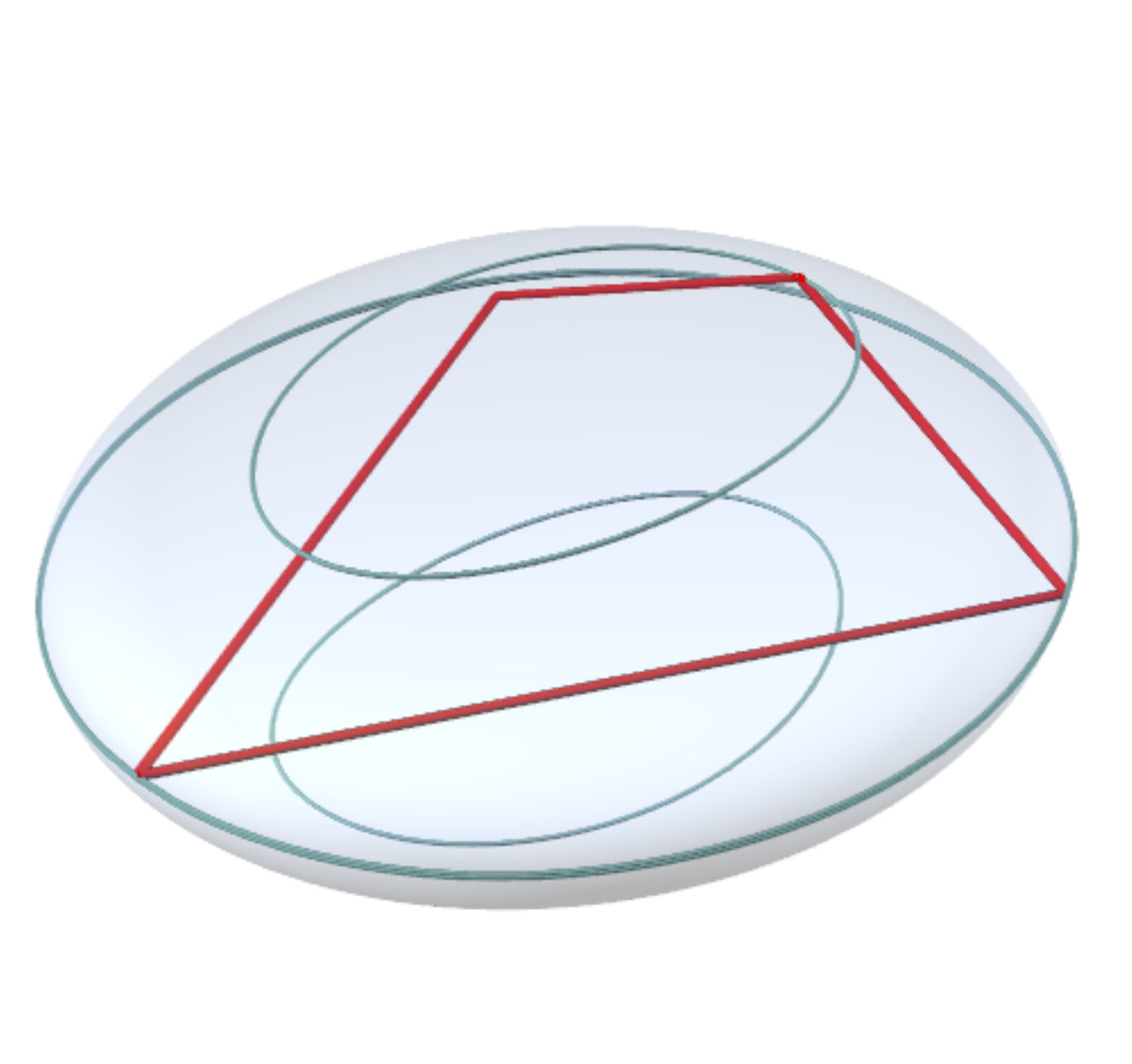}\end{tabular}} &
    \scalebox{0.2}{\begin{tabular}{c}\includegraphics{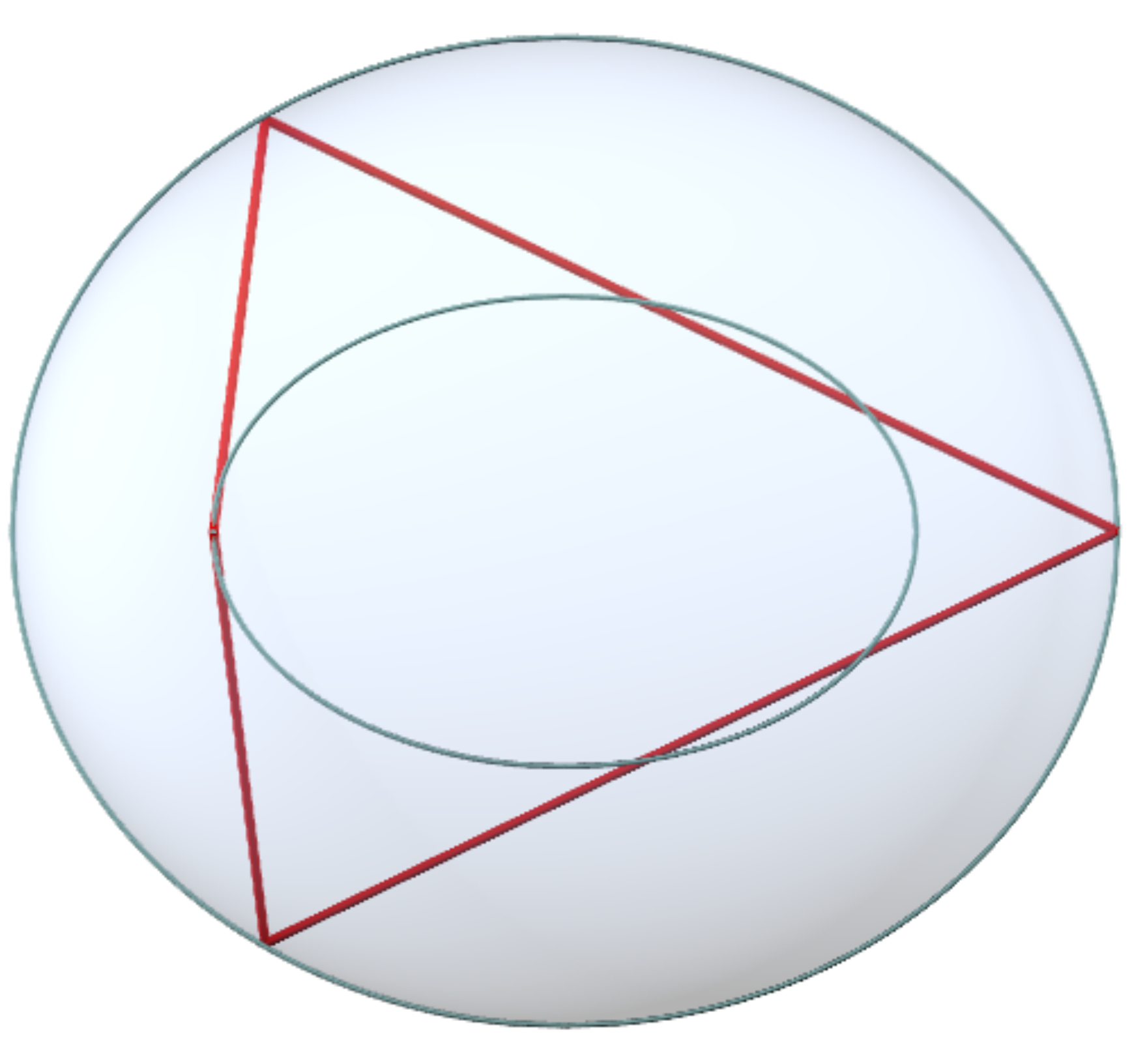}\end{tabular}} &
    \scalebox{0.2}{\begin{tabular}{c}\includegraphics{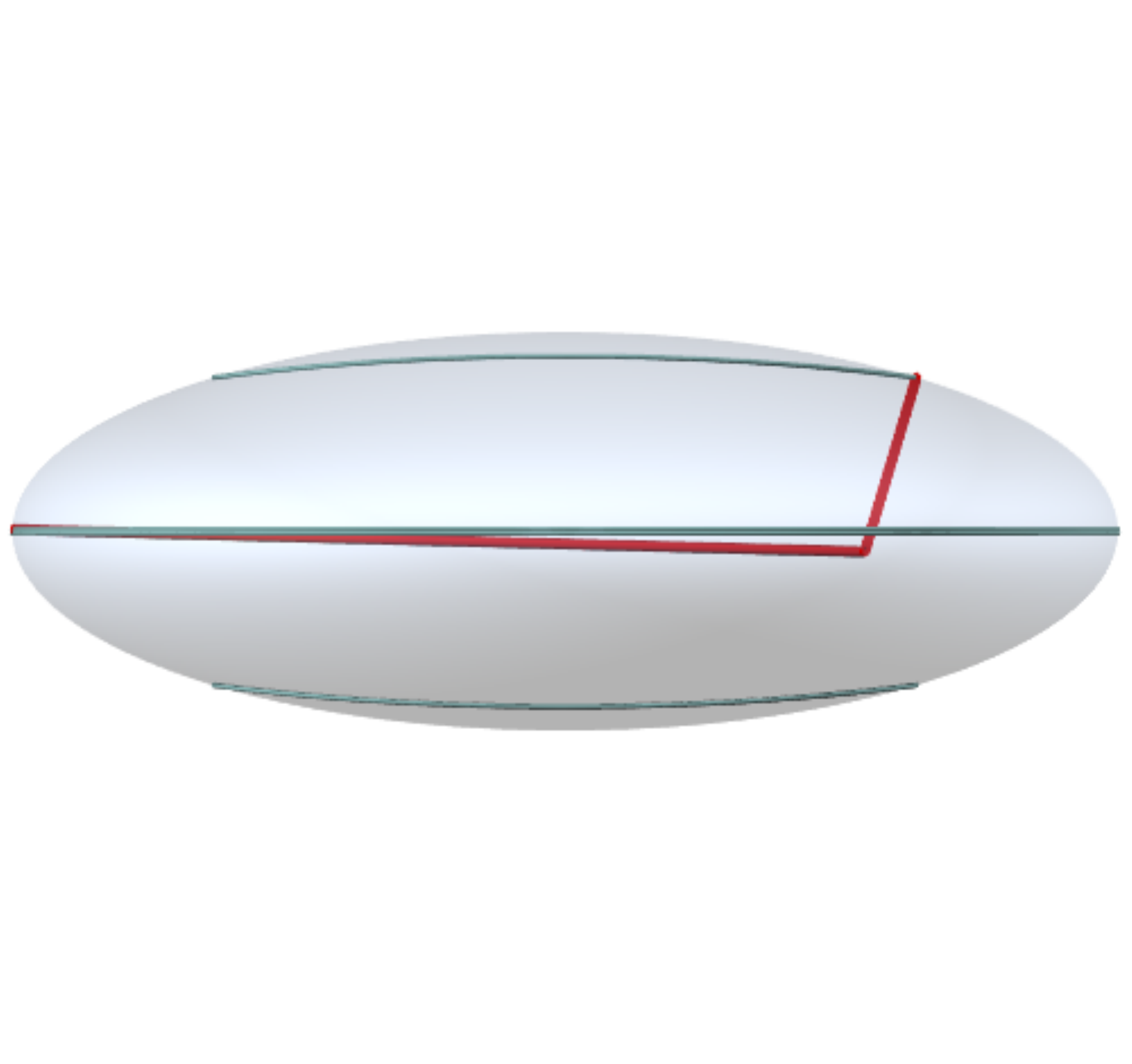}\end{tabular}} &
    \scalebox{0.2}{\begin{tabular}{c}\includegraphics{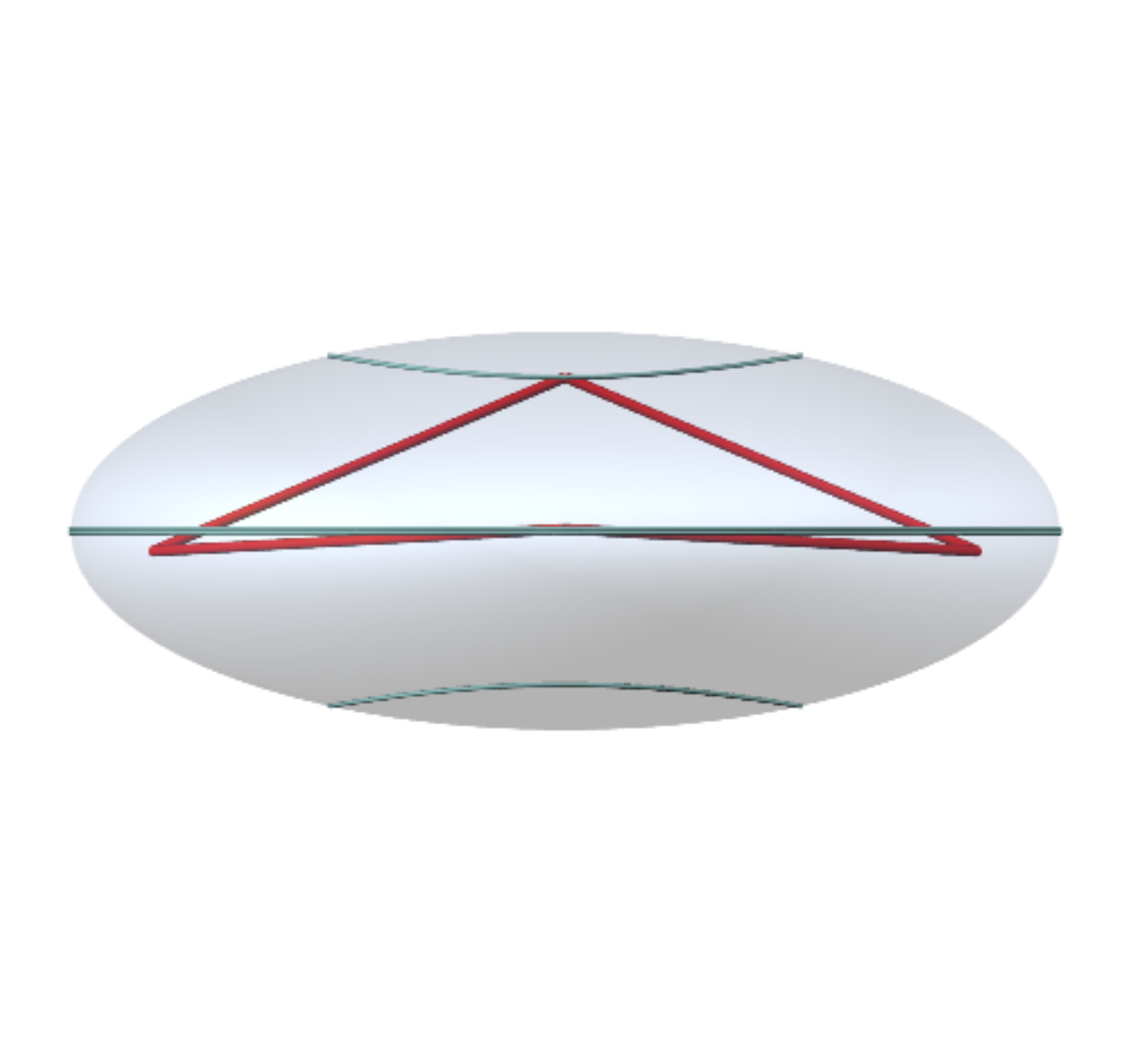}\end{tabular}} &
    \begin{tabular}{c}
      H1H1\\
      $\left( 4, 3, 2 \right)$\\
      0.8\\
      0.13\\
      0.648376\\
      0.130077\\
      $R_2$
    \end{tabular}\\\hline
    \scalebox{0.2}{\begin{tabular}{c}\includegraphics{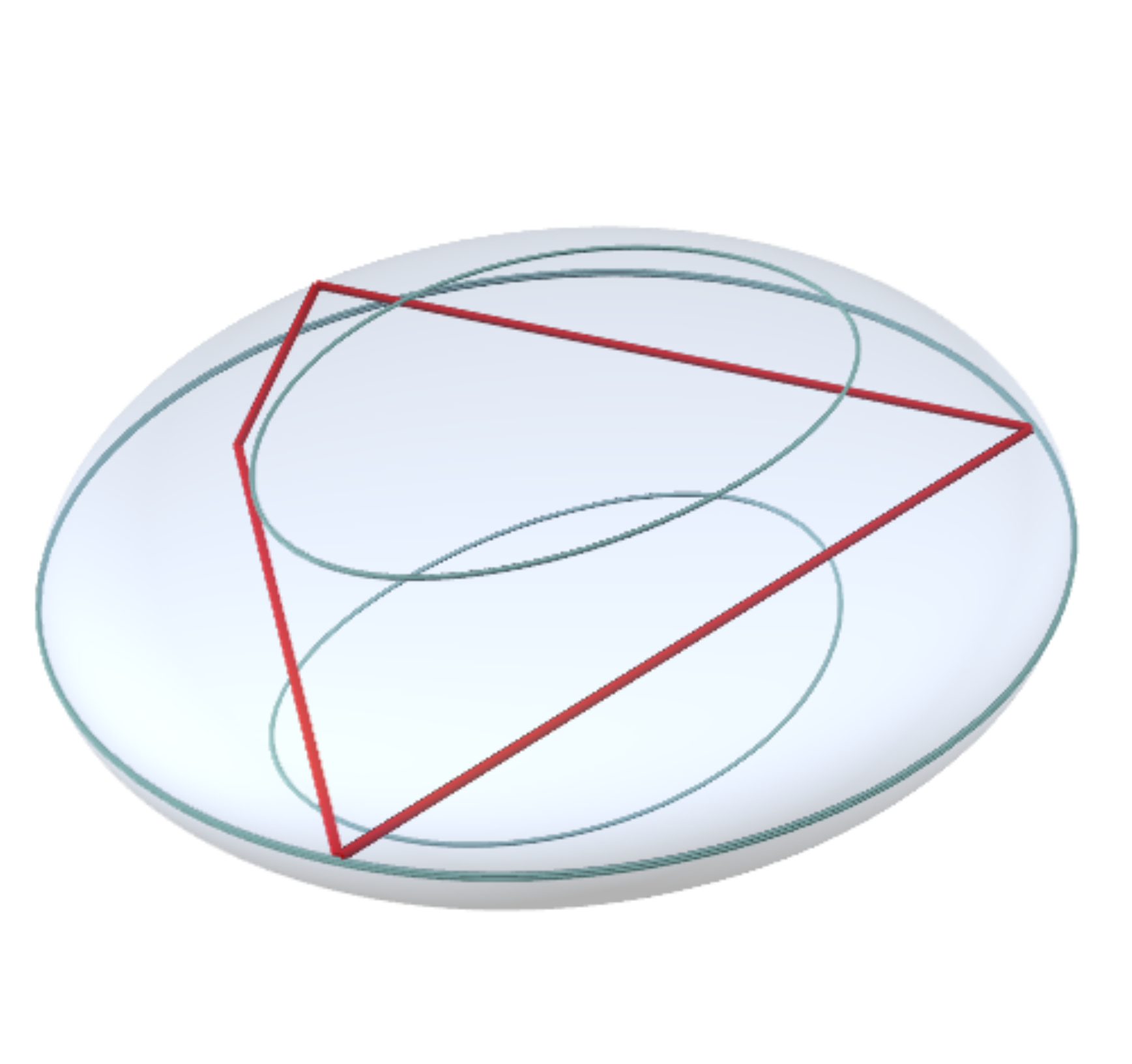}\end{tabular}} &
    \scalebox{0.2}{\begin{tabular}{c}\includegraphics{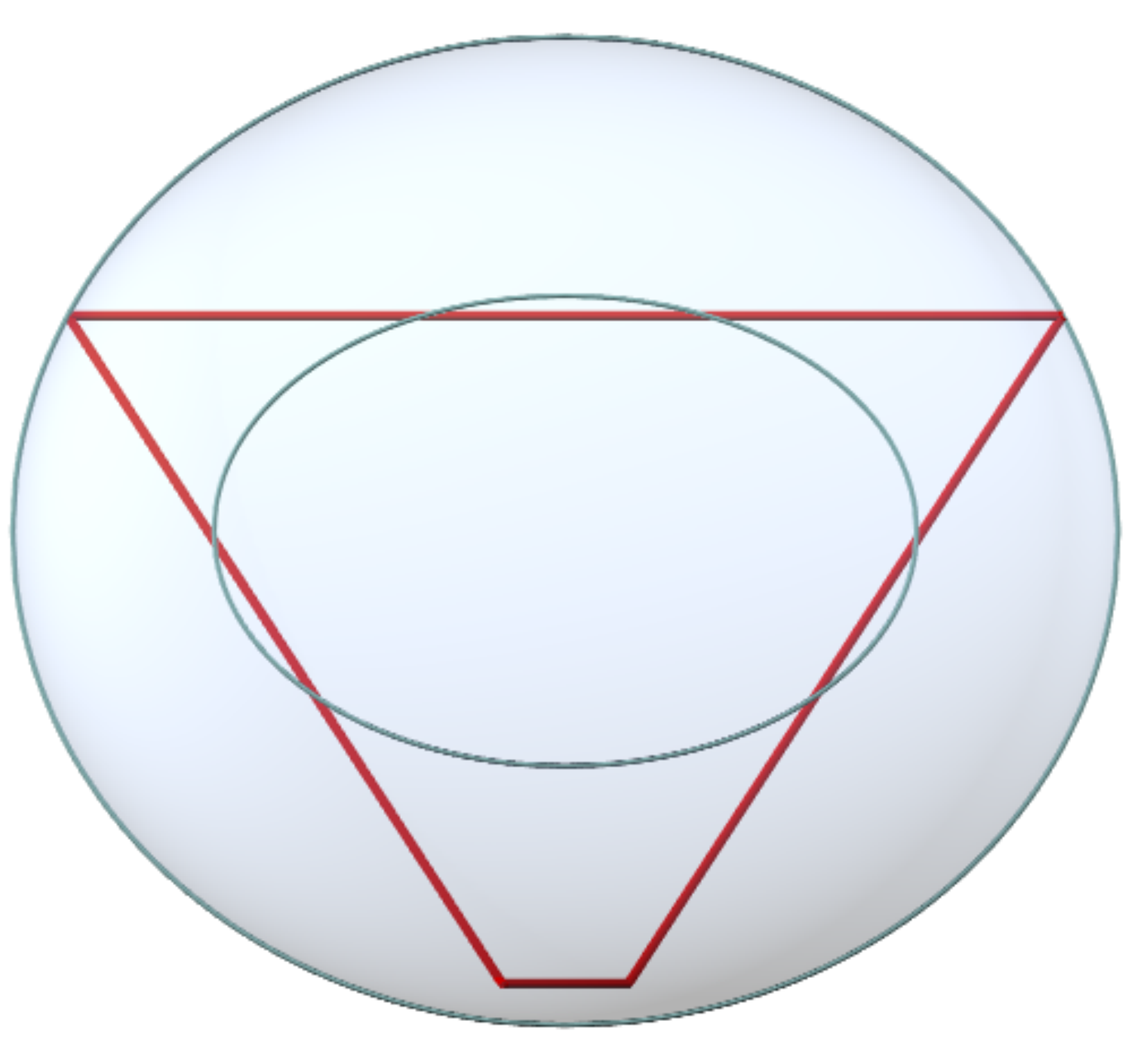}\end{tabular}} &
    \scalebox{0.2}{\begin{tabular}{c}\includegraphics{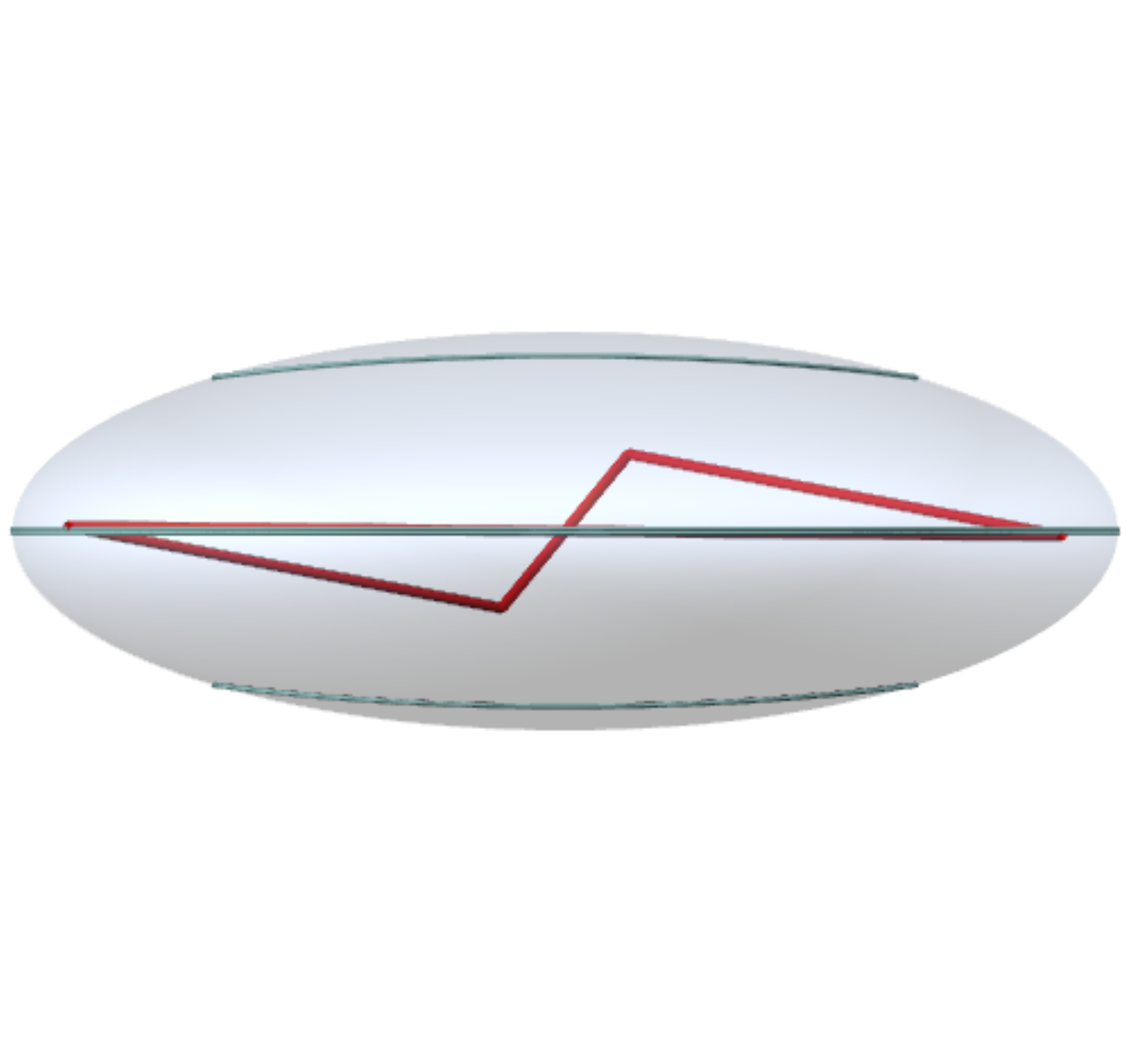}\end{tabular}} &
    \scalebox{0.2}{\begin{tabular}{c}\includegraphics{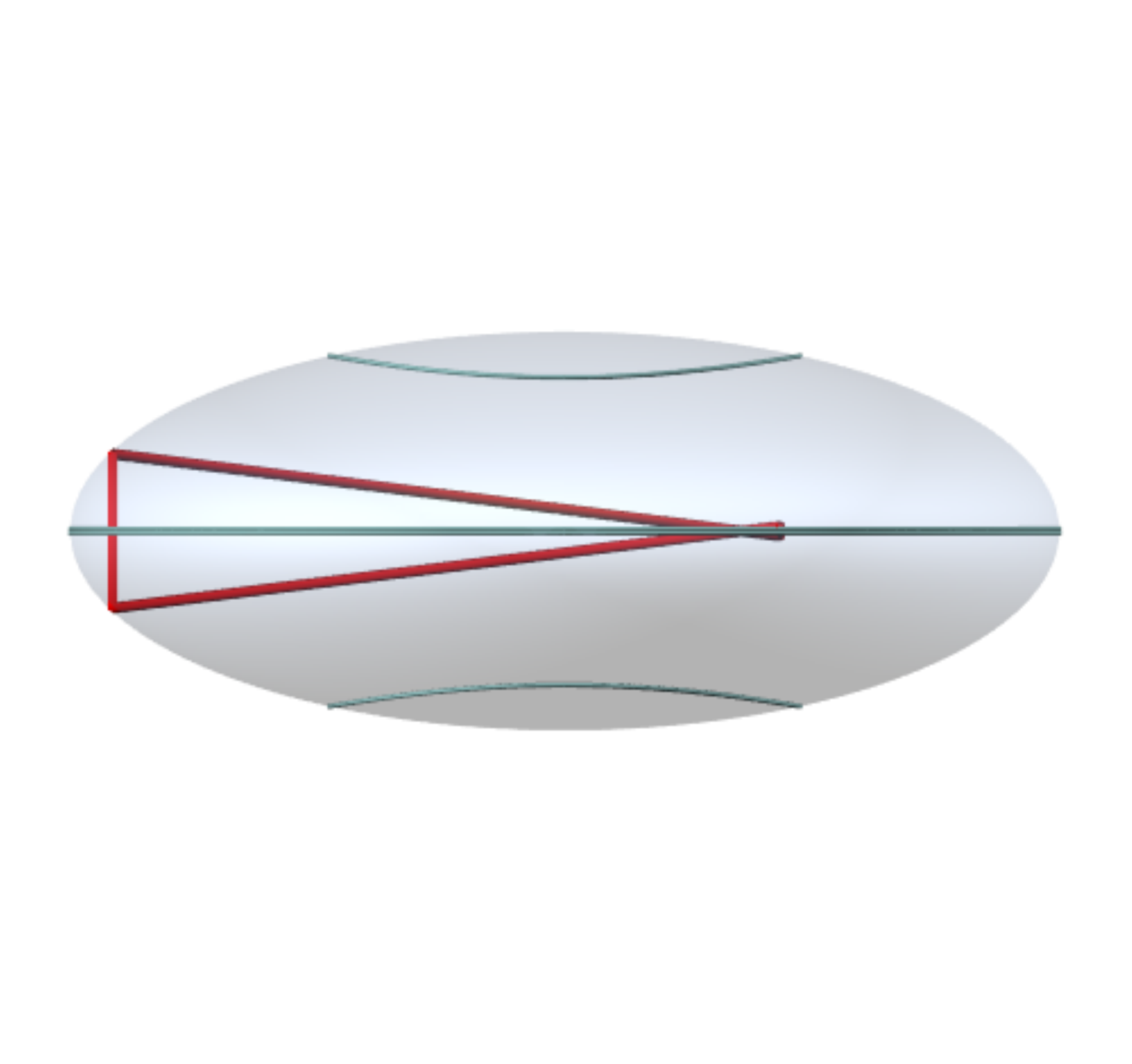}\end{tabular}} &
    \begin{tabular}{c}
      H1H1\\
      $\left( 4, 3, 2 \right)$\\
      0.8\\
      0.13\\
      0.648376\\
      0.130077\\
      $f \circ R_{1 3}$
    \end{tabular}\\\hline
    \scalebox{0.2}{\begin{tabular}{c}\includegraphics{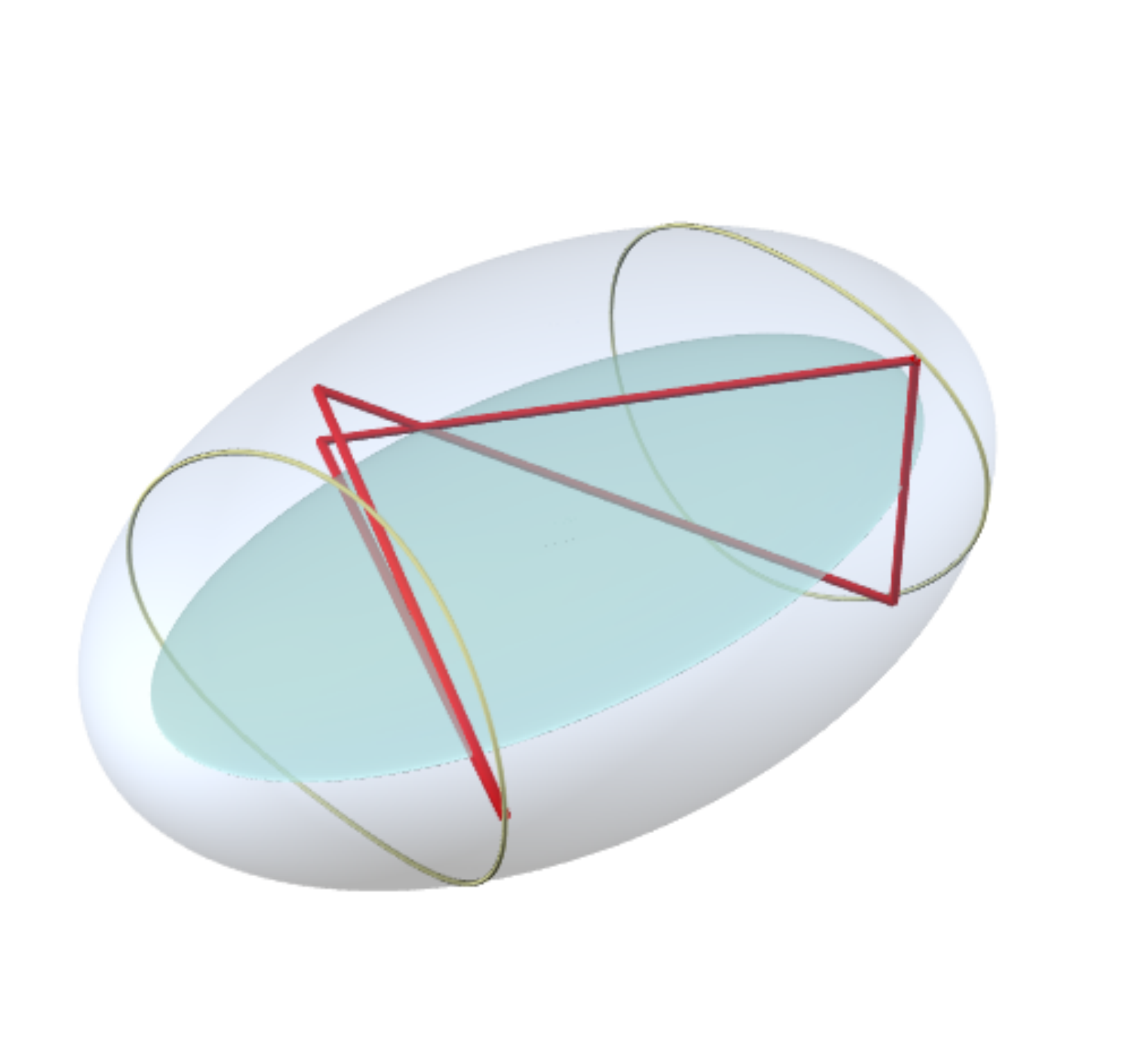}\end{tabular}} &
    \scalebox{0.2}{\begin{tabular}{c}\includegraphics{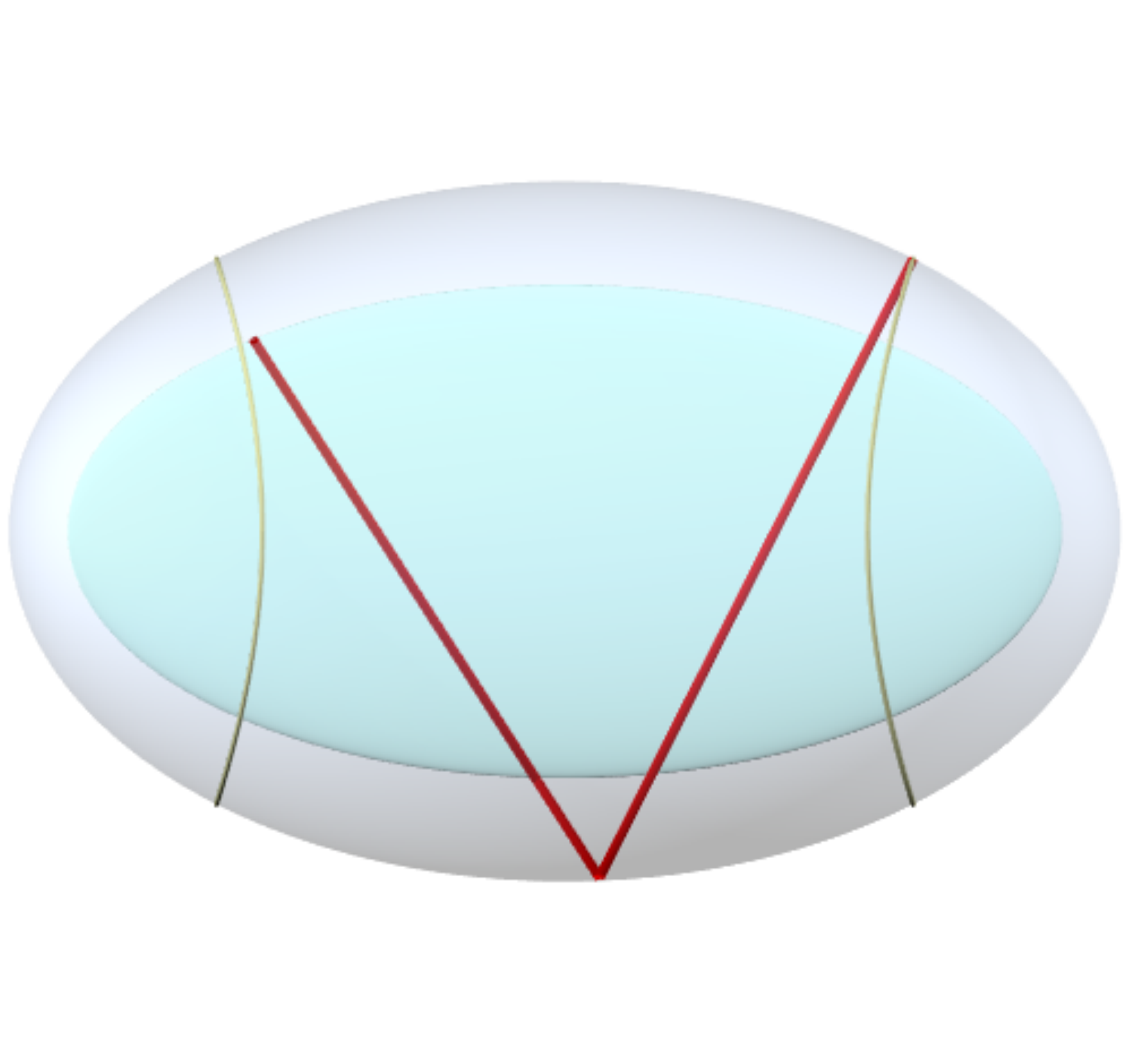}\end{tabular}} &
    \scalebox{0.2}{\begin{tabular}{c}\includegraphics{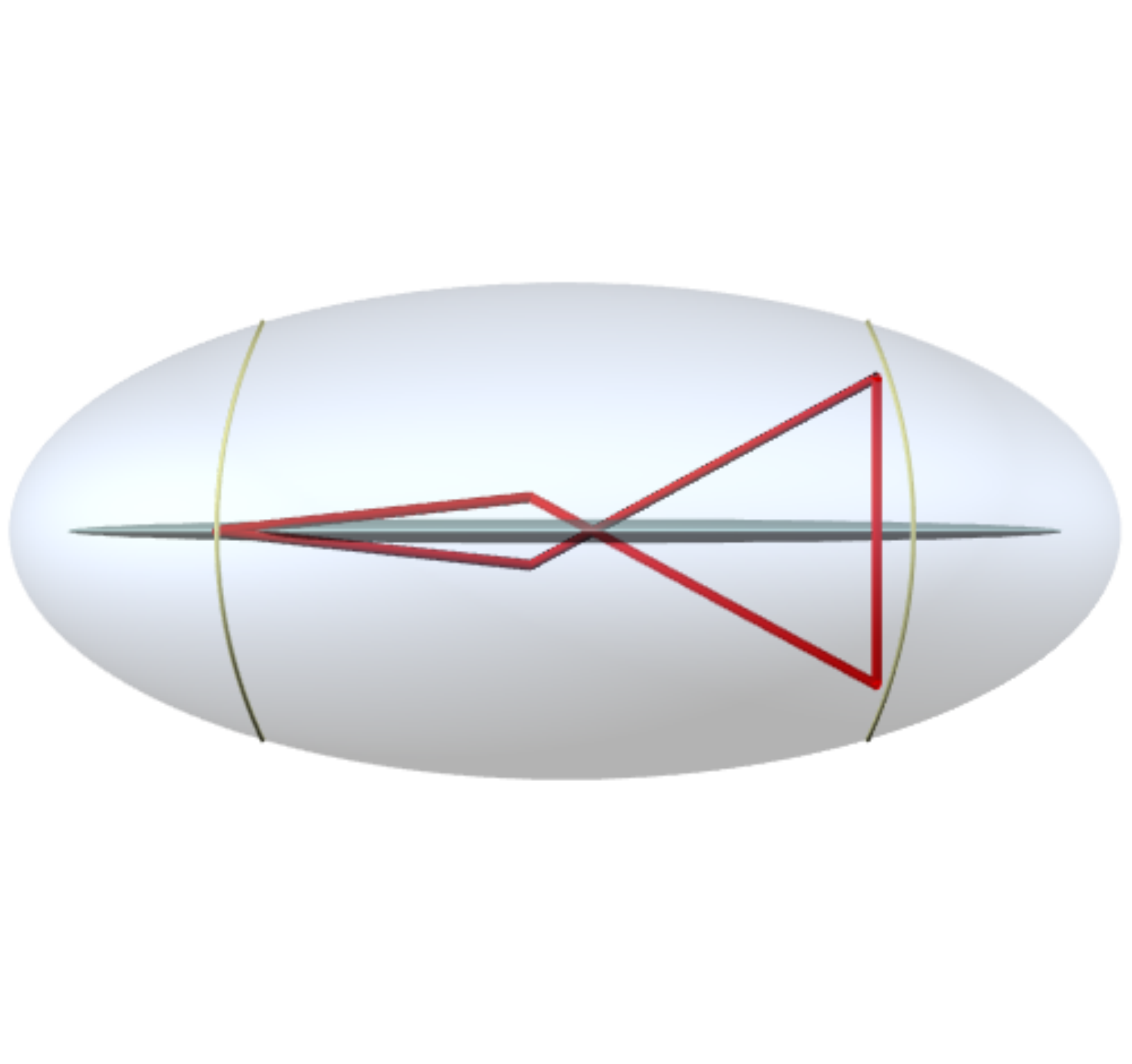}\end{tabular}} &
    \scalebox{0.2}{\begin{tabular}{c}\includegraphics{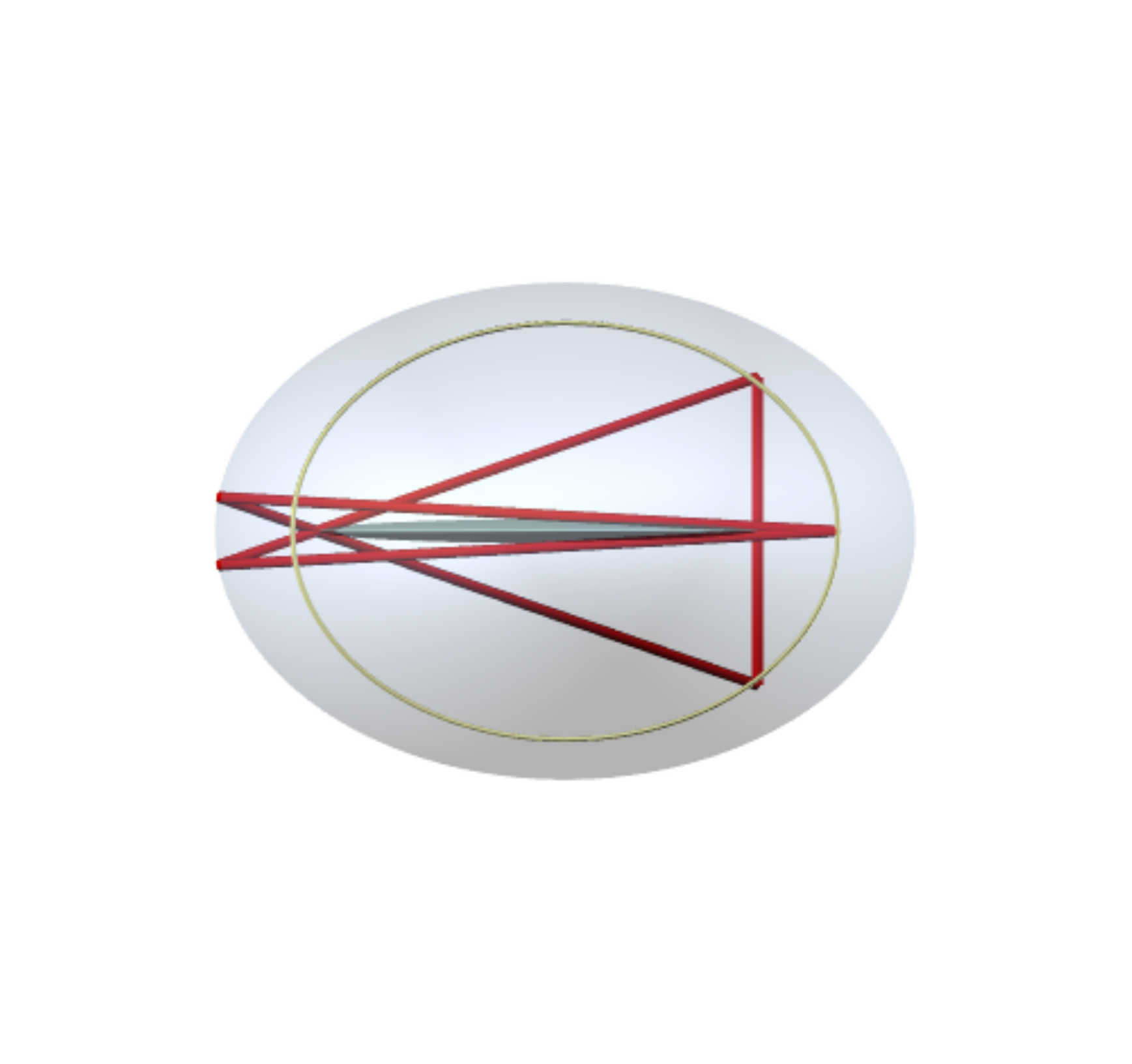}\end{tabular}} &
    \begin{tabular}{c}
      EH2\\
      $\left( 5, 4, 2 \right)$\\
      0.3969\\
      0.2\\
      0.762965\\
      0.199523\\
      $R_1$\\
      $f \circ R_1$
    \end{tabular}\\\hline
    \scalebox{0.2}{\begin{tabular}{c}\includegraphics{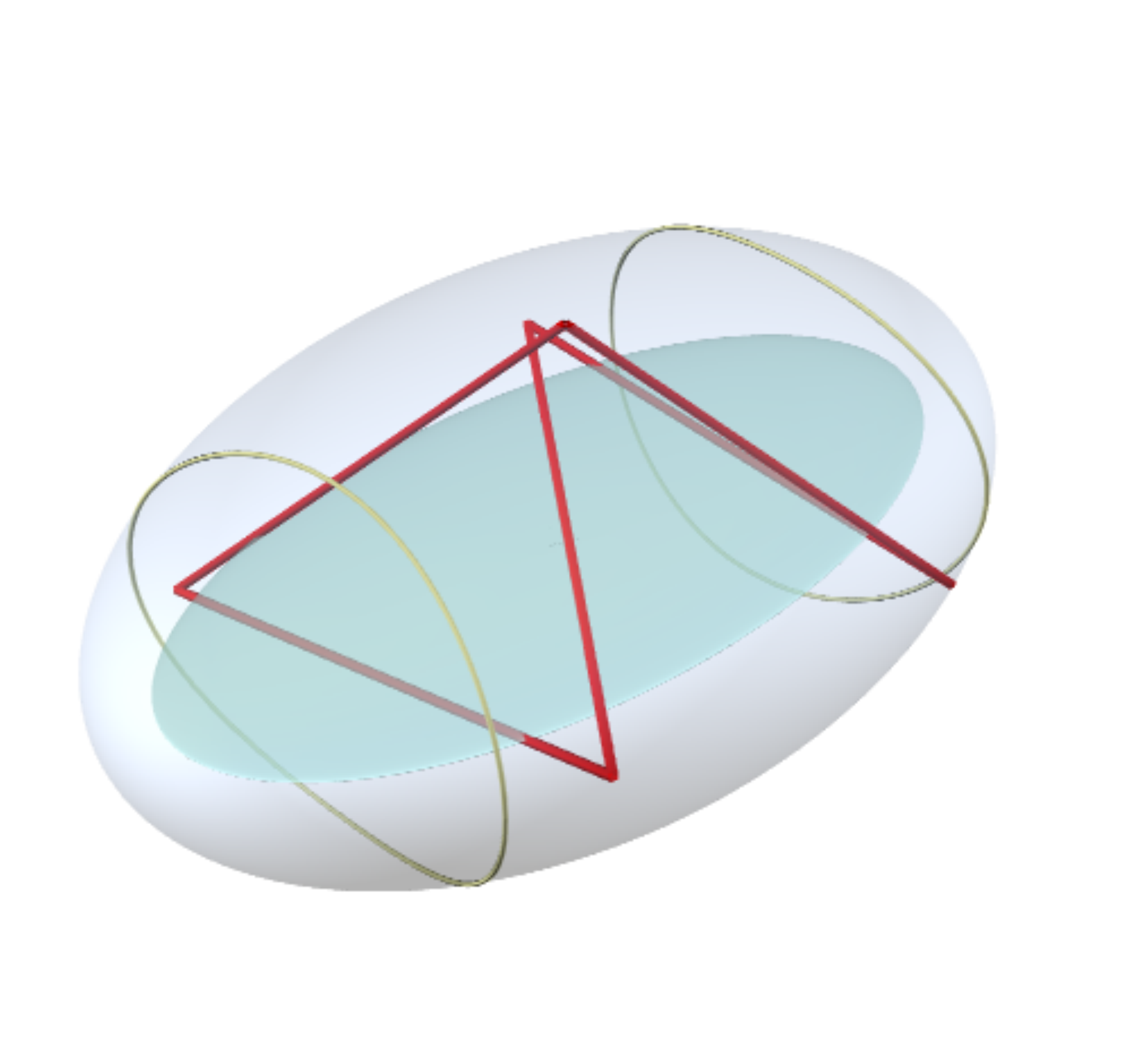}\end{tabular}} &
    \scalebox{0.2}{\begin{tabular}{c}\includegraphics{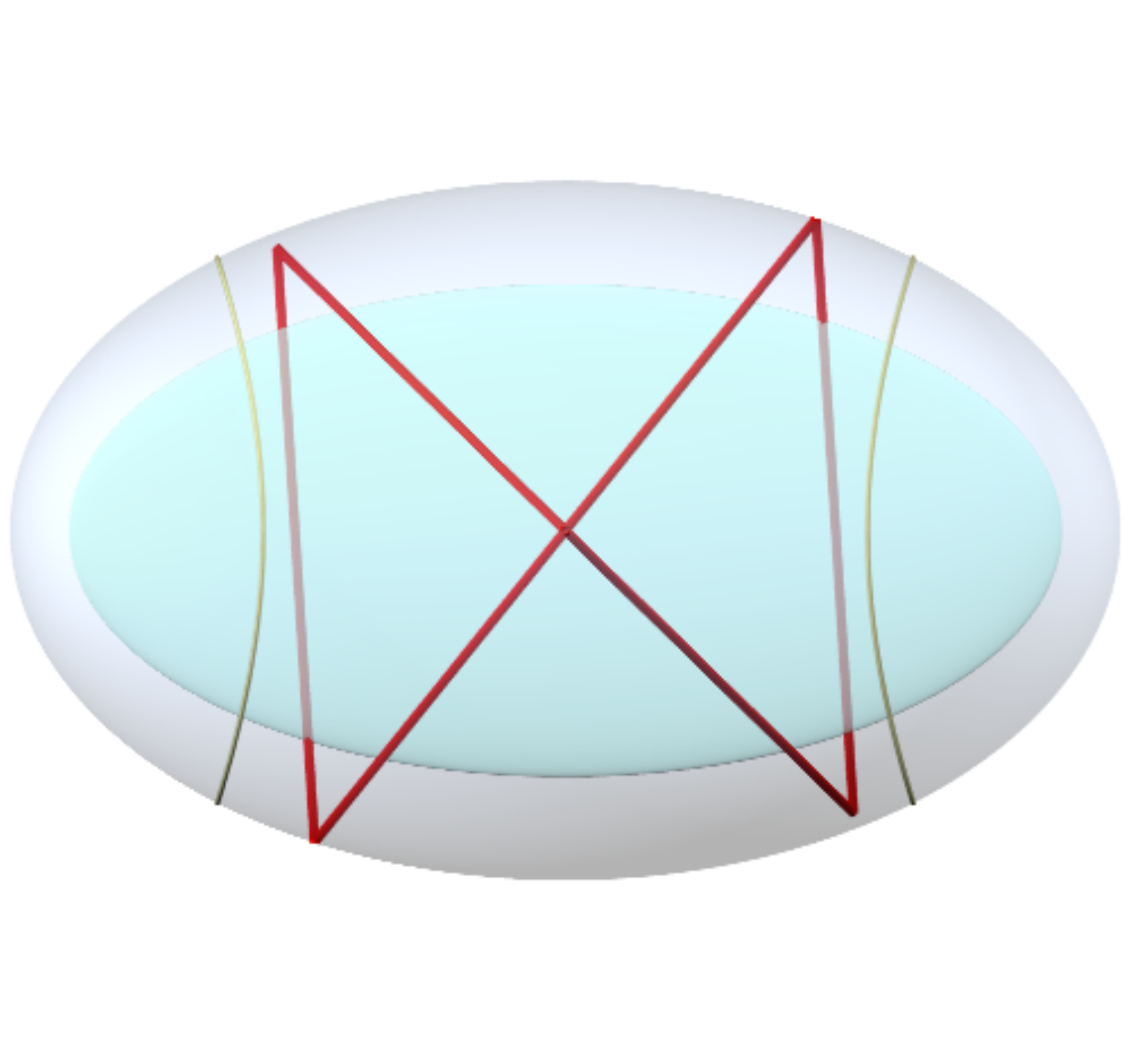}\end{tabular}} &
    \scalebox{0.2}{\begin{tabular}{c}\includegraphics{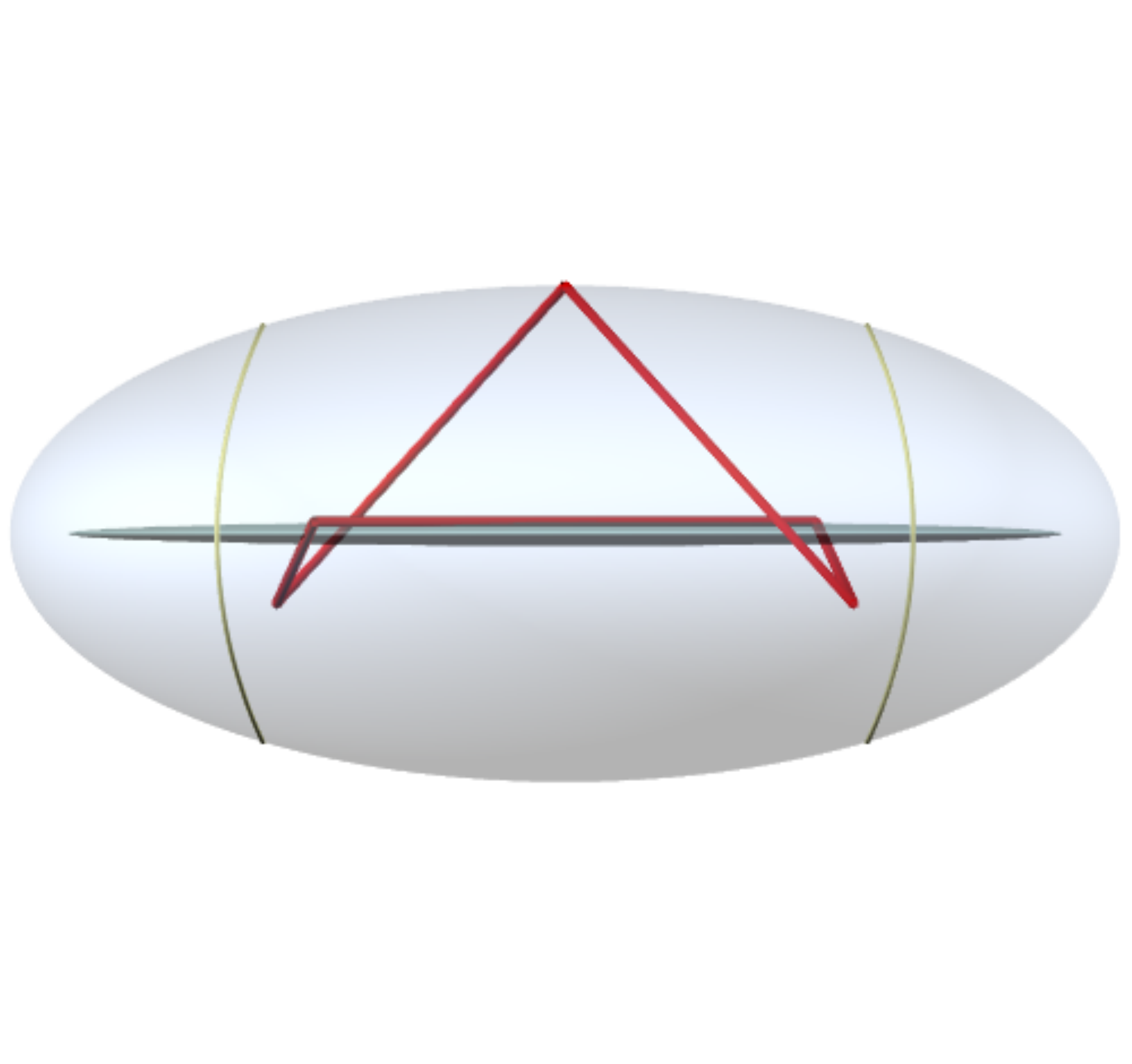}\end{tabular}} &
    \scalebox{0.2}{\begin{tabular}{c}\includegraphics{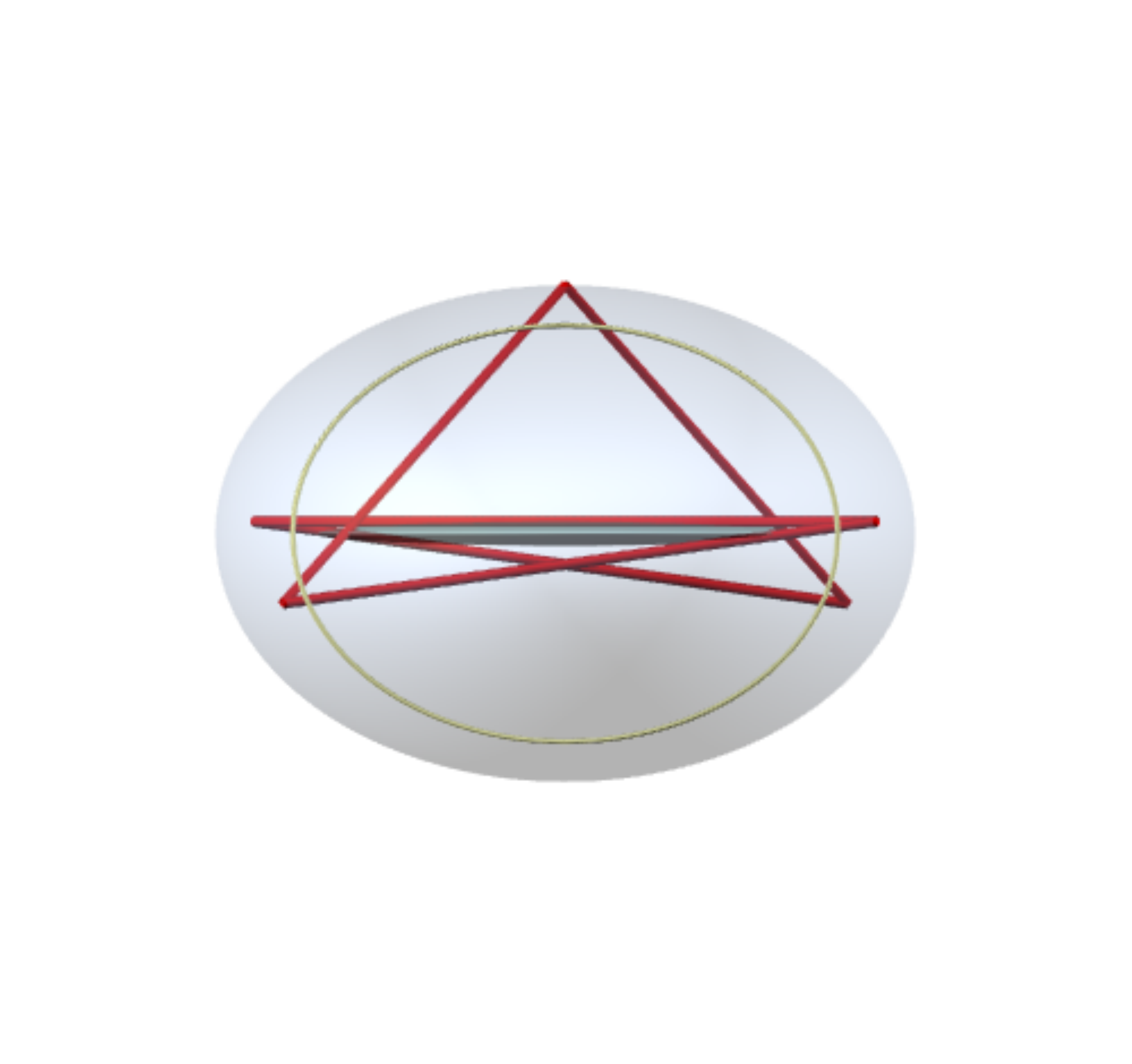}\end{tabular}} &
    \begin{tabular}{c}
      EH2\\
      $\left( 5, 4, 2 \right)$\\
      0.3969\\
      0.2\\
      0.762965\\
      0.199523\\
      $R_{2 3}$\\
      $f \circ R_{2 3}$
    \end{tabular}\\\hline
    \scalebox{0.2}{\begin{tabular}{c}\includegraphics{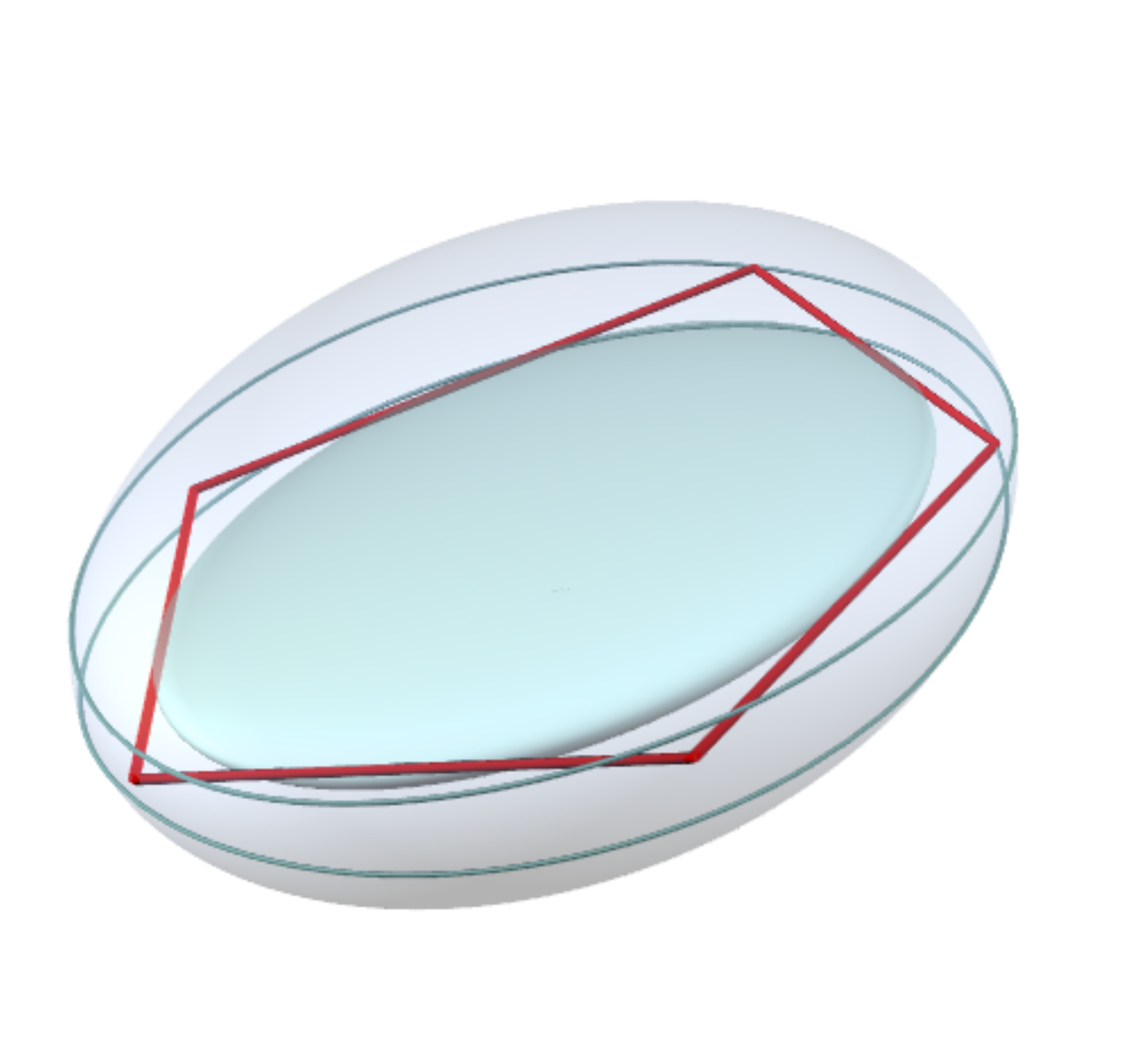}\end{tabular}} &
    \scalebox{0.2}{\begin{tabular}{c}\includegraphics{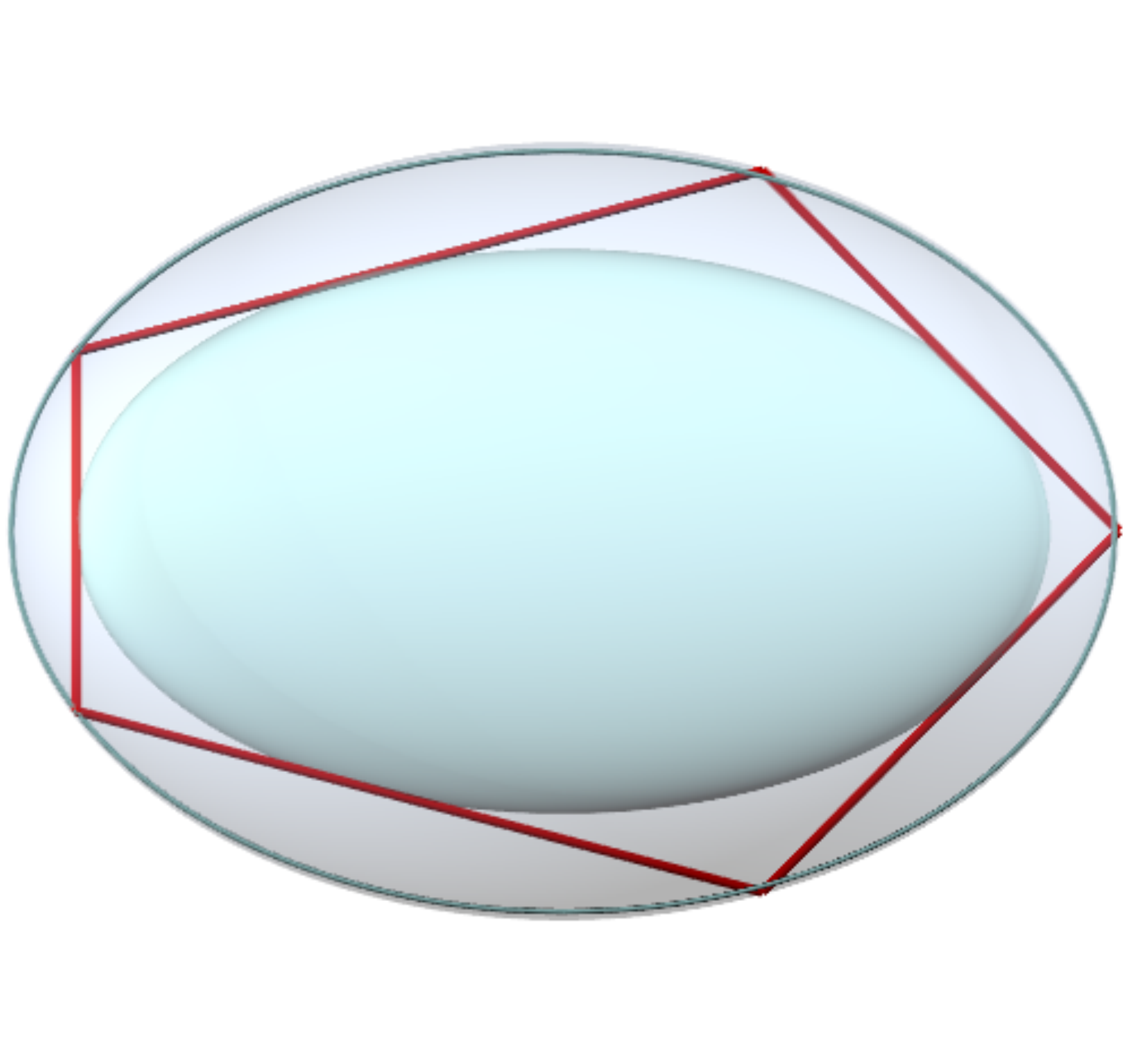}\end{tabular}} &
    \scalebox{0.2}{\begin{tabular}{c}\includegraphics{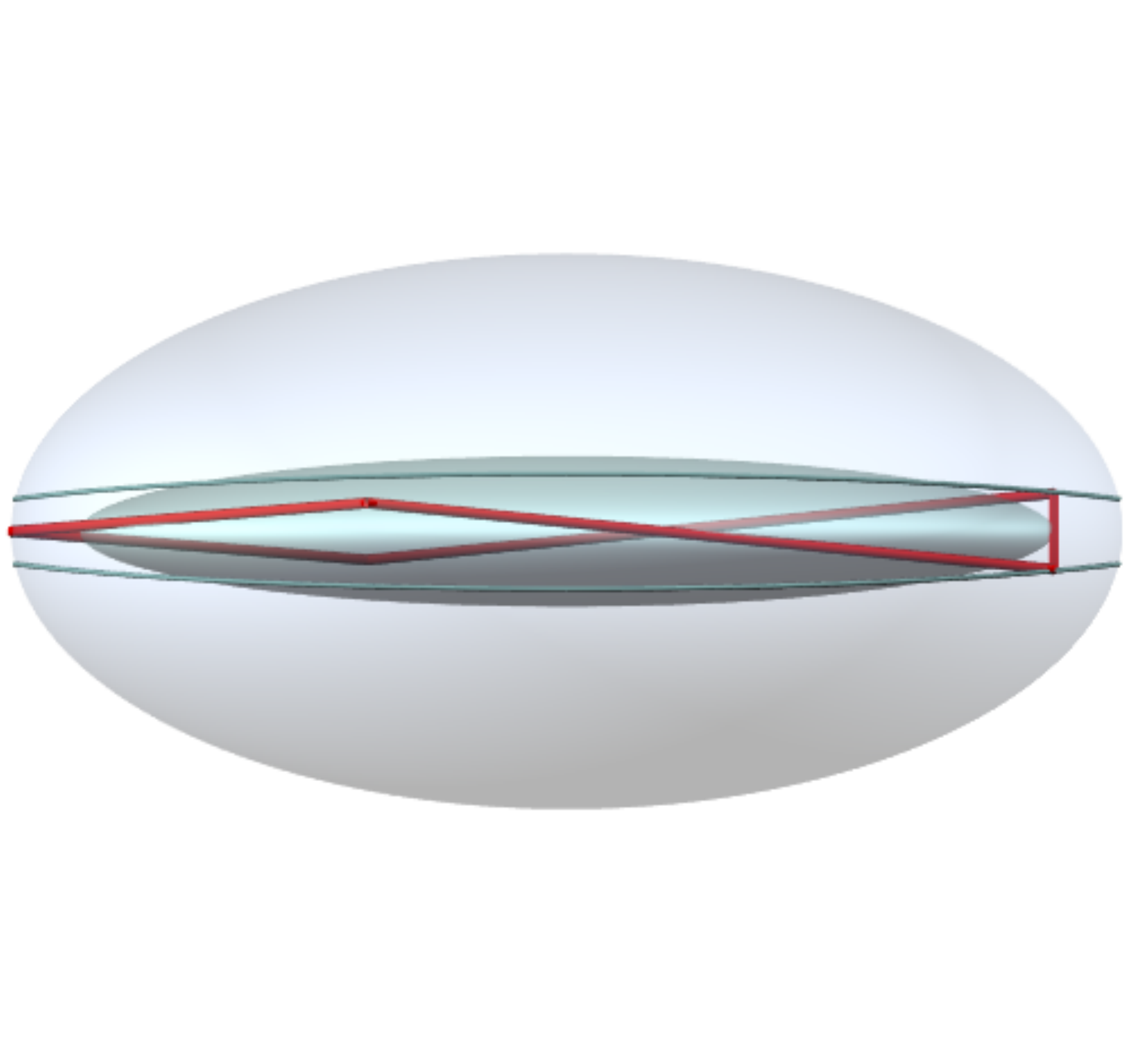}\end{tabular}} &
    \scalebox{0.2}{\begin{tabular}{c}\includegraphics{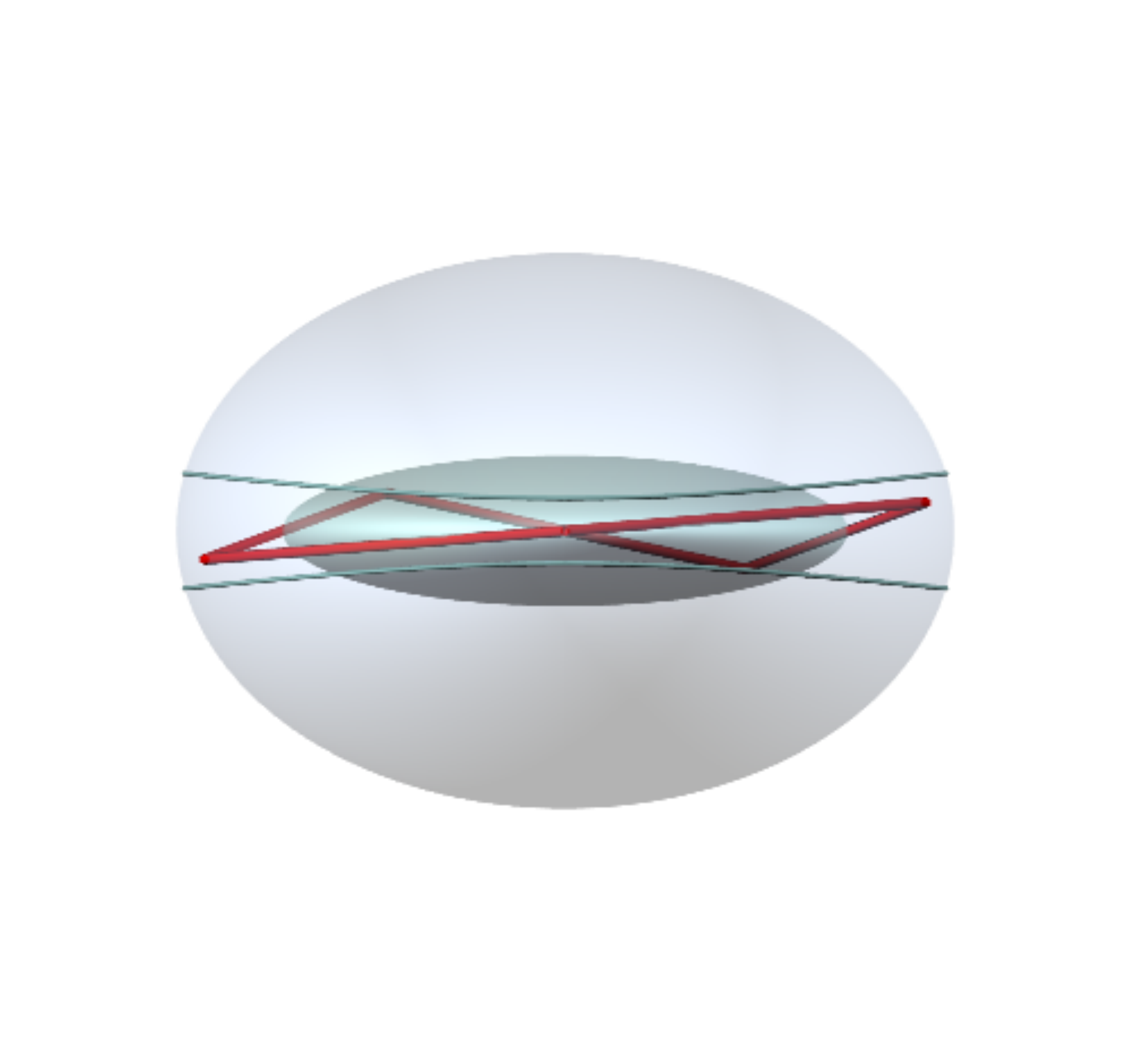}\end{tabular}} &
    \begin{tabular}{c}
      EH1\\
      $\left( 5, 4, 2 \right)$\\
      0.49\\
      0.25\\
      0.260266\\
      0.231635\\
      $R_{1 2}$\\
      $f \circ R_{1 2}$
    \end{tabular}\\\hline
    \scalebox{0.2}{\begin{tabular}{c}\includegraphics{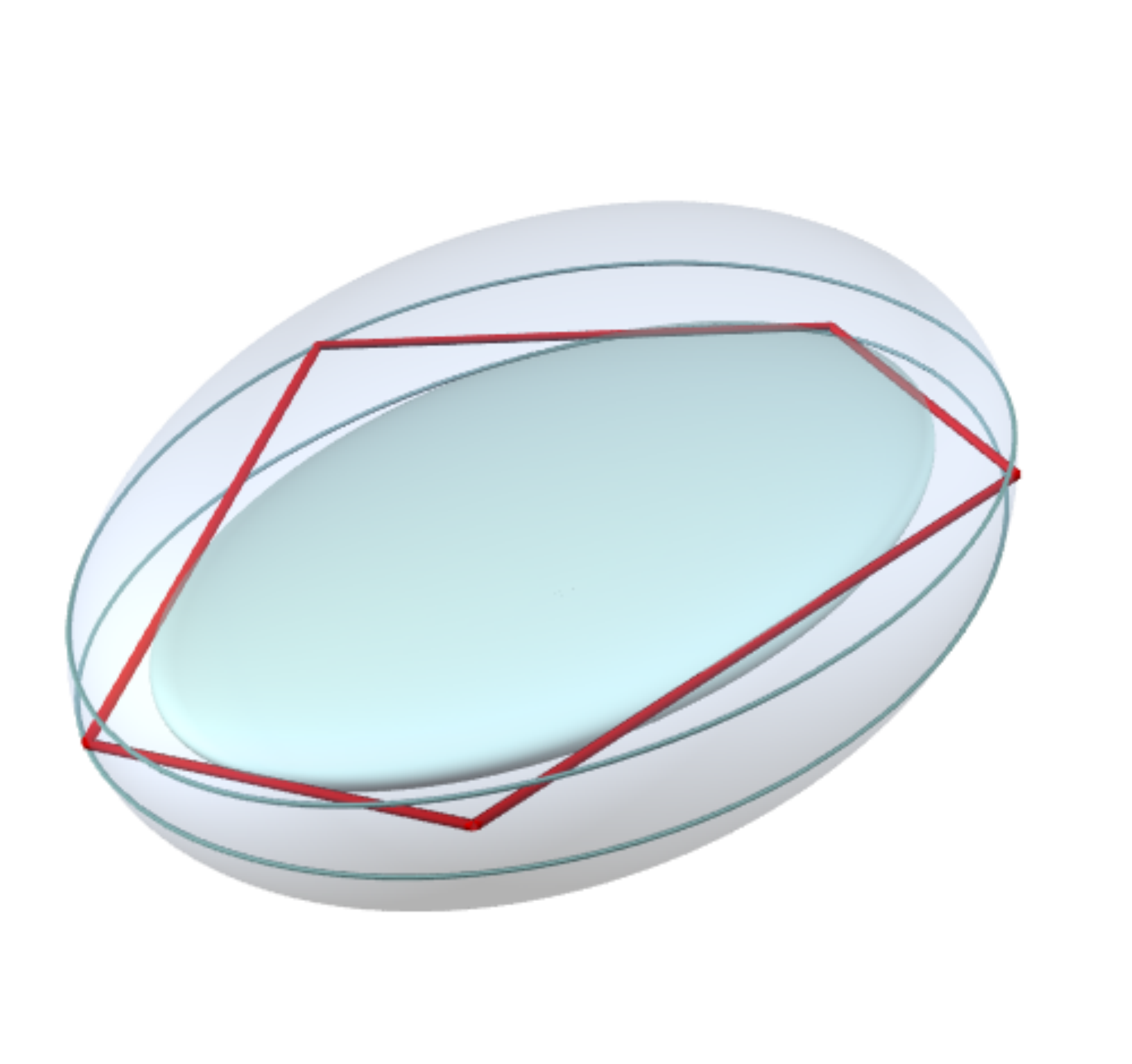}\end{tabular}} &
    \scalebox{0.2}{\begin{tabular}{c}\includegraphics{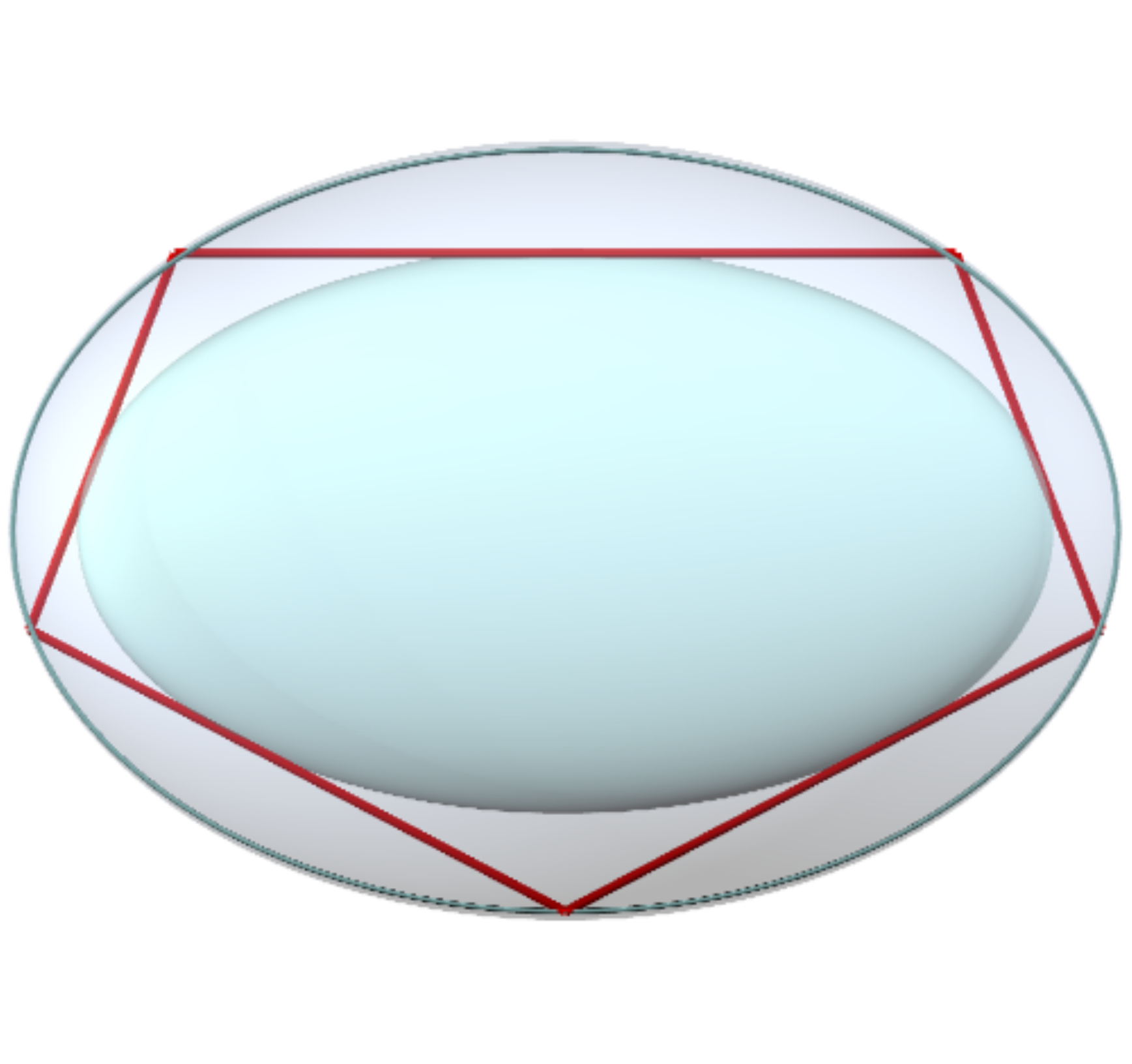}\end{tabular}} &
    \scalebox{0.2}{\begin{tabular}{c}\includegraphics{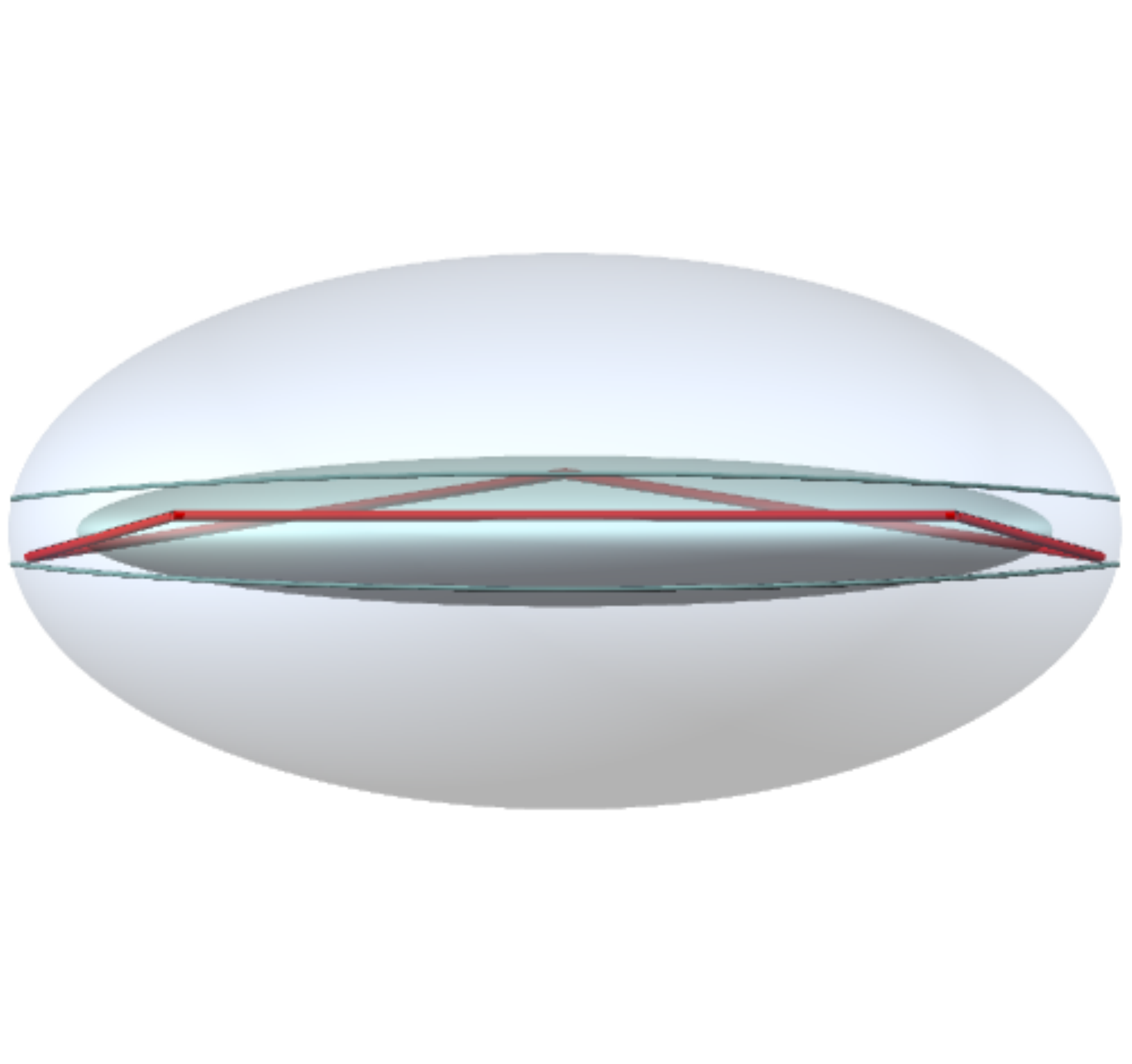}\end{tabular}} &
    \scalebox{0.2}{\begin{tabular}{c}\includegraphics{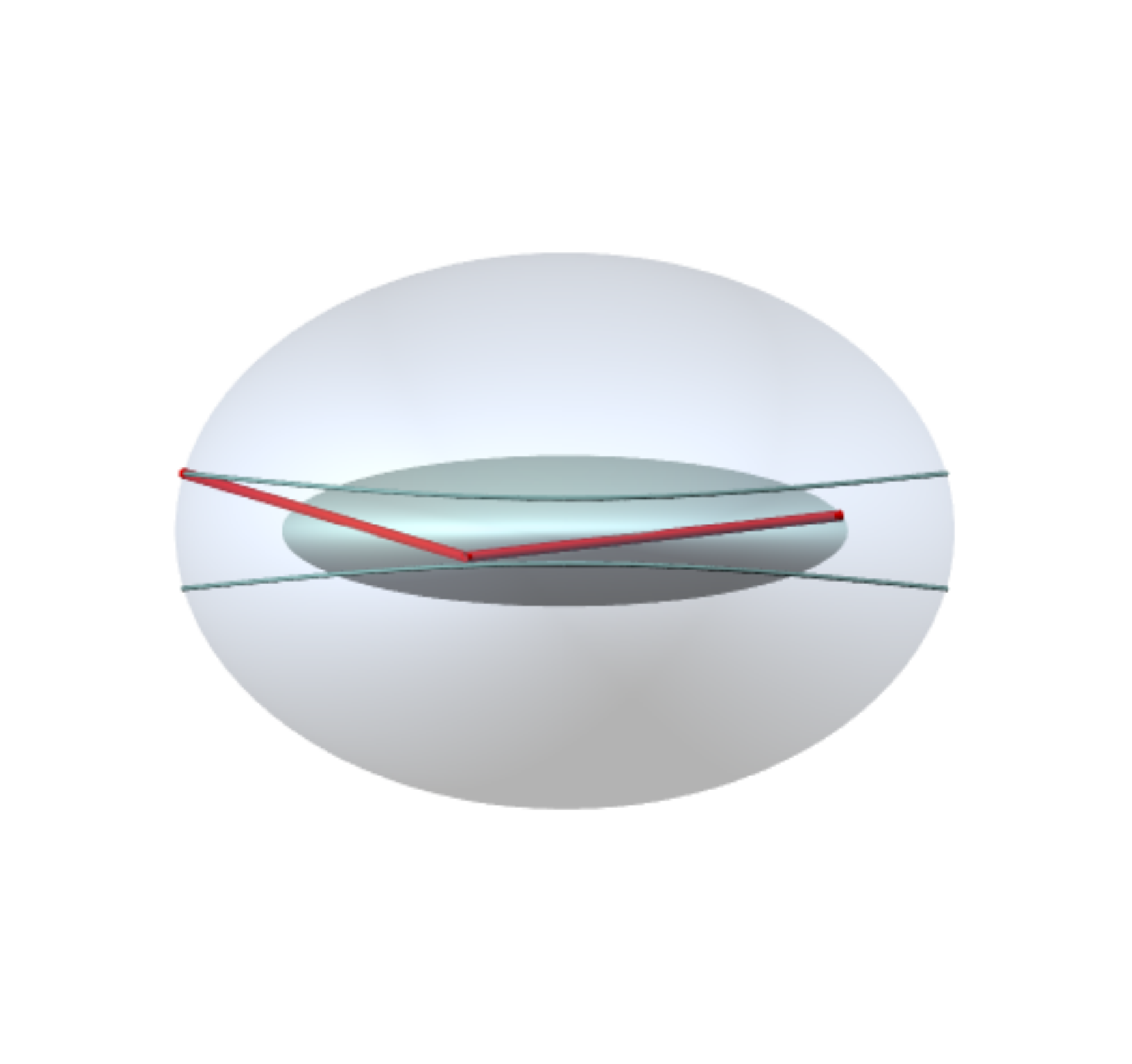}\end{tabular}} &
    \begin{tabular}{c}
      EH1\\
      $\left( 5, 4, 2 \right)$\\
      0.49\\
      0.25\\
      0.260266\\
      0.231635\\
      $R_3$\\
      $f \circ R_3$
    \end{tabular}\\
    \hline
  \end{tabular}
\end{table*}

\begin{table*}
  \caption{\label{tab:SPTs_3D_H1H2_period6}The four classes of SPTs with
  minimal winding numbers $\left( m_0, m_1, m_1 \right) = \left( 6, 4, 2
  \right)$ and H1H2-caustics.
  ``Data'' as in Table~\ref{tab:SPTs_3D_periods45}.}
  \begin{tabular}{c|l|c|c|c}
    3D $\left( x_1 : \uparrow, x_2 : \searrow, x_3 : \swarrow \right)$ & Plane
    $\Pi_1$ $\left( x_2 : \uparrow, x_3 : \rightarrow \right)$ & Plane $\Pi_2$
    $\left( x_1 : \uparrow  , x_3 : \leftarrow \right)$ &
    Plane $\Pi_3$ $\left( x_1 : \uparrow, x_2 : \rightarrow \right)$ & Data\\
    \hline
    \scalebox{0.2}{\begin{tabular}{c}\includegraphics{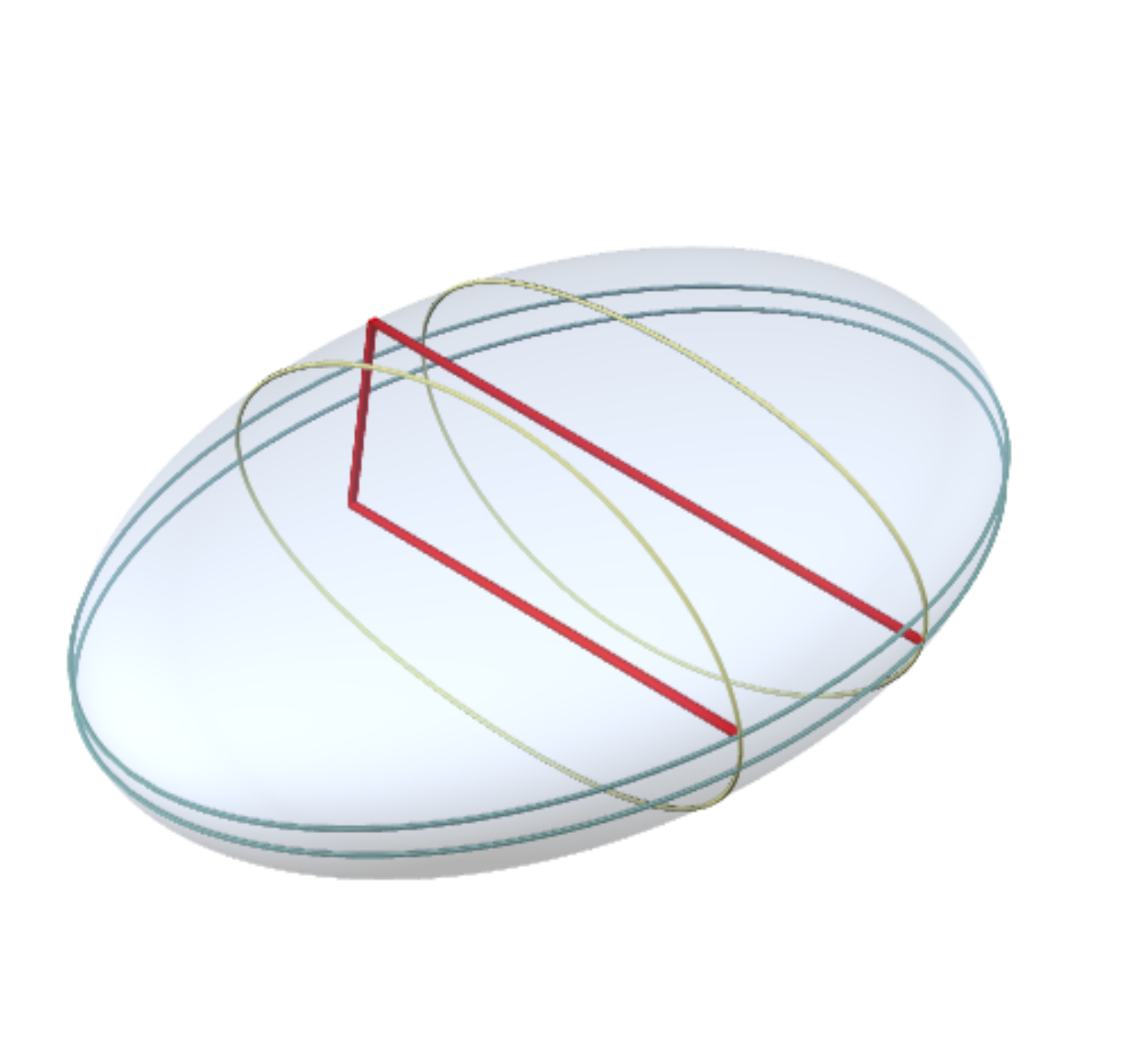}\end{tabular}} &
    \scalebox{0.2}{\begin{tabular}{c}\includegraphics{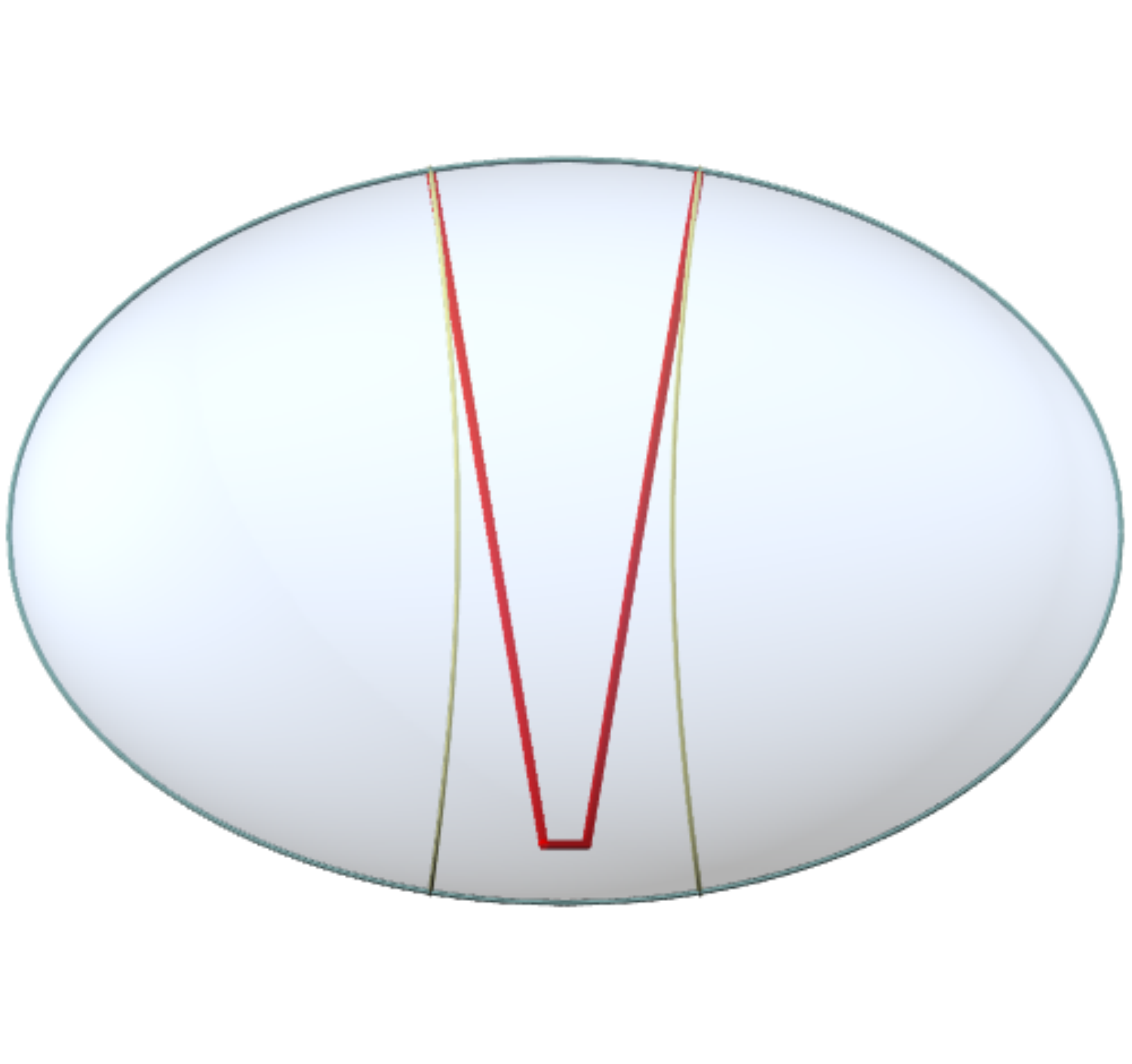}\end{tabular}} &
    \scalebox{0.2}{\begin{tabular}{c}\includegraphics{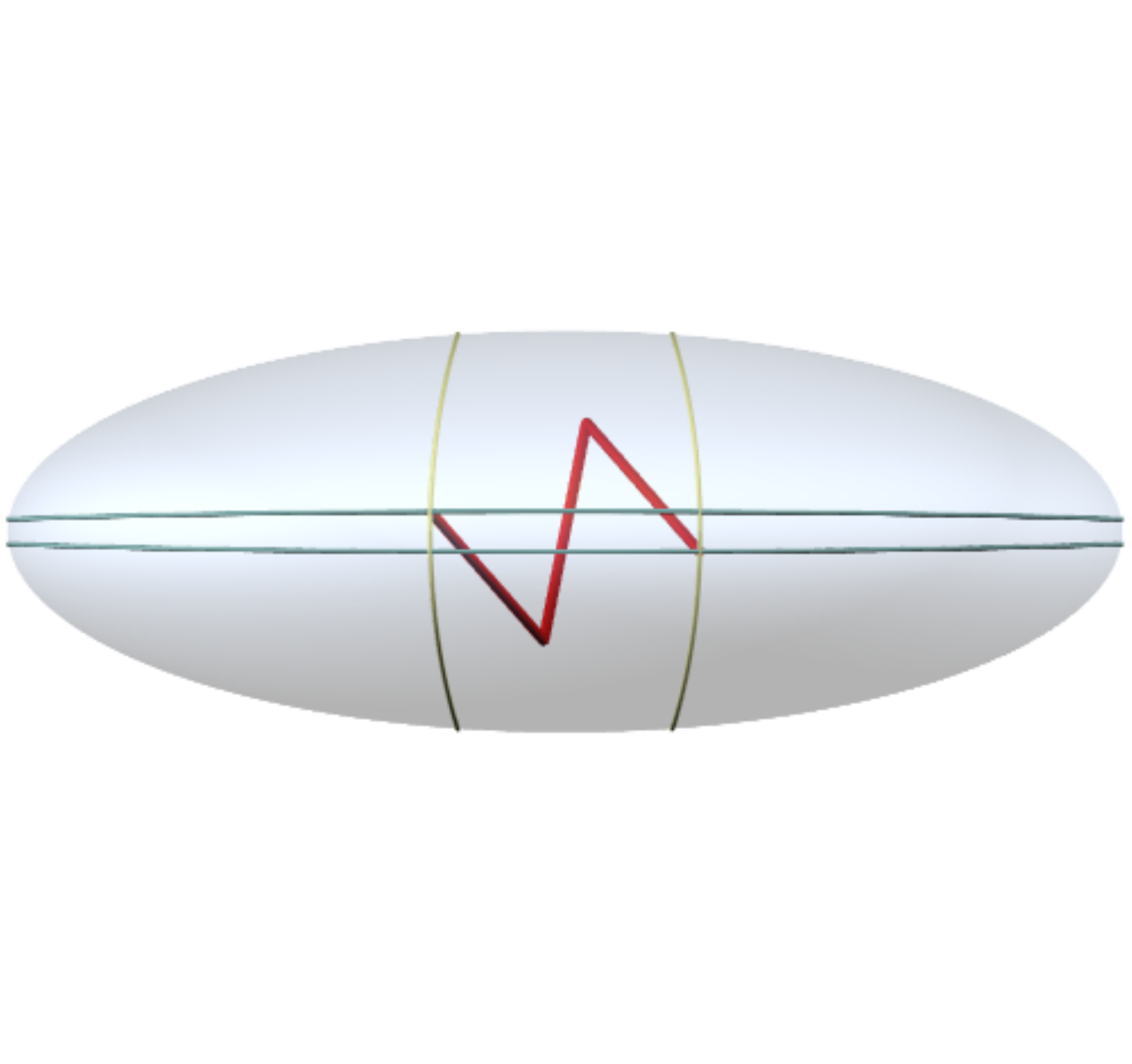}\end{tabular}} &
    \scalebox{0.2}{\begin{tabular}{c}\includegraphics{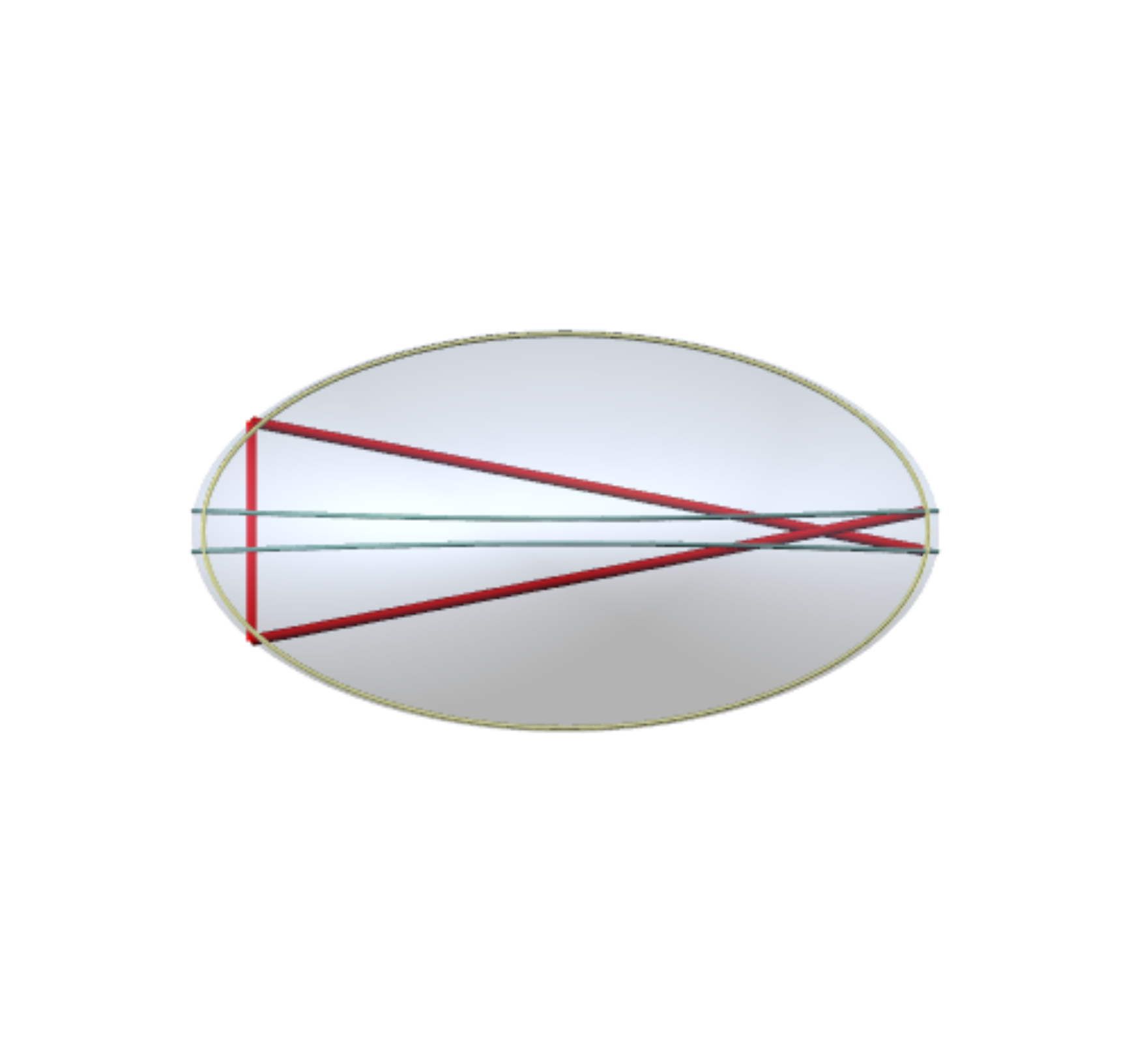}\end{tabular}} &
    \begin{tabular}{c}
      H1H2\\
      $\left( 6, 4, 2 \right)$\\
      0.45\\
      0.13\\
      0.967756\\
      0.133273\\
      $R$\\
      $f \circ R_{1 3}$
    \end{tabular}\\\hline
    \scalebox{0.2}{\begin{tabular}{c}\includegraphics{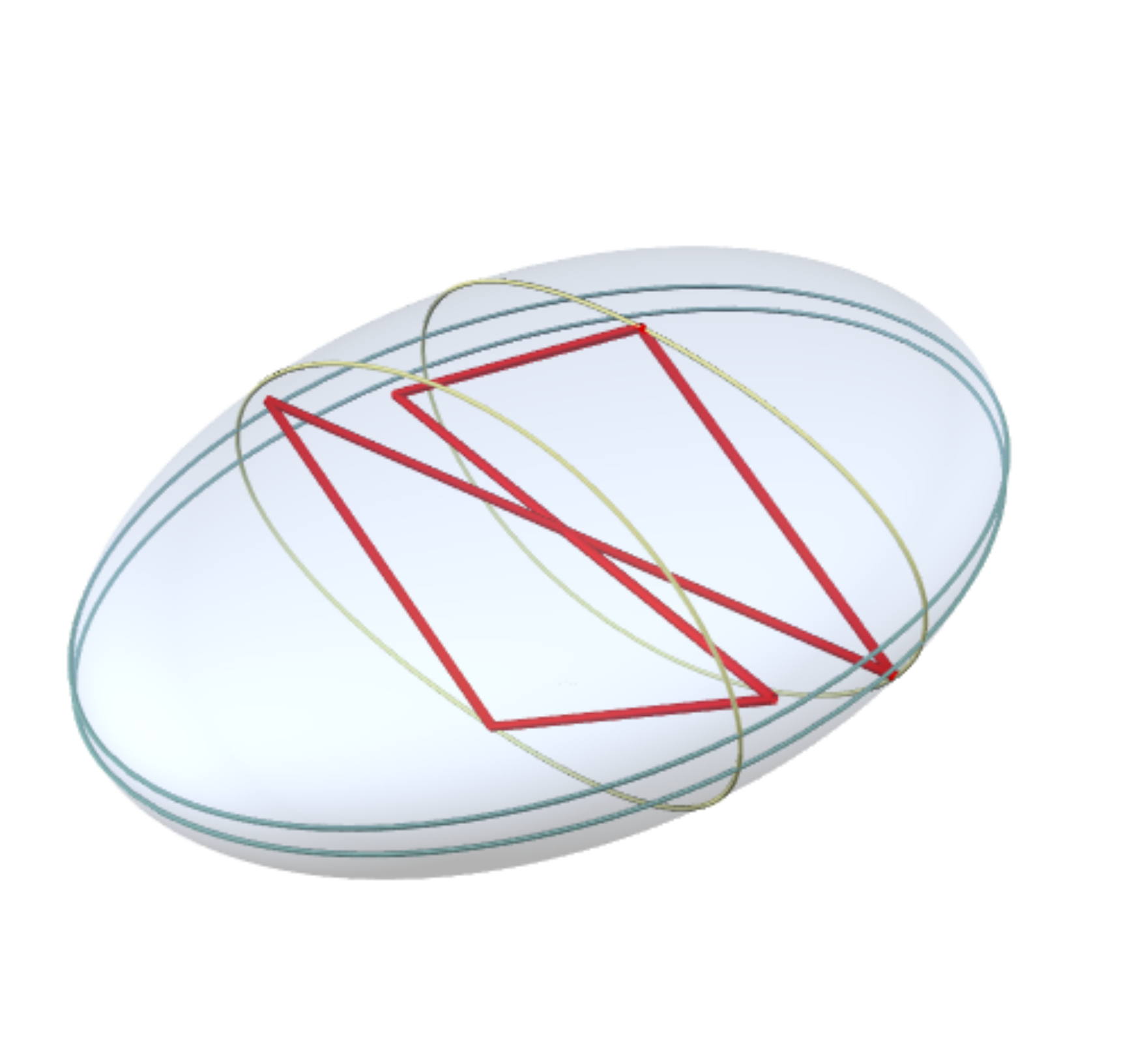}\end{tabular}} &
    \scalebox{0.2}{\begin{tabular}{c}\includegraphics{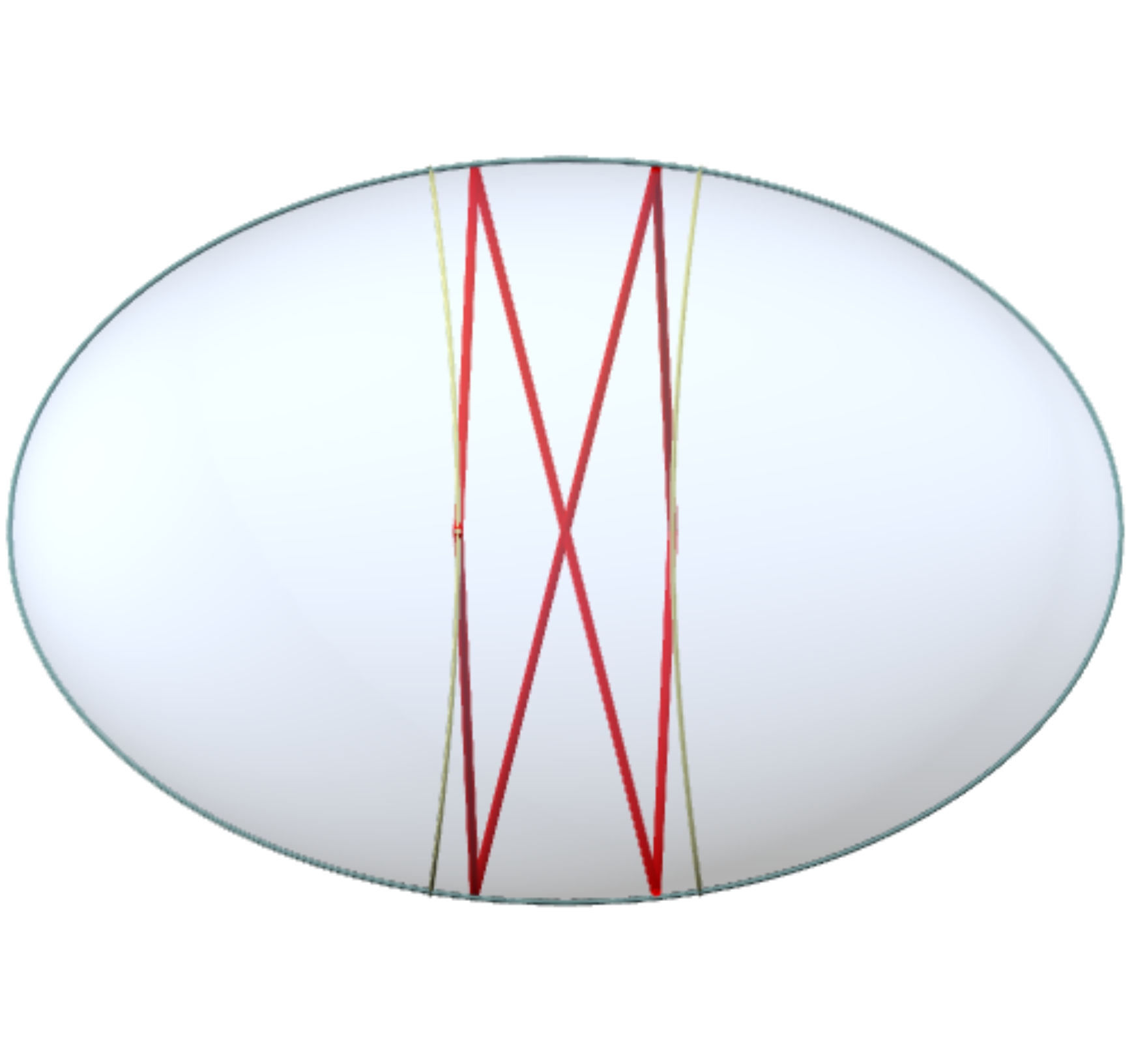}\end{tabular}} &
    \scalebox{0.2}{\begin{tabular}{c}\includegraphics{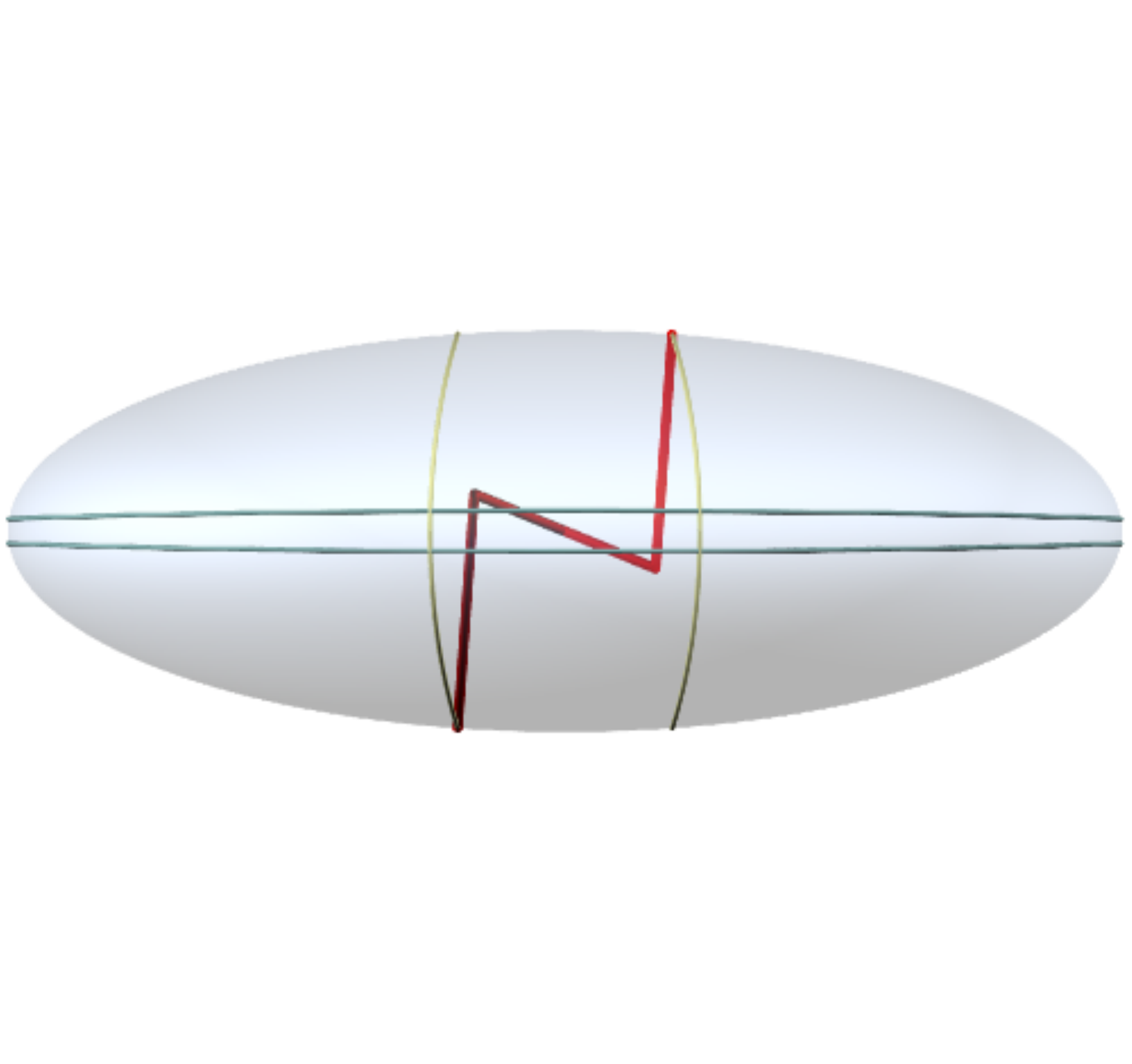}\end{tabular}} &
    \scalebox{0.2}{\begin{tabular}{c}\includegraphics{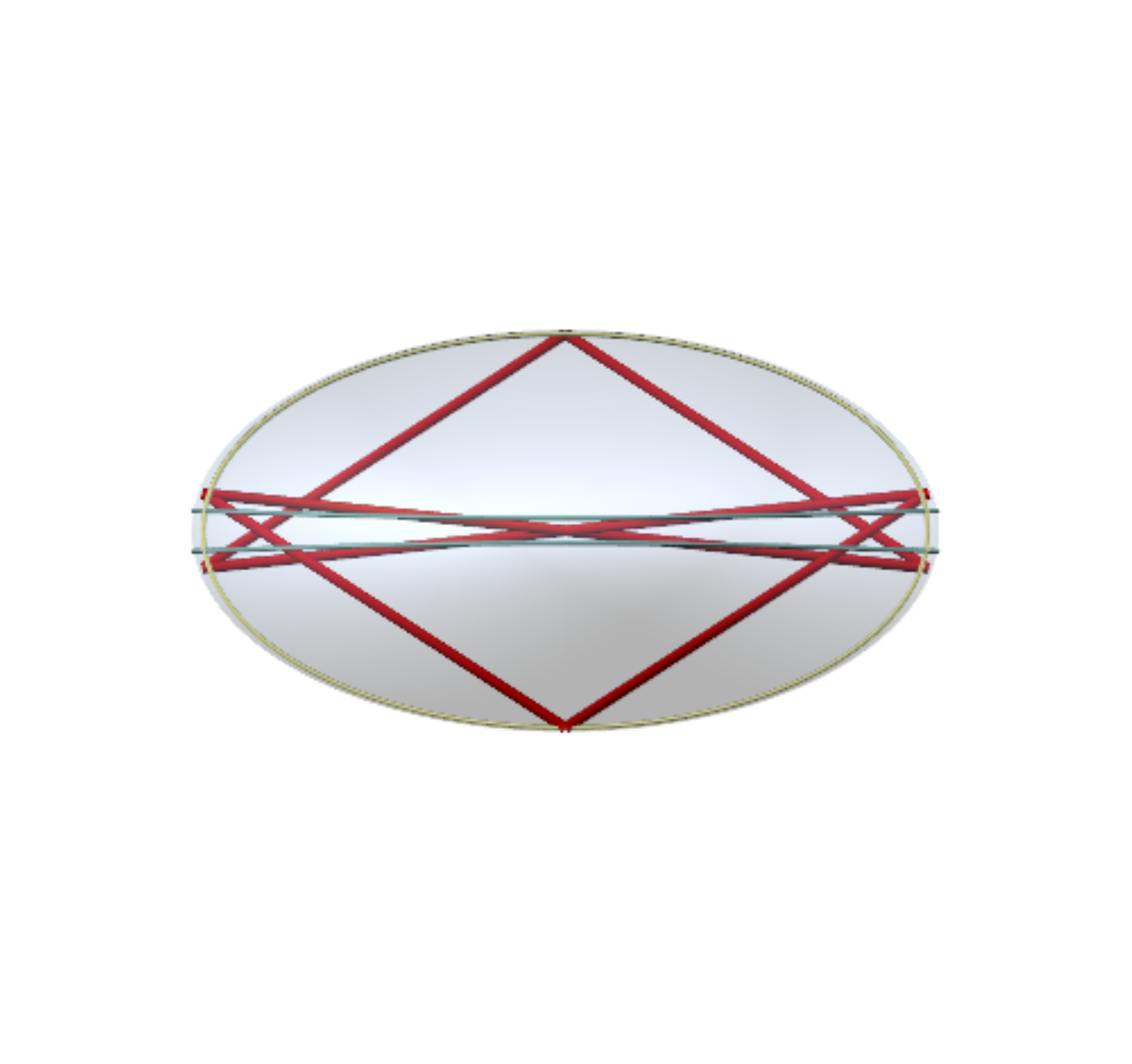}\end{tabular}} &
    \begin{tabular}{c}
      H1H2\\
      $\left( 6, 4, 2 \right)$\\
      0.45\\
      0.13\\
      0.967756\\
      0.133273\\
      $R_2$\\
      $f \circ R_{1 2 3}$
    \end{tabular}\\\hline
    \scalebox{0.2}{\begin{tabular}{c}\includegraphics{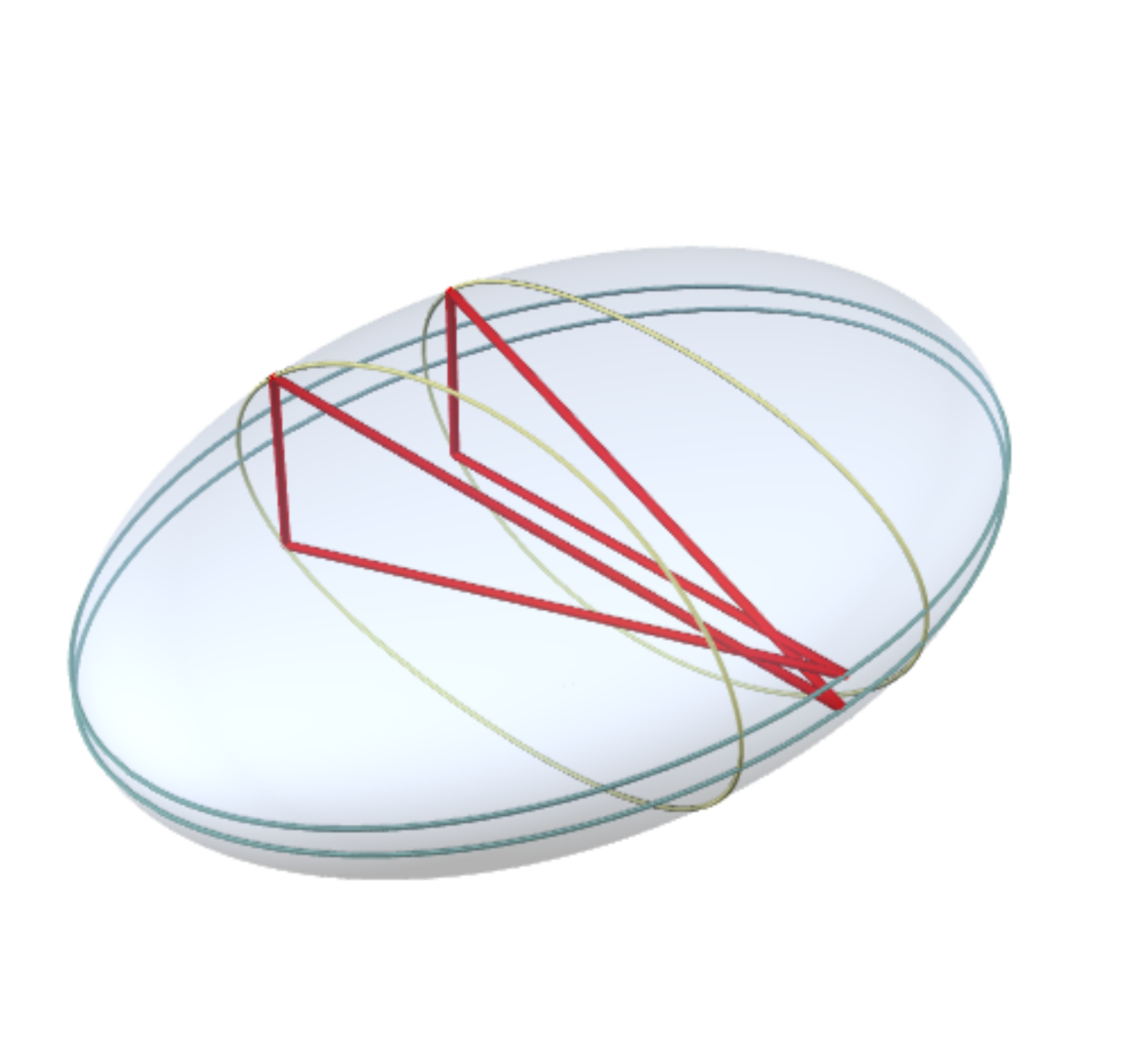}\end{tabular}} &
    \scalebox{0.2}{\begin{tabular}{c}\includegraphics{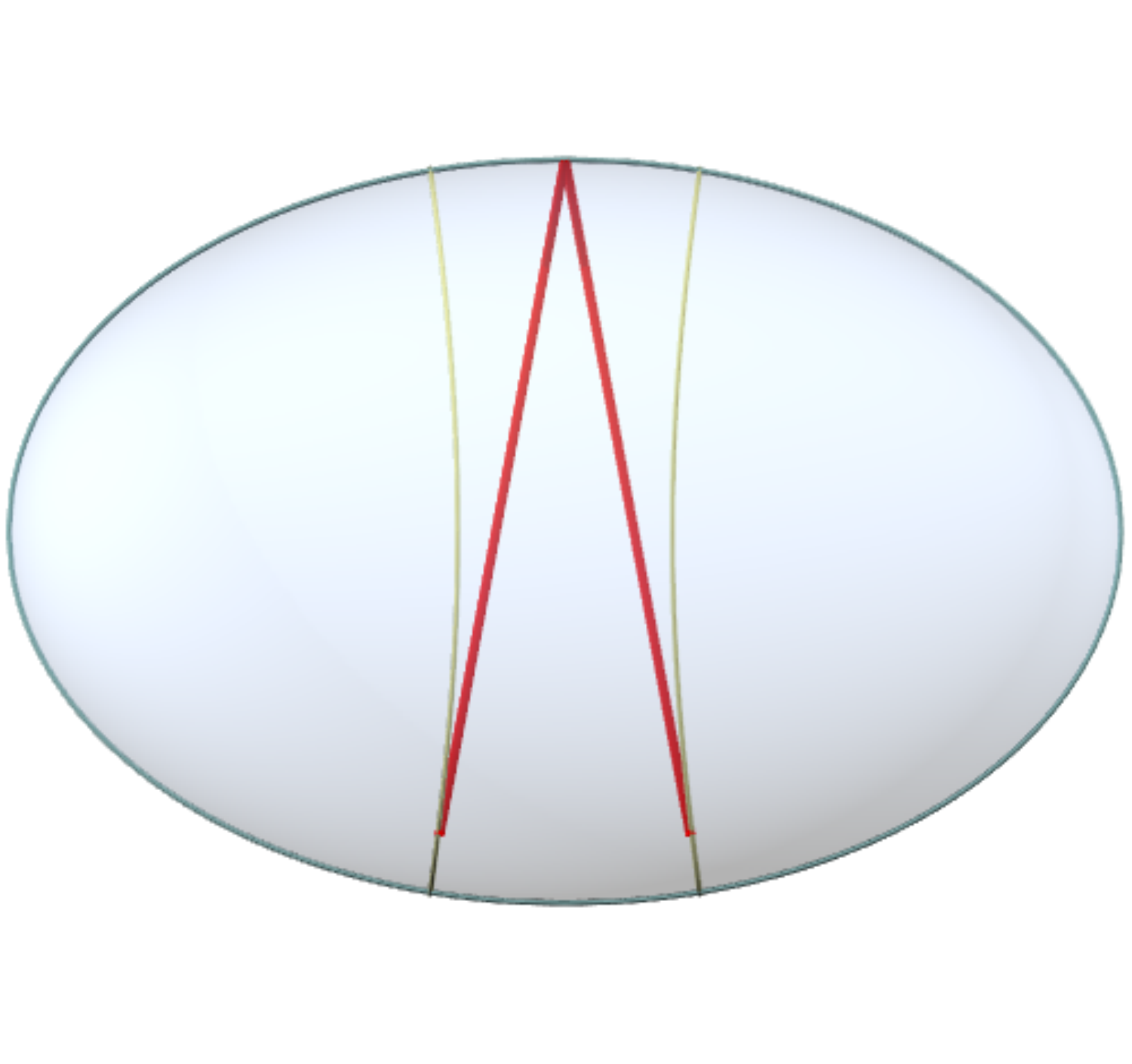}\end{tabular}} &
    \scalebox{0.2}{\begin{tabular}{c}\includegraphics{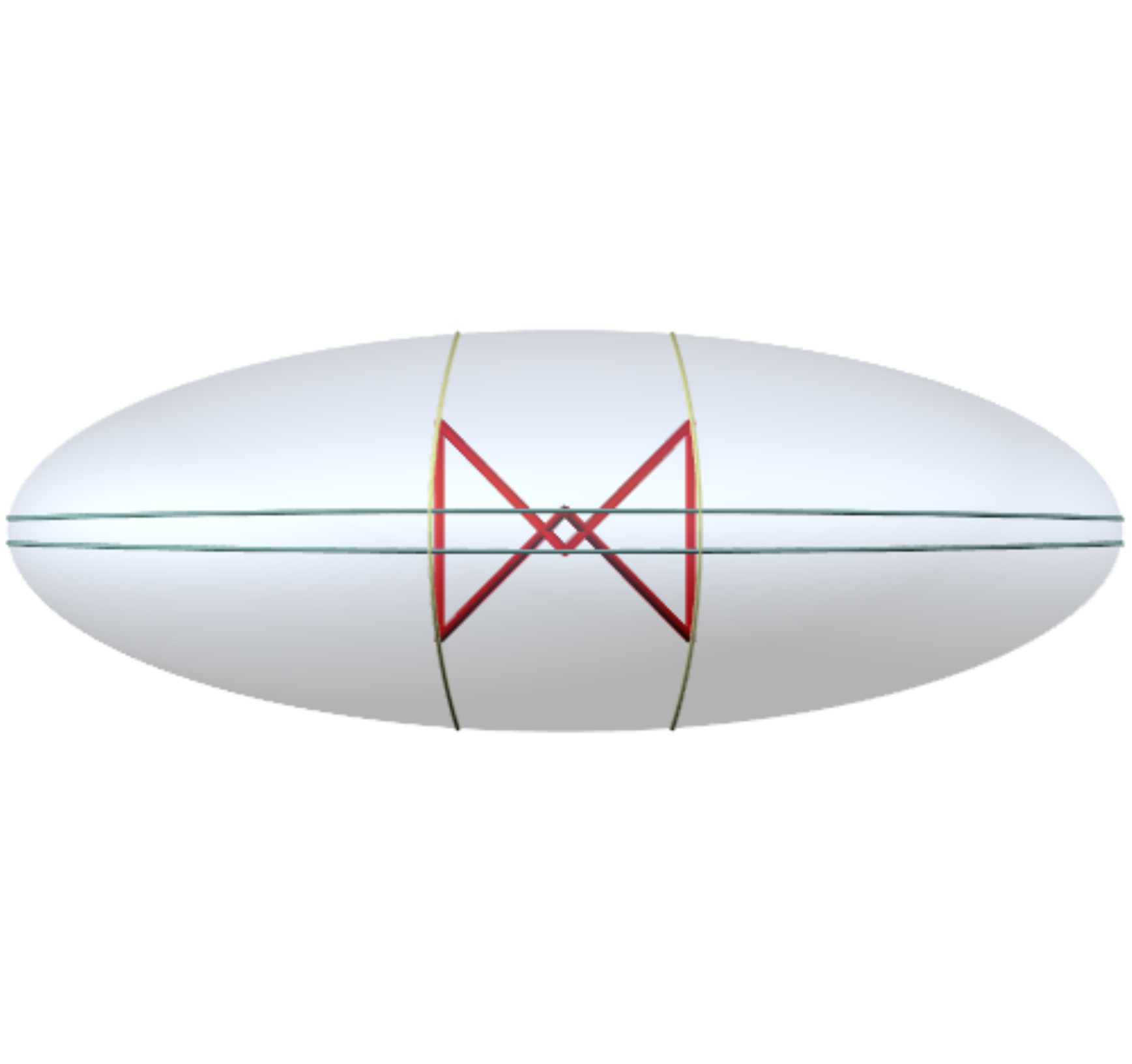}\end{tabular}} &
    \scalebox{0.2}{\begin{tabular}{c}\includegraphics{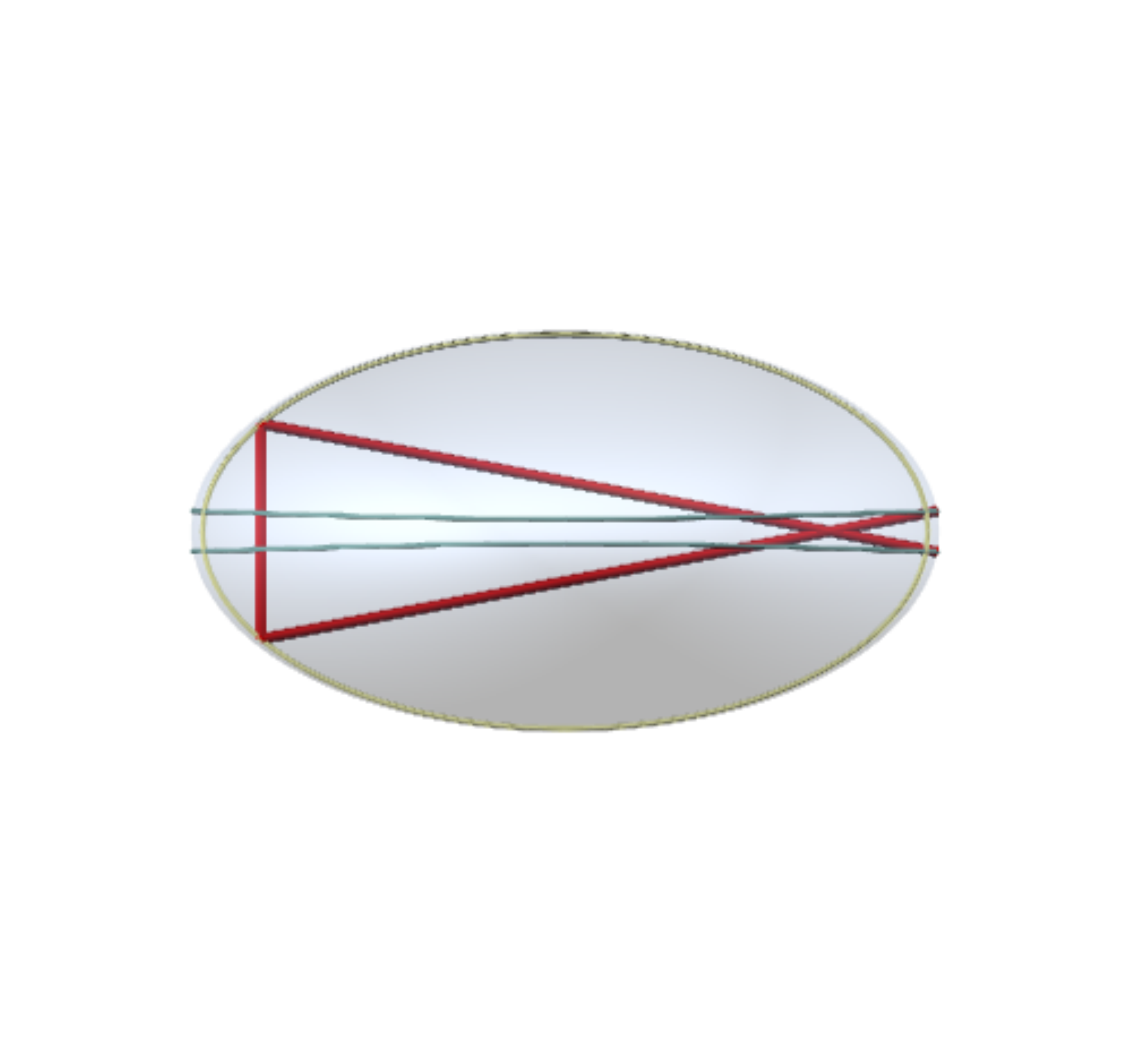}\end{tabular}} &
    \begin{tabular}{c}
      H1H2\\
      $\left( 6, 4, 2 \right)$\\
      0.45\\
      0.13\\
      0.967756\\
      0.133273\\
      $R_3$\\
      $f \circ R_1$
    \end{tabular}\\\hline
    \scalebox{0.2}{\begin{tabular}{c}\includegraphics{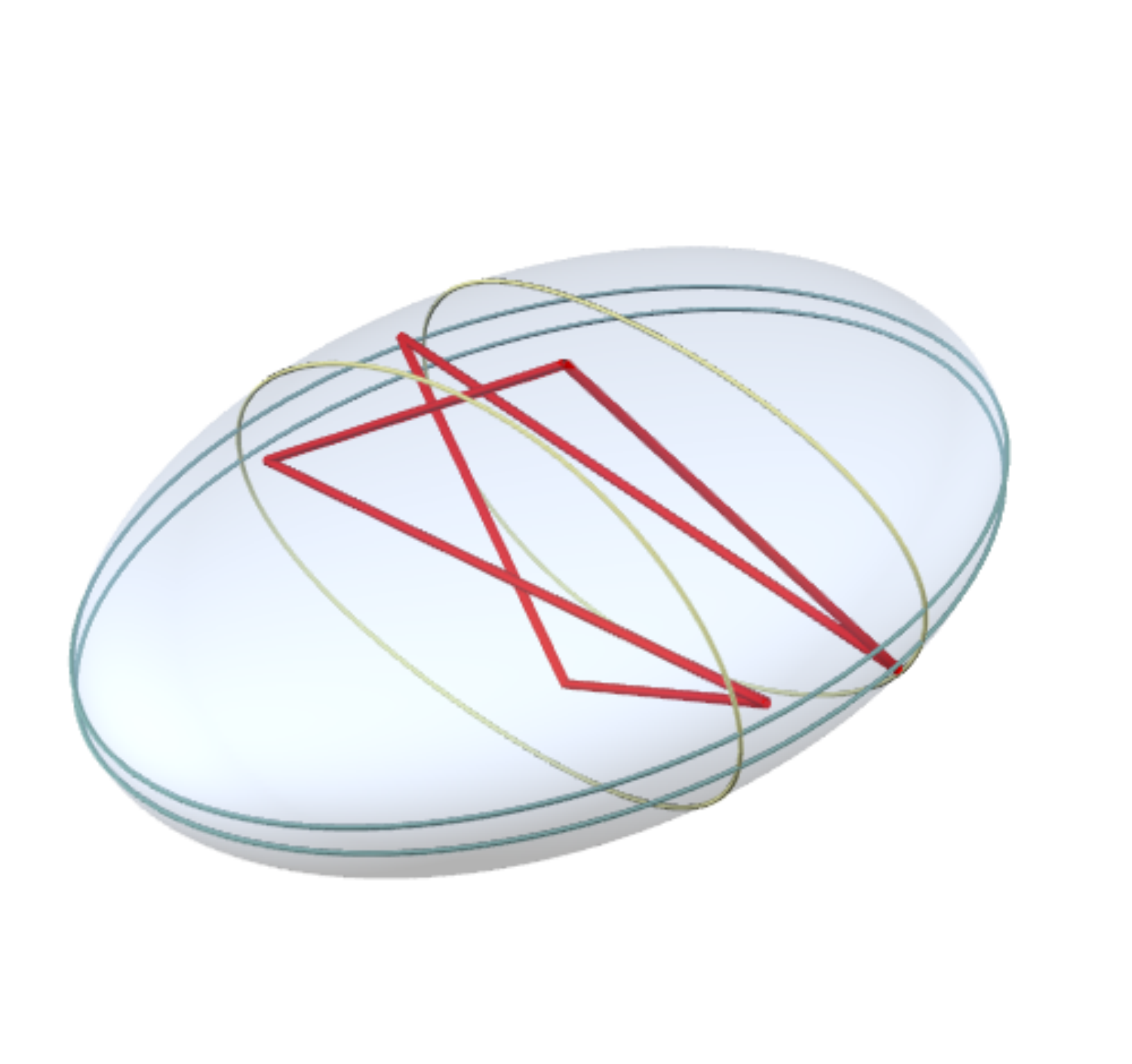}\end{tabular}} &
    \scalebox{0.2}{\begin{tabular}{c}\includegraphics{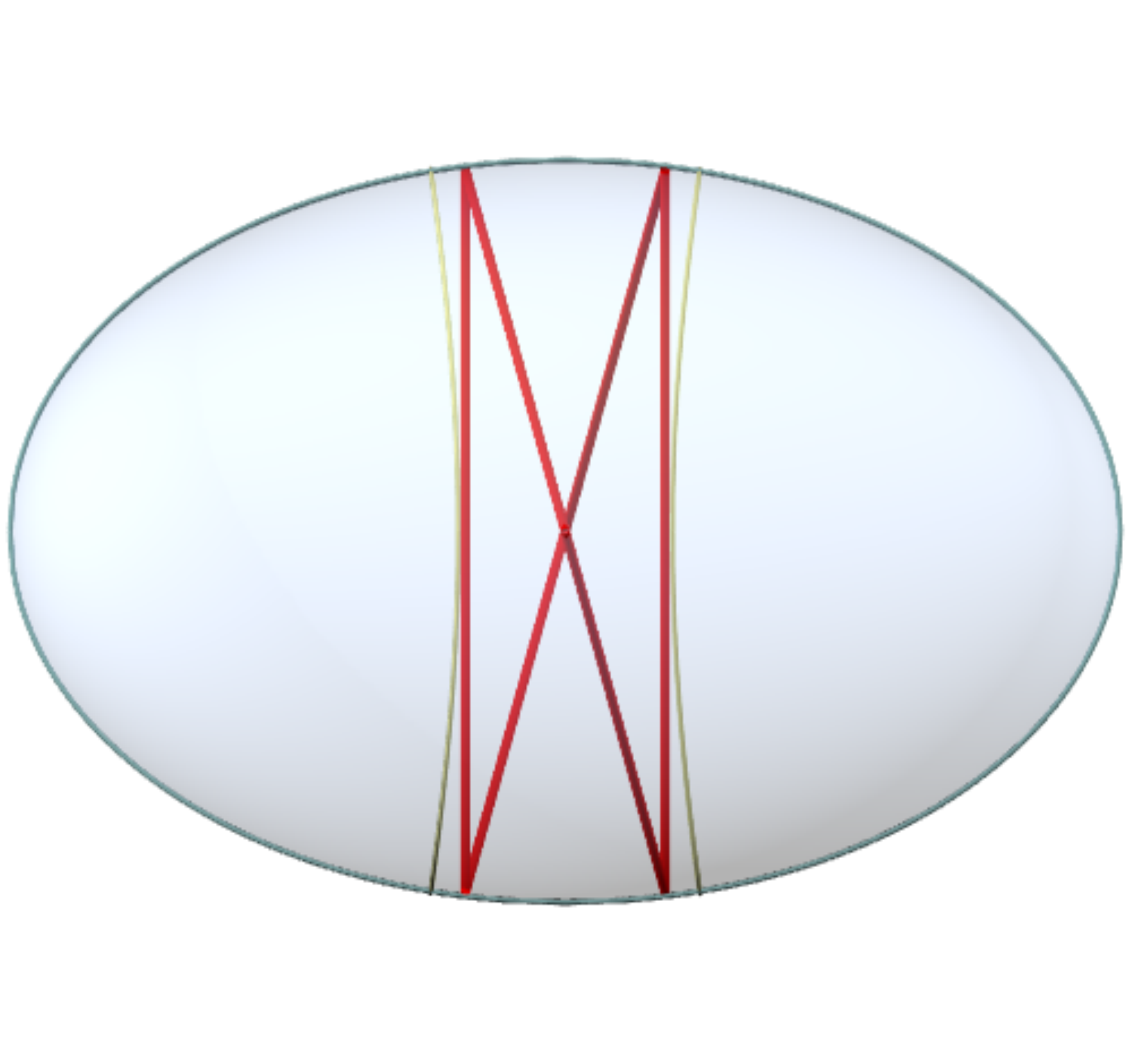}\end{tabular}} &
    \scalebox{0.2}{\begin{tabular}{c}\includegraphics{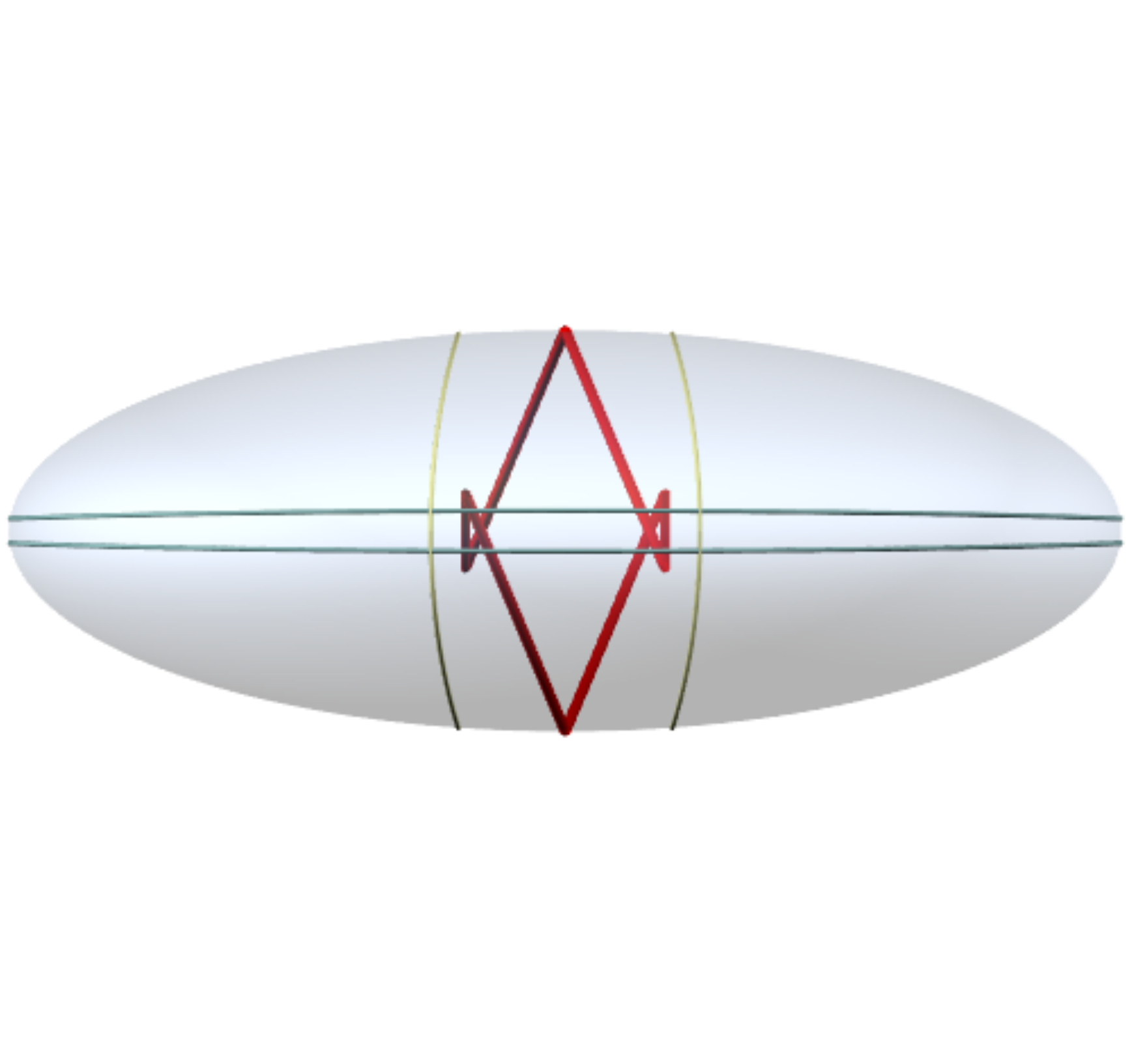}\end{tabular}} &
    \scalebox{0.2}{\begin{tabular}{c}\includegraphics{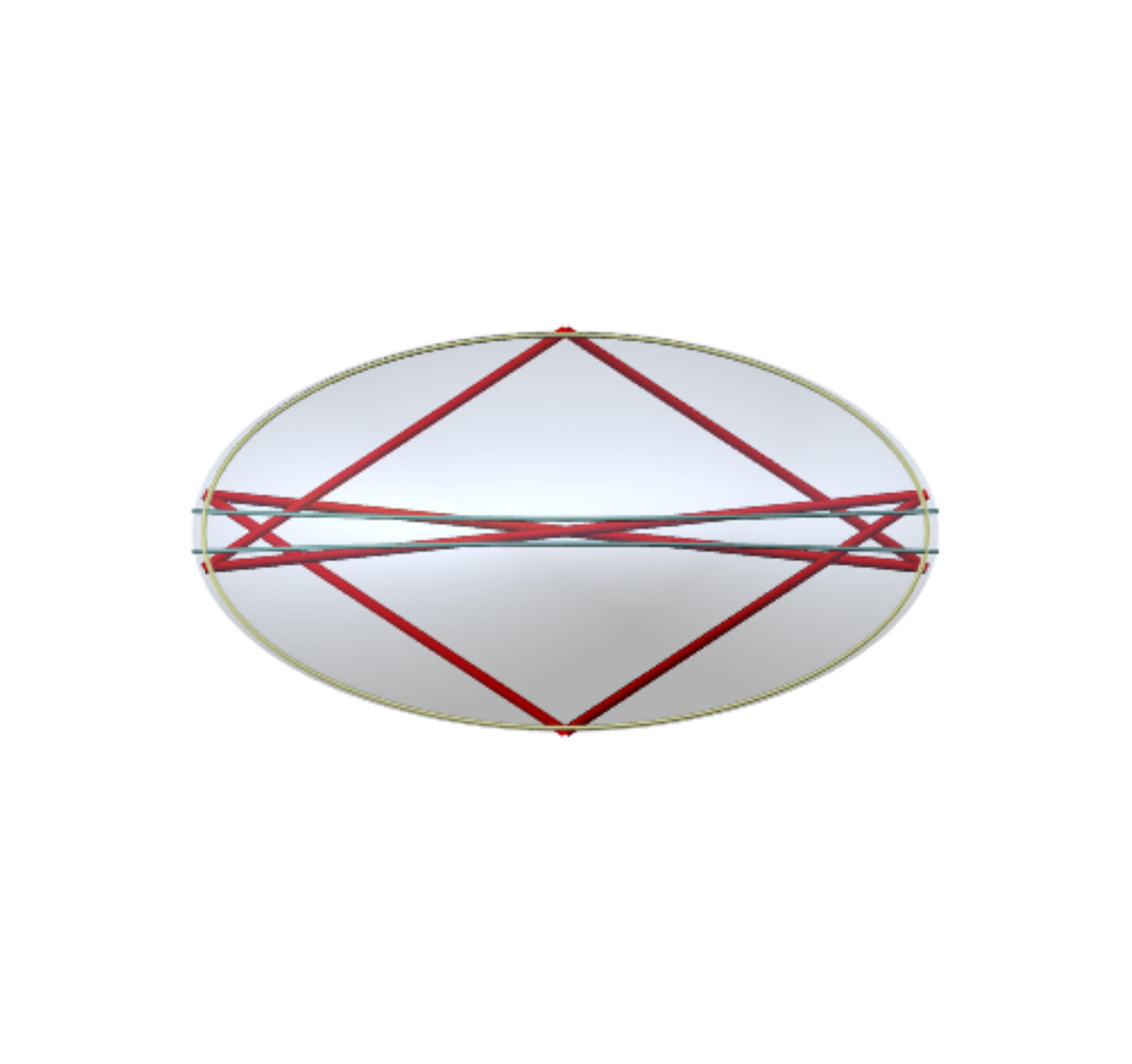}\end{tabular}} &
    \begin{tabular}{c}
      H1H2\\
      $\left( 6, 4, 2 \right)$\\
      0.45\\
      0.13\\
      0.967756\\
      0.133273\\
      $R_{2 3}$\\
      $f \circ R_{1 2}$
    \end{tabular}\\
    \hline
  \end{tabular}
\end{table*}

The other three caustic types can be dealt with the same algorithm. There is just
one difference among them. Namely, we must consider different shapes of ellipsoids
to ensure that Eq.~(\ref{eq:FrequencyWindingNumbers}) has a solution. More
precisely, we must consider almost flat ellipsoids ($a_1$ small) for
caustic types EH1 and H1H1, and almost ``segments'' (both $a_1$ and $a_2$ small)
for caustic types EH2 and H1H2. These results can be found in
Ref.~\onlinecite{CasasRamirez2011}.
This explains why some ellipsoids in
Tables~\ref{tab:SPTs_3D_periods45}--\ref{tab:SPTs_3D_EH?_period6} are so flat.
We have chosen not too extreme ellipsoids whenever it has been possible.

We have found a minimal SPT of each class. Some of them are displayed in
Tables~\ref{tab:SPTs_3D_periods45}--\ref{tab:SPTs_3D_EH?_period6}. Arrows in
their column headings point in the direction of increasing $x_1$, $x_2$ or
$x_3$. Each line of Tables~\ref{tab:SPTs_3D_periods45},
\ref{tab:SPTs_3D_H1H2_period6} and~\ref{tab:SPTs_3D_EH?_period6} shows
four different perspectives of an SPT.
The 3D image corresponds to an isometric view of the first octant.
The three projected planes are viewed from the positive missing axis. We only
display 3D views in Table~\ref{tab:SPTs_3D_H1H1_period6}. Billiard
trajectories are depicted in red. Green and yellow lines represent
intersections of $Q$ with H1-caustics and H2-caustics, respectively. The
E-caustic is also shown in cases EH1 and EH2.

Period four takes only place for four classes of SPTs of type H1H1. Period
five takes only place for eight classes of SPTs, half of type EH1 and half of
type EH2. We display half of these minimal SPTs in
Table~\ref{tab:SPTs_3D_periods45}. Their announced symmetries can be verified by
observing their projections. We can check that trajectories of caustic type
H1H1 give one turn around the $x_1$-axis and have three tangential touches
with the outer (respectively, inner) 1-sheet hyperboloid, trajectories of type
EH1 give one turn around the $x_1$-axis and cross four times the plane
$\Pi_1$, whereas trajectories of type EH2 give two turns around the $x_3$-axis
and cross twice the plane $\Pi_2$. This is consistent with the geometric
interpretation given in Subsec.~\ref{ssec:FrequencyMap} of winding numbers.

We show in Table~\ref{tab:SPTs_3D_H1H2_period6} the minimal representatives of
the classes listed in the third row of Table~\ref{tab:ClassificationSPTsH1H2}.
Since all of them have $\left( 6, 4, 2 \right)$ as winding numbers, they are
$6$-periodic, cross four times the plane $\Pi_2$, and cross twice the plane
$\Pi_3$.

We draw several SPTs of type H1H1 in
Table~\ref{tab:SPTs_3D_H1H1_period6}. They show the difference among
the six classes
$\left( R_2 | R_2 \right)$, $\left( R_3 | R_3
\right)$, $\left( R_2, R_3 | \right)$, $\left( R_2 | R_3 \right)$,
$\left( R_3 | R_2 \right)$, and $\left( | R_2, R_3 \right)$.
We can locate those classes in the last rows
of Table~\ref{tab:ClassificationSPTsH1H1}. One can observe, for instance, that
the SPT of class $\left( R_2, R_3 | \right)$ has two impacts on
the intersection $Q \cap \Pi_2 \cap Q_{\lambda_1}$ and two more on $Q \cap
\Pi_3 \cap Q_{\lambda_1}$. On the contrary, the SPT of class $( | R_2, R_3)$
has two impacts on the intersection $Q \cap \Pi_2
\cap Q_{\lambda_2}$ and two more on $Q \cap \Pi_3 \cap Q_{\lambda_2}$. Since
$\lambda_1 < \lambda_2$, $Q_{\lambda_1}$ and $Q_{\lambda_2}$ are the outer and
inner one-sheet hyperboloids, respectively.

Finally, in Table~\ref{tab:SPTs_3D_EH?_period6} we present half of the 6-periodic
minimal SPTs with an ellipsoidal caustic. They correspond to the third row of
Tables~\ref{tab:ClassificationSPTsEH1}--\ref{tab:ClassificationSPTsEH2}.

\begin{table*}
\caption{\label{tab:SPTs_3D_H1H1_period6}Six SPTs (3D view) for H1H1-caustics.
They show different classes of SPTs corresponding to reversors $R_2$
and/or $R_3$. ``Data'' as in Table~\ref{tab:SPTs_3D_periods45}.}
  \begin{tabular}{c|c|c|c|c|c}
    3D $(x_1 : \uparrow, x_2 : \searrow, x_3 : \swarrow)$ & Data &
    3D $(x_1 : \uparrow, x_2 : \searrow, x_3 : \swarrow)$ & Data &
    3D $(x_1 : \uparrow, x_2 : \searrow, x_3 : \swarrow)$ & Data \\
    \hline
    \scalebox{0.2}{\begin{tabular}{c}\includegraphics{H1H1_3_8_1_4_R2_xyz}\end{tabular}} &
    \begin{tabular}{c}
      H1H1\\
      $(4, 3, 2)$\\
      0.8\\
      0.13\\
      0.648376\\
      0.130077\\
      $(R_2 | R_2)$
    \end{tabular} &
    \scalebox{0.2}{\begin{tabular}{c}\includegraphics{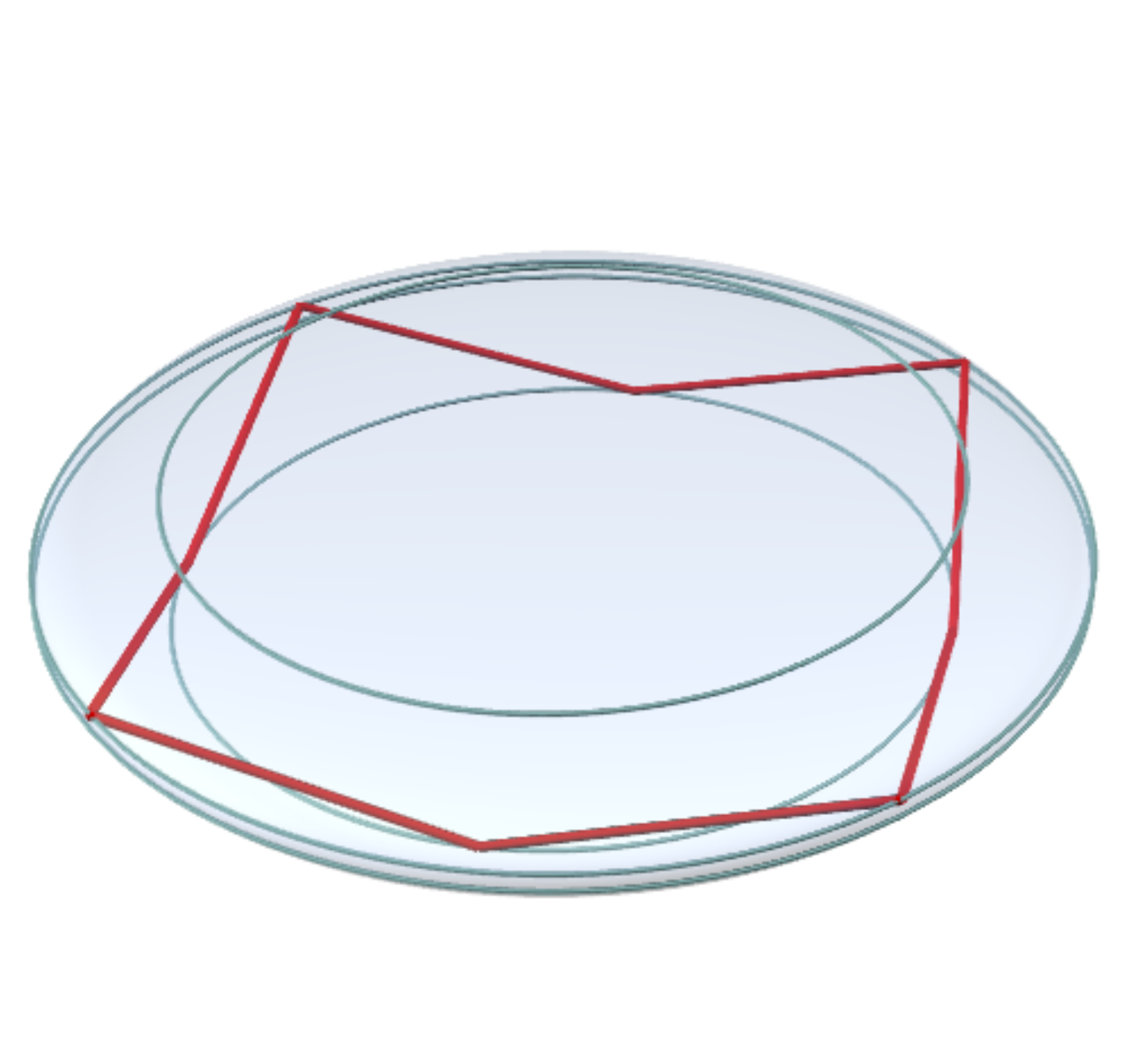}\end{tabular}} &
    \begin{tabular}{c}
      H1H1\\
      $(8, 4, 2)$\\
      0.95\\
      0.05\\
      0.056134\\
      0.457414\\
      $(R_2, R_3 |)$
    \end{tabular} &
    \scalebox{0.2}{\begin{tabular}{c}\includegraphics{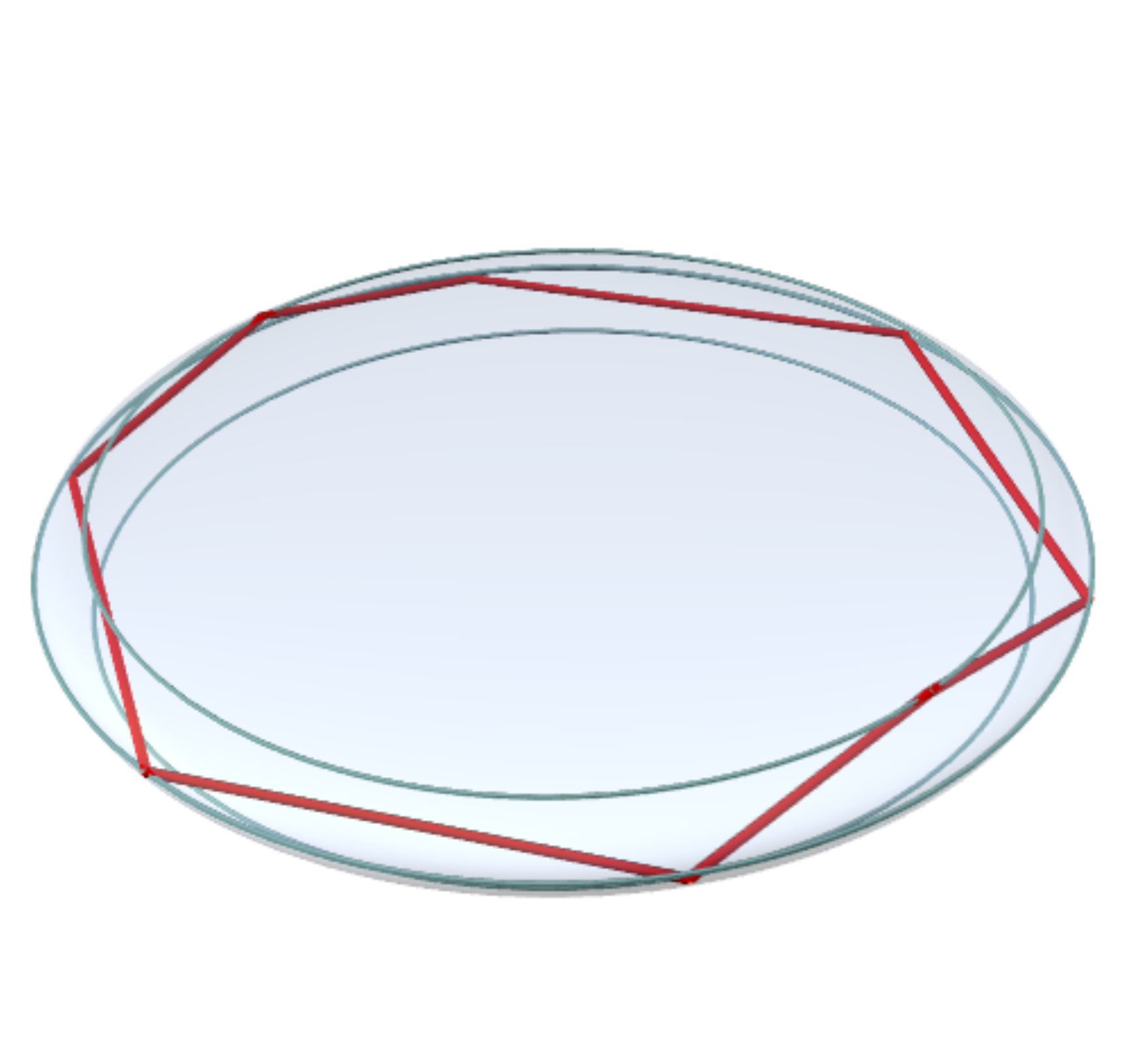}\end{tabular}} &
    \begin{tabular}{c}
      H1H1\\
      $(8, 6, 2)$\\
      0.95\\
      0.05\\
      0.050041\\
      0.229595\\
      $(R_2 | R_3)$
    \end{tabular}\\\hline
    \scalebox{0.2}{\begin{tabular}{c}\includegraphics{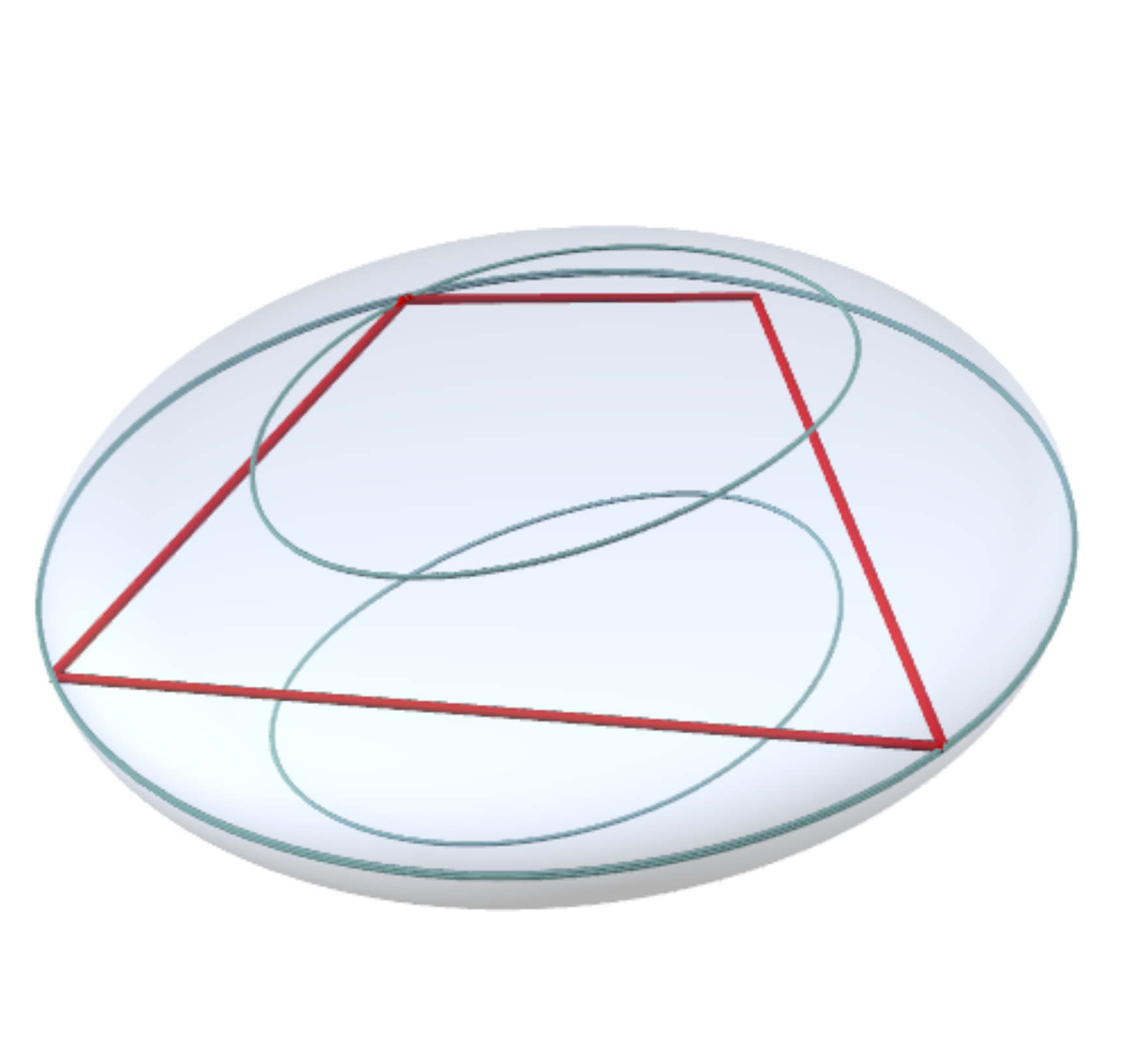}\end{tabular}} &
    \begin{tabular}{c}
      H1H1\\
      $(4, 3, 2)$\\
      0.8\\
      0.13\\
      0.648376\\
      0.130077\\
      $(R_3 | R_3)$
    \end{tabular} &
    \scalebox{0.2}{\begin{tabular}{c}\includegraphics{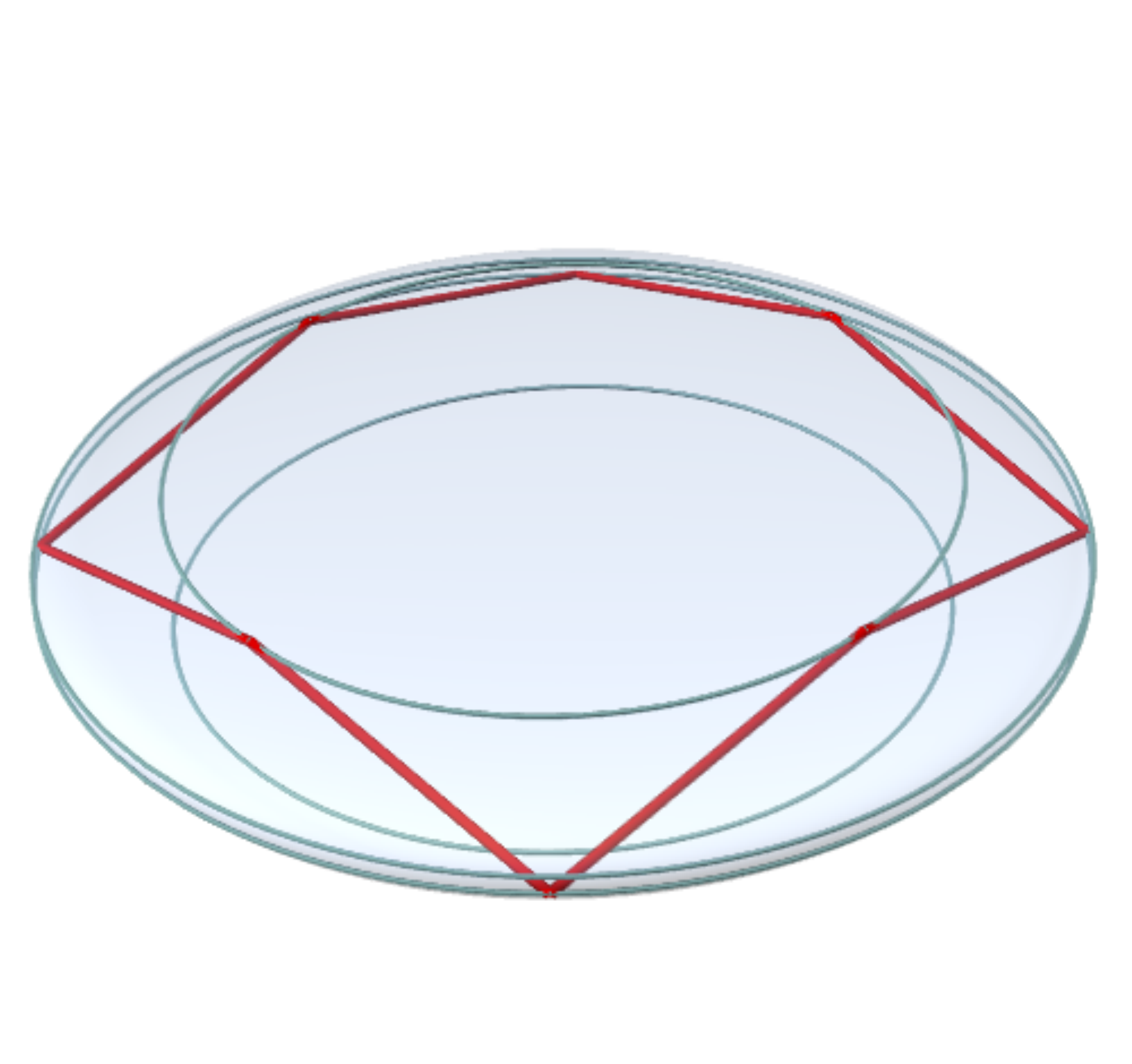}\end{tabular}} &
    \begin{tabular}{c}
      H1H1\\
      $(8, 4, 2)$\\
      0.95\\
      0.05\\
      0.056134\\
      0.457414\\
      $(| R_2, R_3)$
    \end{tabular} &
    \scalebox{0.2}{\begin{tabular}{c}\includegraphics{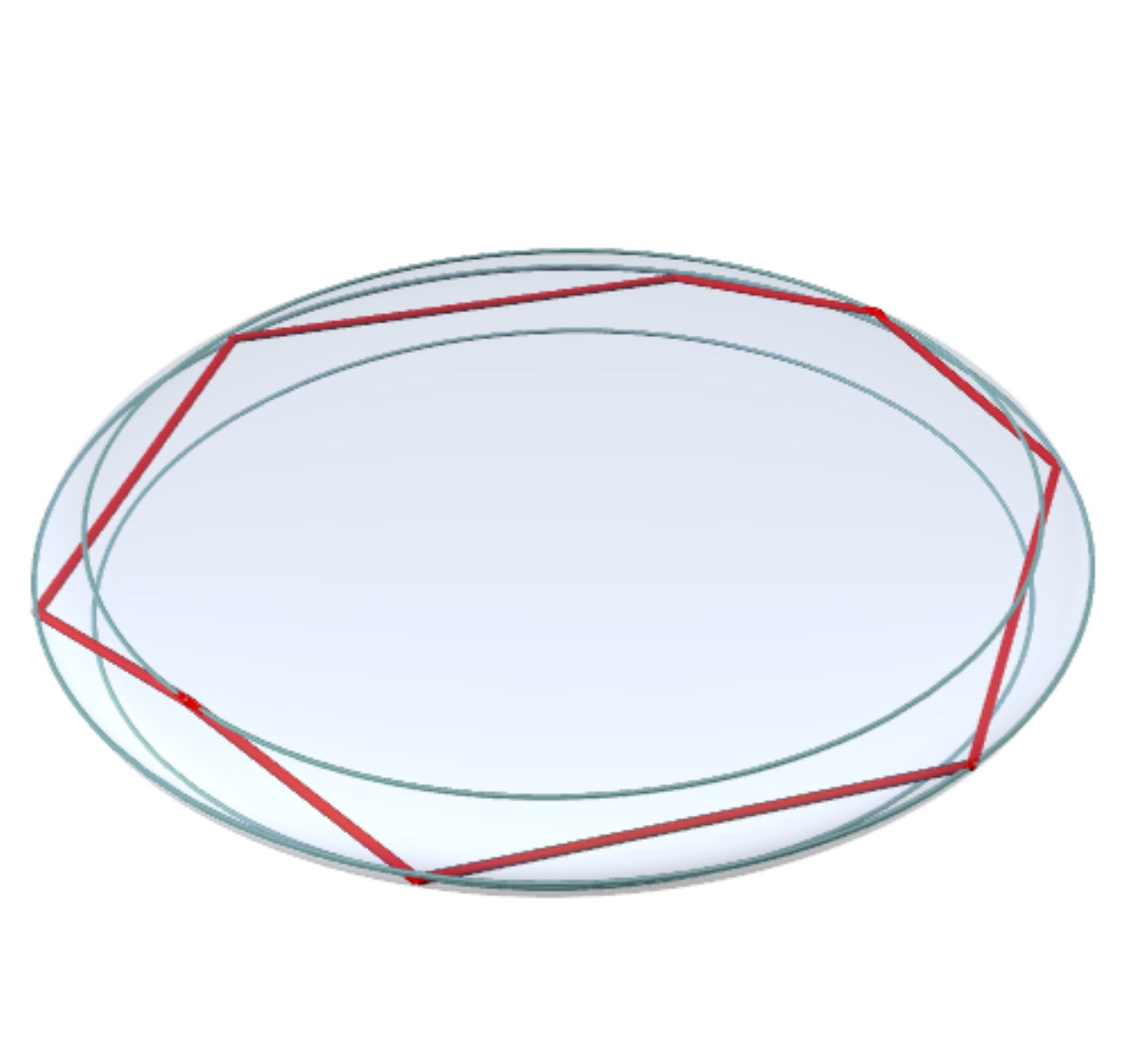}\end{tabular}} &
    \begin{tabular}{c}
      H1H1\\
      $(8, 6, 2)$\\
      0.95\\
      0.05\\
      0.050041\\
      0.229595\\
      $(R_3 | R_2)$
    \end{tabular}\\
    \hline
  \end{tabular}
\end{table*}

Some comments about these trajectories are in order:
\begin{itemize}
  \item The 4-periodic trajectories in Table~\ref{tab:SPTs_3D_periods45} are
  the simplest examples of nonplanar periodic trajectories. Besides, they are
  one of the few simply SPTs. Both trajectories have two points on the same
  symmetry set. However, each point is associated to a different caustic:
  inner or outer. That property is easier to see in the first one. Hence, they
  are of class $(R_2 | R_2)$ and $(f\circ R_{13} | f \circ R_{13})$
  in the notation used in Table~\ref{tab:ClassificationSPTsH1H1}.
  
  \item The 5-periodic trajectories in Table~\ref{tab:SPTs_3D_periods45} are
  the simplest examples of nonplanar periodic trajectories with odd period.
  They have one point on a symmetry set and another on the associated symmetry
  set. We recall that only SPTs with odd period have this property; see
  item~\ref{item:SPOs}) of Theorem~\ref{thm:SymmetricOrbits}.
  
  \item Many projections onto the horizontal plane $\Pi_1$ look like 2D SPTs
  (cf. Table~\ref{tab:SPTs_2D}).
  
  \item Any $R$-SPT is travelled twice in opposite directions, since it hits
  orthogonally the ellipsoid at some point $q \in Q \cap Q_{\lambda_1} \cap
  Q_{\lambda_2}$. Therefore, there exist SPTs of even period $m_0 = 2 l \geq
  6$ with only $l + 1$ distinct impact points on the ellipsoid, although all
  of them are of caustic type H1H2. We show a 6-periodic sample (the simplest
  one) in the first row of Table~\ref{tab:SPTs_3D_H1H2_period6}.
\end{itemize}

\begin{table*}
  \caption{\label{tab:SPTs_3D_EH?_period6}Four minimal SPTs with winding
  numbers $\left( m_0, m_1, m_2 \right) = \left( 6, 4, 2 \right)$ and an
  ellipsoidal caustic. ``Data'' as in Table~\ref{tab:SPTs_3D_periods45}.}
  \begin{tabular}{c|c|c|c|c}
    3D $\left( x_1 : \uparrow, x_2 : \searrow, x_3 : \swarrow \right)$ & Plane
    $\Pi_1$ $\left( x_2 : \uparrow, x_3 : \rightarrow \right)$ & Plane $\Pi_2$
    $\left( x_1 : \uparrow  , x_3 : \leftarrow \right)$ &
    Plane $\Pi_3$ $\left( x_1 : \uparrow, x_2 : \rightarrow \right)$ & Orbit
    data\\
    \hline
    \scalebox{0.2}{\begin{tabular}{c}\includegraphics{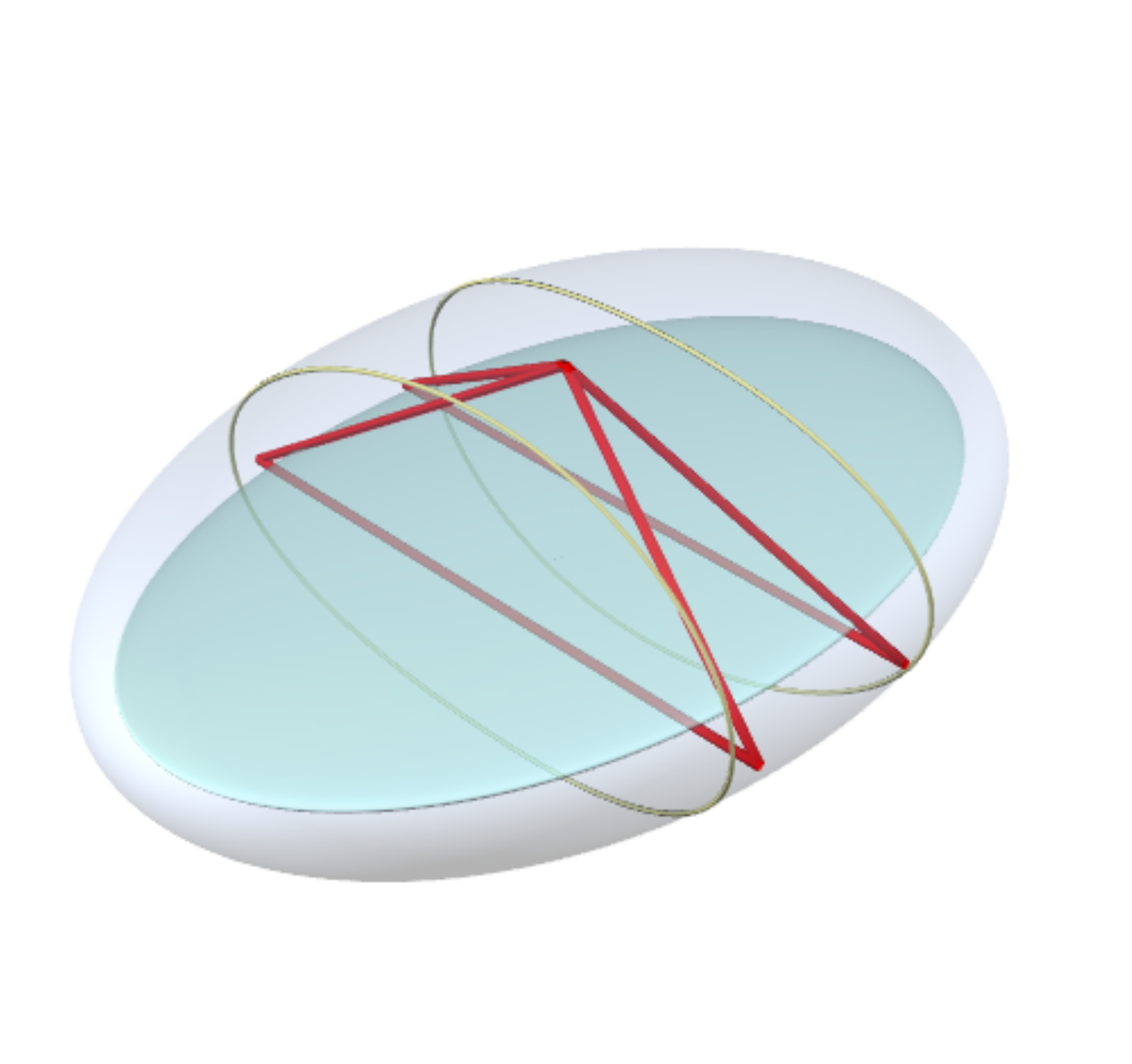}\end{tabular}} &
    \scalebox{0.2}{\begin{tabular}{c}\includegraphics{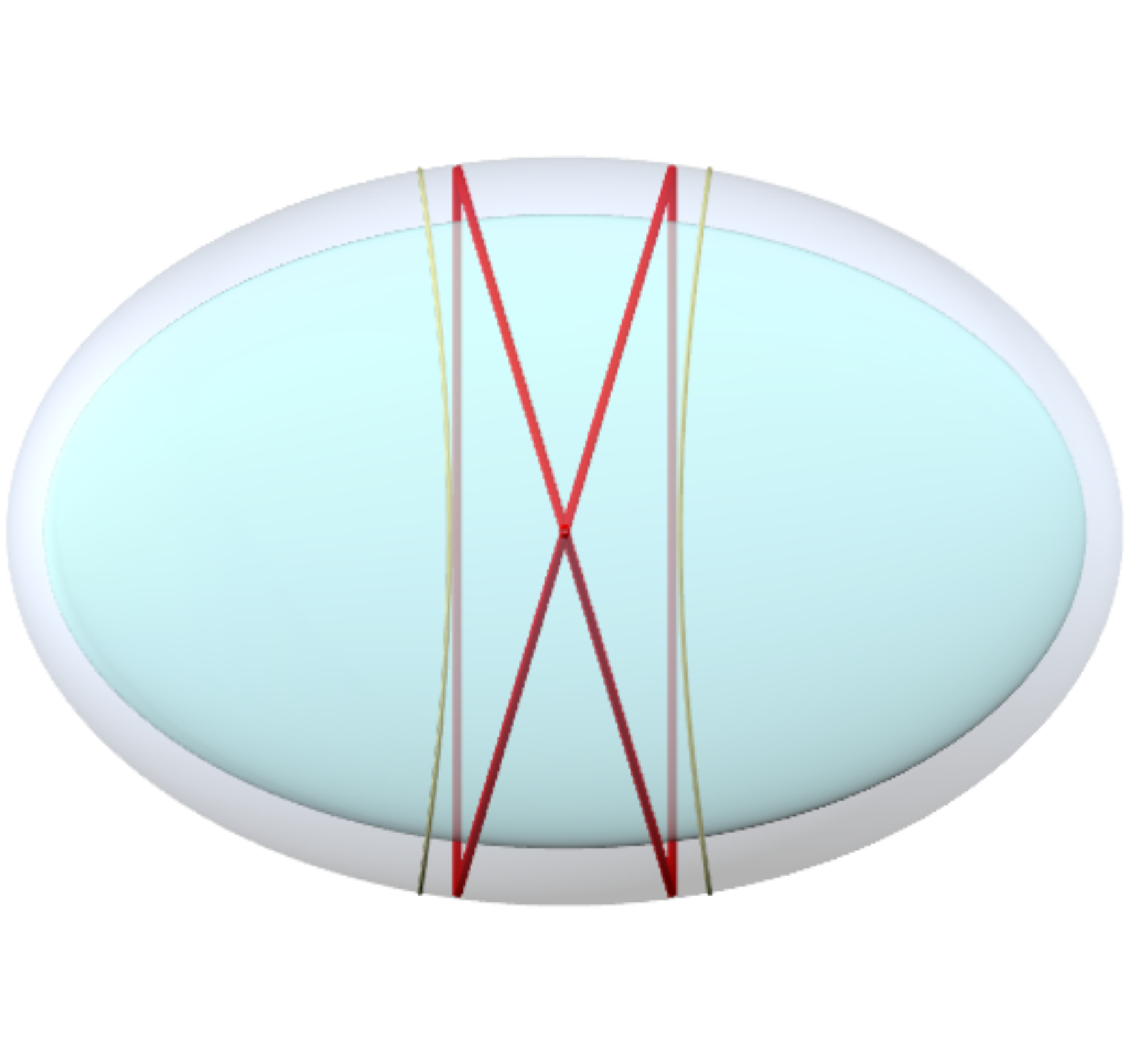}\end{tabular}} &
    \scalebox{0.2}{\begin{tabular}{c}\includegraphics{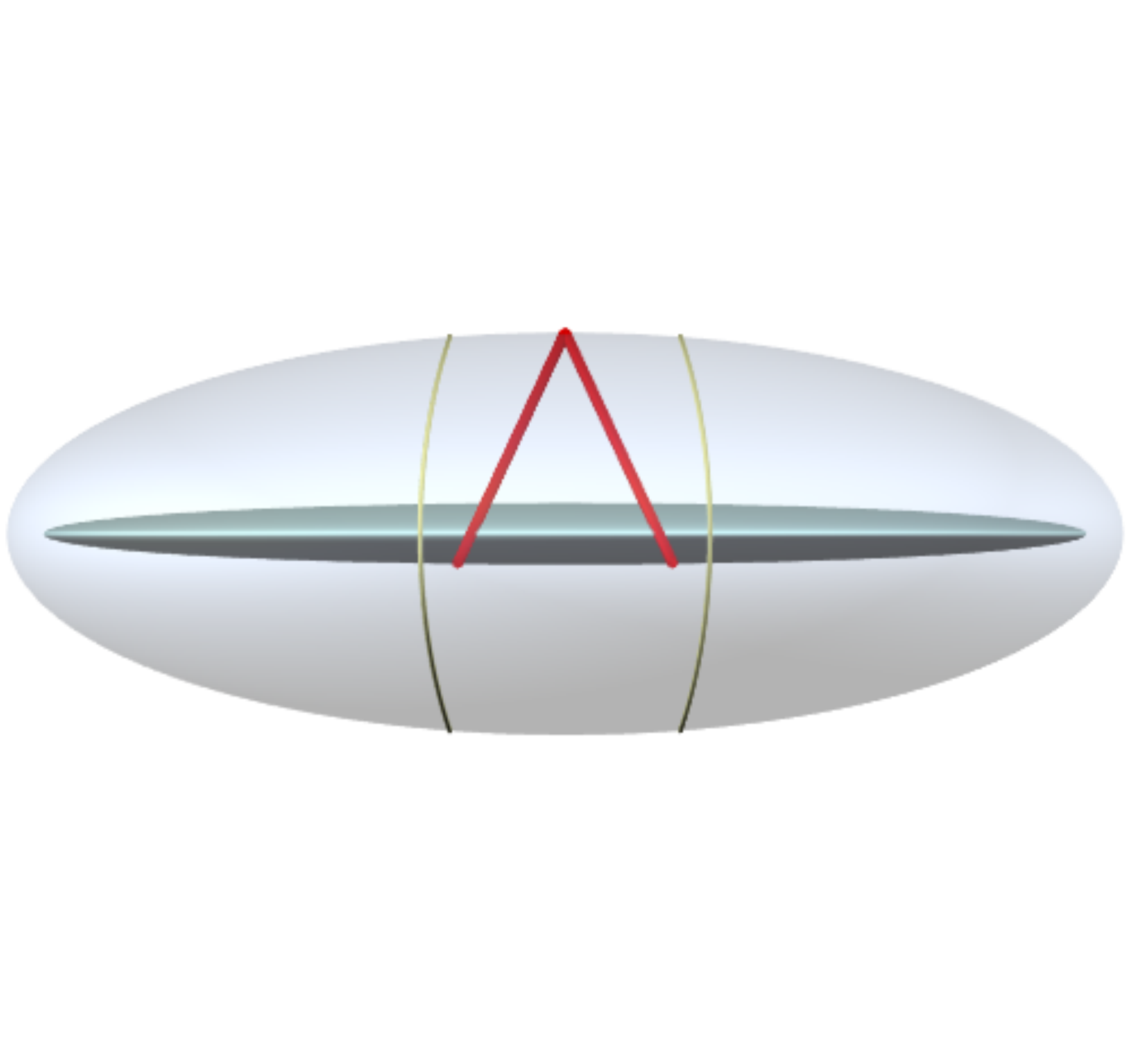}\end{tabular}} &
    \scalebox{0.2}{\begin{tabular}{c}\includegraphics{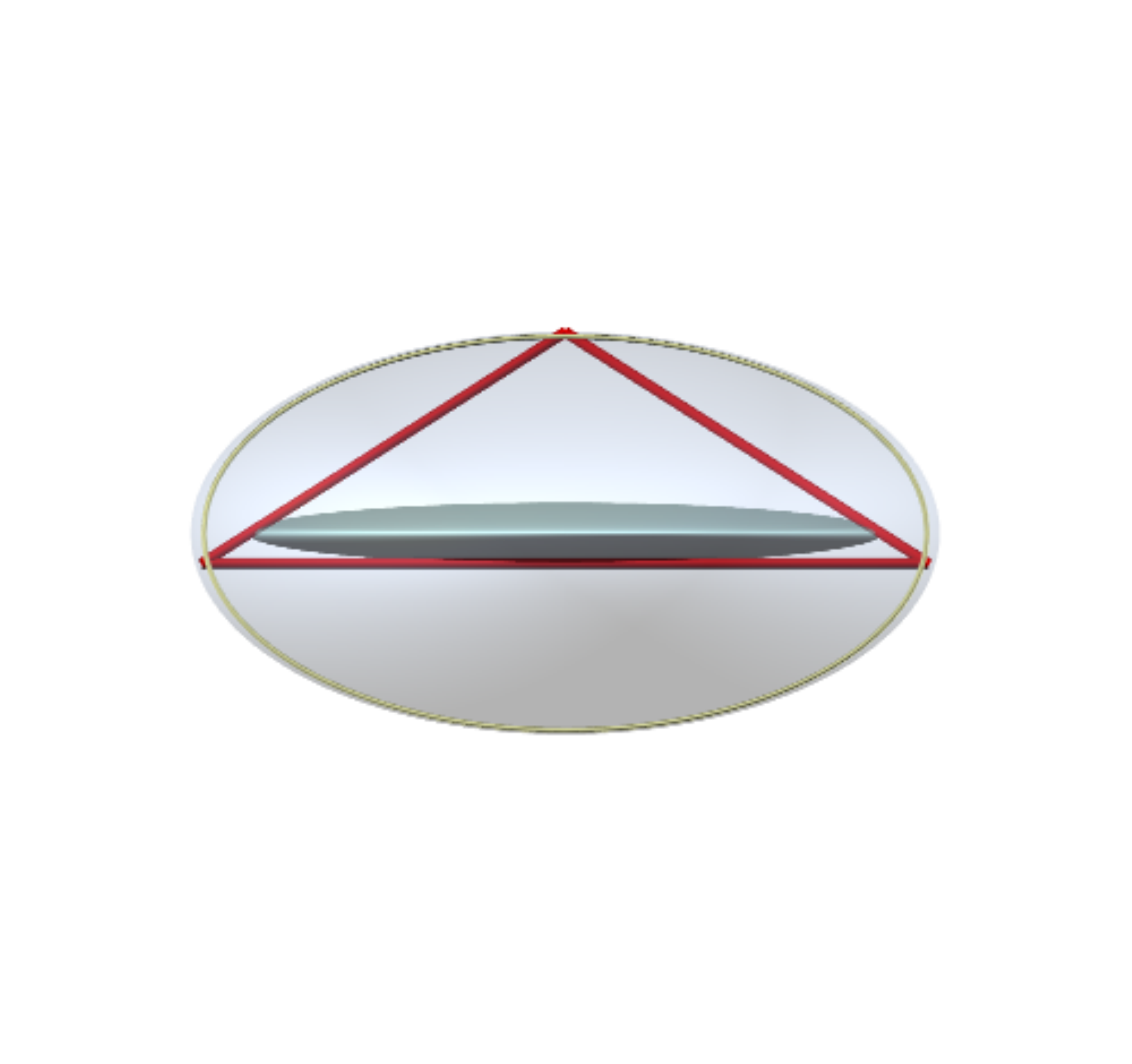}\end{tabular}} &
    \begin{tabular}{c}
      EH2\\
      $\left( 6, 4, 2 \right)$\\
      0.45\\
      0.13\\
      0.962896\\
      0.126968\\
      $R_{2 3}$\\
      $f \circ R_2$
    \end{tabular}\\
    \hline
    \scalebox{0.2}{\begin{tabular}{c}\includegraphics{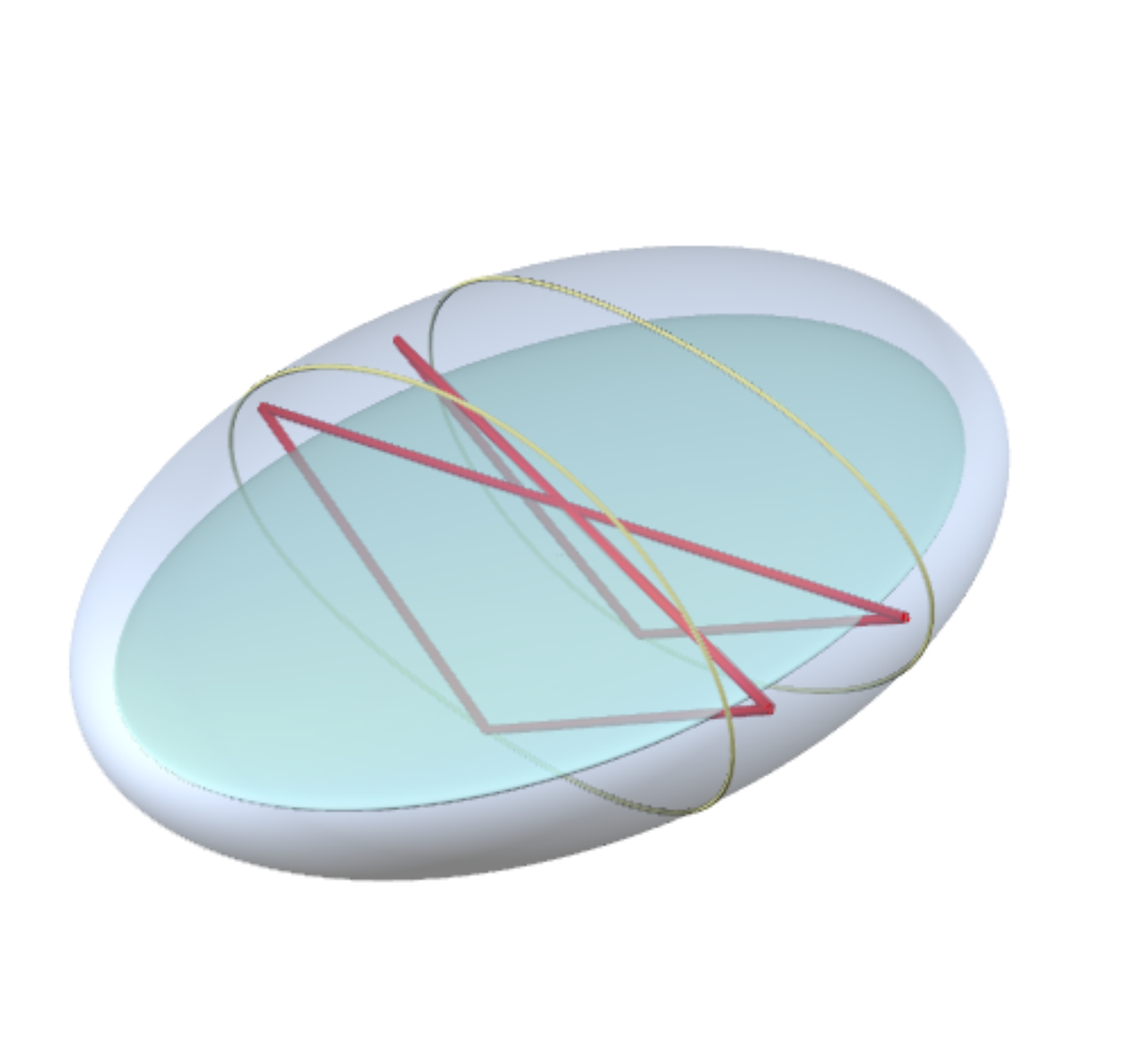}\end{tabular}} &
    \scalebox{0.2}{\begin{tabular}{c}\includegraphics{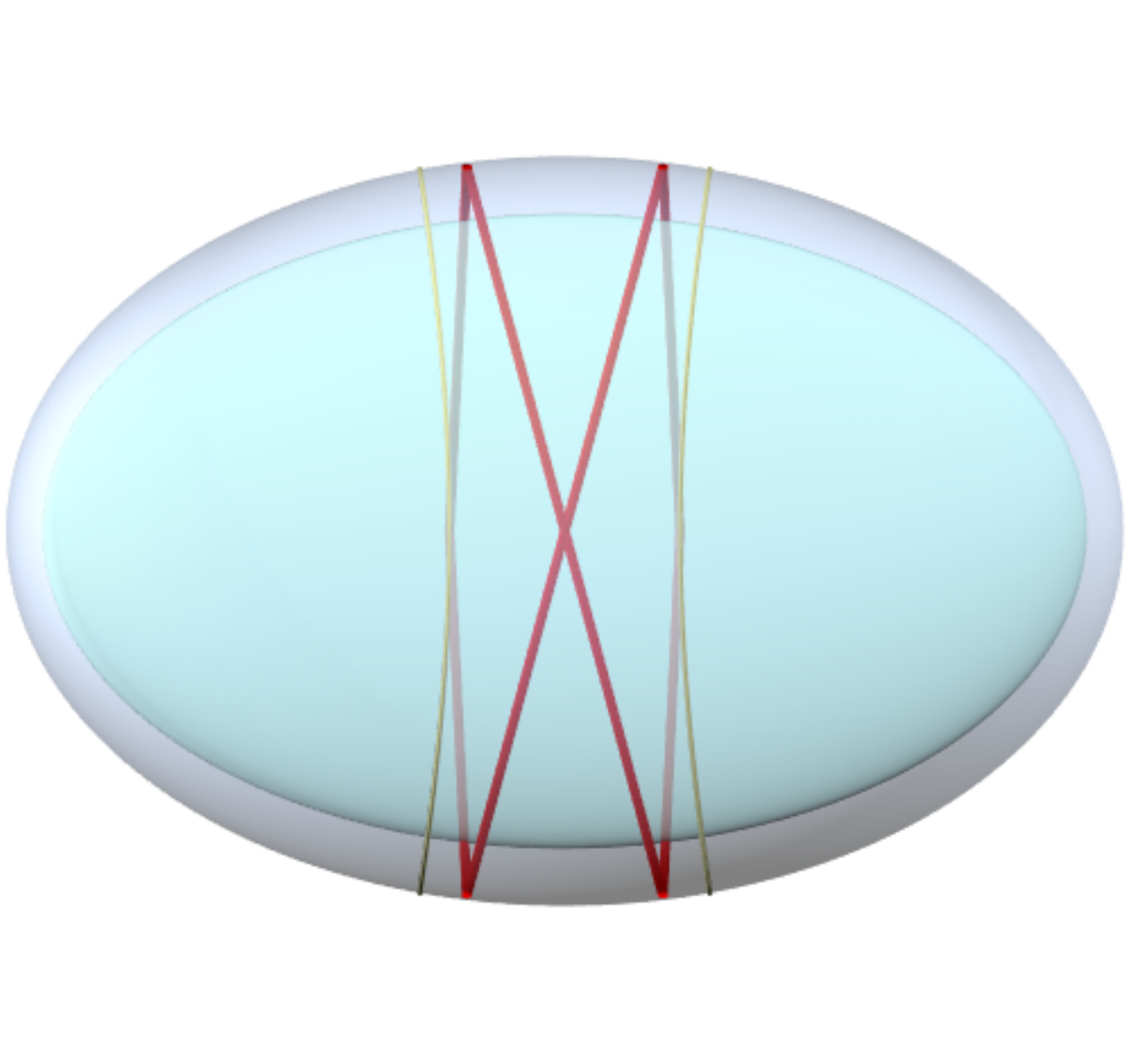}\end{tabular}} &
    \scalebox{0.2}{\begin{tabular}{c}\includegraphics{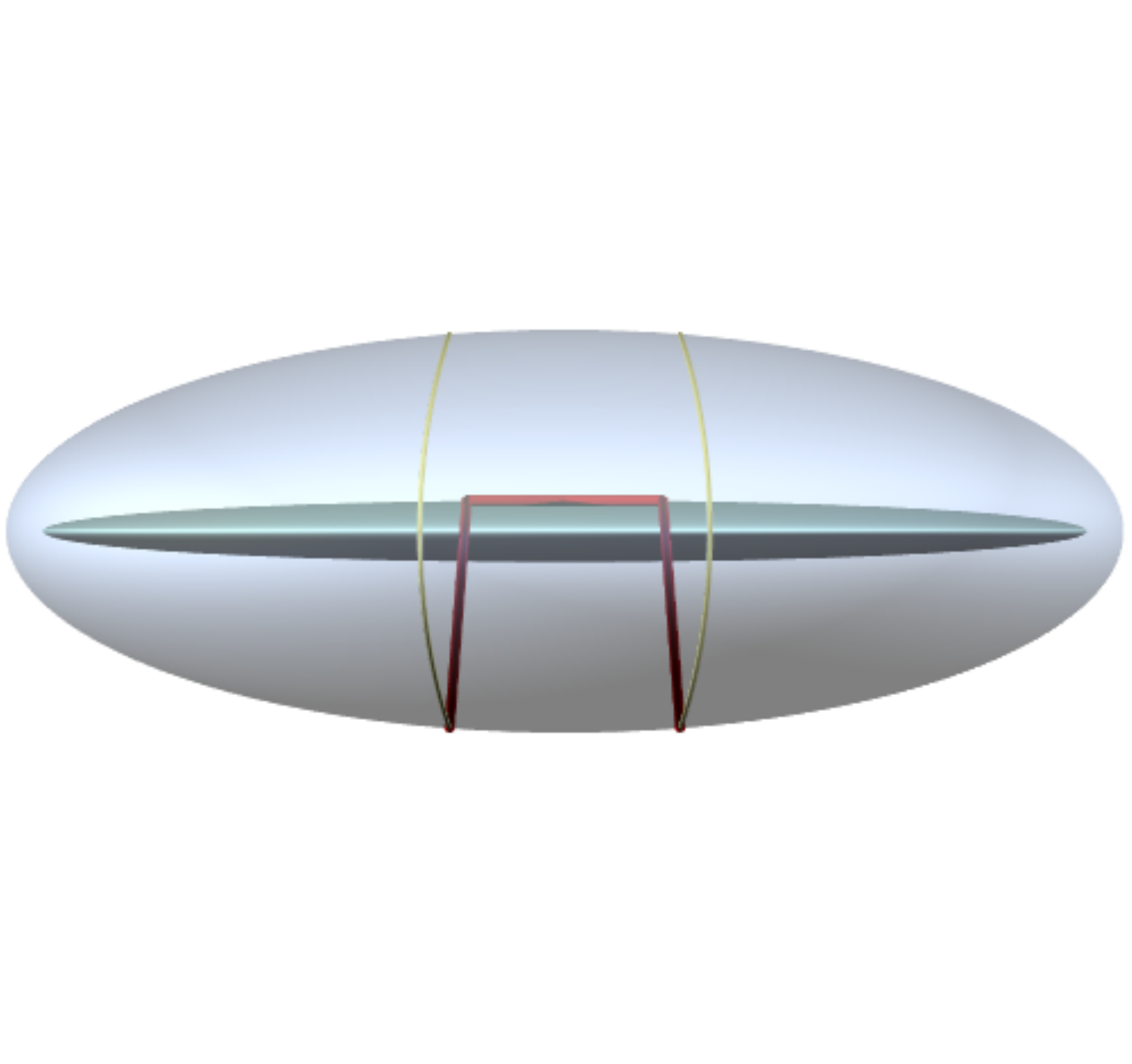}\end{tabular}} &
    \scalebox{0.2}{\begin{tabular}{c}\includegraphics{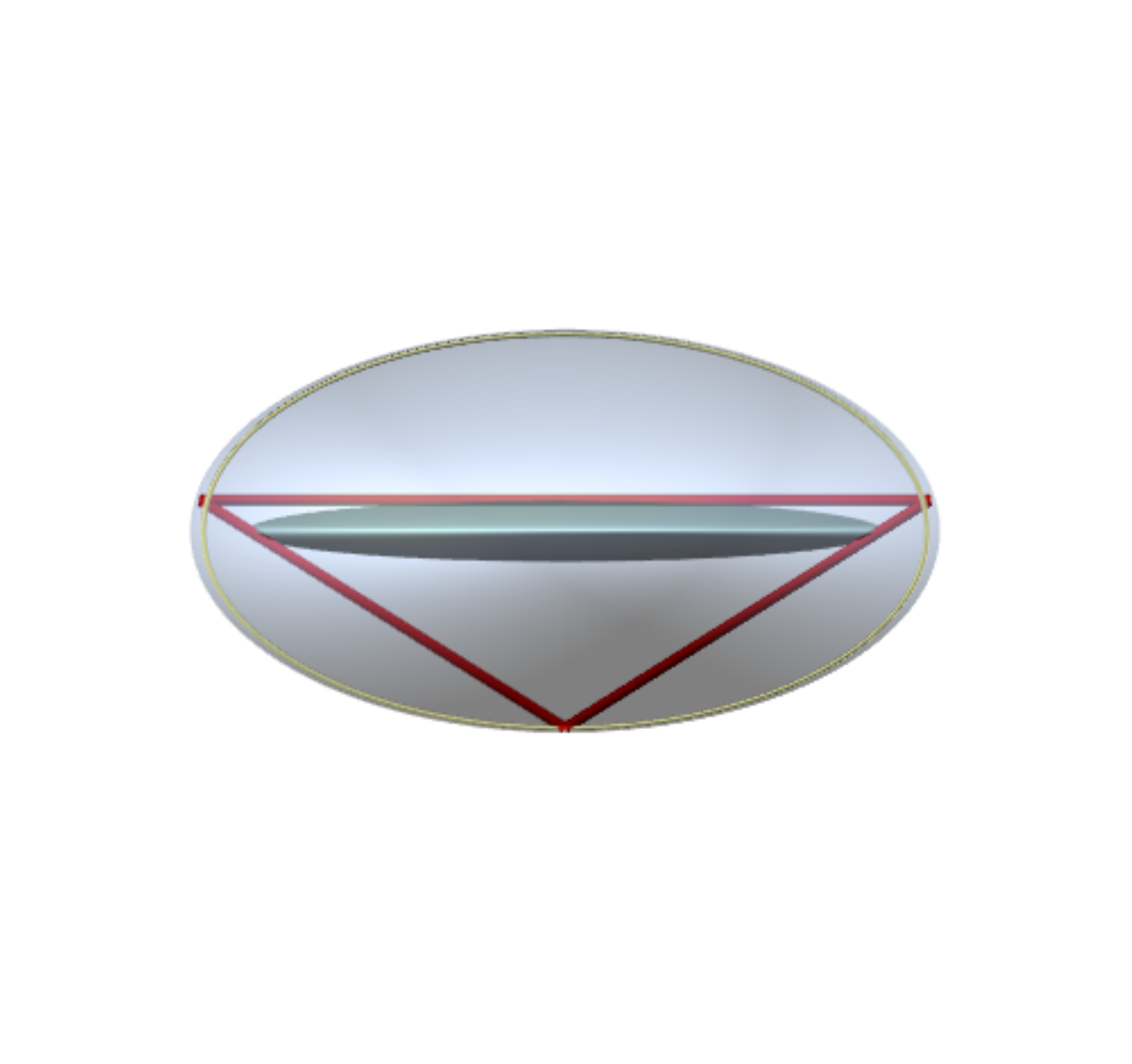}\end{tabular}} &
    \begin{tabular}{c}
      EH2\\
      $\left( 6, 4, 2 \right)$\\
      0.45\\
      0.13\\
      0.962896\\
      0.126968\\
      $R_2$\\
      $f \circ R_{2 3}$
    \end{tabular}\\
    \hline
    \scalebox{0.2}{\begin{tabular}{c}\includegraphics{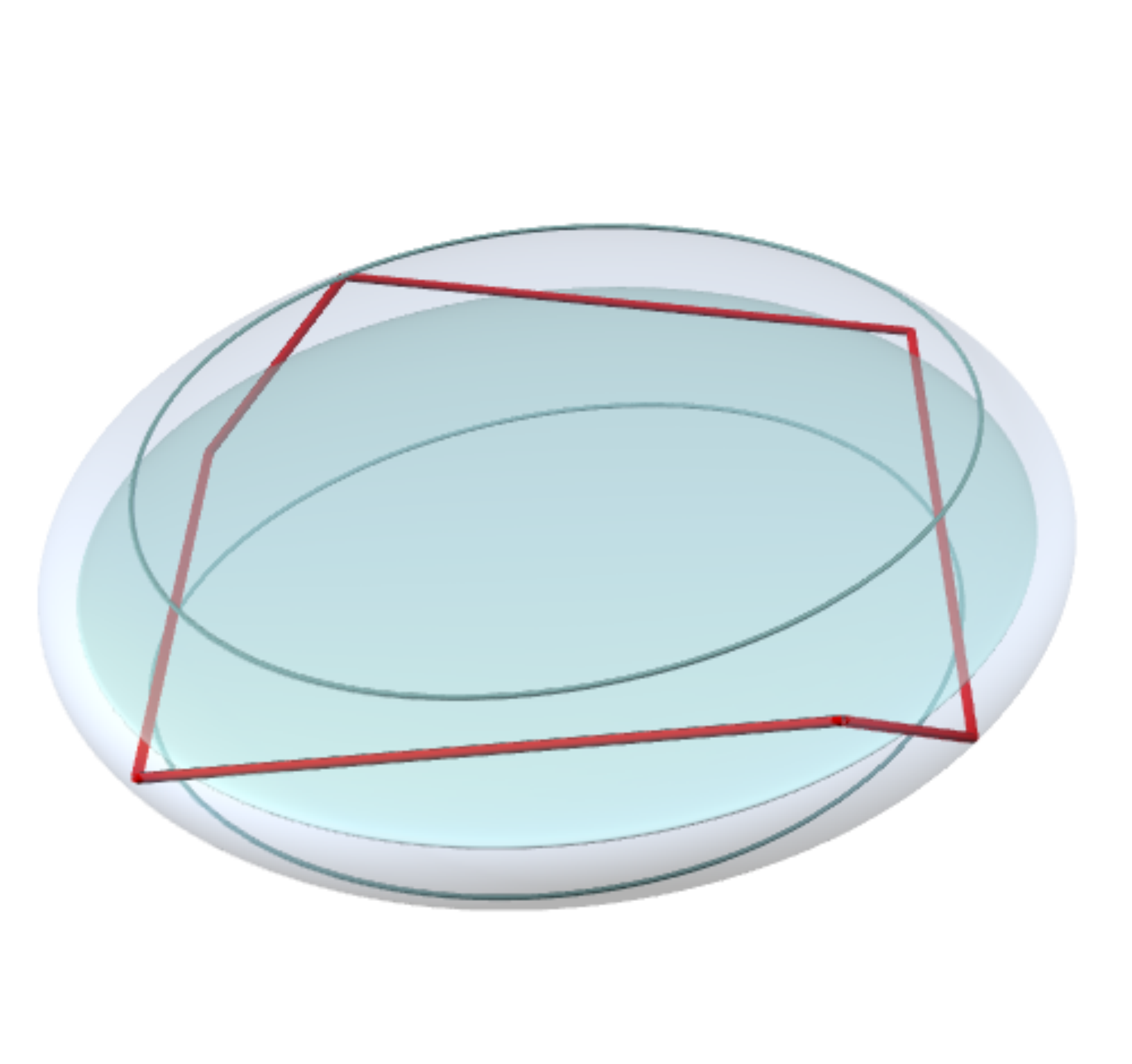}\end{tabular}} &
    \scalebox{0.2}{\begin{tabular}{c}\includegraphics{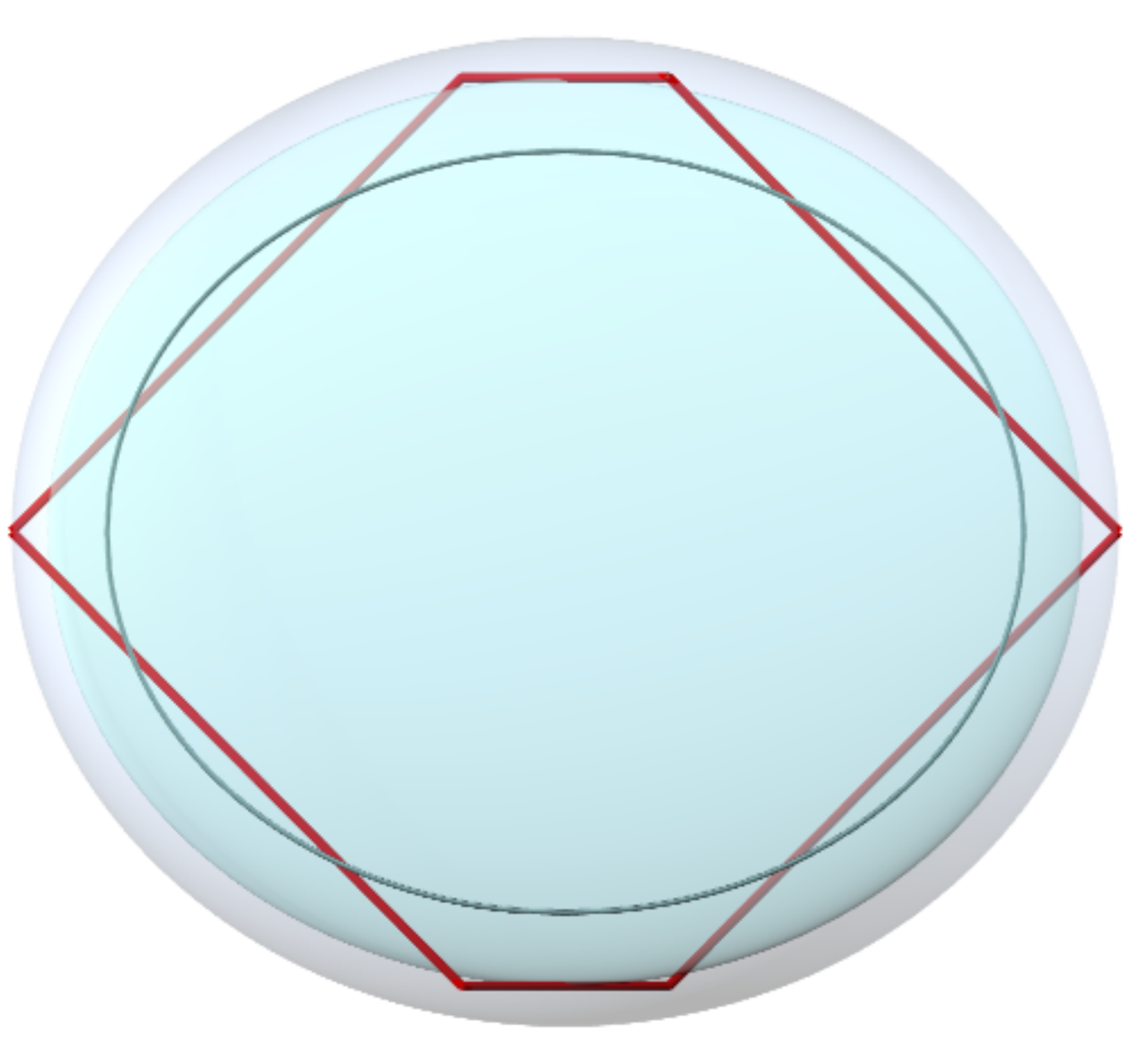}\end{tabular}} &
    \scalebox{0.2}{\begin{tabular}{c}\includegraphics{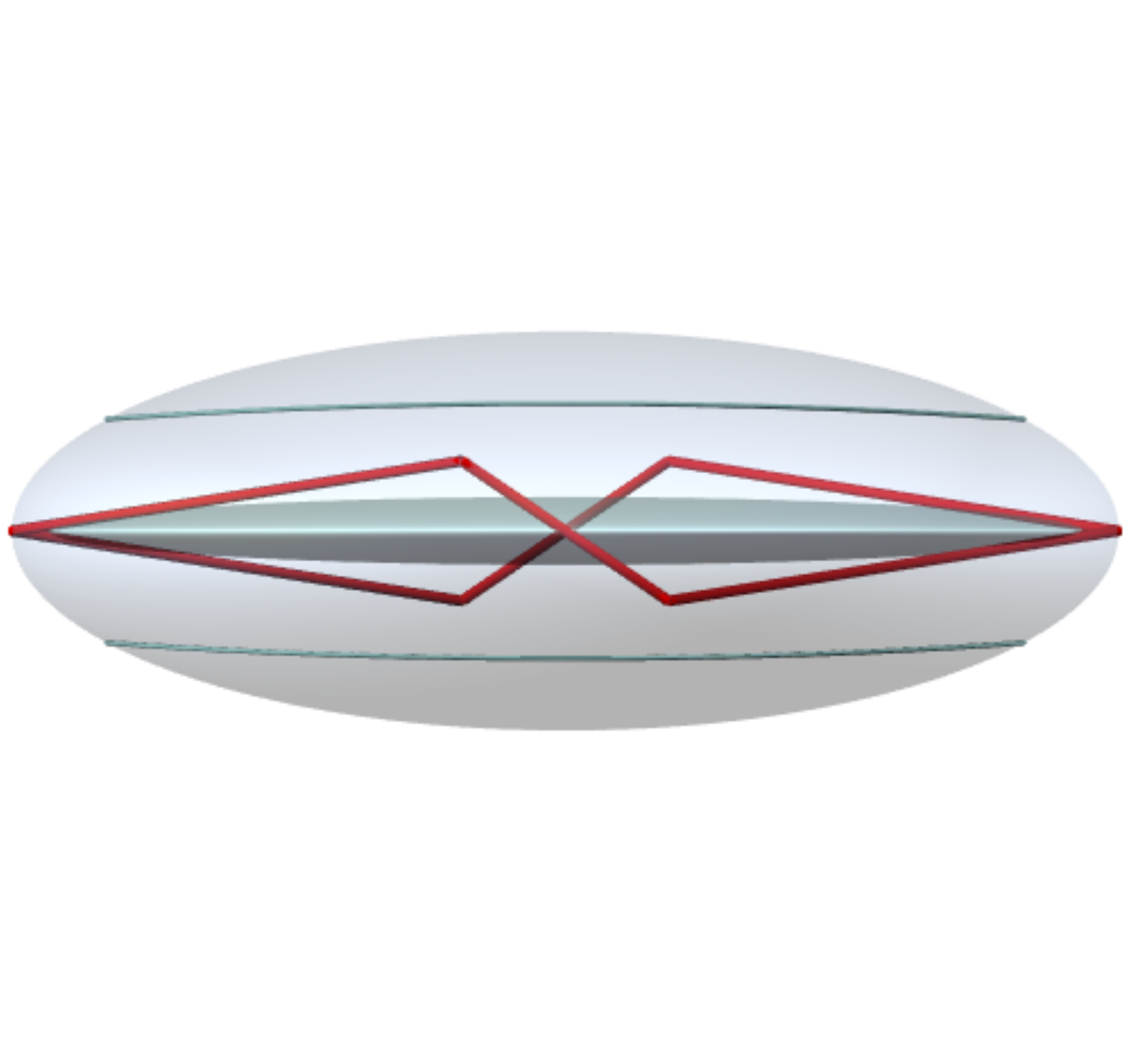}\end{tabular}} &
    \scalebox{0.2}{\begin{tabular}{c}\includegraphics{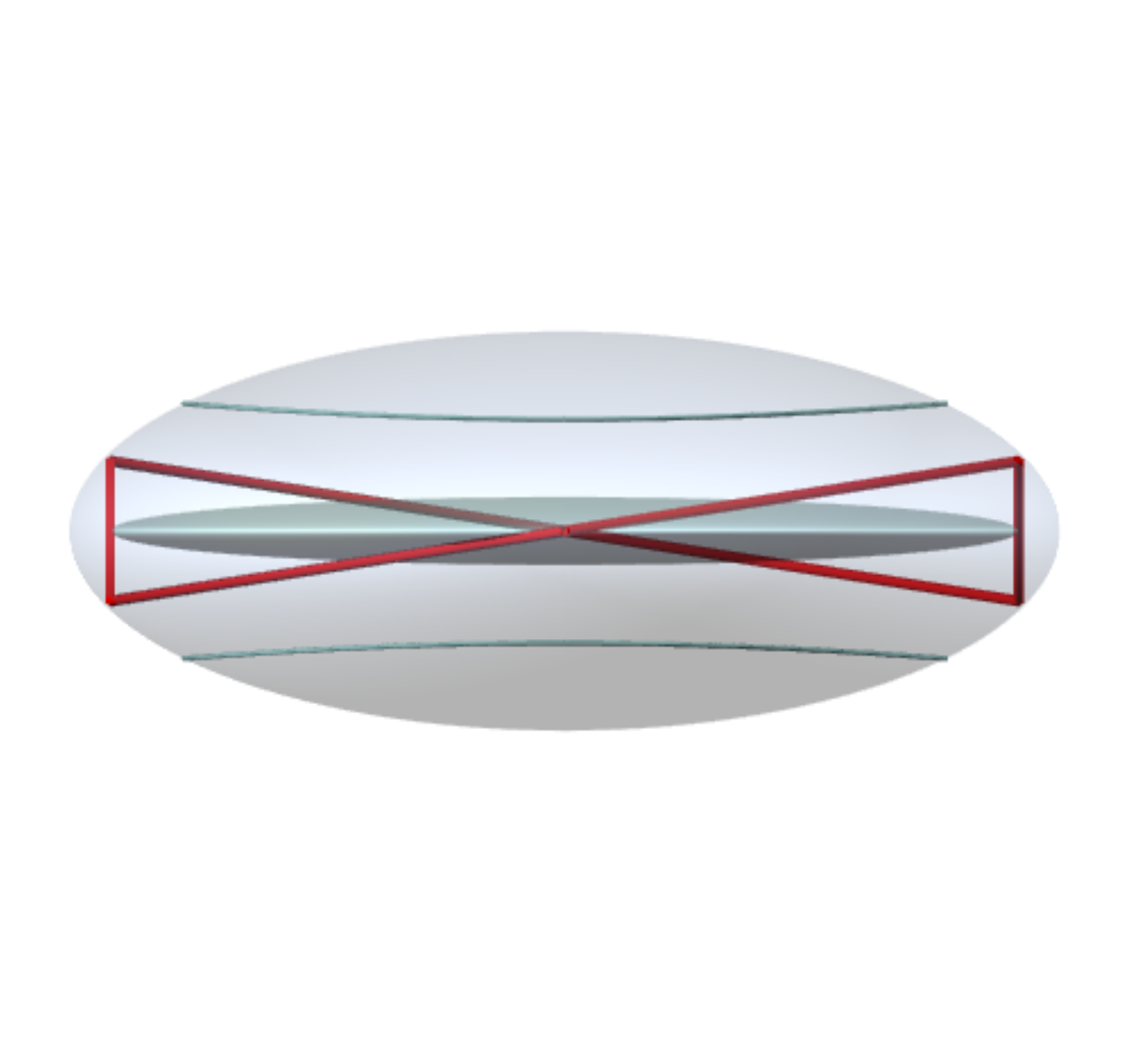}\end{tabular}} &
    \begin{tabular}{c}
      EH1\\
      $\left( 6, 4, 2 \right)$\\
      0.8\\
      0.13\\
      0.403278\\
      0.126231\\
      $R_{1 2}$\\
      $f \circ R_{1 3}$
    \end{tabular}\\
    \hline
    \scalebox{0.2}{\begin{tabular}{c}\includegraphics{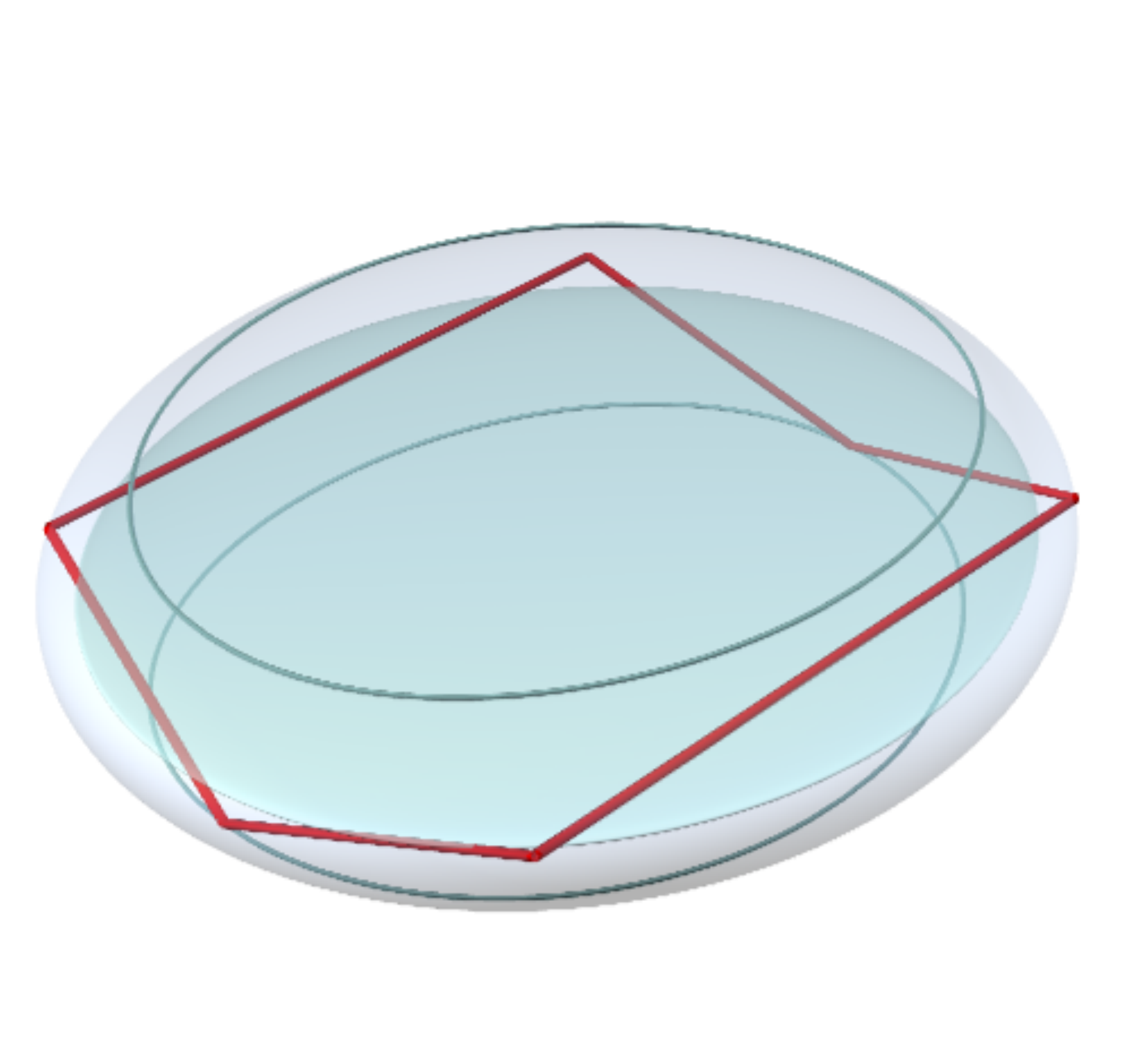}\end{tabular}} &
    \scalebox{0.2}{\begin{tabular}{c}\includegraphics{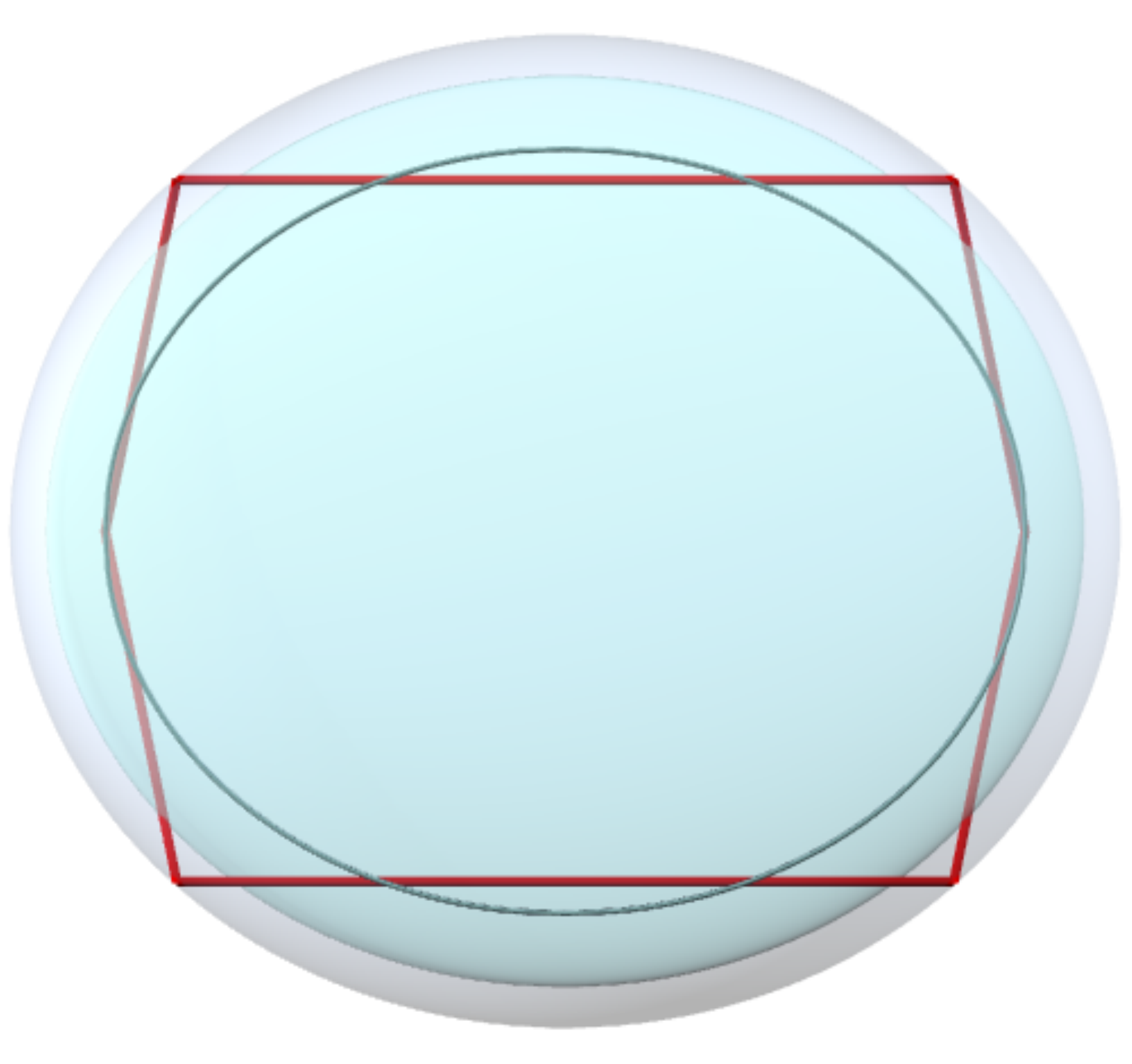}\end{tabular}} &
    \scalebox{0.2}{\begin{tabular}{c}\includegraphics{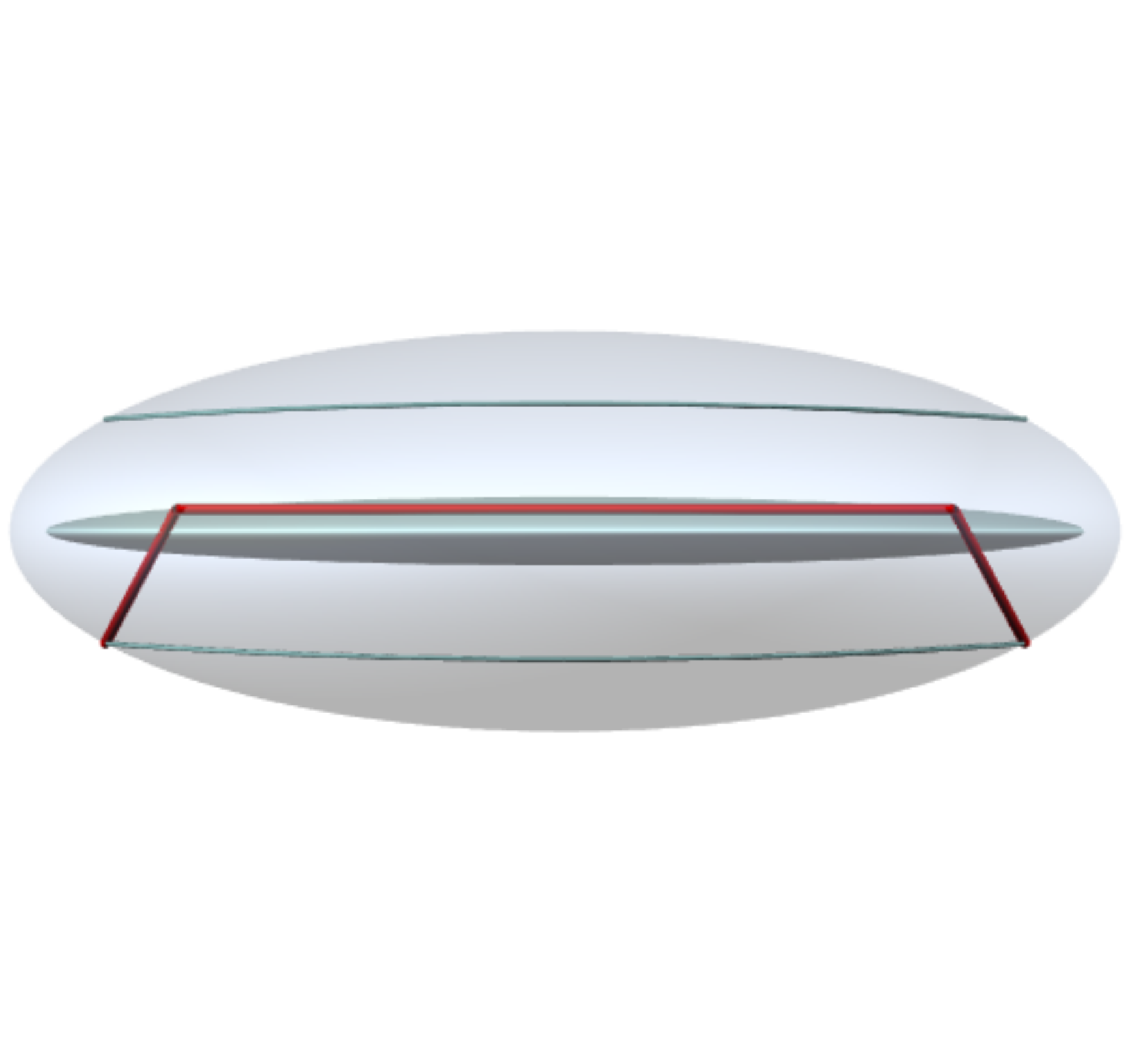}\end{tabular}} &
    \scalebox{0.2}{\begin{tabular}{c}\includegraphics{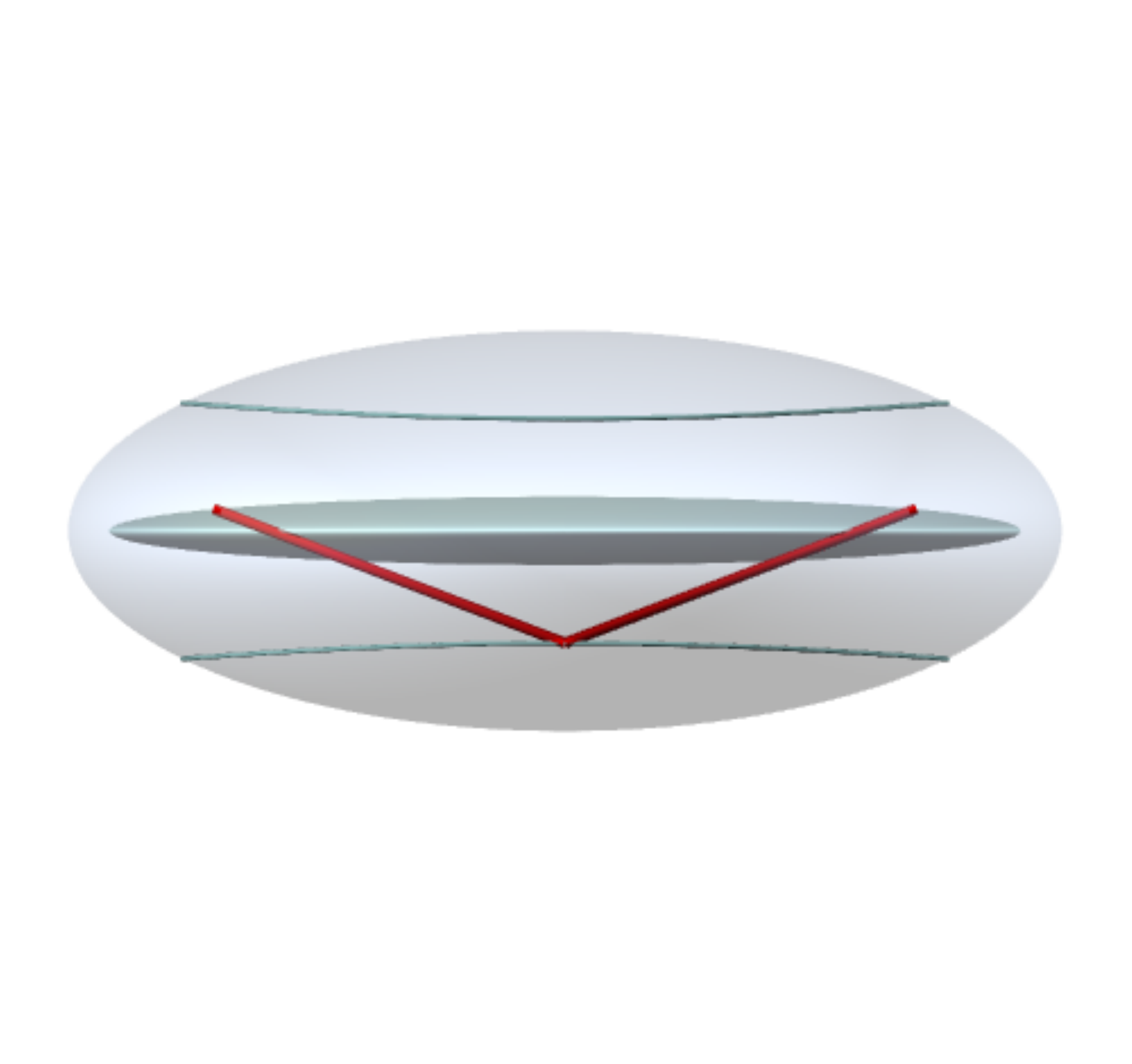}\end{tabular}} &
    \begin{tabular}{c}
      EH1\\
      $\left( 6, 4, 2 \right)$\\
      0.8\\
      0.13\\
      0.403278\\
      0.126231\\
      $R_2$\\
      $f \circ R_3$
    \end{tabular}\\
    \hline
  \end{tabular}
\end{table*}

\section{The general case}\label{sec:nD}

We have characterized STs and SPTs in terms of the vertexes of some rectangles
(for the 2D case) and cuboids (for the 3D case). We have proved that there are
exactly 12 and 112 classes of SPTs in these low-dimensional cases.
See Theorems~\ref{thm:SPTs_2D} and~\ref{thm:SPTs_3D}.

Next, we generalize both characterizations and the final classification to
nondegenerate ellipsoids of $\Rset^{n + 1}$.

We say that two SPTs are of the same class when they have the same type of
caustics and they connect the same couple of vertexes of their respective cuboids.
All SPTs inside the same class are associated to the same reversors.
For instance, looking at Table~\ref{tab:Vertexes_2D},
we see that all SPTs whose type of caustic is E and
that connect the two left (respectively, right) vertexes of the
rectangle~(\ref{eq:Rectangle}) are associated to the reversors $R_x$ and $R_y$
(respectively, $f \circ R_x$ and $f \circ R_y$).

We have the following result.

\begin{theorem}
  Nonsingular STs (respectively, nonsingular SPTs) inside nondegenetate
  ellipsoids of $\Rset^{n + 1}$ are characterized as trajectories
  passing, in elliptic coordinates, through some vertex (respectively, two
  different vertexes) of their cuboid. All nonsingular SPOs are doubly SPOs,
  but a few with caustics of repeated types. There are $2^{2 n} \left( 2^{n +
  1} - 1 \right)$ classes of nonsingular SPTs.
\end{theorem}

This theorem is obtained by means of small refinements of the techniques used
in this article, although some checks become rather cumbersome.

The number $2^{2 n} \left( 2^{n + 1} - 1 \right)$ has a simple explanation.
There are $2^n$ types of caustics
---see Remark~\ref{rem:NumberCausticTypes}---,
and any $(n+1)$-dimensional cuboid has $2^{n+1}$ vertexes,
and so $2^n (2^{n+1} - 1)$ couples of vertexes.

Once fixed the ellipsoid and the caustic type, all $2^n \left( 2^{n + 1} - 1
\right)$ couples of vertexes are \tmtextit{realizable}. That is, each couple
is connected by some SPT. Of course, one should choose suitable winding
numbers for each couple. That choice is guided by the following observation.

Let $\mathcal{C}_{\lambda}$ be a cuboid such that its billiard trajectories
are periodic with winding numbers $m_0, \ldots, m_n$. Let $\hat{\mathfrak{q}}
= \left( \hat{\mu}_0, \ldots, \hat{\mu}_n \right)$ and
$\tilde{\mathfrak{q}}_{} = \left( \tilde{\mu}_0, \ldots, \tilde{\mu}_n
\right)$ be any couple of different vertexes of $\mathcal{C}_{\lambda}$
connected by an SPT.
If some winding number is odd, then $\hat{\mu}_i =
\tilde{\mu}_i \Leftrightarrow m_i \in 2\Zset$. If
all winding numbers are even, then $\hat{\mu}_i = \tilde{\mu}_i
\Leftrightarrow m_i \in 4\Zset$.
Here, $i$ is any integer index such that $0 \le i \le n$.

\section{Conclusions and future work}\label{sec:Conclusions}

We have shown that the billiard map associated to a convex symmetric
hypersurface $Q \subset \Rset^{n + 1}$ is reversible. Even more, it
admits $2^{n + 1}$ factorizations as a composition of two involutions.
Therefore, its SPOs can be classified by the symmetry sets they intersect.
We have carried out this classification for nondegenerate ellipsoids of
$\Rset^{n+1}$. The characterization of STs and SPTs in terms of their
elliptic coordinates has been the key tool.

The existence of SPOs is useful for several reasons. We indicate just two.

First, it reduces the numerical computations needed to find POs. Concretely,
we can halve the dimension of the problem, by restricting the search of
periodic points on the $n$-dimensional symmetry sets. (This issue will be
greatly appreciated when dealing with perturbed ellipsoids, for which no
frequency map exists, since the integrability is lost.) Second, an SPO of a
completely integrable reversible map, persists under symmetric perturbations,
when the symmetry sets are transverse to the Liouville invariant tori of the
map at some point of the SPO.

These two facts are crucial to study billiards inside symmetrically perturbed
ellipsoids, which is our next goal. To be more precise, we plan to generalize the
results about the break-up of resonant invariant curves for billiards inside
perturbed circles~\cite{RamirezRos2006} and perturbed
ellipses~\cite{PintoRamirez2011}.

\begin{acknowledgments}
The first author was supported in part by MCyT-FEDER grant MTM2006-00478 (Spain).
The second author was supported partially by MICINN-FEDER Grant MTM2009-06973
(Spain) and CUR-DIUE Grant 2009SGR859 (Catalonia).
\end{acknowledgments}

\end{document}